\documentclass[11pt, letter]{amsart}
\usepackage{style/package} 
\usepackage{style/command} 
\usepackage{style/bibstyle} 

\begin{document}


\title[Thermodynamic formalism for subsystems and large deviations]
{
    Thermodynamic formalism for subsystems of expanding Thurston maps and large deviations asymptotics
}

\author{Zhiqiang~Li \and Xianghui~Shi \and Yiwei~Zhang}
    
\address{Zhiqiang~Li, School of Mathematical Sciences \& Beijing International Center for Mathematical Research, Peking University, Beijing 100871, China}
\email{zli@math.pku.edu.cn}

\address{Xianghui~Shi, Beijing International Center for Mathematical Research, Peking University, Beijing 100871, China}
\email{xhshi@pku.edu.cn}

\address{Yiwei~Zhang, School of Mathematical Sciences and Big Data, Anhui University of Science and Technology, Huainan, Anhui 232001,China} 
\email{yiweizhang831129@gmail.com}



\begin{abstract}
Expanding Thurston maps were introduced by M.~Bonk and D.~Meyer with motivation from complex dynamics and Cannon's conjecture from geometric group theory via Sullivan's dictionary. In this paper, we introduce subsystems of expanding Thurston maps motivated via Sullivan's dictionary as analogs of some subgroups of Kleinian groups.
We use thermodynamic formalism to prove the Variational Principle and the existence of equilibrium states for strongly irreducible subsystems and real-valued H\"older continuous potentials. 
Here, the sphere $S^{2}$ is equipped with a natural metric, called a visual metric, introduced by M.~Bonk and D.~Meyer.
As an application, we establish large deviation asymptotics for expanding Thurston maps.
\end{abstract} 

\maketitle
\tableofcontents

\section{Introduction}
\label{sec:Introduction}

A Thurston map is a (non-homeomorphic) branched covering map on a topological $2$-sphere $S^{2}$ that is postcritically-finite, meaning that each of its critical points has a finite orbit under iteration.
The most important examples are given by postcritically-finite rational maps on the Riemann sphere $\ccx$.
While Thurston maps are purely topological objects, a deep theorem due to W.P.~Thurston characterizes Thurston maps that are, in a suitable sense, described in the language of topology and combinatorics, equivalent to postcritically-finite rational maps (see \cite{douady1993proof}). 
This suggests that for the relevant rational maps, an explicit analytic expression is not so important but rather a geometric-combinatorial description. 
This viewpoint is both natural and fruitful for considering more general dynamical systems that are not necessarily conformal.

In the early 1980s, D.P.~Sullivan introduced a “dictionary" that is now known as \emph{Sullivan's dictionary}, which connects two branches of conformal dynamics: iterations of rational maps, and actions of Kleinian groups.
Under Sullivan’s dictionary, the counterpart to Thurston’s theorem in geometric group theory is Cannon’s Conjecture \cite{cannon1994combinatorial}. 
An equivalent formulation of Cannon's conjecture, viewed from a quasisymmetric uniformization perspective (\cite[Conjecture~5.2]{bonk2006quasiconformal}), predicts that if the boundary at infinity $\partial_{\infty} G$ of a Gromov hyperbolic group $G$ is homeomorphic to $S^2$, then $\partial_{\infty} G$ equipped with a visual metric is quasisymmetrically equivalent to the Riemann sphere $\ccx$ equipped with the chordal metric. 

Inspired by Sullivan’s dictionary and their interest in Cannon’s conjecture, M.~Bonk and D.~Meyer \cite{bonk2010expanding,bonk2017expanding}, as well as P.~Ha{\"i}ssinsky and K.M.~Pilgrim \cite{haissinsky2009coarse}, studied a subclass of Thurston maps, called \emph{expanding Thurston maps}, by imposing some additional condition of expansion.
These maps are characterized by a contraction property for inverse images (see Subsection~\ref{sub:Thurston_maps} for the precise definition). 
In particular, a postcritically-finite rational map on $\ccx$ is expanding if and only if its Julia set is equal to $\ccx$.
For an expanding Thurston map on $S^{2}$, we can equip $S^2$ with a natural class of metrics $d$, called \emph{visual metrics} (see Subsection~\ref{sub:Thurston_maps} for details), that are snowflake equivalent to each other and are constructed in a similar way as the visual metrics on the boundary $\partial_{\infty} G$ of a Gromov hyperbolic group $G$ (see \cite[Chapter~8]{bonk2017expanding} for details, and see \cite{haissinsky2009coarse} for a related construction). 
In the language above, the following theorem was obtained in \cite{bonk2010expanding,bonk2017expanding,haissinsky2009coarse}, which can be seen as an analog of Cannon’s conjecture for expanding Thurston maps.

\begin{theorem*}[M.~Bonk \& D.~Meyer \cite{bonk2010expanding,bonk2017expanding}; P.~Ha{\"i}ssinky \& K.M.~Pilgrim \cite{haissinsky2009coarse}]
    Let $f \colon S^2 \mapping S^2$ be an expanding Thurston map with no periodic critical points and $d$ be a visual metric for $f$. 
    Then $f$ is topologically conjugate to a rational map if and only if $(S^2, d)$ is quasisymmetrically equivalent to $\ccx$. 
\end{theorem*}

The dynamical systems that we study in this paper are called \emph{subsystems} of expanding Thurston maps, inspired by a translation of the notion of subgroups from geometric group theory via Sullivan’s dictionary.
To clarify this concept, we consider an expanding Thurston map $f \colon S^{2} \mapping S^{2}$ and a Jordan curve $\mathcal{C} \subseteq S^{2}$ that contains the postcritical set $\post{f}$.
The condition $\post{f} \subseteq \mathcal{C}$ ensures that the closure of each connected component of $S^{2} \setminus f^{-n}(\mathcal{C})$ is a closed Jordan region.
We call each such set an \emph{$n$-tile}.
Consider some $1$-tiles and denote their union by $U$.
The restriction $F \define f|_{U}$ is called a \emph{subsystem of $f$ with respect to $\mathcal{C}$}, and the dynamics of $F \colon U \mapping S^2$ generates the \emph{tile maximal invariant set} $\limitset \subseteq S^{2}$, which is the intersection of unions of $n$-tiles contained in $F^{-n}(S^2)$ for $n \in \n$ (see Subsection~\ref{sub:Definition of subsystems} for a detailed discussion).

For expanding Thurston maps, roughly speaking, $1$-tiles together with the maps restricted to those tiles play a role similar to that of generators in the context of Gromov hyperbolic groups. 
For example, one can recover $S^2$ and the original map $f$ from all its $1$-tiles and the dynamics on those tiles.
If we consider all $n$-tiles for some $n \in \n$, we obtain an iterate $f^{n}$ of $f$, which corresponds to a finite-index subgroup of the original group in the group setting.
Given such similarity, it is natural to investigate more general cases, such as dynamics generated by certain $1$-tiles, which leads to our study of subsystems. 
We remark that although the concept of a subsystem shares certain similarities with the notion of a repeller (see \cite{pesin1997dimension,przytycki2010conformal}), the latter typically requires smooth and uniformly expanding assumptions, neither of which are satisfied by a subsystem.

According to Sullivan's dictionary, an expanding Thurston map is associated with a Gromov hyperbolic group whose boundary at infinity is $S^2$.
In this context, a subsystem corresponds to a Gromov hyperbolic group whose boundary at infinity is a subset of $S^{2}$.
In particular, for Gromov hyperbolic groups whose boundary at infinity is a \sierpinski carpet, there is an analog of Cannon's conjecture---the Kapovich--Kleiner conjecture. 
It predicts that these groups arise from some standard situation in hyperbolic geometry.
Similar to Cannon's conjecture, one can reformulate the Kapovich--Kleiner conjecture in an equivalent way as a question related to quasisymmetric uniformization.
For subsystems, it is easy to find examples where the tile maximal invariant set is homeomorphic to the standard \sierpinski carpet (see Subsection~\ref{sub:Definition of subsystems} for examples of subsystems).
In this case, an analog of the Kapovich--Kleiner conjecture for subsystems is under investigation \cite{bonk2024dynamical}. 

In this paper, we delve into the dynamics of subsystems of expanding Thurston maps from the perspective of ergodic theory.
Ergodic theory has played a crucial role in the study of dynamical systems.
The investigation of invariant measures has been a central part of ergodic theory. 
However, a dynamical system may possess a large class of invariant measures, some of which may be more interesting than others. 
It is, therefore, crucial to examine the relevant invariant measures.

The \emph{thermodynamic formalism} serves as a viable mechanism for generating invariant measures endowed with desirable properties.
More precisely, for a continuous transformation on a compact metric space, we can consider the \emph{topological pressure} as a weighted version of the \emph{topological entropy}, with the weight induced by a real-valued continuous function, called \emph{potential}. 
The Variational Principle identifies the topological pressure with the supremum of its measure-theoretic counterpart, the \emph{measure-theoretic pressure}, over all invariant Borel probability measures \cite{bowen1975equilibrium, walters1982introduction}. Under additional regularity assumptions on the transformation and the potential, one gets the existence and uniqueness of an invariant Borel probability measure maximizing the measure-theoretic pressure, called the \emph{equilibrium state} for the given transformation and the potential. 
The study of the existence and uniqueness of the equilibrium states and their various other properties, such as ergodic properties, equidistribution, fractal dimensions, etc., has been the primary motivation for much research in the area.

The ergodic theory for expanding Thurston maps has been investigated in \cite{li2017ergodic} by the first-named author of the current paper.
In \cite{li2018equilibrium}, the first-named author developed the thermodynamic formalism and investigated the existence, uniqueness, and other properties of equilibrium states for expanding Thurston maps.
In particular, for expanding Thurston maps without periodic critical points, using a general framework devised by Y.~Kifer \cite{kifer1990large}, the first-named author established level-$2$ large deviation principles for iterated preimages and periodic points in \cite{li2015weak}.

The current paper is the first in a series of two papers (along with \cite{shi2024uniqueness}) investigating the ergodic theory of subsystems of expanding Thurston maps.
In this paper, we develop the thermodynamic formalism for strongly irreducible subsystems of expanding Thurston maps. 
For these subsystems, we establish the Variational Principle and the existence of equilibrium states. 
Consequently, we derive large deviation asymptotics for (original) expanding Thurston maps.

In the second paper \cite{shi2024uniqueness} in this series, we establish the uniqueness and ergodic properties of equilibrium states for strongly primitive subsystems of expanding Thurston maps.
Building on the results regarding existence and uniqueness of equilibrium states, for strongly primitive subsystems of expanding Thurston maps without periodic critical points, we obtain level-$2$ large deviation principles for the distributions of Birkhoff averages and iterated preimages.

Our investigation of subsystems in this series also has applications in the ergodic properties and large deviation theories of expanding Thurston maps.
In \cite{shi2024entropy}, for any expanding Thurston map, even in the presence of periodic critical points, we prove, by using subsystems, the entropy density of ergodic measures and establish level-$2$ large deviation principles for the distributions of Birkhoff averages, periodic points, and iterated preimages, which generalizes the corresponding results in \cite{li2015weak}. 
Additionally, by constructing suitable subsystems, we show that the entropy map of an expanding Thurston map $f$ is upper semi-continuous if and only if $f$ has no periodic critical points. 
This finding provides a negative answer to the question posed in \cite{li2015weak} and indicates that the method used there to prove large deviation principles is not applicable to expanding Thurston maps with periodic critical points.
In this context, subsystems serve as practical tools for studying expanding Thurston maps.
We expect more applications of subsystems in the future.

\subsection{Main results}%
\label{sub:Main results}

Our main results consist of two parts.
We first establish the Variational Principle and the existence of equilibrium states for subsystems of expanding Thurston maps.
Then as an application, we obtain large deviation asymptotics for expanding Thurston maps.
In particular, these results apply to postcritically-finite rational maps without periodic critical points.

\subsubsection*{Variational Principle and the existence of equilibrium states}%
\label{ssub:Variational Principle and the existence of equilibrium states}

In order to present our results more precisely, we briefly review some key concepts.
We refer the reader to Section~\ref{sec:Preliminaries} for a detailed discussion.

Let $f \colon S^2 \mapping S^2$ be an expanding Thurston map with a Jordan curve $\mathcal{C}\subseteq S^2$ satisfying $\post{f} \subseteq \mathcal{C}$. 
Here $\post{f} \define \bigcup_{n \in \n} \set{f^{n}(c) \describe c \in S^2 \text{ is a critical point of } f}$.
For each $n \in \n_0$, the set of $n$-tiles is 
\[
    \Tile{n} \define \{ X^n \describe X^n \text{ is the closure of a connected component of } S^2 \setminus f^{-n}(\mathcal{C}) \}.    
\]

We say that a map $F \colon \domF \mapping S^2$ is a \emph{subsystem of $f$ with respect to $\mathcal{C}$} if $\domF = \bigcup \mathfrak{X}$ for some non-empty subset $\mathfrak{X} \subseteq \Tile{1}$ and $F = f|_{\domF}$.
We denote by $\subsystem$ the set of subsystems of $f$ with respect to $\mathcal{C}$.

Consider a subsystem $F \in \subsystem$. For each $n \in \n_0$, we define the \emph{set of $n$-tiles of $F$} to be 
\[
    \Domain{n} \define \{ X^n \in \Tile{n} \describe X^n \subseteq F^{-n}(F(\domF)) \},
\]
where we set $F^0 \define \id{S^{2}}$ when $n = 0$. We call each $X^n \in \Domain{n}$ an \emph{$n$-tile} of $F$. 
We define the \emph{tile maximal invariant set} associated with $F$ with respect to $\mathcal{C}$ to be
\[
    \limitset = \limitset(F, \mathcal{C}) \define \bigcap_{n \in \n} \parentheses[\Big]{ \bigcup \Domain{n} }.
\]

We note that $\limitset \subseteq S^2$ is compact and is forward invariant under $F$, i.e., $F(\limitset) \subseteq \limitset$ (see Proposition~\ref{prop:subsystem:properties}~\ref{item:subsystem:properties:limitset forward invariant}).
Hence, we can consider the restriction $F|_{\limitset} \colon \limitset \to \limitset$ and its iterates.
In this context, we denote by $\mathcal{M}(\limitset, F|_{\limitset})$ the set of $F|_{\limitset}$-invariant Borel probability measures on $\limitset$, and by $h_{\mu}(F|_{\limitset})$ the measure-theoretic entropy of $F|_{\limitset}$ for $\mu \in \mathcal{M}(\limitset, F|_{\limitset})$ (see Subsection~\ref{sub:thermodynamic formalism} for precise definitions).

We define the \emph{topological pressure of a subsystem} $F$ with respect to a \emph{potential} $\varphi \in C(S^2)$ by
\[
    \pressure[\varphi] \define \liminf_{n \mapping +\infty} \frac{1}{n} \log \parentheses{ Z_{n}(F, \varphi) },
\]
where \[
    Z_{n}(F, \varphi) \define \sum_{X^n \in \Domain{n}} \myexp[\bigg]{ \sup \set[\bigg]{ \sum_{i = 0}^{n - 1} \varphi(F^{i}(x)) \describe x \in X^n } }.
\]
Observe the difference between $P(F, \varphi)$ and the classical notion of topological pressure $P(F|_{\limitset}, \varphi|_{\limitset})$ for the map $F|_{\limitset} \colon \limitset \mapping \limitset$ (see the remark after Definition~\ref{def:pressure for subsystem} for details).

In the following theorem, under the additional assumption that the Jordan curve $\mathcal{C}$ is forward invariant, we establish the Variational Principle and the existence of the equilibrium state for a strongly irreducible subsystem (see Definition~\ref{def:irreducibility of subsystem}) and a \holder continuous potential with respect to a visual metric.

\begin{theorem}        \label{thm:main:thermodynamic formalism for subsystems}
    Let $f \colon S^2 \mapping S^2$ be an expanding Thurston map and $\mathcal{C} \subseteq S^2$ be a Jordan curve containing $\post{f}$ with the property that $f(\mathcal{C}) \subseteq \mathcal{C}$. 
    Let $d$ be a visual metric on $S^2$ for $f$ and $\phi$ be a real-valued \holder continuous function on $S^2$ with respect to the metric $d$.
    Consider a strongly irreducible subsystem $F \in \subsystem$ and denote its tile maximal invariant set by $\limitset$.
    Then
    \begin{equation}    \label{eq:thm:main:thermodynamic formalism for subsystems:subsystem Variational Principle}
        \pressure = \sup \biggl\{ h_{\mu}(F|_\limitset) + \int_{\limitset} \! \phi \,\mathrm{d}\mu \describe \mu \in \mathcal{M}(\limitset, F|_{\limitset}) \biggr\}.
    \end{equation}
    Moreover, there exists an equilibrium state $\mu_{F, \potential}$ for $F|_{\limitset}$ and $\phi|_{\limitset}$ attaining the supremum in \eqref{eq:thm:main:thermodynamic formalism for subsystems:subsystem Variational Principle}.
\end{theorem}

\begin{rmk}
    When the subsystem $F$ is strongly primitive (see Definition~\ref{def:primitivity of subsystem}), we prove the uniqueness and various ergodic properties of the equilibrium state $\mu_{F, \potential}$, including mixing. 
    We also show that the preimages are equidistributed with respect to $\mu_{F, \potential}$.
    As these proofs employ new techniques beyond the scope of this paper, we present these results in the companion paper \cite{shi2024uniqueness} to enhance accessibility of the current paper for a broader audience who are, for independent reasons, potentially interested in the basic notions of subsystems introduced in this paper.
\end{rmk}

\subsubsection*{Large deviation asymptotics}%
\label{ssub:Large deviation asymptotics}

Our second result is about the large deviations asymptotics for expanding Thurston maps. 
Applying Theorem~\ref{thm:main:thermodynamic formalism for subsystems}, we provide exponential upper bounds for deviations. 

Large deviation theory first arose in probability and usually studies the asymptotic behavior of the probability $\mathbb{P} \bigl\{ \frac{1}{n} \sum_{k = 1}^{n} X_{n} \in I \bigr\}$ as $n \to +\infty$ for a sequence $\sequen{X_{n}}$ of real-valued random valuables and an interval $I \subseteq \real$. 
If a law of large numbers holds, such that $\frac{1}{n} \sum_{k = 1}^{n} X_{k} \to \overline{X}$ as $n \to +\infty$ and $\overline{X} \notin I$, then the above probabilities tend to zero. 
In some cases, it is possible to demonstrate that the convergence is exponentially fast, which brings large deviation theory into play, with its primary goal being to describe the corresponding exponent.  

In order to state the result more precisely, we briefly review some key concepts.

For an expanding Thurston map $f \colon S^2 \mapping S^2$ and a \holder continuous function $\phi \colon S^2 \mapping \real$ (with respect to a given visual metric for $f$ on $S^2$), there exists a unique equilibrium state $\mu_{\phi}$ for $f$ and $\phi$ (see Theorem~\ref{thm:properties of equilibrium state}~\ref{item:thm:properties of equilibrium state:existence and uniqueness}). 
Two real-valued functions $\juxtapose{\varphi}{\psi} \in C(S^2)$ are called co-homologous in $C(S^2)$ (with respect to $f$) if there exists a function $u \in C(S^2)$ such that $\varphi - \psi = u\circ f - u$. 
In the following, we assume that $\phi \colon S^2 \mapping \real$ is \holder continuous (with respect to a given visual metric for $f$ on $S^2$) and is not co-homologous to a constant in $C(S^2)$. 

We write $\mathcal{I}_{\phi} \define [\minenergy, \maxenergy]$, where 
\begin{equation}    \label{eq:def:maxenergy and minenergy}
    \minenergy \define \min_{\mu\in \mathcal{M}(S^2,f)} \int \! \phi \,\mathrm{d}\mu
    \quad \text{ and }\quad 
    \maxenergy \define \max_{\mu\in \mathcal{M}(S^2,f)} \int \! \phi \,\mathrm{d}\mu. 
\end{equation}
Here $\mathcal{M}(S^2,f)$ is the set of $f$-invariant Borel probability measures on $S^2$, which is convex and compact with respect to the weak* topology. 
In particular, we have $\interior{\mathcal{I}_{\phi}} \ne \emptyset$ and $\int \phi \,\mathrm{d}\mu_{\phi} \in (\minenergy, \maxenergy)$ (see Proposition~\ref{prop:properties of rate function}~\ref{item:prop:properties of rate function:non-degenerate}).

We define a \emph{rate function} $I = I_{f, \phi} \colon [\minenergy,\maxenergy] \mapping [0, +\infty)$ for $f$ and $\phi$ by
\begin{equation}    \label{eq:definition of rate function}
    I(\alpha) = I_{f, \phi}(\alpha) \define \inf \biggl\{P(f,\phi) - P_{\mu}(f,\phi)  \describe \int \! \phi \,\mathrm{d}\mu = \alpha,\ \mu\in \mathcal{M}(S^2,f) \biggr\},
\end{equation}
where $P(f, \phi)$ is the topological pressure of $f$ with respect to $\phi$ and $P_{\mu}(f, \phi)$ is the measure-theoretic pressure of $f$ for $\mu$ and $\phi$ (see Subsection~\ref{sub:thermodynamic formalism} for definitions).

Now we are able to state our second result.

\begin{theorem}[Large deviations asymptotics]    \label{thm:large deviation asymptotics}
    Let $f \colon S^2 \mapping S^2$ be an expanding Thurston map and $d$ be a visual metric on $S^2$ for $f$. 
    Let $\phi$ be a real-valued \holder continuous function on $S^2$ (with respect to the metric $d$) that is not co-homologous to a constant in $C(S^2)$. 
    Let $\mu_{\phi}$ be the unique equilibrium state for the map $f$ and the potential $\phi$. 
    Denote $\gamma_{\phi} \define \int \! \phi \,\mathrm{d}\mu_{\phi}$.
    Then for each $\alpha \in (\minenergy, \maxenergy)$, there exists an integer $N \in \n$ and a real number $C_{\alpha} > 0$ such that for each integer $n \geqslant N$,
    \begin{equation*}    \label{eq:large deviation asymptotics}
        \mu_{\phi}\Bigl(\Bigl\{ x \in S^2 \describe \sgn{\alpha - \gamma_{\phi}} \frac{1}{n} S_n\phi(x) \geqslant \sgn{\alpha - \gamma_{\phi}} \alpha \Bigr\} \Bigr)
        \leqslant C_{\alpha} e^{-I(\alpha)n},
    \end{equation*}
    where $S_{n}\phi(x) \define \sum_{i = 0}^{n - 1} \phi(f^{i}(x))$, $\minenergy$ and $\maxenergy$ are defined by \eqref{eq:def:maxenergy and minenergy}, and $I(\alpha)$ is defined by \eqref{eq:definition of rate function}.
\end{theorem}

For the rate function, $I(\alpha) = 0$ holds if and only if $\alpha = \gamma_{\phi}$ (see Proposition~\ref{prop:properties of rate function}~\ref{item:prop:properties of rate function:second order differentiable strictly convex boundary behavior}).
Hence, $I$ gives control on exponential rates of the convergence.
Moreover, it follows from Proposition~\ref{prop:properties of rate function}~\ref{item:prop:properties of rate function:expression of rate function} and \cite[Theorems~1.1~(5) and~1.2]{das2021thermodynamic} that the exponents $I(\alpha)$ in Theorem~\ref{thm:large deviation asymptotics} are the best possible.

\subsubsection*{Applications to rational maps}%
\label{ssub:Applications to rational maps}

Recall that a postcritically-finite rational map is expanding if and only if it has no periodic critical points (see \cite[Proposition~2.3]{bonk2017expanding}).
Therefore, when we restrict our attention to postcritically-finite rational maps, we obtain the following corollaries of Theorems~\ref{thm:main:thermodynamic formalism for subsystems}, \ref{thm:large deviation asymptotics}, and Remark~\ref{rem:chordal metric visual metric qs equiv}.

\begin{corollary}    \label{coro:pcf:thermodynamic formalism for subsystems}
    Let $f \colon \ccx \mapping \ccx$ be a postcritically-finite rational map without periodic critical points on the Riemann sphere $\ccx$ and $\mathcal{C} \subseteq \ccx$ be a Jordan curve containing $\post{f}$ with the property that $f(\mathcal{C}) \subseteq \mathcal{C}$.
    Let $\potential$ be a real-valued \holder continuous function on $\ccx$ with respect to the chordal metric.
    Consider a strongly irreducible subsystem $F \in \subsystem$ and denote its tile maximal invariant set by $\limitset$.
    Then 
    \begin{equation}    \label{eq:coro:pcf:thermodynamic formalism for subsystems:subsystem Variational Principle}
        \pressure = \sup \biggl\{ h_{\mu}(F|_\limitset) + \int_{\limitset} \! \phi \,\mathrm{d}\mu \describe \mu \in \mathcal{M}(\limitset, F|_{\limitset}) \biggr\}.
    \end{equation}
    Moreover, there exists an equilibrium state $\mu_{F, \potential}$ for $F|_{\limitset}$ and $\phi|_{\limitset}$ attaining the supremum in \eqref{eq:coro:pcf:thermodynamic formalism for subsystems:subsystem Variational Principle}.
\end{corollary}

\begin{corollary}[Large deviations asymptotics]    \label{coro:pcf:large deviation asymptotics}
    Let $f \colon \ccx \mapping \ccx$ be a postcritically-finite rational map without periodic critical points on the Riemann sphere $\ccx$.
    Let $\potential$ be a real-valued \holder continuous function on $\ccx$ (with respect to the chordal metric) that is not co-homologous to a constant in $C(\ccx)$.
    Let $\mu_{\phi}$ be the unique equilibrium state for the map $f$ and the potential $\phi$. 
    Denote $\gamma_{\phi} \define \int \! \phi \,\mathrm{d}\mu_{\phi}$.
    Then for each $\alpha \in (\minenergy, \maxenergy)$, there exists an integer $N \in \n$ and a real number $C_{\alpha} > 0$ such that for each integer $n \geqslant N$,
    \begin{equation*}    \label{eq:coro:pcf:large deviation asymptotics:large deviation asymptotics}
        \mu_{\phi}\Bigl(\Bigl\{ x \in S^2 \describe \sgn{\alpha - \gamma_{\phi}} \frac{1}{n}S_n\phi(x) \geqslant \sgn{\alpha - \gamma_{\phi}} \alpha \Bigr\} \Bigr)
        \leqslant C_{\alpha} e^{-I(\alpha)n},
    \end{equation*}
    where $S_{n}\phi(x) \define \sum_{i = 0}^{n - 1} \phi(f^{i}(x))$, $\minenergy$ and $\maxenergy$ are defined by \eqref{eq:def:maxenergy and minenergy}, and $I(\alpha)$ is defined by \eqref{eq:definition of rate function}.
\end{corollary}

\subsection{Strategy and organization of the paper}%
\label{sub:Strategy and organization of the paper}

We now discuss the strategy of the proofs of our main results and describe the organization of the paper.

To prove Theorem~\ref{thm:main:thermodynamic formalism for subsystems}, we define appropriate variants of the Ruelle operator called the \emph{split Ruelle operator}.
A similar technique was first used in \cite{baladi2002dynamical}, where the authors employed a version of the split Ruelle operator technique to reduce the problem to a uniformly expanding situation.
In the context of expanding Thurston maps, the concepts of splitting the Ruelle operators and split Ruelle operators were first introduced in \cite{li2018prime,li2024prime:split}.
In this paper, we generalize these concepts to subsystems.

By definition, a subsystem $F$ may not be a branched covering map on $S^2$.
Consequently, the local degree of $F$ at $x \in \domF$ may not make sense.
This leads to inadequate combinatorial structures when studying the dynamics. 
As a result, the Ruelle operator may not be well-defined and continuous. 
To address this, instead of a single number, we use four numbers arranged in the form of a $2 \times 2$ matrix, called the \emph{local degree matrix}, to describe the local degree (see Subsection~\ref{sub:Local degree}). 
Moreover, we show that such a local degree matrix is well-behaved under iteration.
We can then define the split Ruelle operator in our context (see Subsection~\ref{sub:Split Ruelle operators}). 
The idea is to ``split'' the Ruelle operator into two ``pieces'' so that the continuity is preserved under iteration in each piece, and to piece them together to obtain the split Ruelle operator on the product space.
Next, we prove the existence of an eigenfunction of this operator and the existence of an eigenmeasure of the adjoint operator.
With some additional work, we establish the existence of an equilibrium state and the Variational Principle.

To prove Theorem~\ref{thm:large deviation asymptotics}, we construct suitable subsystems and apply them within the thermodynamic formalism to bound deviations using the topological pressures of these subsystems.
This strategy is based on the approximation by subsystems in Takahasi's work~\cite{Takahasi2020} within the context of continued fractions.
However, constructing appropriate subsystems for approximation is not straightforward in our setting, where the dynamical systems are branched covering maps on $S^2$.
Indeed, we need to utilize the geometric properties of visual metrics and their interplay with the associated combinatorial structures (see Subsection~\ref{sub:Pair structures}).

\smallskip

We now describe the structure of the paper.

After fixing some notation in Section~\ref{sec:Notation}, we review some notions from ergodic theory and go over some key concepts and results on Thurston maps in Section~\ref{sec:Preliminaries}.
In Section~\ref{sec:The Assumptions}, we state the assumptions regarding some of the objects in this paper, which we will repeatedly refer to later as the \emph{Assumptions} in Section~\ref{sec:The Assumptions}.

In Section~\ref{sec:Subsystems}, we define subsystems of expanding Thurston maps, introduce relevant concepts, and prove preliminary results about them.

Section~\ref{sec:Thermodynamic formalism for subsystems} focuses on thermodynamic formalism for subsystems, with the main results being Theorems~\ref{thm:subsystem characterization of pressure} and \ref{thm:existence of equilibrium state for subsystem}.
We introduce the split Ruelle operators in Subsection~\ref{sub:Split Ruelle operators}, and then, in the remainder of this section, we use these operators to establish the Variational Principle and the existence of equilibrium states for strongly irreducible subsystem.

Section~\ref{sec:Large deviation asymptotics for expanding Thurston maps} aims to prove Theorem~\ref{thm:large deviation asymptotics}.
We first study the properties of the rate function (Subsection~\ref{sub:The rate function}), and then introduce the pair structures (Subsection~\ref{sub:Pair structures}).
Subsection~\ref{sub:Key bounds} is devoted to the proof of key bounds in Proposition~\ref{prop:LDA:key bounds}.
Finally, in Subsection~\ref{sub:Proof of the large deviation asymptotics}, we use the key bounds to establish Theorem~\ref{thm:large deviation asymptotics}.


\section{Notation}
\label{sec:Notation}

Let $\cx$ be the complex plane and $\ccx$ be the Riemann sphere. 
Let $S^2$ denote an oriented topological $2$-sphere.
We use $\n$ to denote the set of integers greater than or equal to $1$ and write $\n_0 \define \{0\} \cup \n$. 
For $x \in \real$, we define $\lfloor x \rfloor$ as the greatest integer $\leqslant x$, and $\lceil x \rceil$ the smallest integer $\geqslant x$.
For $x \in \real$, we denote the sign function by $\sgn{x}$, which takes the value $-1$, $1$, or $0$ depending on whether $x$ is negative, positive, or zero.
The cardinality of a set $A$ is denoted by $\card{A}$.

Let $g \colon X \mapping Y$ be a map between two sets $X$ and $Y$. We denote the restriction of $g$ to a subset $Z$ of $X$ by $g|_{Z}$.

Consider a map $f \colon X \mapping X$ on a set $X$. 
The inverse map of $f$ is denoted by $f^{-1}$. 
We write $f^n$ for the $n$-th iterate of $f$, and $f^{-n}\define (f^n)^{-1}$, for each $n \in \n$. 
We set $f^0 \define \id{X}$, the identity map on $X$. 
For a real-valued function $\varphi \colon X \mapping \real$, we write
\begin{equation}    \label{eq:def:Birkhoff average}
    S_n \varphi(x) = S^f_n \varphi(x) \define \sum_{j=0}^{n-1} \varphi \bigl( f^j(x) \bigr)
\end{equation}
for $x \in X$ and $n \in \n_0$. 
We omit the superscript $f$ when the map $f$ is clear from the context. Note that when $n = 0$, by definition we always have $S_0 \varphi = 0$.

Let $(X,d)$ be a metric space. For each subset $Y \subseteq X$, we denote the diameter of $Y$ by $\diam{d}{Y} \define \sup\{d(x, y) \describe \juxtapose{x}{y} \in Y\}$, the interior of $Y$ by $\interior{Y}$, and the characteristic function of $Y$ by $\indicator{Y}$, which maps each $x \in Y$ to $1 \in \real$ and vanishes otherwise. 
For each $r > 0$ and each $x \in X$, we denote the open (\resp closed) ball of radius $r$ centered at $x$ by $B_{d}(x,r)$ (\resp $\overline{B_{d}}(x,r)$). 
We often omit the metric $d$ in the subscript when it is clear from the context.

For a compact metric space $(X, d)$, we denote by $C(X)$ (\resp $B(X)$) the space of continuous (\resp bounded Borel) functions from $X$ to $\real$, by $\mathcal{M}(X)$ the set of finite signed Borel measures, and $\mathcal{P}(X)$ the set of Borel probability measures on $X$.
For $\mu \in \mathcal{M}(X)$, we denote by $\norm{\mu}$ the total variation norm of $\mu$, and by $\supp{\mu}$ the support of $\mu$ (the smallest closed set $A \subseteq X$ such that $|\mu|(X \setminus A) = 0$). 
Additionally, for $u \in C(X)$, we denote
\[
    \functional{\mu}{u} \define \int \! u \,\mathrm{d}\mu,
\]
and define the finite signed Borel measure $u\mu \in \mathcal{M}(X)$ as
\[
    (u \mu)(A) \define \int_{A} \! u \,\mathrm{d} \mu \qquad \text{for Borel set } A \subseteq X.
\]
For a point $x \in X$, we define $\delta_x$ as the Dirac measure supported on $\{x\}$.
For a continuous map $g \colon X \mapping X$, we set $\mathcal{M}(X, g)$ to be the set of $g$-invariant Borel probability measures on $X$.
If we do not specify otherwise, we equip $C(X)$ with the uniform norm $\normcontinuous{\cdot}{X} \define \uniformnorm{\cdot}$, and equip $\mathcal{M}(X)$, $\mathcal{P}(X)$, and $\mathcal{M}(X, g)$ with the weak$^*$ topology.
According to the Riesz representation theorem (see for example, \cite[Theorems~7.17 and 7.8]{folland2013real}), we identify the dual of $C(X)$ with the space $\mathcal{M}(X)$.

The space of real-valued \holder continuous functions with an exponent $\holderexp \in (0,1]$ on a compact metric space $(X, d)$ is denoted as $\holderspace$. For each $\phi \in \holderspace$, \[
    \holderseminorm{\phi} \define \sup\left\{ \frac{|\phi(x) - \phi(y)|}{d(x, y)^{\holderexp}} \describe \juxtapose{x}{y} \in X, \, x \ne y \right\},
\]
and the \holder norm is defined as $\holdernorm{\phi} \define \holderseminorm{\phi} + \normcontinuous{\phi}{X}$.

\section{Preliminaries}
\label{sec:Preliminaries}

\subsection{Thermodynamic formalism}  
\label{sub:thermodynamic formalism}

We first review some basic concepts from ergodic theory and dynamical systems. 
We refer the reader to \cite[Chapter~20]{katok1995introduction} and \cite[Chapter~9]{walters1982introduction} for more detailed studies of these concepts.

\smallskip

Let $(X,d)$ be a compact metric space and $g \colon X \mapping X$ a continuous map. Given $n \in \n$, 
\[
    d^n_g(x, y) \define \operatorname{max} \bigl\{  d \bigl(g^k(x), g^k(y) \bigr) \describe k \in \{0, \, 1, \, \dots, \, n-1 \} \!\bigr\}, \quad \text{ for } \juxtapose{x}{y} \in X,
\]
defines a metric on $X$. A set $F \subseteq X$ is \emph{$(n, \epsilon)$-separated} (with respect to $g$), for some $n \in \n$ and $\epsilon > 0$, if for each pair of distinct points $\juxtapose{x}{y} \in F$, we have $d^n_g(x, y) \geqslant \epsilon$.

For each real-valued continuous function $\psi \in C(X)$, the following two limits exist and are equal (see for example, \cite[Subsection~20.2]{katok1995introduction}), which we denote by $P(g, \psi)$:
\begin{equation}  \label{eq:def:topological pressure}
    P(g, \psi) \define \lim \limits_{\epsilon \to 0^{+}} \limsup\limits_{n \to +\infty} \frac{1}{n} \log  N_{d}(g, \psi, \varepsilon, n) 
    = \lim \limits_{\epsilon \to 0^{+}} \liminf\limits_{n \to +\infty} \frac{1}{n} \log  N_{d}(g, \psi, \varepsilon, n),
\end{equation}
where $N_{d}(g, \psi, \varepsilon, n) \define \sup \set[\big]{\sum_{x \in E} \myexp[\big]{S_n\psi(x)} \describe E \subseteq X \text{ is } (n, \varepsilon)\text{-separated with respect to } g}$.
We call $P(g, \psi)$ the \emph{topological pressure} of $g$ with respect to the \emph{potential} $\psi$. 
In particular, when $\psi = 0$, the quantity $h_{\operatorname{top}}(g) \define P(g, 0)$ is called the \emph{topological entropy} of $g$. 
Note that $P(g, \psi)$ is independent of $d$ as long as the topology on $X$ defined by $d$ remains the same (see for example, \cite[Subsection~20.2]{katok1995introduction}).
Moreover, the topological pressure is well-behaved under iteration. 
Indeed, if $n \in \n$, then $P(g^{n}, S_{n}\psi) = n P(g, \psi)$ (see for example, \cite[Theorem~9.8~(i)]{walters1982introduction}).

We denote by $\mathcal{B}$ the $\sigma$-algebra of all Borel sets on $X$. 
A \emph{measurable partition} $\xi$ of $X$ is a collection $\xi = \{ A_{i} \describe i \in J \}$ consisting of countably many (i.e., either finite or countably infinite) mutually disjoint sets in $\mathcal{B}$, where $J$ is the index set. 
The measurable partition $\xi$ is finite if the index set $J$ is a finite set.

Let $\xi = \{A_j \describe j \in J\}$ and $\eta = \{B_k \describe k \in K\}$ be measurable partitions of $X$, where $J$ and $K$ are the corresponding index sets. 
We say that $\xi$ is a \emph{refinement} of $\eta$ if for each $A_j \in \xi$, there exists $B_k \in \eta$ such that $A_j \subseteq B_k$. 
The \emph{common refinement} (or \emph{join}) $\xi \vee \eta$ of $\xi$ and $\eta$ defined as\[
    \xi \vee \eta \define \{A_j \cap B_k \describe j \in J, \, k \in K\}
\]
is also a measurable partition. 
Put $g^{-1}(\xi) \define \bigl\{ g^{-1}(A_j) \describe j \in J \bigr\}$, and for $n \in \n$ define \[
    \xi^n_g \define \bigvee \limits_{j = 0}^{n - 1} g^{-j}(\xi) = \xi \vee g^{-1}(\xi) \vee \cdots \vee g^{-(n-1)}(\xi).
\]

Let $\xi = \{A_j \describe j \in J \}$ be a measurable partition of $X$ and $\mu \in \mathcal{M}(X, g)$ be a $g$-invariant Borel probability measure on $X$. For $x \in X$, we denote by $\xi(x)$ the unique element of $\xi$ that contains $x$. 
The \emph{information function} $I_{\mu}$ maps a measurable partition $\xi$ of $X$ to a $\mu$-a.e.\ defined real-valued function on $X$ in the following way:
\begin{equation}   \label{eq:def:information function}
    I_{\mu}(\xi)(x) \define -\log (\mu(\xi(x))), \qquad \text{for } x \in X.
\end{equation} 

The \emph{entropy} of $\xi$ is $H_{\mu}(\xi) \define -\sum_{j\in J} \mu(A_j) \log\left(\mu (A_j)\right) \in [0, +\infty]$, where $0 \log 0$ is defined to be zero. 
One can show that (see for example, \cite[Chapter~4]{walters1982introduction}) if $H_{\mu}(\xi) < +\infty$, then the following limit exists:
\begin{equation}    \label{eq:def:measure-theoretic entropy with respect to partition}
    h_{\mu}(g, \xi) \define \lim\limits_{n \to +\infty} \frac{1}{n} H_{\mu}(\xi^n_g) \in [0, +\infty).
\end{equation}
The quantity $h_{\mu}(g, \xi)$ is called the \emph{measure-theoretic entropy of $g$ relative to $\xi$}.
The \emph{measure-theoretic entropy} of $g$ for $\mu$ is defined as
\begin{equation}   \label{eq:def:measure-theoretic entropy}
h_{\mu}(g) \define \sup\{h_{\mu}(g,\xi) \describe \xi \text{ is a measurable partition of } X  \text{ with } H_{\mu}(\xi) < +\infty\}.   
\end{equation}
If $\mu \in \mathcal{M}(X, g)$ and $n \in \n$, then (see for example, \cite[Proposition~4.3.16~(4)]{katok1995introduction})
\begin{equation}    \label{eq:measure-theoretic entropy well-behaved under iteration}
    h_{\mu}(g^{n}) = n h_{\mu}(g).
\end{equation}
If $t \in [0, 1]$ and $\nu \in \mathcal{M}(X, g)$ is another measure, then (see for example, \cite[Theorem~8.1]{walters1982introduction})
\begin{equation}    \label{eq:measure-theoretic entropy is affine}
    h_{t \mu + (1 - t)\nu}(g) = t h_{\mu}(g) + (1 - t) h_{\nu}(g).
\end{equation}

For each real-valued continuous function $\psi \in C(X)$, the \emph{measure-theoretic pressure} $P_\mu(g, \psi)$ of $g$ for the measure $\mu \in \mathcal{M}(X, g)$ and the potential $\psi$ is
\begin{equation}  \label{eq:def:measure-theoretic pressure}
    P_\mu(g, \psi) \define h_\mu (g) + \int \! \psi \,\mathrm{d}\mu.
\end{equation}

The topological pressure is related to the measure-theoretic pressure by the so-called \emph{Variational Principle}.
It states that (see for example, \cite[Theorem~20.2.4]{katok1995introduction})
\begin{equation}  \label{eq:Variational Principle for pressure}
    P(g, \psi) = \sup \{P_\mu(g, \psi) \describe \mu \in \mathcal{M}(X, g)\}
\end{equation}
for $\psi \in C(X)$.
In particular, when $\psi$ is the constant function $0$,
\begin{equation}  \label{eq:Variational Principle for entropy}
    h_{\operatorname{top}}(g) = \sup\{h_{\mu}(g) \describe\mu \in \mathcal{M}(X, g)\}.
\end{equation}
A measure $\mu$ that attains the supremum in \eqref{eq:Variational Principle for pressure} is called an \emph{equilibrium state} for the map $g$ and the potential $\psi$. A measure $\mu$ that attains the supremum in \eqref{eq:Variational Principle for entropy} is called a \emph{measure of maximal entropy} of $g$.

\subsection{Thurston maps}%
\label{sub:Thurston_maps}
In this subsection, we go over some key concepts and results on Thurston maps, and expanding Thurston maps in particular. 
For a more thorough treatment of the subject, we refer to \cite{bonk2017expanding}.

\smallskip

Let $S^2$ denote an oriented topological $2$-sphere. A continuous map $f \colon S^2 \mapping S^2$ is called a \emph{branched covering map} on $S^2$ if for each point $x\in S^2$, there exists a positive integer $d\in \n$, open neighborhoods $U$ of $x$ and $V$ of $y \define f(x)$, open neighborhoods $U'$ and $V'$ of $0$ in $\ccx$, and orientation-preserving homeomorphisms $\varphi \colon U \mapping U'$ and $\eta \colon V \mapping V'$ such that $\varphi(x) = 0, \eta(y) = 0$, and $\big(\eta \circ f\circ \varphi^{-1}\bigr)(z) = z^d$ for each $z \in U'$. 
The positive integer $d$ above is called the \emph{local degree} of $f$ at $x$ and is denoted by $\deg_f(x)$ or $\deg(f, x)$.

The \emph{degree} of $f$ is\[
    \deg{f} = \sum_{x\in f^{-1}(y)} \deg_{f}(x)
\]
for $y\in S^2$ and is independent of $y$. If $f \colon S^2 \mapping S^2$ and $g \colon S^2 \mapping S^2$ are two branched covering maps on $S^2$, then so is $f\circ g$, and
\begin{equation}    \label{eq:composed local degree for branched covering map}
    \deg(f\circ g, x) = \deg(g, x) \deg(f, g(x))
\end{equation}
for $x \in S^2$, and moreover, $\deg(f\circ g) = (\deg{f})(\deg{g})$.

A point  $x\in S^2$ is a \emph{critical point} of $f$ if $\deg_f(x) \geqslant 2$. The set of critical points of $f$ is denoted by $\crit{f}$. A point $y\in S^2$ is a \emph{postcritical point} of $f$ if $y = f^n(x)$ for some $x \in \crit{f}$ and $n\in \n$. The set of postcritical points of $f$ is denoted by $\post{f}$. Note that $\post{f} = \post{f^n}$ for all $n\in \n$.

\begin{definition}[Thurston maps]
    A Thurston map is a branched covering map $f \colon S^2 \mapping S^2$ on $S^2$ with $\deg f \geqslant 2$ and $\card{\post{f}}< +\infty$.
\end{definition}

We now recall the notation for cell decompositions of $S^2$ used in \cite{bonk2017expanding} and \cite{li2017ergodic}. A \emph{cell of dimension $n$} in $S^2$, $n \in \{1, \, 2\}$, is a subset $c \subseteq S^2$ that is homeomorphic to the closed unit ball $\overline{\mathbb{B}^n}$ in $\real^n$, where $\mathbb{B}^{n}$ is the open unit ball in $\real^{n}$. 
We define the \emph{boundary of $c$}, denoted by $\partial c$, to be the set of points corresponding to $\partial \mathbb{B}^n$ under such a homeomorphism between $c$ and $\overline{\mathbb{B}^n}$. The \emph{interior of $c$} is defined to be $\inte{c} = c \setminus \partial c$. 
For each point $x\in S^2$, the set $\{x\}$ is considered as a \emph{cell of dimension $0$} in $S^2$. For a cell $c$ of dimension $0$, we adopt the convention that $\partial c = \emptyset$ and $\inte{c} = c$. 

We record the following definition of cell decompositions from \cite[Definition~3.2]{bonk2017expanding}.

\begin{definition}[Cell decompositions]    \label{def:cell decomposition}
    Let $\mathbf{D}$ be a collection of cells in $S^2$. 
    We say that $\mathbf{D}$ is a \emph{cell decomposition of $S^2$} if the following conditions are satisfied:
    \begin{enumerate}[label= (\roman*)]
        \smallskip
        
        \item the union of all cells in $\mathbf{D}$ is equal to $S^2$,
        
        \smallskip
        
        \item if $c \in \mathbf{D}$, then $\partial c$ is a union of cells in $\mathbf{D}$,
        
        \smallskip
        
        \item for $\juxtapose{c_1}{c_2} \in \mathbf{D}$ with $c_1 \ne c_2$, we have $\inte{c_1} \cap \inte{c_2} = \emptyset$,
        
        \smallskip
        
        \item every point in $S^2$ has a neighborhood that meets only finitely many cells in $\mathbf{D}$.
    \end{enumerate}
\end{definition}

We record \cite[Lemma~5.3]{bonk2017expanding} here to review some facts about cell decompositions.

\begin{lemma}    \label{lem:intersection of cells}
    Let $\mathbf{D}$ be a cell decomposition of $S^2$.
    \begin{enumerate}[label=\rm{(\roman*)}]
        \smallskip
        
        \item     \label{item:lem:intersection of cells:intersection of two tiles} 
            If $\sigma$ and $\tau$ are two distinct cells in $\mathbf{D}$ with $\sigma \cap \tau \ne \emptyset$, then one of the following statements hold: $\sigma \subseteq \partial \tau$, $\tau \subseteq \partial \sigma$, or $\sigma \cap \tau = \partial\sigma \cap \partial\tau$ and this intersection consists of cells in $\mathbf{D}$ of dimension strictly less than $\min\{\dim\sigma, \dim\tau\}$.

        \smallskip
        
        \item     \label{item:lem:intersection of cells:intersection with union of tiles}
            If $\sigma, \, \tau_1, \, \dots, \, \tau_n$ are cells in $\mathbf{D}$ and $\inte{\sigma}\, \cap\, (\tau_1 \, \cup \cdots \cup\, \tau_n) \ne \emptyset$, then $\sigma \subseteq \tau_i$ for some $i \in \oneton$.
    \end{enumerate}
\end{lemma}

\begin{definition}[Refinements]
    Let $\mathbf{D}'$ and $\mathbf{D}$ be two cell decompositions of $S^2$. We say that $\mathbf{D}'$ is a \emph{refinement} of $\mathbf{D}$ if the following conditions are satisfied:
    \begin{enumerate}[label = (\roman*)]
        \smallskip

        \item every cell $c \in \mathbf{D}$ is the union of all cells $c' \in \mathbf{D}'$ with $c' \subseteq c$.

        \smallskip

        \item for every cell $c' \in \mathbf{D}'$ there exists a cell $c \in \mathbf{D}$ with $c' \subseteq c$.
    \end{enumerate}
\end{definition}

\begin{definition}[Cellular maps and cellular Markov partitions]
    Let $\mathbf{D}'$ and $\mathbf{D}$ be two cell decompositions of $S^2$. We say that a continuous map $f \colon S^2 \mapping S^2$ is \emph{cellular} for $(\mathbf{D}', \mathbf{D})$ if for every cell $c \in \mathbf{D}'$, the restriction $f|_c$ of $f$ to $c$ is a homeomorphism of $c$ onto a cell in $\mathbf{D}$. 
    We say that $(\mathbf{D}',\mathbf{D})$ is a \emph{cellular Markov partition} for $f$ if $f$ is cellular for $(\mathbf{D}', \mathbf{D})$ and $\mathbf{D}'$ is a refinement of $\mathbf{D}$.
\end{definition}

Let $f \colon S^2 \mapping S^2$ be a Thurston map, and $\mathcal{C}\subseteq S^2$ be a Jordan curve containing $\post{f}$. 
Then the pair $f$ and $\mathcal{C}$ induces natural cell decompositions $\mathbf{D}^n(f,\mathcal{C})$ of $S^2$, for each $n \in \n_0$, in the following way.

By the Jordan curve theorem, the set $S^2 \setminus \mathcal{C}$ has two connected components. We call the closure of one of them the \emph{white $0$-tile} for $(f,\mathcal{C})$, denoted by $X^0_{\white}$, and the closure of the other one the \emph{black $0$-tile} for $(f,\mathcal{C})$, denoted be $X^0_{\black}$. 
The set of $0$-\emph{tiles} is $\mathbf{X}^0(f, \mathcal{C}) \define \bigl\{X^0_{\black}, \, X^0_{\white} \bigr\}$. 
The set of $0$-\emph{vertices} is $\mathbf{V}^0(f, \mathcal{C}) \define \post{f}$. 
We set $\overline{\mathbf{V}}^0(f, \mathcal{C}) \define \bigl\{ \{x\} \describe x\in \mathbf{V}^0(f,\mathcal{C}) \bigr\}$. 
The set of $0$-\emph{edges} $\mathbf{E}^0(f,\mathcal{C})$ is the set of the closures of the connected components of $\mathcal{C} \setminus \post{f}$. 
Then we get a cell decomposition\[
    \mathbf{D}^0(f,\mathcal{C}) \define \mathbf{X}^0(f, \mathcal{C}) \cup \mathbf{E}^0(f,\mathcal{C}) \cup \overline{\mathbf{V}}^0(f,\mathcal{C})
\]
of $S^2$ consisting of \emph{cells of level }$0$, or $0$-\emph{cells}.

We can recursively define the unique cell decomposition $\mathbf{D}^n(f,\mathcal{C})$, $n\in \n$, consisting of $n$-\emph{cells} such that $f$ is cellular for $(\mathbf{D}^{n+1}(f,\mathcal{C}), \mathbf{D}^n(f,\mathcal{C}))$. We refer to \cite[Lemma~5.12]{bonk2017expanding} for more details. 
We denote by $\mathbf{X}^n(f,\mathcal{C})$ the set of $n$-cells of dimension 2, called $n$-\emph{tiles}; by $\mathbf{E}^n(f,\mathcal{C})$ the set of $n$-cells of dimension $1$, called $n$-\emph{edges}; by $\overline{\mathbf{V}}^n(f,\mathcal{C})$ the set of $n$-cells of dimension $0$; and by $\mathbf{V}^n(f,\mathcal{C})$ the set $\{x \describe \{x\} \in \overline{\mathbf{V}}^n(f,\mathcal{C})\}$, called the set of $n$-\emph{vertices}. 
The $k$-\emph{skeleton}, for $k\in \{0, \, 1 \}$, of $\mathbf{D}^n(f,\mathcal{C})$ is the union of all $n$-cells of dimension $k$ in this cell decomposition.

We record \cite[Proposition~5.16]{bonk2017expanding} here in order to summarize properties of the cell decompositions $\mathbf{D}^n(f,\mathcal{C})$ defined above.

\begin{proposition}[M.~Bonk \& D.~Meyer \cite{bonk2017expanding}]     \label{prop:properties cell decompositions}
    Let $\juxtapose{k}{n} \in \n_0$, $f \colon S^2 \mapping S^2$ be a Thurston map, $\mathcal{C} \subseteq S^2$ be a Jordan curve with $\post{f} \subseteq \mathcal{C}$, and $m \define \card{\post{f}}$.
    \begin{enumerate}[label=\rm{(\roman*)}]
        \smallskip

        \item     \label{item:prop:properties cell decompositions:cellular} 
            The map $f^k$ is cellular for $(\mathbf{D}^{n+k}(f,\mathcal{C}), \mathbf{D}^n(f,\mathcal{C}))$. 
            In particular, if $c$ is any $(n+k)$-cell, then $f^k(c)$ is an $n$-cell, and $f^k|_c$ is a homeomorphism of $c$ onto $f^k(c)$.
        
        \smallskip

        \item     \label{item:prop:properties cell decompositions:union of cells}
        Let $c$ be an $n$-cell. Then $f^{-k}(c)$ is equal to the union of all $(n+k)$-cell $c'$ with $f^k(c') = c$.
        
        \smallskip

        \item     \label{item:prop:properties cell decompositions:skeletons}
            The $1$-skeleton of $\mathbf{D}^n(f,\mathcal{C})$ is equal to $f^{-n}(\mathcal{C})$. The $0$-skeleton of $\mathbf{D}^n(f,\mathcal{C})$ is the set $\mathbf{V}^n(f,\mathcal{C}) = f^{-n}(\post{f})$, and we have $\mathbf{V}^n(f,\mathcal{C}) \subseteq \mathbf{V}^{n+k}(f,\mathcal{C})$.

        \smallskip

        \item     \label{item:prop:properties cell decompositions:cardinality}
            $\card{\mathbf{X}^n(f,\mathcal{C})} = 2(\deg f)^n$, $\card{\mathbf{E}^n(f,\mathcal{C})} = m(\deg f)^n$, and $\card{\mathbf{V}^n(f,\mathcal{C})} \leqslant m (\deg f)^n$.

        \smallskip

        \item     \label{item:prop:properties cell decompositions:edge is boundary of tile}
            The $n$-edges are precisely the closures of the connected components of $f^{-n}(\mathcal{C}) \setminus f^{-n}(\post{f})$. The $n$-tiles are precisely the closures of the connected components of $S^2\setminus f^{-n}(\mathcal{C})$.

        \smallskip

        \item     \label{item:prop:properties cell decompositions:tile is gon}
        Every $n$-tile is an $m$-gon, i.e., the number of $n$-edges and the number of $n$-vertices contained in its boundary are equal to $m$.

        \smallskip

        \item     \label{item:prop:properties cell decompositions:iterate of cell decomposition}
            Let $F \define f^k$ be an iterate of $f$ with $k\in \n$. Then $\mathbf{D}^n(F,\mathcal{C}) = \mathbf{D}^{nk}(f,\mathcal{C})$.
    \end{enumerate}
\end{proposition}

We record \cite[Lemma~5.17]{bonk2017expanding}.

\begin{lemma}[M.~Bonk \& D.~Meyer \cite{bonk2017expanding}]    \label{lem:cell mapping properties of Thurston map}
    Let $\juxtapose{k}{n} \in \n_0$, $f \colon S^2 \mapping S^2$ be a Thurston map, and $\mathcal{C} \subseteq S^2$ be a Jordan curve with $\post{f} \subseteq \mathcal{C}$.
    \begin{enumerate}[label=\rm{(\roman*)}]
        \smallskip

        \item     \label{item:lem:cell mapping properties of Thurston map:i}
        If $c \subseteq S^2$ is a topological cell such that $f^{k}|_{c}$ is a homeomorphism onto its image and $f^{k}(c)$ is an $n$-cell, then $c$ is an $(n + k)$-cell.

        \smallskip

        \item \label{item:lem:cell mapping properties of Thurston map:ii}
        If $X$ is an $n$-tile and $p \in S^2$ is a point with $f^{k}(p) \in \inte{X}$, then there exists a unique $(n + k)$-tile $X'$ with $p \in X'$ and $f^{k}(X') = X$.
    \end{enumerate}
\end{lemma}

\begin{remark}\label{rem:intersection of two tiles}
    Note that for each $n$-edge $e^{n} \in \mathbf{E}^{n}(f,\mathcal{C})$, $n \in \n_0$, there exist exactly two $n$-tiles in $\mathbf{X}^n(f,\mathcal{C})$ containing $e^{n}$.
\end{remark}

For $n\in \n_0$, we define the \emph{set of black $n$-tiles} as\[
    \textbf{X}^n_{\black}(f,\mathcal{C}) \define \left\{ X \in \mathbf{X}^n (f,\mathcal{C}) \describe f^n(X) = X^0_{\black} \right\},
\]
and the \emph{set of white $n$-tiles} as\[
    \mathbf{X}^n_{\white}(f,\mathcal{C}) \define \left\{X\in \mathbf{X}^n(f,\mathcal{C}) \describe f^n(X) = X^0_{\white}\right\}.
\]

From now on, if the map $f$ and the Jordan curve $\mathcal{C}$ are clear from the context, we will sometimes omit $(f,\mathcal{C})$ in the notation above.

We denote, for each $x \in S^2$ and each $n \in \z$, the \emph{$n$-bouquet of $x$}
\begin{equation}    \label{eq:Un bouquet of point}
    U^n(x) \define \bigcup \bigl\{ Y^n \in \mathbf{X}^n \describe \text{there exists } X^n \in \mathbf{X}^n \text{ with } x\in X^n, \, X^n \cap Y^n \ne \emptyset \bigr\}
\end{equation}
if $n \geqslant 0$, and set $U^n(x) \define S^2$ otherwise. 

For each $n \in \n_0$, we define the \emph{$n$-partition} $O_{n}$ of $S^{2}$ induced by $(f, \mathcal{C})$ as
\begin{equation}    \label{eq:partition induced by cell decomposition}
    O_{n} \define \{ \inte{X^n} \describe X^n \in \mathbf{X}^{n}\} \cup \{ \inte{e^{n}} \describe e^{n} \in \mathbf{E}^{n} \} \cup \overline{\mathbf{V}}^{n}. 
\end{equation}

We can now give a definition of expanding Thurston maps.

\begin{definition}[Expansion]     \label{def:expanding_Thurston_maps}
    A Thurston map $f \colon S^2 \mapping S^2$ is called \emph{expanding} if there exists a metric $d$ on $S^2$ that induces the standard topology on $S^2$ and a Jordan curve $\mathcal{C} \subseteq S^2$ containing $\post{f}$ such that
    \begin{equation}    \label{eq:definition of expansion}
        \lim_{n \to +\infty} \max\{ \diam{d}{X} \describe X \in \mathbf{X}^n(f,\mathcal{C}) \} = 0.
    \end{equation}
\end{definition}

\begin{remark}\label{rem:Expansion_is_independent}
    It is clear from Proposition~\ref{prop:properties cell decompositions}~\ref{item:prop:properties cell decompositions:iterate of cell decomposition} and Definition~\ref{def:expanding_Thurston_maps} that if $f$ is an expanding Thurston map, so is $f^n$ for each $n\in \n$. We observe that being expanding is a topological property of a Thurston map and independent of the choice of the metric $d$ that generates the standard topology on $S^2$. By Lemma~6.2 in \cite{bonk2017expanding}, it is also independent of the choice of the Jordan curve $\mathcal{C}$ containing $\post{f}$. More precisely, if $f$ is an expanding Thurston map, then\[
        \lim_{n \to +\infty} \max \{ \diam{\widetilde{d}}{X} \describe X\in \mathbf{X}^n(f,\widetilde{\mathcal{C}}) \} = 0,        
    \]
    for each metric $\widetilde{d}$ that generates the standard topology on $S^2$ and each Jordan curve $\widetilde{\mathcal{C}} \subseteq S^2$ that contains $\post{f}$.
\end{remark}

For an expanding Thurston map $f$, we can fix a particular metric $d$ on $S^2$ called a \emph{visual metric for $f$}. 
For the existence and properties of such metrics, see \cite[Chapter~8]{bonk2017expanding}. 
For a visual metric $d$ for $f$, there exists a unique constant $\Lambda > 1$ called the \emph{expansion factor} of $d$ (see \cite[Chapter~8]{bonk2017expanding} for more details). 
One major advantage of a visual metric $d$ is that in $(S^2,d)$ we have good quantitative control over the sizes of the cells in the cell decompositions discussed above. 
We summarize several results of this type (\cite[Proposition~8.4, Lemmas~8.10, and~8.11]{bonk2017expanding}) in Lemma~\ref{lem:visual_metric}.

\begin{lemma}[M.~Bonk \& D.~Meyer \cite{bonk2017expanding}]    \label{lem:visual_metric}
    Let $f \colon S^2 \mapping S^2$ be an expanding Thurston map, and $\mathcal{C} \subseteq S^2$ be a Jordan curve containing $\post{f}$. 
    Let $d$ be a visual metric on $S^2$ for $f$ with expansion factor $\Lambda > 1$. 
    Then there exist constants $C \geqslant 1$, $K \geqslant 1$, and $n_0 \in \n_0$ with the following properties:
    \begin{enumerate}[label=\rm{(\roman*)}]

        \smallskip
        
        \item     \label{item:lem:visual_metric:distinct cell separated} 
            $d(\sigma,\tau) \geqslant C^{-1}\Lambda^{-n}$ whenever $\sigma$ and $\tau$ are disjoint $n$-cells for some $n \in \n_0$.

        \smallskip

        \item     \label{item:lem:visual_metric:diameter of cell}
            $C^{-1}\Lambda^{-n} \leqslant \diam{d}{\tau} \leqslant C\Lambda^{-n}$ for all $n$-edges and all $n$-tiles $\tau$ and for all $n \in \n_0$.
        
        \smallskip
        
        \item     \label{item:lem:visual_metric:bouquet bounded by ball}
            $B_{d}(x, K^{-1}\Lambda^{-n}) \subseteq U^n(x) \subseteq B_{d}(x, K\Lambda^{-n})$ for each $x \in S^2$ and each $n \in \n_0$.

        \smallskip
        
        \item     \label{item:lem:visual_metric:ball bounded by bouquet}
            $U^{n+n_0}(x) \subseteq B_{d}(x,r) \subseteq U^{n-n_0}(x)$ where $n \define \lceil -\log{r}/\log{\Lambda} \rceil$ for all $r > 0$ and $x \in S^2$.
        
        \smallskip
                
        \item     \label{item:lem:visual_metric:tile contain ball and bounded by ball}
            For every $n$-tile $X^n \in \mathbf{X}^n(f,\mathcal{C})$, $n\in \n_0$, there exists a point $p\in X^n$ such that $B_{d}(p, C^{-1}\Lambda^{-n}) \subseteq X^n \subseteq B_{d}(p, C\Lambda^{-n})$.
    \end{enumerate}

    Conversely, if $\widetilde{d}$ is a metric on $S^2$ satisfying conditions $\textnormal{(i)}$ and $\textnormal{(ii)}$ for some constant $C \geqslant 1$, then $\widetilde{d}$ is a visual metric with expansion factor $\Lambda > 1$.
\end{lemma}

Recall $U^n(x)$ is defined in \eqref{eq:Un bouquet of point}.

\begin{remark}\label{rem:chordal metric visual metric qs equiv}
    If $f \colon \ccx \mapping \ccx$ is a rational expanding Thurston map, then a visual metric is quasisymmetrically equivalent to the chordal metric on the Riemann sphere $\ccx$ (see \cite[Theorem~18.1~(ii)]{bonk2017expanding}). 
    Here the chordal metric $\sigma$ on $\ccx$ is given by $\sigma (z, w) \define \frac{2\abs{z - w}}{\sqrt{1 + \abs{z}^2} \sqrt{1 + \abs{w}^2}}$ for all $\juxtapose{z}{w} \in \cx$, and $\sigma(\infty, z) = \sigma(z, \infty) \define \frac{2}{\sqrt{1 + \abs{z}^2}}$ for all $z \in \cx$. 
    We also note that quasisymmetric embeddings of bounded connected metric spaces are \holder continuous (see \cite[Section~11.1 and Corollary~11.5]{heinonen2001lectures}). 
    Accordingly, the classes of \holder continuous functions on $\ccx$ equipped with the chordal metric and on $S^2 = \ccx$ equipped with any visual metric for $f$ are the same (up to a change of the \holder exponent).
\end{remark}

A Jordan curve $\mathcal{C} \subseteq S^2$ is \emph{$f$-invariant} if $f(\mathcal{C}) \subseteq \mathcal{C}$. If $\mathcal{C}$ is $f$-invariant with $\post{f} \subseteq \mathcal{C}$, then the cell decompositions $\mathbf{D}^{n}(f, \mathcal{C})$ have nice compatibility properties. In particular, $\mathbf{D}^{n+k}(f, \mathcal{C})$ is a refinement of $\mathbf{D}^{n}(f, \mathcal{C})$, whenever $\juxtapose{n}{k} \in \n_0$. Intuitively, this means that each cell $\mathbf{D}^{n}(f, \mathcal{C})$ is ``subdivided'' by the cells in $\mathbf{D}^{n+k}(f, \mathcal{C})$. A cell $c\in \mathbf{D}^{n}(f, \mathcal{C})$ is actually subdivided by the cells in $\mathbf{D}^{n+k}(f, \mathcal{C})$ ``in the same way'' as the cell $f^n(c) \in \mathbf{D}^{0}(f, \mathcal{C})$ by the cells in $\mathbf{D}^{k}(f, \mathcal{C})$. 

For convenience, we record \cite[Proposition~12.5~(ii)]{bonk2017expanding} here.

\begin{proposition}[M.~Bonk \& D.~Meyer \cite{bonk2017expanding}]    \label{prop:cell decomposition: invariant Jordan curve}
    Let $\juxtapose{k}{n} \in \n_0$, $f \colon S^2 \mapping S^2$ be a Thurston map, and $\mathcal{C} \subseteq S^2$ be an $f$-invariant Jordan curve with $\post{f} \subseteq \mathcal{C}$. Then every $(n+k)$-tile $X^{n+k}$ is contained in a unique $k$-tile $X^k$.
\end{proposition}

We are interested in $f$-invariant Jordan curves that contain $\post{f}$, since for such a Jordan curve $\mathcal{C}$, we get a cellular Markov partition $\big( \mathbf{D}^{1}(f,\mathcal{C}), \mathbf{D}^{0}(f,\mathcal{C}) \bigr)$ for $f$. 
According to Example~15.11 in \cite{bonk2017expanding}, such $f$-invariant Jordan curves containing $\post{f}$ need not exist. 
However, M.~Bonk and D.~Meyer \cite[Theorem~15.1]{bonk2017expanding} proved that there exists an $f^n$-invariant Jordan curve $\mathcal{C}$ containing $\post{f}$ for each sufficiently large $n$ depending on $f$. We record it below for the convenience of the reader.

\begin{lemma}[M.~Bonk \& D.~Meyer \cite{bonk2017expanding}]    \label{lem:invariant_Jordan_curve}
    Let $f \colon S^2 \mapping S^2$ be an expanding Thurston map, and $\widetilde{\mathcal{C}} \subseteq S^2$ be a Jordan curve with $\post{f} \subseteq \widetilde{\mathcal{C}}$. Then there exists an integer $N(f, \widetilde{\mathcal{C}}) \in \n$ such that for each $n \geqslant N(f,\widetilde{\mathcal{C}})$ there exists an $f^n$-invariant Jordan curve $\mathcal{C}$ isotopic to $\widetilde{\mathcal{C}}$ rel. $\post{f}$.
\end{lemma}

We summarize the existence, uniqueness, and some basic properties of equilibrium states for expanding Thurston maps in the following theorem.

\begin{theorem}[\cite{li2018equilibrium}]     \label{thm:properties of equilibrium state}
    Let $f \colon S^2 \mapping S^2$ be an expanding Thurston map and $d$ a visual metric on $S^2$ for $f$. 
    Let $\juxtapose{\phi}{\gamma} \in C^{0,\holderexp}(S^2,d)$ be real-valued \holder continuous functions with an exponent $\holderexp \in (0,1]$. 
    Then the following statements are satisfied:
    \begin{enumerate}[label=\rm{(\roman*)}]
        \smallskip
        
        \item     \label{item:thm:properties of equilibrium state:existence and uniqueness}
        There exists a unique equilibrium state $\mu_{\phi}$ for the map $f$ and the potential $\phi$.

        \smallskip
        
        \item     \label{item:thm:properties of equilibrium state:derivative}
        For each $t \in \real$, we have $\frac{\mathrm{\mathrm{d}}}{\mathrm{d}t}P(f,\phi + t \gamma) = \int \! \gamma \,\mathrm{d}\mu_{\phi + t\gamma}$.

        \smallskip
        
        \item     \label{item:thm:properties of equilibrium state:edge measure zero}
        If $\mathcal{C} \subseteq S^2$ is a Jordan curve containing $\post{f}$ with the property that $f^{n_{\mathcal{C}}}(\mathcal{C}) \subseteq \mathcal{C}$ for some $n_{\mathcal{C}} \in \n$, then $\mu_{\phi} \bigl( \bigcup_{i=0}^{+\infty} f^{-i}(\mathcal{C}) \bigr) = 0$.

        \smallskip

        \item     \label{item:thm:properties of equilibrium state:co-homologous}
        Let $\mu_{\gamma}$ be the unique equilibrium state for the map $f$ and the potential $\gamma$.
        Then $\mu_{\potential} = \mu_{\gamma}$ if and only if there exist a constant $K \in \real$ and a continuous function $u \in C(S^2)$ such that $\potential - \gamma = K + u \circ f - u$. 

    \end{enumerate}
\end{theorem}

Theorem~\ref{thm:properties of equilibrium state}~\ref{item:thm:properties of equilibrium state:existence and uniqueness} is part of \cite[Theorem~1.1]{li2018equilibrium}. 
Theorem~\ref{thm:properties of equilibrium state}~\ref{item:thm:properties of equilibrium state:derivative} follows immediately from \cite[Theorem~6.13]{li2018equilibrium} and the uniqueness of equilibrium states in Theorem~\ref{thm:properties of equilibrium state}~\ref{item:thm:properties of equilibrium state:existence and uniqueness}. 
Theorem~\ref{thm:properties of equilibrium state}~\ref{item:thm:properties of equilibrium state:edge measure zero} was established in \cite[Proposition~7.1]{li2018equilibrium}.
Theorem~\ref{thm:properties of equilibrium state}~\ref{item:thm:properties of equilibrium state:co-homologous} was established in \cite[Theorem~8.2]{li2018equilibrium}. 
\section{The Assumptions}
\label{sec:The Assumptions}

We state below the hypotheses under which we will develop our theory in most parts of this paper. 
We will selectively use some of the following assumptions in the remaining part of this paper.

\begin{assumptions}
\quad
    \begin{enumerate}[label=\textrm{(\arabic*)}]
        \smallskip

        \item \label{assumption:expanding Thurston map}
            $f \colon S^2 \mapping S^2$ is an expanding Thurston map.

        \smallskip

        \item \label{assumption:Jordan curve}
            $\mathcal{C} \subseteq S^2$ is a Jordan curve containing $\post{f}$ with the property that there exists an integer $n_{\mathcal{C}} \in \n$ such that $f^{n_{\mathcal{C}}}(\mathcal{C}) \subseteq \mathcal{C}$ and $f^m(\mathcal{C}) \not\subseteq \mathcal{C}$ for each $m \in \oneton[n_{\mathcal{C}} - 1]$.
        
        \smallskip

        \item \label{assumption:subsystem}
            $F \in \subsystem$ is a subsystem of $f$ with respect to $\mathcal{C}$.
        
        \smallskip

        \item \label{assumption:visual metric and expansion factor}
            $d$ is a visual metric on $S^2$ for $f$ with expansion factor $\Lambda > 1$. 

        \smallskip

        \item \label{assumption:holder exponent}
        $\holderexp \in (0, 1]$.

        \smallskip

        \item \label{assumption:holder potential}
        $\potential \in C^{0,\holderexp}(S^2, d)$ is a real-valued H\"{o}lder continuous function with an exponent $\holderexp$.
        Denote $\minenergy \define \min\limits_{\mu\in \mathcal{M}(S^2,f)} \int \! \potential \,\mathrm{d}\mu$, $\maxenergy \define \max\limits_{\mu\in \mathcal{M}(S^2,f)} \int \! \potential \,\mathrm{d}\mu$, and $\mathcal{I}_{\potential} \define [\minenergy, \maxenergy]$.

        \smallskip

        \item \label{assumption:equilibrium state}
            $\mu_{\potential}$ is the unique equilibrium state for the map $f$ and the potential $\potential$. Denote $\gamma_{\potential} \define \int \! \potential \,\mathrm{d}\mu_{\potential}$. 

        \smallskip

        \item \label{assumption:0-edge}
            $e^0 \in \mathbf{E}^0(f,\mathcal{C})$ is a $0$-edge.
    \end{enumerate}
\end{assumptions}

Note that the notion of subsystems in \ref{assumption:subsystem} will be introduced in Definition~\ref{def:subsystems}.  
    
Observe that by Lemma~\ref{lem:invariant_Jordan_curve}, for each $f$ in \ref{assumption:expanding Thurston map}, there exists at least one Jordan curve $\mathcal{C}$ that satisfies \ref{assumption:Jordan curve}. 
Since for a fixed $f$, the number $n_{\mathcal{C}}$ is uniquely determined by $\mathcal{C}$ in \ref{assumption:Jordan curve}, in the remaining part of the paper, we will say that a quantity depends on $\mathcal{C}$ even if it also depends on $n_{\mathcal{C}}$.

Recall that the expansion factor $\Lambda$ of a visual metric $d$ on $S^2$ for $f$ is uniquely determined by $d$ and $f$. 
We will say that a quantity depends on $f$ and $d$ if it depends on $\Lambda$.

In the discussion below, depending on the conditions we will need, we will sometimes say ``Let $f$, $\mathcal{C}$, $d$, $\potential$ satisfy the Assumptions in Section~\ref{sec:The Assumptions}.", and sometimes say ``Let $f$ and $\mathcal{C}$ satisfy the Assumptions in Section~\ref{sec:The Assumptions}.'', etc. 
\section{Subsystems}
\label{sec:Subsystems}

In this section, we set the stage for subsequent discussions. 
We start with the definition of subsystems of expanding Thurston maps and the associated cell decompositions.

In Subsection~\ref{sub:Basic properties of subsystems}, we prove a few preliminary results for subsystems, which will be used frequently later. 
We categorize these results according to their assumptions.

In Subsection~\ref{sub:Local degree}, we introduce the notion of the local degree for a subsystem. 
Although a subsystem may not be a branched covering map or possess structures as good as those found in Thurston maps, such as the structures of flowers (see \cite[Section~5.6]{bonk2017expanding}), we propose a solution to these challenges, namely, rather than relying on a single number, we use a $2 \times 2$ matrix to represent the local degree, consisting of $4$ numbers. Our findings suggest that this approach is a natural way to capture the behavior of the local degrees under iteration similar to \eqref{eq:composed local degree for branched covering map} (see Lemma~\ref{lem:iteration of local degree matrix}).

In Subsection~\ref{sub:Tile matrix}, we introduce tile matrices of subsystems, which are useful tools to describe the combinatorial information of subsystems.

In Subsection~\ref{sub:Irreducible and primitive subsystems}, we define irreducible (\resp strongly irreducible) subsystems, which have nice dynamical properties. 
We will use these concepts frequently in Section~\ref{sec:Thermodynamic formalism for subsystems}. 
These requirements can be weakened, but for the brevity of the presentation, we will require the subsystem to be irreducible or strongly irreducible.

In Subsection~\ref{sub:Distortion lemmas}, we investigate some distortion estimates that serve as the cornerstones for the analysis of thermodynamic formalism for subsystems.

\subsection{Definition of subsystems}%
\label{sub:Definition of subsystems}

We first introduce the definition of subsystems along with relevant concepts and notations that will be used frequently throughout this section.
Additionally, we provide examples to illustrate these notions.

\begin{definition}    \label{def:subsystems}
	Let $f \colon S^2 \mapping S^2$ be an expanding Thurston map with a Jordan curve $\mathcal{C}\subseteq S^2$ satisfying $\post{f} \subseteq \mathcal{C}$. 
	We say that a map $F \colon \domF \mapping S^2$ is a \emph{subsystem of $f$ with respect to $\mathcal{C}$} if $\domF = \bigcup \mathfrak{X}$ for some non-empty subset $\mathfrak{X} \subseteq \Tile{1}$ and $F = f|_{\domF}$.
	We denote by $\subsystem$ the set of all subsystems of $f$ with respect to $\mathcal{C}$.
	Define \[
		\operatorname{Sub}_{*}(f, \mathcal{C}) \define \{ F \in \subsystem \describe \domF \subseteq F(\domF) \}.
	\]
\end{definition}

Consider a subsystem $F \in \subsystem$. 
For each $n \in \n_0$, we define the \emph{set of $n$-tiles of $F$} to be
\begin{equation}    \label{eq:definition of tile of subsystem}
	\Domain{n} \define \{ X^n \in \Tile{n} \describe X^n \subseteq F^{-n}(F(\domF)) \},
\end{equation}
where we set $F^0 \define \id{S^{2}}$ when $n = 0$. We call each $X^n \in \Domain{n}$ an \emph{$n$-tile} of $F$. 
We define the \emph{tile maximal invariant set} associated with $F$ with respect to $\mathcal{C}$ to be
\begin{equation}    \label{eq:def:limitset}
	\limitset(F, \mathcal{C}) \define \bigcap_{n \in \n} \Bigl( \bigcup \Domain{n} \Bigr), 
\end{equation}
which is a compact subset of $S^{2}$. 
Indeed, $\limitset(F, \mathcal{C})$ is forward invariant with respect to $F$, namely, $F(\limitset(F, \mathcal{C})) \subseteq \limitset(F, \mathcal{C})$ (see Proposition~\ref{prop:subsystem:properties}~\ref{item:subsystem:properties:limitset forward invariant}). 
We denote by $\limitmap$ the map $F|_{\limitset(F, \mathcal{C})} \colon \limitset(F, \mathcal{C}) \mapping \limitset(F, \mathcal{C})$.

Let $\juxtapose{X^0_{\black}}{X^0_{\white}} \in \mathbf{X}^0(f, \mathcal{C})$ be the black $0$-tile and the white $0$-tile, respectively. 
We define the \emph{color set of $F$} as \[
	\colourset \define \bigl\{ \colour \in \colours \describe X^0_{\colour} \in \Domain{0} \bigr\}.
\]
For each $n \in \n_0$, we define the \emph{set of black $n$-tiles of $F$} as\[
	\bFTile{n} \define \bigl\{ X \in \Domain{n} \describe F^{n}(X) = X^0_{\black} \bigr\},
\] 
and the \emph{set of white $n$-tiles of $F$} as\[
	\wFTile{n} \define \bigl\{ X \in \Domain{n} \describe F^{n}(X) = X^0_{\white} \bigr\}. 
\]
Moreover, for each $n \in \n_0$ and each pair of colors $\juxtapose{\colour}{\ccolour} \in \colours$ we define
\[
	\ccFTile{n}{\colour}{\ccolour} \define \bigl\{ X \in \cFTile{n} \describe X \subseteq X^0_{\ccolour} \bigr\}.
\]
In other words, for example, a tile $X \in \ccFTile{n}{\black}{\white}$ is a \emph{black $n$-tile of $F$ contained in $\whitetile$}, i.e., an $n$-tile of $F$ that is contained in the white $0$-tile $X^0_{\white}$ as a set, and is mapped by $F^{n}$ onto the black $0$-tile $\blacktile$.

By abuse of notation, we often omit $(F, \mathcal{C})$ in the notations above when it is clear from the context.

\begin{remark}\label{rem:not branch covering map}
	Note that an expanding Thurston map $f \colon S^2 \mapping S^2$ itself is a subsystem of $f$ with respect to every Jordan curve $\mathcal{C} \subseteq S^2$ satisfying $\post{f} \subseteq \mathcal{C}$, and we have $\Dom{n}(f, \mathcal{C}) = \Tile{n}$ for each $n \in \n_0$ and $\Omega(f, \mathcal{C}) = S^2$ in this case. 
	Generally, the map $F$ defined in Definition~\ref{def:subsystems} is not a branched covering map and $\domF$ may not be connected. 
	Therefore, the results in \cite{das2021thermodynamic} cannot be applied directly to our context.
\end{remark}

We discuss the following examples separately.

\begin{example}    \label{exam:subsystems}
Let $f$, $\mathcal{C}$, $F$ satisfy the Assumptions in Section~\ref{sec:The Assumptions}.
	\begin{enumerate}[label=(\roman*)]
		\smallskip

		\item     \label{item:exam:subsystems:emptyset}
			The map $F$ satisfies $\domF = X^1_{\black} \subseteq X^0_{\white}$ for some $X^1_{\black} \in \mathbf{X}^1_{\black}(f, \mathcal{C})$. 
			In this case, $F$ is not surjective, and $\limitset$ is an empty set.
		
		\smallskip

		\item     \label{item:exam:subsystems:singleton}
			The map $F$ satisfies $\domF = X^1_{\colour} \subseteq X^0_{\colour}$ for some $\colour \in \colours$ and $X^1_{\colour} \in \mathbf{X}^1_{\colour}(f, \mathcal{C})$. 
			In this case, $F$ is not surjective, and one can check that $\limitset$ is non-empty and $\card{\limitset} = 1$.

		\smallskip

		\item     \label{item:exam:subsystems:cantor set}
			The map $F$ satisfies $\domF = X^1_{\colour} \cup \widetilde{X}^1_{\colour} \subseteq X^0_{\colour}$ for some $\colour \in \colours$ and disjoint $\juxtapose{X^1_{\colour}}{\widetilde{X}^1_{\colour}} \in \mathbf{X}^1_{\colour}(f, \mathcal{C})$. 
			In this case, $F$ is not surjective, and one can check that $\limitset$ is non-empty and uncountable. 

		\smallskip

		\item     \label{item:exam:subsystems:limitmap not surjective}
			The map $F$ satisfies $\domF = X^1_{\black} \cup \widetilde{X}^1_{\black}$ for some $\juxtapose{X^1_{\black}}{\widetilde{X}^1_{\black}} \in \mathbf{X}^1_{\black}(f, \mathcal{C})$ satisfying $X^1_{\black} \subseteq \inte[\big]{X^0_{\black}}$ and $\widetilde{X}^1_{\black} \subseteq \inte[\big]{X^0_{\white}}$. 
			In this case, $F$ is not surjective and $\limitset = \{p,\, q\}$ for some $p \in X^1_{\black}$ and $q \in \widetilde{X}^1_{\black}$. 
			One sees that $F(\limitset) = \{p\} \ne \limitset$ since $F(p) = p$ and $F(q) = p$. 

		\smallskip

		\item     \label{item:exam:subsystems:strongly irreducible but not primitive} 
			The map $F$ satisfies $\domF = X^1_{\black} \cup X^1_{\white}$ for some $X^1_{\black} \in \mathbf{X}^1_{\black}(f, \mathcal{C})$ and $X^1_{\white} \in \cTile{1}{\white}$ satisfying $X^1_{\black} \subseteq \inte[\big]{X^0_{\white}}$ and $X^1_{\white} \subseteq \inte[\big]{X^0_{\black}}$. 
			In this case, $F$ is surjective and $\limitset = \{p,\, q\}$ for some $p \in X^1_{\black}$ and $q \in X^1_{\white}$. 
			One sees that $F(\limitset) = \limitset$ since $F(p) = q$ and $F(q) = p$.

		\smallskip

		\item     \label{item:exam:subsystems:Sierpinski carpet}
			The map $F \colon \domF \mapping S^2$ is represented by Figure~\ref{fig:subsystem:example:carpet}.
			Here $S^{2}$ is identified with a pillow that is obtained by gluing two squares together along their boundaries.
			Moreover, each square is subdivided into $3\times 3$ subsquares, and $\dom{F}$ is obtained from $S^2$ by removing the interior of the middle subsquare $X^{1}_{\white} \in \cTile{1}{\white}$ and $X^{1}_{\black} \in \cTile{1}{\black}$ of the respective squares. 
			In this case, $\limitset$ is a \sierpinski carpet. 
			It consists of two copies of the standard square \sierpinski carpet glued together along the boundaries of the squares.
			\begin{figure}[H]
				\centering
				\begin{overpic}[width=12cm, tics=20]{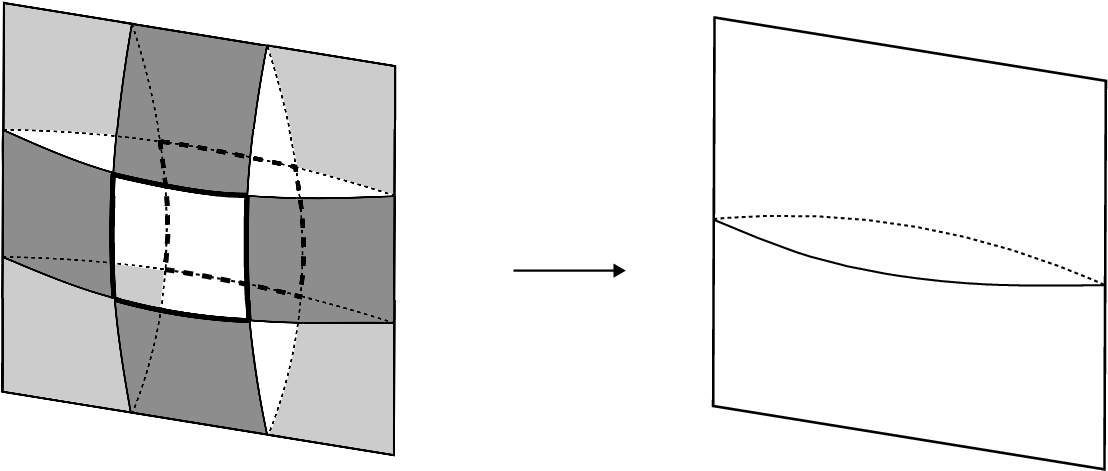}
					\put(50,20){$F$}
					\put(16,41){$\domF$}
					\put(83,40){$S^2$}
				\end{overpic}
				\caption{A \sierpinski carpet subsystem.} 
				\label{fig:subsystem:example:carpet}
			\end{figure}

		\smallskip

		\item     \label{item:exam:subsystems:Sierpinski gasket}
			The map $F \colon \domF \mapping S^2$ is represented by Figure~\ref{fig:subsystem:example:gasket}.
			Here $S^{2}$ is identified with a pillow that is obtained by gluing two equilateral triangles together along their boundaries.
			Moreover, each triangle is subdivided into $4$ small equilateral triangles, and $\dom{F}$ is obtained from $S^2$ by removing the interior of the middle small triangle $X^{1}_{\black} \in \cTile{1}{\black}$ and $X^{1}_{\white} \in \cTile{1}{\white}$ of the respective triangle. 
			In this case, $\limitset$ is a \sierpinski gasket. 
			It consists of two copies of the standard \sierpinski gasket glued together along the boundaries of the triangles.
			\begin{figure}[H]
				\centering
				\begin{overpic}[width=12cm, tics=20]{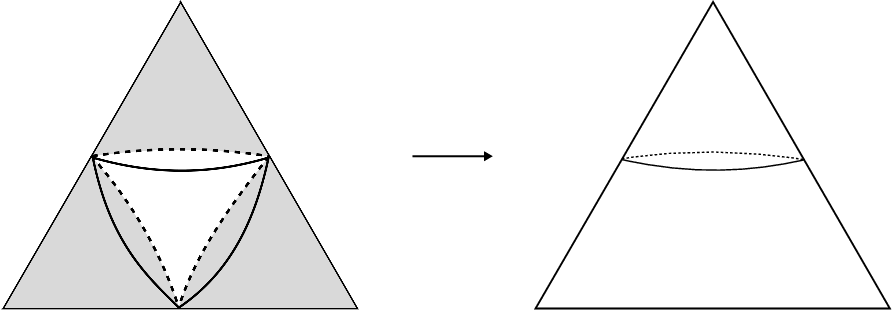}
					\put(50,19){$F$}
					\put(0,23){$\domF$}
					\put(90,24){$S^2$}
				\end{overpic}
				\caption{A \sierpinski gasket subsystem.} 
				\label{fig:subsystem:example:gasket}
			\end{figure}
	\end{enumerate}
\end{example}

\subsection{Basic properties of subsystems}%
\label{sub:Basic properties of subsystems}

In this subsection, we collect and prove a few preliminary results for subsystems.

\begin{proposition}    \label{prop:subsystem:properties}
	Let $f$, $\mathcal{C}$, $F$ satisfy the Assumptions in Section~\ref{sec:The Assumptions}.
	Consider arbitrary $\juxtapose{n}{k} \in \n_0$.
	Then the following statements hold:
	\begin{enumerate}[label=\rm{(\roman*)}]
		\smallskip
		\item    \label{item:subsystem:properties:homeo} 
			If $X \in \Domain{n + k}$ is any $(n + k)$-tile of $F$, then $F^k(X)$ is an $n$-tile of $F$, and $F^{k}|_{X}$ is a homeomorphism of $X$ onto $F^{k}(X)$. 
			As a consequence, we have $\bigl\{ F^{k}(X) \describe X \in \Domain{n + k} \bigr\} \subseteq \Domain{n}$.
		
		\smallskip
		
		\item    \label{item:subsystem:properties:preimage} 
			For each $X^n \in \Domain{n}$, the set $\bigcup \bigl\{ Y \in \Domain{n + k} \describe Y \subseteq F^{-k}(X^n) \bigr\}$ is equal to the union of all $(n + k)$-tiles (of $F$) $X \in \Domain{n + k}$ with $F^{k}(X) = X^n$.
		
		\smallskip

		\item     \label{item:subsystem:properties:limitset forward invariant}
			The tile maximal invariant set $\limitset$ is forward invariant with respect to $F$, i.e., $F(\limitset) \subseteq \limitset$.
	\end{enumerate}
\end{proposition}   
\begin{proof} 
	\ref{item:subsystem:properties:homeo} 
	If $X \in \Domain{n + k}$, then $X \in \Tile{n + k}$ and $X \subseteq F^{-(n + k)}(F(\domF)) $ (recall \eqref{eq:definition of tile of subsystem}).
	Thus $F^k(X) \subseteq F^k\bigl(F^{-(n + k)}(F(\domF))\bigr) \subseteq F^{-n}(F(\domF))$. Since $F^{k}|_{X} = f^{k}|_{X}$, it follows from Proposition~\ref{prop:properties cell decompositions}~\ref{item:prop:properties cell decompositions:cellular} that $F^k(X)$ is an $n$-tile of $F$ and $F^{k}|_{X}$ is a homeomorphism of $X$ onto $F^{k}(X)$.
	
	\smallskip

	\ref{item:subsystem:properties:preimage}  
	It suffices to show that if $X^n \in \Domain{n}$, then\[
		\bigl\{ X \in \Domain{n + k} \describe F^{k}(X) = X^n \bigr\} = \bigl\{ X \in \Domain{n + k} \describe X \subseteq F^{-k}(X^n) \bigr\}.	
	\]
	If $X \in \Domain{n + k}$ and $F^k(X) = X^n$, then $X \subseteq F^{-k}\bigl( F^{k}(X) \bigr) \subseteq F^{-k}(X^n)$. 
	For the converse direction, suppose that $X \in \Domain{n + k}$ and $X \subseteq F^{-k}(X^n)$. Then $F^{k}(X) \subseteq F^{k}(F^{-k}(X^n)) \subseteq X^n$. Thus $F^{k}(X) = X^n$ since both $X^n$ and $F^{k}(X)$ are $n$-tiles by statement~(i).

	\smallskip

	\ref{item:subsystem:properties:limitset forward invariant} 
	We write $\limitset^n \define \domain{n}$ for each $n \in \n_0$. Then it follows from statement~\ref{item:subsystem:properties:homeo} that $F(\limitset^{n+1}) \subseteq \limitset^{n}$ for each $n \in \n_{0}$. Thus $F(\limitset) = F\Bigl( \bigcap\limits_{n = 1}^{+\infty} \limitset^{n} \Bigr) \subseteq \bigcap\limits_{n = 1}^{+\infty} F(\limitset^{n}) \subseteq \bigcap\limits_{n = 0}^{+\infty} \limitset^{n} \subseteq \limitset$ by \eqref{eq:def:limitset}.
\end{proof}

\begin{proposition}    \label{prop:subsystem:properties invariant Jordan curve}
	Let $f$, $\mathcal{C}$, $F$ satisfy the Assumptions in Section~\ref{sec:The Assumptions}. 
	We assume in addition that $f(\mathcal{C}) \subseteq \mathcal{C}$.
	Then the following statements hold:
	\begin{enumerate}[label=\rm{(\roman*)}]
		\smallskip
		\item    \label{item:subsystem:properties invariant Jordan curve:decreasing relation of domains} 
			$\domain{n + k} \subseteq \domain{n} \subseteq \domain{1} = \domF$ for all $\juxtapose{n}{k} \in \n$.

		\smallskip
		
		\item    \label{item:subsystem:properties invariant Jordan curve:relation between color and location of tile} 
			$\cFTile{m} = \ccFTile{m}{\colour}{\black} \cup \ccFTile{m}{\colour}{\white}$ for each $m \in \n_{0}$ and each $\colour \in \colours$.
				
		\smallskip

		\item     \label{item:subsystem:properties invariant Jordan curve:backward invariant limitset outside invariant Jordan curve}
			$F^{-1}(\limitset \setminus \mathcal{C}) \subseteq \limitset \setminus \mathcal{C}$.
			Moreover, if $\mathcal{C} \subseteq \interior{\domF}$, then $F^{-1}(\mathcal{C}) \subseteq \limitset$, and in particular, $F^{-1}(\limitset) = \limitset$.
		
		\smallskip

		\item     \label{item:prop:subsystem:properties invariant Jordan curve:limitset no isolated points}
			If $\limitset$ satisfies $\card[\big]{\limitset \cap \inte[\big]{X^0_{\colour}}} \geqslant 2$ for each $\colour \in \colourset$, then $\limitset$ has no isolated points, i.e., $\limitset$ is perfect.
	\end{enumerate}
\end{proposition}
\begin{proof}
	\ref{item:subsystem:properties invariant Jordan curve:decreasing relation of domains} 
	By \eqref{eq:definition of tile of subsystem}, it is clear that $\domain{1} = \domF$ and $\domain{n} \subseteq \domF$ for each $n \in \n$. 
	Hence, it suffices to show that
	\begin{equation}    \label{eq:temp:domain relation}
		\domain{n + 1} \subseteq \domain{n}
	\end{equation}
	for each $n \in \n$. We prove \eqref{eq:temp:domain relation} by induction on $n \in \n$. 

	For $n = 1$, \eqref{eq:temp:domain relation} holds trivially. 

	We now assume that \eqref{eq:temp:domain relation} holds for $n = \ell$ for some $\ell \in \n$. 
	Let $X^{\ell + 2} \in \Domain{\ell + 2}$ be arbitrary. 
	It suffices to show that $X^{\ell + 2} \subseteq X^{\ell + 1}$ for some $X^{\ell + 1} \in \Domain{\ell + 1}$.
	Since $X^{\ell + 2} \subseteq \domain{1}$ and $f(\mathcal{C}) \subseteq \mathcal{C}$, by Proposition~\ref{prop:cell decomposition: invariant Jordan curve} and Lemma~\ref{lem:intersection of cells}~\ref{item:lem:intersection of cells:intersection with union of tiles}, there exists a unique $X^1 \in \Domain{1}$ containing $X^{\ell + 2}$. 
	Denote $Y^{\ell + 1} \define F\bigl(X^{\ell + 2}\bigr)$.
	Note that $Y^{\ell + 1} \in \Domain{\ell + 1}$ by Proposition~\ref{prop:subsystem:properties}~\ref{item:subsystem:properties:homeo}.
	Then by the induction hypothesis, we have $Y^{\ell + 1} \subseteq \domain{\ell}$. 
	Similarly, there exists a unique $Y^{\ell} \in \Domain{\ell}$ containing $Y^{\ell + 1}$. 
	We set $X^{\ell + 1} \define (F|_{X^{1}})^{-1}\bigl(Y^{\ell}\bigr)$. 
	It is clear that $X^{\ell + 2} \subseteq X^{\ell + 1}$ since $Y^{\ell + 1} \subseteq Y^{\ell}$. 
	By Lemma~\ref{lem:cell mapping properties of Thurston map}~\ref{item:lem:cell mapping properties of Thurston map:i}, we have $X^{\ell + 1} \in \Tile{\ell + 1}$. 
	Then it follows from \eqref{eq:definition of tile of subsystem} and \[
		X^{\ell + 1} \subseteq 
		F^{-1}\bigl(Y^{\ell}\bigr) \subseteq 
		F^{-1}\bigl( F^{-\ell}(F(\domF)) \bigr) = F^{-(\ell + 1)}(F(\domF))
	\]
	that $X^{\ell + 1} \in \Domain{\ell + 1}$.
	The induction step is now complete.

	\smallskip

	\ref{item:subsystem:properties invariant Jordan curve:relation between color and location of tile} 
	Let $m \in \n_{0}$ and $\colour \in \colours$ be arbitrary. 
	Since $f(\mathcal{C}) \subseteq \mathcal{C}$, it follows immediately from Proposition~\ref{prop:cell decomposition: invariant Jordan curve} that each $m$-tile $X^{m}_{\colour} \in \cTile{m}{\colour}$ is contained in exactly one of $\blacktile$ and $\whitetile$. Thus \ref{item:subsystem:properties invariant Jordan curve:relation between color and location of tile} holds.

	\smallskip

	\ref{item:subsystem:properties invariant Jordan curve:backward invariant limitset outside invariant Jordan curve} 	
	Let $y \in \limitset \setminus \mathcal{C}$ be arbitrary. 
	Since $f(\mathcal{C}) \subseteq \mathcal{C}$ and $y \notin \mathcal{C}$, it suffices to show that $x \in \limitset$ for each $x \in F^{-1}(y)$. 
	Let $x \in F^{-1}(y)$ be arbitrary. 
	Without loss of generality we may assume that $y \in \inte[\big]{X^0_{\black}}$.
	Then by Lemma~\ref{lem:cell mapping properties of Thurston map}~\ref{item:lem:cell mapping properties of Thurston map:ii} there exists a unique $1$-tile $X^{1} \in \Domain{1}$ with $x \in X^{1}$ and $F(X^{1}) = X^{0}_{\black}$.
	Since $y \in \limitset \cap \inte[\big]{X^0_{\black}}$, by \eqref{eq:def:limitset} and Lemma~\ref{lem:visual_metric}~\ref{item:lem:visual_metric:diameter of cell}, for a sufficiently large integer $n_{0} \in \n$ there exists a sequence of tiles $\{Y^{n_{0} + k}\}_{k \in \n}$ satisfying $Y^{n_{0} + k} \in \Domain{n_0 + k}$ and $y \in Y^{n_0 + k + 1} \subseteq Y^{n_0 + k} \subseteq \inte[\big]{X^0_{\black}}$.
	Then it follows from Proposition~\ref{prop:subsystem:properties}~\ref{item:subsystem:properties:homeo} and Lemma~\ref{lem:cell mapping properties of Thurston map}~\ref{item:lem:cell mapping properties of Thurston map:i} that 
	$x \in X^{n_0 + k + 1} \define (F|_{X^{1}})^{-1} \bigl( Y^{n_0 + k} \bigr) \in \Domain{n_0 + k + 1}$ for each $k \in \n$. 
	Hence $x \in \bigcap_{k \in \n} X^{n_0 + k + 1} \subseteq \limitset$ by \eqref{eq:def:limitset} and statement~\ref{item:subsystem:properties invariant Jordan curve:decreasing relation of domains}.
	
	We now assume that $\mathcal{C} \subseteq \interior{\domF}$. We first show that $F^{-1}(\mathcal{C}) \subseteq \limitset$.

	Let $y \in \mathcal{C}$ and $x \in F^{-1}(y)$ be arbitrary. It suffices to show that $x \in \limitset$. 
	Since $x \in \domF = \bigcup \Domain{1}$, there exists a $X^{1} \in \Domain{1}$ such that $x \in X^{1}$.
	Then we have $F(X^{1}) = X^{0}_{\colour}$ for some $\colour \in \colours$.
	Since $y \in \mathcal{C} \subseteq X^{0}_{\colour}$, by Proposition~\ref{prop:cell decomposition: invariant Jordan curve}, there exists $Y^{1} \in \Tile{1}$ such that $y \in Y^{1} \subseteq X^{0}_{\colour}$.
	We claim that $Y^{1} \in \Domain{1}$. Indeed, since $y \in \mathcal{C} \subseteq \interior{\domF}$, we have $Y^{1} \cap \interior{\domF} \ne \emptyset$. 
	Thus $\inte{Y^{1}} \cap \interior{\domF} \ne \emptyset$. Then $\inte{Y^{1}} \cap \domF \ne \emptyset$ and it follows from Lemma~\ref{lem:intersection of cells}~\ref{item:lem:intersection of cells:intersection with union of tiles}  and Definition~\ref{def:cell decomposition}~(iii) that $Y^{1} \in \Domain{1}$. 
	Applying Proposition~\ref{prop:subsystem:properties}~\ref{item:subsystem:properties:homeo} and Lemma~\ref{lem:cell mapping properties of Thurston map}~\ref{item:lem:cell mapping properties of Thurston map:i}, we have $X^{2} \define (F|_{X^{1}})^{-1}(Y^{1}) \in \Domain{2}$ and $x \in X^{2} \subseteq X^{1}$ since $F(x) = y \in Y^{1}$. 
	Then we have that $F^{2}(X^{2}) = X^{0}_{\ccolour}$ for some $\ccolour \in \colours$.
	Similarly, since $F(y) \in F(\mathcal{C}) \subseteq \mathcal{C}$, there exists a $1$-tile $Z^{1} \in \Tile{1}$ such that $F(y) \in Z^{1} \subseteq X^{0}_{\ccolour}$.
	By the same argument as before, with $y$ and $Y^{1}$ replaced by $F(y)$ and $Z^{1}$, respectively, we deduce that $Z^{1} \in \Domain{1}$.
	Then we have $X^{3} \define (F^{2}|_{X^{2}})^{-1}(Z^{1}) \in \Domain{3}$ and $x \in X^{3} \subseteq X^{2}$ since $F^2(x) = F(y) \in Z^{1}$.
	Thus by induction, there exists a sequence of tiles $\{X^{n}\}_{n \in \n}$ such that $X^{n} \in \Domain{n}$ and $x \in X^{n + 1} \subseteq X^{n}$ for each $n \in \n$. 
	Hence $x \in \bigcap_{n \in \n} X^{n} \subseteq \limitset$ by \eqref{eq:def:limitset}. 

	To verify that $F^{-1}(\limitset) = \limitset$, by Proposition~\ref{prop:subsystem:properties}~\ref{item:subsystem:properties:limitset forward invariant}, it suffices to show that $F^{-1}(\limitset) \subseteq \limitset$.
	Since $F^{-1}(\limitset \setminus \mathcal{C}) \subseteq \limitset \setminus \mathcal{C} \subseteq \limitset$ and $F^{-1}(\mathcal{C}) \subseteq \limitset$, the proof is complete.

	\smallskip

	\ref{item:prop:subsystem:properties invariant Jordan curve:limitset no isolated points} 
	We fix a visual metric $d$ on $S^2$ for $f$. 
	It suffices to show that for each $p \in \limitset$ and each $r > 0$, the set $(B_{d}(p, r) \setminus \{p\}) \cap \limitset$ is non-empty. Since $p \in \limitset$, for each $n \in \n_0$ there exists $X^{n} \in \Domain{n}$ which contains $p$. Thus by Lemma~\ref{lem:visual_metric}~\ref{item:lem:visual_metric:diameter of cell}, for each sufficiently large integer $n$ there exists $X^{n} \in \Domain{n}$ such that $p \in X^{n}$ and $X^{n} \subseteq B_{d}(p, r)$. We fix such an integer $n \in \n_0$ and an $n$-tile $X^{n} \in \Domain{n}$. 
	Then $F^{n}(X^n) = X^0_{\colour}$ for some $\colour \in \colourset$, and $F^{n}|_{X^n}$ is a homeomorphism of $X^n$ onto $X^0_{\colour}$. 
	Then it follows from statement~\ref{item:subsystem:properties invariant Jordan curve:backward invariant limitset outside invariant Jordan curve} that $(F^{n}|_{X^n})^{-1}(\limitset \setminus \mathcal{C}) \subseteq F^{-n}(\limitset \setminus \mathcal{C}) \subseteq \limitset \setminus \mathcal{C} \subseteq \limitset$. 
	Thus, by the assumption in statement~\ref{item:prop:subsystem:properties invariant Jordan curve:limitset no isolated points}, we have $\card{X^{n} \cap \limitset} \geqslant \card[\big]{\limitset \cap \inte[\big]{X^0_{\colour}}} \geqslant 2$, which completes the proof.
\end{proof}

\begin{proposition}    \label{prop:sursubsystem properties}
	Let $f$ and $\mathcal{C}$ satisfy the Assumptions in Section~\ref{sec:The Assumptions}. 
	We assume in addition that $f(\mathcal{C}) \subseteq \mathcal{C}$. 
	Consider $F \in \sursubsystem$. 
	Then the following statements hold:
	\begin{enumerate}[label=\rm{(\roman*)}]
		\smallskip
		\item    \label{item:prop:sursubsystem properties:F maps n+1 tile to n tile} 
			$\bigl\{ F(X) \describe X \in \Domain{n + 1} \bigr\} = \Domain{n}$ for each $n \in \n_0$.
			In particular, if $F(\domF) = S^2$, then $\cFTile{n} \ne \emptyset$ for each $\colour \in \colours$ and each $n \in \n_{0}$.

		\smallskip
		
		\item    \label{item:prop:sursubsystem properties:property of domain and limitset} 
			$F(\limitset) = \limitset \ne \emptyset$.
	\end{enumerate}
\end{proposition}
\begin{proof}
	\ref{item:prop:sursubsystem properties:F maps n+1 tile to n tile} 
	Let $n \in \n_0$ be arbitrary. We first establish
	\begin{equation}    \label{eq:F maps n+1 tile to n tile}
		\bigl\{ F(X) \describe X \in \Domain{n + 1} \bigr\} = \Domain{n}.
	\end{equation}
	It follows from Proposition~\ref{prop:subsystem:properties}~\ref{item:subsystem:properties:homeo} that $\bigl\{ F(X) \describe X \in \Domain{n + 1} \bigr\} \subseteq \Domain{n}$. 
	Thus, it suffices to show that for each $X^n \in \Domain{n}$, there exists $X \in \Domain{n + 1}$ such that $F(X) = X^n$. 

	Fix arbitrary $X^n \in \Domain{n}$. 
	Since $\domF \subseteq F(\domF)$, we have $X^n \subseteq F(\domF)$ by Proposition~\ref{prop:subsystem:properties invariant Jordan curve}~\ref{item:subsystem:properties invariant Jordan curve:decreasing relation of domains}. 
	Then by Proposition~\ref{prop:cell decomposition: invariant Jordan curve}, there exists a unique $X^0 \in \Tile{0}$ containing $X^n$. Thus $\inte{X^0} \cap F(\domF)$ is non-empty as it contains $\inte{X^n}$. Noting that $F(\domF)$ is a union of $0$-tiles in $\Tile{0}$, by Lemma~\ref{lem:intersection of cells}, we conclude that there exists $X^1 \in \Domain{1}$ such that $F(X^1) = X^0$. 

	We denote by $X$ the set $(F|_{X^1})^{-1}(X^n)$. Since $F(X) = X^n$, it suffices to show that $X \in \Domain{n + 1}$. 
	Note that $X \in \Tile{n + 1}$ by Proposition~\ref{prop:subsystem:properties}~\ref{item:subsystem:properties:homeo} and Lemma~\ref{lem:cell mapping properties of Thurston map}~\ref{item:lem:cell mapping properties of Thurston map:i}. 
	Since\[
		X = (F|_{X^1})^{-1}(X^n) \subseteq F^{-1}(X^n) \subseteq F^{-(n + 1)}(F(\domF)),
	\]
	it follows from \eqref{eq:definition of tile of subsystem} that $X \in \Domain{n + 1}$, which establishes \eqref{eq:F maps n+1 tile to n tile}.

	If $F(\domF) = S^2$, then $\Domain{0} = \Tile{0}$. Then it follows from \eqref{eq:F maps n+1 tile to n tile}, Proposition~\ref{prop:subsystem:properties}~\ref{item:subsystem:properties:homeo}, and induction that $\cFTile{n} \ne \emptyset$ for each $\colour \in \colours$.

	\smallskip

	\ref{item:prop:sursubsystem properties:property of domain and limitset}
	We write $\limitset^n \define \domain{n}$ for each $n \in \n_0$. 
	Then it follows from statement~\ref{item:prop:sursubsystem properties:F maps n+1 tile to n tile} and induction that $\limitset^{n} \ne \emptyset$ and $F(\limitset^{n+1}) = \limitset^{n}$ for each $n \in \n_{0}$. 
	Thus, by Proposition~\ref{prop:subsystem:properties invariant Jordan curve}~\ref{item:subsystem:properties invariant Jordan curve:decreasing relation of domains} and \eqref{eq:def:limitset} we have $\limitset = \bigcap_{n=1}^{+\infty} \limitset^{n} \ne \emptyset$. 

	By Proposition~\ref{prop:subsystem:properties}~\ref{item:subsystem:properties:limitset forward invariant} we have $F(\limitset) \subseteq \limitset$. 
	For the reverse inclusion $\limitset \subseteq F(\limitset)$, let $x \in \limitset$ be arbitrary. 
	For each $n \in \n$, since $x \in \limitset^{n - 1} = F(\limitset^{n})$, we can find $y_n \in \limitset^{n}$ with $F(y_n) = x$. 
	By compactness of $S^2$ and the nested property of $\sequen{\limitset^n}$, the sequence $\sequen{y_n}$ has a limit point $y \in \limitset$. The continuity of $F$ then implies $x = F(y) \in F(\limitset)$.
\end{proof}

\subsection{Local degree}%
\label{sub:Local degree}

In this subsection, we define the degree and local degrees for a subsystem. We show that local degrees are well-behaved under iteration.

\begin{definition}    \label{def:subsystem local degree}
	Let $f$, $\mathcal{C}$, $F$ satisfy the Assumptions in Section~\ref{sec:The Assumptions}.
	The \emph{degree} of $F$ is defined as \[
		\deg{(F)} \define \sup \bigl\{ \card[\big]{F^{-1}(\{y\})} \describe y \in S^2 \bigr\}.
	\]

	Fix arbitrary $x \in S^2$ and $n \in \n$. We define the \emph{black degree} of $F^n$ at $x$ as\[ 
		\ccndegF{\black}{}{n}{x} \define \card{ \neighbortile{n}{\black}{}{x} },
	\]
	where $\neighbortile{n}{\black}{}{x} \define \{X \in \bFTile{n} \describe x \in X \}$ is the \emph{set of black $n$-tiles} of $F$ at $x$.
	Similarly, we define the \emph{white degree} of $F^n$ at $x$ as\[
		\ccndegF{\white}{}{n}{x} \define \card{ \neighbortile{n}{\white}{}{x} },
	\]
	where $\neighbortile{n}{\white}{}{x} \define \{X \in \wFTile{n} \describe x \in X \}$ is the \emph{set of white $n$-tiles} of $F$ at $x$. 
	Moreover, the \emph{local degree} of $F^{n}$ at $x$ is defined as
	\begin{equation}    \label{eq:subsystem local degree greater than color degree}
		\ccndegF{}{}{n}{x} \define \max\{\ccndegF{\black}{}{n}{x}, \ccndegF{\white}{}{n}{x}\},
	\end{equation}
	and the \emph{set of $n$-tiles} of $F$ at $x$ is $\neighbortile{n}{}{}{x} \define \{X \in \Domain{n} \describe x \in X \}$.
	Furthermore, for each pair of colors $\juxtapose{\colour}{\ccolour} \in \colours$ we define
	\begin{align*}
		\neighbortile{n}{\colour}{\ccolour}{x} &\define \{X \in \ccFTile{n}{\colour}{\ccolour} \describe x \in X \},  \\
		\ccndegF{\colour}{\ccolour}{n}{x} &\define \card{ \neighbortile{n}{\colour}{\ccolour}{x} }, 
	\end{align*}
	and the \emph{local degree matrix} of $F^{n}$ at $x$ is
	\[
		\qquad \quad \Deg{n}{x} \define \begin{bmatrix}
	 		\ccndegF{\black}{\black}{n}{x} & \ccndegF{\white}{\black}{n}{x} \\
	 		\ccndegF{\black}{\white}{n}{x} & \ccndegF{\white}{\white}{n}{x}
	 	\end{bmatrix}.
	\]
\end{definition}

One sees that $\locdegF{x} \geqslant 1$ if and only if $x \in \domF$, and that if $\locdegF{x} > 1$ then $x \in \crit{f}$.

\begin{remark}\label{rem:number of tile in neighborhood for F}
	If we assume that $f(\mathcal{C}) \subseteq \mathcal{C}$, then for each $n \in \n$, by Proposition~\ref{prop:cell decomposition: invariant Jordan curve}, each $n$-tile $X^n \in \Domain{n}$ is contained in exactly one of $X^0_{\black}$ and $X^0_{\white}$ so that $\Domain{n} = \bigcup_{\juxtapose{\colour}{\ccolour} \in \colours} \ccFTile{n}{\colour}{\ccolour}$. 
	Thus it follows immediately from Definition~\ref{def:subsystem local degree} that $\neighbortile{n}{}{}{x} = \bigcup_{\juxtapose{\colour}{\ccolour} \in \colours} \neighbortile{n}{\colour}{\ccolour}{x}$ and $\sumnorm{\Deg{n}{x}} = \card{\neighbortile{n}{}{}{x}}$ for each $n \in \n$ and each $x \in S^{2}$, where\[
		\sumnorm{\Deg{n}{x}} \define \sum_{\juxtapose{\colour}{\ccolour} \in \colours} \ccndegF{\colour}{\ccolour}{n}{x}.
	\]
	Moreover, for all $n \in \n$, $x \in S^{2}$, and $\colour \in \colours$, by Proposition~\ref{prop:subsystem:properties invariant Jordan curve}~\ref{item:subsystem:properties invariant Jordan curve:relation between color and location of tile}, we have
	\begin{equation}    \label{eq:color degree is the sum of colour-position degree}
		\ccndegF{\colour}{}{n}{x} = \sum_{\ccolour \in \colours} \ccndegF{\colour}{\ccolour}{n}{x}.
	\end{equation}
\end{remark}

To describe the local degree for a subsystem, instead of a single number, we use $4$ numbers written in the form of a $2 \times 2$ matrix. Indeed, we prove that the local degree matrix is well-behaved under iteration.

\begin{lemma}    \label{lem:iteration of local degree matrix}
	Let $f$, $\mathcal{C}$, $F$ satisfy the Assumptions in Section~\ref{sec:The Assumptions}. 
	We assume in addition that $f(\mathcal{C}) \subseteq \mathcal{C}$. 
	Then for all $x \in S^2$ and $\juxtapose{n}{m} \in \n$, we have
	\begin{equation}    \label{eq:composed local degree matrix well-defined}
		\Deg{n + m}{x} = \Deg{n}{x} \Deg{m}{f^{n}(x)}.
	\end{equation}
\end{lemma}
\begin{proof}
	Let $x \in S^2$, $\juxtapose{n}{m} \in \n$, and $\juxtapose{\colour}{\ccolour} \in \colours$ be arbitrary. It suffices to show that\[
		\ccndegF{\colour}{\ccolour}{n + m}{x} = \sum_{\cccolour \in \colours} \ccndegF{\cccolour}{\ccolour}{n}{x} \ccndegF{\colour}{\cccolour}{m}{f^{n}(x)}.
	\]

	By Propositions~\ref{prop:cell decomposition: invariant Jordan curve} and \ref{prop:subsystem:properties invariant Jordan curve}~\ref{item:subsystem:properties invariant Jordan curve:decreasing relation of domains}, every $X^{n + m} \in \neighbortile{n + m}{\colour}{\ccolour}{x}$ is contained in a unique $X^n \in \Domain{n}$. 
	Moreover, $X^n$ is contained in $X^0_{\ccolour}$ since $X^{n + m}$ is contained in $X^0_{\ccolour}$ by definition. 
	Thus $X^n$ is contained in $\bigcup_{\cccolour \in \colours} \neighbortile{n}{\cccolour}{\ccolour}{x}$.

	Now let $\cccolour \in \colours$ and $X^n \in \neighbortile{n}{\cccolour}{\ccolour}{x}$ be arbitrary, and define\[
		\mathcal{X} \define \{ X^{n + m} \in \neighbortile{n + m}{\colour}{\ccolour}{x} \describe X^{n + m} \subseteq X^n \}.
	\]
	By Proposition~\ref{prop:subsystem:properties}~\ref{item:subsystem:properties:homeo}, each $X^{n + m} \in \mathcal{X}$ is mapped by $F^n$ homeomorphically to $Y^m \define F^{n}(X^{n + m})$, and $Y^m \in \neighbortile{m}{\colour}{\cccolour}{f^n(x)}$ since $Y^m \subseteq F^{n}(X^n) = X^0_{\cccolour}$ and $Y^m$ has the same color as $X^{n + m}$. 
	Moreover, it follows from Lemma~\ref{lem:cell mapping properties of Thurston map}~\ref{item:lem:cell mapping properties of Thurston map:i} and Proposition~\ref{prop:subsystem:properties}~\ref{item:subsystem:properties:homeo} that the map $F^{n}|_{X^{n}}$ induces a bijection $X^{n + m} \mapsto F^{n}(X^{n + m})$ between $\mathcal{X}$ and $\neighbortile{m}{\colour}{\cccolour}{f^n(x)}$. Thus $\card{\mathcal{X}} = \card{\neighbortile{m}{\colour}{\cccolour}{f^n(x)}} = \ccndegF{\colour}{\cccolour}{m}{f^{n}(x)}$, which is independent of $X^n$.

	Combining the two arguments above, we get\[
		\card{\neighbortile{n + m}{\colour}{\ccolour}{x}} 
		= \sum_{\cccolour \in \colours} \card{\neighbortile{n}{\cccolour}{\ccolour}{x}} \ccndegF{\colour}{\cccolour}{m}{f^{n}(x)}.
	\]
	This completes the proof.
\end{proof}

\subsection{Tile matrix}%
\label{sub:Tile matrix}

In this subsection, we introduce a $2 \times 2$ matrix called the tile matrix to describe tiles of a subsystem according to their colors and locations.
We show that the tile matrix is well-behaved under iteration.

\begin{definition}[Tile matrices]     \label{def:tile matrix} 
	Let $f$, $\mathcal{C}$, $F$ satisfy the Assumptions in Section~\ref{sec:The Assumptions}. 
	We define the \emph{tile matrix} of $F$ with respect to $\mathcal{C}$ as
	\begin{equation}    \label{eq:definition of tile matrix}
			A = A(F, \mathcal{C}) \define \begin{bmatrix}
			N_{\white \white} & N_{\black \white} \\
			N_{\white \black} & N_{\black \black}
		\end{bmatrix},
	\end{equation}
	where \[
		N_{\colour \ccolour} = N_{\colour\ccolour}(A) \define \operatorname{card}{\! \bigl\{ X \in \cFTile{1} \describe X \subseteq X^0_{\ccolour} \bigr\}} = \card[\big]{\ccFTile{1}{\colour}{\ccolour}}
	\]
	for each pair of colors $\juxtapose{\colour}{\ccolour} \in \colours$. For example, $N_{\black \white}$ is the number of black tiles in $\Domain{1}$ which are contained in the white $0$-tile $X^0_{\white}$.
\end{definition}

\begin{remark}    \label{rem:tile matrix associated with set of tiles}
	Note that the tile matrix $A(F, \mathcal{C})$ of $F$ with respect to $\mathcal{C}$ is completely determined by the set $\Domain{1}$. 
	Thus for each integer $n \in \n_0$ and each set of $n$-tiles $\mathbf{T} \subseteq \Tile{n}$, similarly, we can define the tile matrix of $\mathbf{T}$ and denote it by $A(\mathbf{T})$. 
	For example, when $\mathbf{T} = \Domain{n}$ for some $n \in \n_0$, we define\[
		A(\Domain{n}) \define 
		\begin{bmatrix}
			N_{\white \white}(\Domain{n}) & N_{\black \white}(\Domain{n}) \\
			N_{\white \black}(\Domain{n}) & N_{\black \black}(\Domain{n})
		\end{bmatrix},
	\]
	where $N_{\colour \ccolour}(\Domain{n}) \define \card[\big]{\set[\big]{ X \in \Domain{n} \describe X \in \cTile{n}{\colour}, \, X \subseteq X^0_{\ccolour} }} = \card{\ccFTile{n}{\colour}{\ccolour}}$ for each pair of colors $\juxtapose{\colour}{\ccolour} \in \colours$.
\end{remark}

\begin{proposition}    \label{prop:power of tile matrix}
	Let $f$, $\mathcal{C}$, $F$ satisfy the Assumptions in Section~\ref{sec:The Assumptions}. 
	We assume in addition that $f(\mathcal{C}) \subseteq \mathcal{C}$.
	Then for each integer $n \in \n$, we have
	\begin{equation}    \label{eq:power of tile matrix}
		 A(\Domain{n}) = \bigl(A\bigl(\Domain{1} \bigr)\bigr)^n,
	\end{equation}
	i.e., the tile matrix of $\Domain{n}$ equals the $n$-th power of the tile matrix of $\Domain{1}$.
\end{proposition}
\begin{remark}\label{rem:power of tile matrix}
	Note that if the map $F(\domF) = S^2$, then $A(\Dom{0})$ is a $2 \times 2$ identity matrix and \eqref{eq:power of tile matrix} holds for $n = 0$.
\end{remark}
\begin{proof}
	For convenience, we write for each $n \in \n_0$, 
	\[
		\begin{bmatrix}
			\white_{n} & \black_{n} \\
			\white_{n}' & \black_{n}'
		\end{bmatrix}
		\define A(\Domain{n}) =
		\begin{bmatrix}
			N_{\white \white}(A(\Domain{n})) & N_{\black \white}(A(\Domain{n})) \\
			N_{\white \black}(A(\Domain{n})) & N_{\black \black}(A(\Domain{n}))
		\end{bmatrix}.
	\]

	Let $\triojuxtapose{k}{\ell}{m} \in \n_0$ with $m \geqslant \ell \geqslant k$ be arbitrary. 
	By Proposition~\ref{prop:subsystem:properties}~\ref{item:subsystem:properties:homeo}, the map $F^{k}$ preserves colors of tiles of $F$, i.e., if $X^{m}$ is an $m$-tile of $F$, then $F^{k}(X^{m})$ is an $(m - k)$-tile of $F$ with the same color as $X^{m}$. 
	Moreover, if $Y^{\ell}$ is an $\ell$-tile of $F$, then it follows from Lemma~\ref{lem:cell mapping properties of Thurston map}~\ref{item:lem:cell mapping properties of Thurston map:i} and Proposition~\ref{prop:subsystem:properties}~\ref{item:subsystem:properties:homeo} that the map $F^{k}|_{Y^{\ell}}$ induces a bijection $X^{m} \mapsto F^{k}(X^{m})$ between the $m$-tiles of $F$ contained in $Y^{\ell}$ and the $(m - k)$-tiles of $F$ contained in the $(\ell - k)$-tile $Y^{\ell - k} \define F^{k}(Y^{\ell})$.

	If we use this for $m = k + 1$ and $\ell = k$, then we see that a white $k$-tile of $F$ contains $\white_{1}$ white and $\black_{1}$ black $(k + 1)$-tiles of $F$, and similarly each black $k$-tile of $F$ contains $\white_{1}'$ white and $\black_{1}'$ black $(k + 1)$-tiles of $F$. 
	This leads to the identity\[
		\begin{bmatrix}
		\white_{k + 1} & \black_{k + 1} \\
		\white_{k + 1}' & \black_{k + 1}'
		\end{bmatrix}
		= 
		\begin{bmatrix}
		\white_{k} & \black_{k} \\
		\white_{k}' & \black_{k}'
		\end{bmatrix}
		\begin{bmatrix}
		\white_{1} & \black_{1} \\
		\white_{1}' & \black_{1}'
		\end{bmatrix}
	\]
	for each $k \in \n_0$. 
	This implies \eqref{eq:power of tile matrix}.
\end{proof}

The following proposition is not used in this paper but should be of independent interest.

\begin{proposition}    \label{prop:no degenerate no die}
	Let $f$, $\mathcal{C}$, $F$ satisfy the Assumptions in Section~\ref{sec:The Assumptions}.
	We assume in addition that $f(\mathcal{C}) \subseteq \mathcal{C}$. 
	If the tile matrix $A$ of $F$ with respect to $\mathcal{C}$ is not degenerate, then for each $n \in \n_0$ and each $X^n \in \Domain{n}$ we have $X^n \cap \limitset \ne \emptyset$.
\end{proposition}
We say that the tile matrix $A$ of $F$ with respect to $\mathcal{C}$ is \emph{degenerate} if $A$ has one of the following forms:
\[
	\begin{bmatrix}
		a & b \\ 0 & 0
	\end{bmatrix},
\begin{bmatrix}
		0 & 0 \\ a & b
	\end{bmatrix}
\]
where $\juxtapose{a}{b} \in \n_0$.
\begin{proof}
	Fix arbitrary integer $n \in \n_0$ and tile $X^n \in \Domain{n}$. 
	Recall that we set $F^0 = \id{S^2}$. 
	By Proposition~\ref{prop:subsystem:properties}~\ref{item:subsystem:properties:homeo}, we have $F^{n}(X^n) = X^0_{\colour} \in \Tile{0}$ for some $\colour \in \colours$ and $F^{n}|_{X^n}$ is a homeomorphism of $X^n$ onto $X^0_{\colour}$. 
	Since the tile matrix $A$ is not degenerate, there exists a tile $Y^{1}_{\colour} \in \Domain{1}$ such that $Y^{1}_{\colour} \subseteq X^0_{\colour}$. 
	Hence by Lemma~\ref{lem:cell mapping properties of Thurston map}~\ref{item:lem:cell mapping properties of Thurston map:i} and Proposition~\ref{prop:subsystem:properties}~\ref{item:subsystem:properties:homeo}, there exists a tile $X^{n + 1} \in \Domain{n + 1}$ such that $X^{n + 1} \subseteq X^{n}$ and $F^{n}(X^{n + 1}) = Y^{1}_{\colour}$. 
	Since $F^{n + 1}(X^{n + 1}) \in \Domain{0}$, similarly, there exists a tile $X^{n + 2} \in \Domain{n + 2}$ such that $X^{n + 2} \subseteq X^{n + 1}$. 
	Thus by induction, there exists a sequence of tiles $\{X^{n + k} \}_{k \in \n_0}$ that satisfies $X^{n + k} \in \Domain{n + k}$ and $X^{n + k + 1} \subseteq X^{n + k}$ for each $k \in \n_0$. 
	By Lemma~\ref{lem:visual_metric}~\ref{item:lem:visual_metric:diameter of cell}, the set $\bigcap_{k \in \n_0} X^{n + k}$ is the intersection of a nested sequence of closed sets with diameters convergent to zero. 
	Thus, it contains exactly one point in $S^2$. 
	Since $\bigcap_{k \in \n_0} X^{n + k} \subseteq X^{n} \cap \limitset$ by \eqref{eq:def:limitset} and Proposition~\ref{prop:subsystem:properties invariant Jordan curve}~\ref{item:subsystem:properties invariant Jordan curve:decreasing relation of domains}, the proof is complete.
\end{proof}

\subsection{Irreducible and primitive subsystems}%
\label{sub:Irreducible and primitive subsystems}

In this subsection, we specialize in irreducible (\resp strongly irreducible) subsystems and primitive (\resp strongly primitive) subsystems, which have additional properties.

Let $f$, $\mathcal{C}$, $F$ satisfy the Assumptions in Section~\ref{sec:The Assumptions}.

\begin{definition}[Irreducibility]    \label{def:irreducibility of subsystem}
	We say that $F$ is an \emph{irreducible} (\resp a \emph{strongly irreducible}) subsystem (of $f$ with respect to $\mathcal{C}$) if for each pair of colors $\juxtapose{\colour}{\ccolour} \in \colours$, there exists an integer $n = n(\colour, \ccolour) \in \n$ and $X^{n} \in \cFTile{n}$ satisfying $X^{n} \subseteq X^0_{\ccolour}$ (\resp $X^{n} \subseteq \inte[\big]{X^0_{\ccolour}}$).
	We denote by $n_{F}$ the constant $\max_{\juxtapose{\colour}{\ccolour} \in \colours} n(\colour, \ccolour)$, which depends only $F$ and $\mathcal{C}$.
\end{definition}

Obviously, if $F$ is irreducible then $\colourset = \colours$ and $F(\domF) = S^{2}$.

Note that the subsystem $F$ defined in Example~\ref{exam:subsystems}~\ref{item:exam:subsystems:Sierpinski gasket} is strongly irreducible if the front side of the left pillow shown in Figure~\ref{fig:subsystem:example:gasket} is black.

\begin{definition}[Primitivity]    \label{def:primitivity of subsystem}
	We say that $F$ is a \emph{primitive} (\resp \emph{strongly primitive}) subsystem (of $f$ with respect to $\mathcal{C}$) if there exists an integer $n_{F} \in \n$ such that for each pair of colors $\juxtapose{\colour}{\ccolour} \in \colours$ and each integer $n \geqslant n_{F}$, there exists $X^n \in \cFTile{n}$ satisfying $X^n \subseteq X^0_{\ccolour}$ (\resp $X^n \subseteq \inte[\big]{X^0_{\ccolour}}$). 
\end{definition}

Obviously, if $F$ is primitive (\resp strongly primitive), then $F$ is irreducible (\resp strongly irreducible).

Note that the subsystem $F$ defined in Example~\ref{exam:subsystems}~\ref{item:exam:subsystems:Sierpinski carpet} is strongly primitive.

\begin{remark}\label{rem:expanding Thurston map is strongly primitive subsystem of itself}
	By \cite[Lemma~5.10]{li2018equilibrium}, an expanding Thurston map $f$ is a strongly primitive subsystem of itself with respect to every Jordan curve $\mathcal{C} \subseteq S^2$ satisfying $\post{f} \subseteq \mathcal{C}$.
\end{remark}

\begin{definition}	\label{def:irreducible and primitive matrix}
	A matrix all of whose entries are positive (\resp non-negative) is called \emph{positive} (\resp \emph{non-negative}).
	Let $A$ be a square non-negative matrix.
	If for any $i, \, j$ there is $n \in \n$ such that $(A^{n})_{ij} > 0$, then $A$ is called \emph{irreducible}; otherwise $A$ is called \emph{reducible}.
	If some power of $A$ is positive, $A$ is called \emph{primitive}.
\end{definition}

\begin{remark}\label{rem:equivalence between irreducibility and primitivity for subsystem and tile matrix}
	If we assume that $f(\mathcal{C}) \subseteq \mathcal{C}$, then it follows immediately from Definitions~\ref{def:tile matrix} and Proposition~\ref{prop:power of tile matrix} that $F$ is irreducible (\resp primitive) if and only if the tile matrix of $F$ is irreducible (\resp primitive).   
\end{remark}

\begin{proposition}    \label{prop:irreducible subsystem properties}
	Let $f$, $\mathcal{C}$, $F$ satisfy the Assumptions in Section~\ref{sec:The Assumptions}.  
	We assume in addition that $f(\mathcal{C}) \subseteq \mathcal{C}$.
	Then the following statements hold:
	\begin{enumerate}[label=\rm{(\roman*)}]
		\smallskip
		\item    \label{item:prop:irreducible subsystem properties:preimages is dense in limitset} 
			If $F$ is irreducible, then $\bigcup_{i \in \n} F^{-i}(x)$ is dense in $\limitset$ for each $x \in S^2$.
		\smallskip
		
		\item    \label{item:prop:irreducible subsystem properties:limitset non degenerate to Jordan curve} 
			If $F$ is strongly irreducible, then $\limitset \cap \inte[\big]{X^0_{\colour}} \ne \emptyset$ for each $\colour \in \colours$.
	\end{enumerate}
\end{proposition}
\begin{proof}
	\ref{item:prop:irreducible subsystem properties:preimages is dense in limitset}
	Fix an arbitrary point $x \in S^2$. 
	It suffices to show that the closure of $\bigcup_{i \in \n} F^{-i}(x)$ in $S^2$ contains $\limitset$. 
	By \eqref{eq:definition of tile of subsystem}, Proposition~\ref{prop:subsystem:properties invariant Jordan curve}~\ref{item:subsystem:properties invariant Jordan curve:decreasing relation of domains}, \eqref{eq:def:limitset}, and \eqref{eq:definition of expansion}, it suffices to show that for each $n \in \n$ and each $X^n \in \Domain{n}$, $X^n \cap \bigcup_{i \in \n} F^{-i}(x) \ne \emptyset$. 

	Fix arbitrary $n \in \n$ and $X^n \in \Domain{n}$. Since $x \in S^{2} = X^0_{\black} \cup X^{0}_{\white}$, there exists $\colour \in \colours$ such that $x \in X^0_{\colour}$. By Proposition~\ref{prop:subsystem:properties}~\ref{item:subsystem:properties:homeo}, $X^n$ is mapped by $F^n$ homeomorphically to a $0$-tile $X^0_{\ccolour}$ for some $\ccolour \in \colours$. 
	Since $F$ is irreducible, by Definition~\ref{def:irreducibility of subsystem}, there exist $k \in \n$ and $Y^{k} \in \cFTile{k}$ such that $Y^{k} \subseteq X^0_{\ccolour}$ and $F^{k}\bigl(Y^{k}\bigr) = X^0_{\colour}$. 
	Then it follows from Lemma~\ref{lem:cell mapping properties of Thurston map}~\ref{item:lem:cell mapping properties of Thurston map:i} and Proposition~\ref{prop:subsystem:properties}~\ref{item:subsystem:properties:homeo} that $X^{k + n} \define (F^{n}|_{X^n})^{-1}\bigl(Y^{k}\bigr) \in \Domain{k + n}$. 
	Since $x \in X^0_{\colour} = F^{k + n}(X^{k + n})$ and $X^{k + n} \subseteq X^{n}$, we conclude that $X^n \cap \bigcup_{i \in \n} F^{-i}(x) \ne \emptyset$.

	\smallskip

	\ref{item:prop:irreducible subsystem properties:limitset non degenerate to Jordan curve}
	Fix arbitrary $\colour \in \colours$. 
	Assume that $F \in \subsystem$ is strongly irreducible.
	Then by Definition~\ref{def:irreducibility of subsystem}, there exist $n \in \n$ and $X^{n} \in \Domain{n}$ such that $X^{n} \subseteq \inte[\big]{X^0_{\colour}}$ and $F^{n}(X^{n}) = X^0_{\colour}$. 
	Thus, by Lemma~\ref{lem:cell mapping properties of Thurston map}~\ref{item:lem:cell mapping properties of Thurston map:i} and Proposition~\ref{prop:subsystem:properties}~\ref{item:subsystem:properties:homeo}, there exists $X^{2 n} \in \Domain{2 n}$ such that $X^{2 n} \subseteq X^{n}$ and $F^{n}(X^{2 n}) = X^{n}$. Since $X^{n} \subseteq \inte[\big]{X^0_{\colour}}$ and $F^{2 n}(X^{2 n}) = X^{0}_{\colour}$, similarly, there exists $X^{3 n} \in \Domain{3 n}$ such that $X^{3 n} \subseteq X^{2 n}$ and $F^{2 n}(X^{3 n}) = X^{n}$. Thus by induction, there exists a sequence of tiles $\bigl\{ X^{k n} \bigr\}_{k \in \n}$ such that $X^{k n} \in \Domain{k n}$ and $X^{(k + 1) n} \subseteq X^{k n} \subseteq \inte[\big]{X^0_{\colour}}$ for each $k \in \n$. Hence the set $\bigcap_{k \in \n} X^{k n} \subseteq \inte[\big]{X^0_{\colour}}$ is the intersection of a nested sequence of closed sets so that it is non-empty. 
	Since $\bigcap_{k \in \n} X^{k n} \subseteq \limitset$ by \eqref{eq:def:limitset} and Proposition~\ref{prop:subsystem:properties invariant Jordan curve}~\ref{item:subsystem:properties invariant Jordan curve:decreasing relation of domains}, we deduce that $\limitset \cap \inte[\big]{X^0_{\colour}} \ne \emptyset$.
\end{proof}

\begin{lemma}    \label{lem:strongly irreducible:tile in interior tile}
	Let $f$, $\mathcal{C}$, $F$ satisfy the Assumptions in Section~\ref{sec:The Assumptions}. 
	Suppose that $F \in \subsystem$ is irreducible (\resp strongly irreducible) and let $n_F \in \n$ be the constant in Definition~\ref{def:irreducibility of subsystem}, which depends only on $F$ and $\mathcal{C}$. 
	Then for each $k \in \n_{0}$, each $\colour \in \colours$, and each $k$-tile $X^k \in \Domain{k}$, there exists an integer $n \in \n$ with $n \leqslant n_{F}$ and $X^{k + n}_{\colour} \in \cFTile{k + n}$ satisfying $X^{k + n}_{\colour} \subseteq X^k$ (\resp $X^{k + n}_{\colour} \subseteq \inte{X^k}$).
\end{lemma}
\begin{proof}
	Assume that $F$ is irreducible.
	Fix arbitrary $k \in \n_{0}$, $\colour \in \colours$, and $X^k \in \Domain{k}$. 
	By Proposition~\ref{prop:subsystem:properties}~\ref{item:subsystem:properties:homeo}, $X^k$ is mapped by $F^k$ homeomorphically to $X^0 \define F^k(X^k) \in \Tile{0}$. 
	Since $F$ is irreducible, by Definition~\ref{def:irreducibility of subsystem}, there exists $n \in \n$ with $n \leqslant n_{F}$ such that there exists $X^n_{\colour} \in \cFTile{n}$ satisfying $X^n_{\colour} \subseteq X^0$. 
	Then it follows from Lemma~\ref{lem:cell mapping properties of Thurston map}~\ref{item:lem:cell mapping properties of Thurston map:i} and Proposition~\ref{prop:subsystem:properties}~\ref{item:subsystem:properties:homeo} that $X^{k + n}_{\colour} \define (F^k|_{X^k})^{-1}(X^n_{\colour})$ is an $(k + n)$-tile satisfying $X^{k + n}_{\colour} \subseteq X^k$. 
	Since $F^{k + n}(X^{k + n}_{\colour}) = F^{n}(X^n_{\colour}) = X^0_{\colour}$ and $X^{k + n}_{\colour} \subseteq F^{-k}(X^n_{\colour}) \subseteq F^{-(k + n)}(F(\domF))$, we conclude that $X^{k + n}_{\colour} \in \cFTile{k + n}$ and the proof is complete in this case.

	For strongly irreducible $F$, by the same argument as above we can prove the corresponding results.
\end{proof}

\begin{lemma}    \label{lem:strongly primitive:tile in interior tile for high enough level}
	Let $f$, $\mathcal{C}$, $F$ satisfy the Assumptions in Section~\ref{sec:The Assumptions}. 
	We assume in addition that $F \in \subsystem$ is primitive (\resp strongly primitive). 
	Let $n_F \in \n$ be the constant from Definition~\ref{def:primitivity of subsystem}, which depends only on $F$ and $\mathcal{C}$. 
	Then for each $n \in \n$ with $n \geqslant n_{F}$, each $m \in \n_0$, each $\colour \in \colours$, and each $m$-tile $X^m \in \Domain{m}$, there exists an $(n + m)$-tile $X^{n + m}_{\colour} \in \cFTile{n + m}$ such that $X^{n + m}_{\colour} \subseteq X^m$ (\resp $X^{n + m}_{\colour} \subseteq \inte{X^m}$).
\end{lemma}
\begin{proof}
	Assume that $F$ is primitive.
	Let integer $n \geqslant n_F$, $m \in \n_0$, $\colour \in \colours$, and $X^m \in \Domain{m}$ be arbitrary. 
	By Proposition~\ref{prop:subsystem:properties}~\ref{item:subsystem:properties:homeo}, $X^m$ is mapped by $F^{m}|_{X^{m}}$ homeomorphically to $X^0 \define F^m(X^m)$. 
	Since $F$ is primitive, there exists $X^n_{\colour} \in \cFTile{n}$ satisfying $X^n_{\colour} \subseteq X^0$. 
	Then it follows from Lemma~\ref{lem:cell mapping properties of Thurston map}~\ref{item:lem:cell mapping properties of Thurston map:i} and Proposition~\ref{prop:subsystem:properties}~\ref{item:subsystem:properties:homeo} that $X^{n + m}_{\colour} \define (F^m|_{X^m})^{-1}(X^n_{\colour})$ is an $(n + m)$-tile satisfying $X^{n + m}_{\colour} \subseteq X^m$. 
	Since $F^{n + m}(X^{n + m}_{\colour}) = F^{n}(X^n_{\colour}) = X^0_{\colour}$ and $X^{n + m}_{\colour} \subseteq F^{-m}(X^n_{\colour}) \subseteq F^{-(n + m)}(F(\domF))$, we conclude that $X^{n + m}_{\colour} \in \cFTile{n + m}$.

	For strongly primitive $F$, by the same argument as above we can prove the corresponding results.
\end{proof}

\subsection{Distortion lemmas}%
\label{sub:Distortion lemmas}

For the convenience of the reader, we first record the following lemma from \cite[Lemma~3.13]{li2018equilibrium}, which generalizes \cite[Lemma~15.25]{bonk2017expanding}.
\begin{lemma}[M.~Bonk \& D.~Meyer \cite{bonk2017expanding}, Z.~Li \cite{li2018equilibrium}]     \label{lem:basic_distortion}
    Let $f$, $\mathcal{C}$, $d$, $\Lambda$ satisfy the Assumptions in Section~\ref{sec:The Assumptions}.
    Then there exists a constant $C_0 > 1$, depending only on $f$, $\mathcal{C}$, and $d$, with the following property:

    If $\juxtapose{n}{k} \in \n_0$, $X^{n+k}\in \mathbf{X}^{n+k}(f,\mathcal{C})$, and $\juxtapose{x}{y} \in X^{n + k}$, then
    \begin{equation}     \label{eq:basic_distortion}
        C_0^{-1}d(x, y) \leqslant d(f^n(x), f^n(y)) / \Lambda^{n} \leqslant C_0 d(x, y).
    \end{equation}
\end{lemma}

The next distortion lemma for expanding Thurston maps follows immediately from \cite[Lemma~5.1]{li2018equilibrium}.

\begin{lemma}    \label{lem:distortion_lemma}
	Let $f$, $\mathcal{C}$, $d$, $\Lambda$, $\phi$, $\holderexp$ satisfy the Assumptions in Section~\ref{sec:The Assumptions}.
    Then there exists a constant $C_1 \geqslant 0$ depending only on $f$, $\mathcal{C}$, $d$, $\phi$, and $\holderexp$ such that for all $n \in \n_0$, $X^n \in \Tile{n}$, and $\juxtapose{x}{y} \in X^n$,
    \begin{equation}    \label{eq:distortion_lemma}
        \left| S_n\phi(x) - S_n\phi(y) \right| \leqslant C_1 d(f^n(x), f^n(y))^{\holderexp} \leqslant \Cdistortion.
    \end{equation}
    Quantitatively, we choose
    \begin{equation}     \label{eq:const:C_1}
        C_1 \define C_0  \holderseminorm{\potential}  \big/ \bigl( 1 - \Lambda^{-\holderexp} \bigr),
    \end{equation}
    where $C_0 > 1$ is the constant depending only on $f$, $\mathcal{C}$, and $d$ from Lemma~\ref{lem:basic_distortion}.
\end{lemma}

We establish the following distortion lemma for subsystems of expanding Thurston maps.
\begin{lemma}    \label{lem:distortion lemma for subsystem}
	Let $f$, $\mathcal{C}$, $F$, $d$, $\Lambda$, $\phi$, $\holderexp$ satisfy the Assumptions in Section~\ref{sec:The Assumptions}. 
	We assume in addition that $f(\mathcal{C}) \subseteq \mathcal{C}$.
	Then the following statements hold:
	\begin{enumerate}[label=\rm{(\roman*)}]
		\smallskip
		
		\item     \label{item:lem:distortion lemma for subsystem:holder bound on same color tile}
		For each $n \in \n_{0}$, each $\colour \in \colourset$, and each pair of points $\juxtapose{x}{y} \in X^0_{\colour}$, we have
		\begin{equation}    \label{eq:same color distortion bound for split operator}
			\frac{ \sum_{X^n \in \cFTile{n}} \myexp[\big]{ S^{F}_{n} \phi\bigl( (F^n|_{X^n})^{-1}(x) \bigr) } }{ \sum_{X^n \in \cFTile{n}} \myexp[\big]{ S^{F}_{n} \phi\bigl( (F^n|_{X^n})^{-1}(x) \bigr) } }  
			\leqslant \myexp[\big]{ C_1 d(x, y)^{\holderexp} } 
			\leqslant \myexp[\big]{ \Cdistortion },
		\end{equation}
		where $C_1 \geqslant 0$ is the constant defined in \eqref{eq:const:C_1} in Lemma~\ref{lem:distortion_lemma} and depends only on $f$, $\mathcal{C}$, $d$, $\phi$, and $\holderexp$.
			
		\smallskip

		\item     \label{item:lem:distortion lemma for subsystem:uniform bound}
		If $F$ is irreducible, then there exists a constant $\Csplratio \geqslant 1$ depending only on $F$, $\mathcal{C}$, $d$, $\phi$, and $\holderexp$ such that for each $n \in \n_{0}$, each pair of colors $\juxtapose{\colour}{\ccolour} \in \colourset$, each $x \in X^0_{\colour}$, and each $y \in X^0_{\ccolour}$, we have
		\begin{equation}    \label{eq:distinct color distortion bound for split operator}
			\frac{ \sum_{X^n_{\colour} \in \ccFTile{n}{\colour}{}} \myexp[\big]{ S^{F}_{n} \phi\bigl( (F^n|_{X^n_{\colour}})^{-1}(x) \bigr) } }{ \sum_{X^n_{\ccolour} \in \ccFTile{n}{\ccolour}{}} \myexp[\big]{ S^{F}_{n} \phi\bigl( (F^n|_{X^n_{\ccolour}})^{-1}(y) \bigr) } } \leqslant \Csplratio.
		\end{equation}
		Quantitatively, we choose 
		\begin{equation}    \label{eq:const:Csplratio}
			\Csplratio \define \CsplratioExpression,
		\end{equation}
		where $n_{F} \in \n$ is the constant in Definition~\ref{def:irreducibility of subsystem} and depends only on $F$ and $\mathcal{C}$, and $C_1 \geqslant 0$ is the constant defined in \eqref{eq:const:C_1} in Lemma~\ref{lem:distortion_lemma} and depends only on $f$, $\mathcal{C}$, $d$, $\phi$, and $\holderexp$.
	\end{enumerate}
\end{lemma}
\begin{proof}
	\ref{item:lem:distortion lemma for subsystem:holder bound on same color tile}  
	We fix arbitrary $n \in \n_0$, $\colour \in \colourset$, and $\juxtapose{x}{y} \in X^0_{\colour}$. 
	For each $X^n \in \cFTile{n}$, it follows from Proposition~\ref{prop:subsystem:properties}~\ref{item:subsystem:properties:homeo} that $F^{n}|_{X^n}$ is a homeomorphism from $X^n$ onto $X^0_{\colour}$. Then by Lemma~\ref{lem:distortion_lemma}, we have\[
		\myexp[\big]{ S^{F}_{n} \phi\bigl( (F^n|_{X^n})^{-1}(x)\bigr) - S^{F}_{n} \phi\bigl( (F^n|_{X^n})^{-1}(y)\bigr) } \leqslant \myexp[\big]{ C_1 d(x, y)^{\holderexp} }.
	\]
	Thus\[
		\myexp[\big]{ S^{F}_{n} \phi\bigl( (F^n|_{X^n})^{-1}(x)\bigr) } 
		\leqslant \myexp[\big]{ C_1 d(x, y)^{\holderexp} } \myexp[\big]{ S^{F}_{n} \phi\bigl( (F^n|_{X^n})^{-1}(y)\bigr) }.
	\]
	By summing the last inequality over all $X^n \in \cFTile{n}$, we can conclude that \eqref{eq:same color distortion bound for split operator} holds.

	\smallskip

	\ref{item:lem:distortion lemma for subsystem:uniform bound}
	We assume that $F$ is irreducible and let $n_F \in \n$ be the constant from Definition~\ref{def:irreducibility of subsystem}, which depends only on $F$ and $\mathcal{C}$.

	For convenience, we write $I^{n}_{\colour}(z) \define \sum_{X^n \in \cFTile{n}} \myexp[\big]{ S^{F}_{n} \phi\bigl( (F^n|_{X^n})^{-1}(z) \bigr) }$ for $n \in \n_0$, $\colour \in \colourset$, and $z \in X^0_{\colour}$.
	Since $F$ is irreducible, by Definition~\ref{def:irreducibility of subsystem} and Proposition~\ref{prop:sursubsystem properties}~\ref{item:prop:sursubsystem properties:F maps n+1 tile to n tile}, we have $\colourset = \colours$ and $\cFTile{n} \ne \emptyset$ for each $n \in \n_0$ and each $\colour \in \colours$. 
	Thus $I^{n}_{\colour}(z) > 0$ for each $n \in \n_0$, each $\colour \in \colours$, and each $z \in X^0_{\colour}$.

	In the rest of the proof we fix arbitrary $\juxtapose{\colour}{\ccolour} \in \colours$, $x \in X^{0}_{\colour}$, and $y \in X^{0}_{\ccolour}$.

	We first consider an arbitrary integer $n \in \{0,\, 1,\, 2,\, \dots, \, n_{F} \}$. 
	Since $\card{\ccFTile{n}{\colour}{}} \leqslant \card{\cTile{n}{\colour}} = (\deg{f})^{n}$ and $\ccFTile{n}{\ccolour}{} \ne \emptyset$, we have
	\begin{equation}    \label{eq:temp:distortion for I for small n}
		\frac{ I_{\colour}^{n}(x) }{ I_{\ccolour}^{n}(y) } 
		\leqslant \frac{ \card{\ccFTile{n}{\colour}{}} \myexp{ n \uniformnorm{\potential}} }{ \card{\ccFTile{n}{\ccolour}{}} \myexp{ - n \uniformnorm{\potential} } } 
		\leqslant (\deg{f})^{n_{F}} \myexp{ {2 n_{F} \uniformnorm{\phi}} }
	\end{equation}
	Thus \eqref{eq:distinct color distortion bound for split operator} holds. 

	We now consider an arbitrary integer $n > n_{F}$. 
	Since $F$ is irreducible, by Definition~\ref{def:irreducibility of subsystem}, there exists $n_{\ccolour \colour} \in \n$ with $n_{\ccolour \colour} \leqslant n_{F}$ such that there exists $Y^{n_{\ccolour \colour}}_{\ccolour} \in \ccFTile{n_{\ccolour \colour}}{\ccolour}{}$ satisfying $Y^{n_{\ccolour \colour}}_{\ccolour} \subseteq X^{0}_{\colour}$.
	Then for each $X \in \cFTile{n - n_{\ccolour \colour}}$, it follows from Lemma~\ref{lem:cell mapping properties of Thurston map}~\ref{item:lem:cell mapping properties of Thurston map:i} and Proposition~\ref{prop:subsystem:properties}~\ref{item:subsystem:properties:homeo} that $Y \define ( F^{n - n_{\ccolour \colour}}|_{ X } )^{-1} \bigl( Y^{n_{\ccolour \colour}}_{\ccolour} \bigr) \in \ccFTile{n}{\ccolour}{}$.
	This defines an injective map from $\cFTile{n - n_{\ccolour \colour}}$ to $\ccFTile{n}{\ccolour}{}$ that maps each $X \in \cFTile{n - n_{\ccolour \colour}}$ to $Y = ( F^{n - n_{\ccolour \colour}}|_{ X } )^{-1} \bigl( Y^{n_{\ccolour \colour}}_{\ccolour} \bigr) \in \ccFTile{n}{\ccolour}{}$.
	Moreover, applying Lemma~\ref{lem:distortion_lemma}, we have \[
		\begin{split}
			S^{F}_{n} \phi\bigl( (F^n|_{Y})^{-1}(y) \bigr) 
			&\geqslant S_{n - n_{\ccolour \colour}}^{F} \phi\bigl( (F^n|_{Y})^{-1}(y) \bigr) - n_{\ccolour \colour} \uniformnorm{\phi}  \\
			&\geqslant S^{F}_{n - n_{\ccolour \colour}} \phi\bigl( (F^{n - n_{\ccolour \colour}}|_{X})^{-1}(x) \bigr) - \Cdistortion - n_{F} \uniformnorm{\phi}
		\end{split}
	\]
	since $(F^n|_{Y})^{-1}(y) \in Y \subseteq X \in \Tile{n - n_{\ccolour \colour}}$ and $n_{\ccolour \colour} \leqslant n_{F}$.
	Hence, 
	\begin{equation}    \label{eq:temp:lem:distortion lemma for subsystem:item:lem:distortion lemma for subsystem:uniform bound:bound for different color}
		I^{n - n_{\ccolour \colour}}_{\colour}(x) \leqslant I^{n}_{\ccolour}(y) \myexp[\big]{ n_{F} \uniformnorm{\phi} + \Cdistortion }.
	\end{equation}

	In order to establish \eqref{eq:distinct color distortion bound for split operator}, it suffices to show that\[
		I^{n}_{\colour}(x) \leqslant (\deg{f})^{n_{F}} \myexp{ n_{F} \uniformnorm{\phi} } I^{n - n_{\ccolour \colour}}_{\colour}(x).
	\] 
	For each $X \in \cFTile{n - n_{\ccolour \colour}}$, we set $E(X) \define \{ X^n \in \cFTile{n} \describe F^{n_{\ccolour \colour}}(X^n) = X \}$.
	It follows immediately from Proposition~\ref{prop:subsystem:properties}~\ref{item:subsystem:properties:homeo} that $\cFTile{n} = \bigcup_{X \in \cFTile{n - n_{\ccolour \colour}}} E(X)$. 
	By Proposition~\ref{prop:properties cell decompositions}~\ref{item:prop:properties cell decompositions:union of cells}, for each $X \in \cFTile{n - n_{\ccolour \colour}}$, we have $\card{E(X)} \leqslant (\deg{f})^{n_{\ccolour \colour}} \leqslant (\deg{f})^{n_{F}}$.
	Moreover, for each $X^{n} \in E(X)$, \[
		\begin{split}
			S^{F}_{n} \phi\bigl( \bigl(F^n|_{X^{n}}\bigr)^{-1}(x) \bigr) 
			&= S^{F}_{n_{\ccolour \colour}} \phi\bigl( (F^n|_{X^{n}})^{-1}(x) \bigr) + S^{F}_{n - n_{\ccolour \colour}} \phi\bigl( (F^{n - n_{\ccolour \colour}}|_{X})^{-1}(x) \bigr) \\
			&\leqslant n_{F} \uniformnorm{\phi} + S^{F}_{n - n_{\ccolour \colour}} \phi\bigl( (F^{n - n_{\ccolour \colour}}|_{X})^{-1}(x) \bigr).
		\end{split}
	\]
	Therefore, we get \[
		\begin{split}
			I^{n}_{\colour}(x) &= \sum_{ X \in \cFTile{n - n_{\ccolour \colour}} }  \sum\limits_{ X^n \in E(X) } \myexp[\big]{ S^{F}_{n} \phi\bigl( (F^n|_{X^n})^{-1}(z) \bigr) } \\
			&\leqslant \sum_{ X \in \cFTile{n - n_{\ccolour \colour}} } (\deg{f})^{n_{F}} \myexp{ n_{F} \uniformnorm{\phi} } \myexp[\big]{ S^{F}_{n - n_{\ccolour \colour}} \phi\bigl( (F^{n - n_{\ccolour \colour}}|_{X})^{-1}(x) \bigr) }  \\
			&= (\deg{f})^{n_{F}} \myexp{ n_{F} \uniformnorm{\phi} } I^{n - n_{\ccolour \colour}}_{\colour}(x).
		\end{split}
	\]
	Combining this with \eqref{eq:temp:lem:distortion lemma for subsystem:item:lem:distortion lemma for subsystem:uniform bound:bound for different color}, we establish \eqref{eq:distinct color distortion bound for split operator} by choosing $\Csplratio$ as in \eqref{eq:const:Csplratio}, which depends only on $F$, $\mathcal{C}$, $d$, $\phi$, and $\holderexp$.
\end{proof} 
\section{Thermodynamic formalism for subsystems}
\label{sec:Thermodynamic formalism for subsystems}

This section focuses on thermodynamic formalism for subsystems, with the main results being Theorems~\ref{thm:subsystem characterization of pressure} and \ref{thm:existence of equilibrium state for subsystem}. 
These theorems establish the Variational Principle and demonstrate the existence of equilibrium states for subsystems of expanding Thurston maps.

In Subsection~\ref{sub:Pressures for subsystems}, we define the topological pressure of a subsystem with respect to a potential via the tile structures determined by the subsystem.

In Subsection~\ref{sub:Split Ruelle operators} we define appropriate variants of the Ruelle operator called split Ruelle operators (Definition~\ref{def:split ruelle operator}). One cannot use the Ruelle operator introduced in \cite{li2018equilibrium} for a subsystem in the proofs of Theorems~\ref{thm:subsystem characterization of pressure} and \ref{thm:existence of equilibrium state for subsystem} directly. 
One fundamental problem is that the direct adaptation of the Ruelle operator $\mathcal{L}_{\phi} \colon C(S^2) \mapping C(S^2)$ may not be well-defined for subsystems. 
More precisely, it is possible that $\mathcal{L}_{\phi}(u) \notin C(S^2)$ for some $u \in C(S^2)$ due to the subtle combinatorial structures of subsystems. 
Our strategy here is to ``split" the Ruelle operator into two pieces so that in each piece, the continuity is preserved under iteration. We note that the concept of splitting the Ruelle operators was first introduced in \cite{li2018prime}. We extend this definition to split Ruelle operators for subsystems. 

We next develop the thermodynamic formalism to establish the existence of equilibrium states for subsystems. 
We follow \cite{li2018equilibrium} in several places, but we have made some notable changes to address the challenges posed by the subsystems.
We first introduce the split sphere $\widetilde{S}$ as the disjoint union of $X^0_{\black}$ and $X^0_{\white}$. 
Then the product of function spaces and the product of measure spaces can be identified naturally with the space of functions and the space of measures on $\widetilde{S}$, respectively (Remark~\ref{rem:disjoint union}). 
Since the split Ruelle operator $\splopt$ acts on the product of function spaces, its adjoint operator $\dualsplopt$ acts on the product of measure spaces. We have to deal with functions and measures on the split sphere $\widetilde{S}$ while the measures we want (for example, the equilibrium states) should be on $S^{2}$. 
Therefore, we make extra efforts to convert the measures on $\widetilde{S}$ provided by thermodynamic formalism into measures on $S^{2}$. 
Moreover, by the local degree defined in Subsection~\ref{sub:Local degree}, we establish Theorem~\ref{thm:subsystem:eigenmeasure existence and basic properties} and Proposition~\ref{prop:subsystem edge Jordan curve has measure zero} so that we show that an eigenmeasure for $\dualsplopt$ is a Gibbs measure on $S^2$ under our identifications (Proposition~\ref{prop:subsystem properties of eigenmeasure}). 
Then we construct an $f$-invariant Gibbs measure $\equstate$ with desired properties in Theorem~\ref{thm:existence of f invariant Gibbs measure}. 
Finally, in Theorem~\ref{thm:existence of equilibrium state for subsystem}, we prove that such $\equstate$ is an equilibrium state.

\subsection{Pressures for subsystems}%
\label{sub:Pressures for subsystems}

In this subsection, we define the topological pressure for subsystems.

\begin{definition}    \label{def:pressure for subsystem}
	Let $f$, $\mathcal{C}$, $F$ satisfy the Assumptions in Section~\ref{sec:The Assumptions}.
	For a real-valued function $\varphi \colon S^2 \mapping \real$, we denote \[
		Z_{n}(F, \varphi ) \define \sum_{X^n \in \Domain{n}} \myexp[\big]{ \sup\bigl\{S^F_n \varphi(x) \describe x \in X^n \bigr\} } 
	\]
	for each $n \in \n$. 
	We define the \emph{topological pressure} of the subsystem $F$ with respect to the \emph{potential} $\varphi$ by
	\begin{equation}    \label{eq:pressure of subsystem}
		\pressure[\varphi] \define \liminf_{n \mapping +\infty} \frac{1}{n} \log ( Z_{n}(F, \varphi) ).
	\end{equation}
	In particular, when $\varphi$ is the constant function $0$, the quantity $h_{\operatorname{top}}(F) \define \pressure[0]$ is called the \emph{topological entropy} of the subsystem $F$.
\end{definition}

\begin{rmk}
	We note that the definition \eqref{eq:pressure of subsystem} used here differs from the classical definition \eqref{eq:def:topological pressure} in the context of Subsection~\ref{sub:thermodynamic formalism}, which applies to $F|_{\limitset}$ and $\varphi|_{\limitset}$.
	We study the relationship between these two topological pressures in Proposition~\ref{prop:characterization of pressures of subsystems in separated sets}, and ultimately demonstrate that, for a strongly irreducible subsystem and a \holder continuous potential, they coincide (see Theorem~\ref{thm:subsystem characterization of pressure}).
\end{rmk}

\begin{lemma}    \label{lem:partition function is submultiplicative}
	Let $f$, $\mathcal{C}$, $F$ satisfy the Assumptions in Section~\ref{sec:The Assumptions}.
	We assume in addition that $f(\mathcal{C}) \subseteq \mathcal{C}$. 
	Consider a real-valued function $\varphi \colon S^2 \mapping \real$.
	Then for all $\juxtapose{k}{\ell} \in \n$, we have
	\begin{equation}    \label{eq:Z_n submultiplicative}
		Z_{k + \ell}(F, \varphi) \leqslant Z_{k}(F,\varphi) Z_{\ell}(F, \varphi).
	\end{equation}
\end{lemma}
\begin{proof}
	For each $\juxtapose{k}{\ell} \in \n$,
	\begin{align*}
		Z_{k + \ell}(F,\varphi) &= \sum_{X^{k + \ell} \in \Domain{k + \ell}} \myexp[\big]{ \sup \set[\big]{  S^F_{k + \ell} \varphi(x) \describe x \in X^{k + \ell} } } \\
		&\leqslant \sum_{X^{k + \ell} \in \Domain{k + \ell}} \myexp[\big]{\sup \set[\big]{ S^F_{k} \varphi(x) \describe x \in X^{k + \ell} } }   \myexp[\big]{ \sup \set[\big]{ S^F_{\ell} \varphi(F^{k}(x)) \describe x \in X^{k + \ell} } }  \\
		&\leqslant \sum_{X^{\ell} \in \Domain{\ell}} \myexp[\big]{ \sup \set[\big]{ S^F_{\ell} \varphi(y) \describe y \in X^{\ell} } }  
			\sum_{ \substack{X^{k + \ell} \in \Domain{k + \ell} \\ F^{k}(X^{k + \ell}) = X^{\ell}} } 
			\myexp[\big]{ \sup \set[\big]{ S^F_{k} \varphi(x) \describe x \in X^{k + \ell} } }  \\
		&\leqslant \sum_{X^{\ell} \in \Domain{\ell}} \myexp[\big]{ \sup \set[\big]{ S^F_{\ell} \varphi(y) \describe y \in X^{\ell} } }
			\sum_{X^{k} \in \Domain{k}} \myexp[\big]{ \sup \set[\big]{ S^F_{k} \varphi(y) \describe y \in X^{k} } }  \\
		&= Z_{k}(F,\varphi) Z_{\ell}(F,\varphi).
	\end{align*}
	Here, the first inequality uses the fact that $S^F_{k + \ell}\varphi(x) = S^F_{k}\varphi(x) + S^F_{\ell}\varphi \bigl( F^k(x) \bigr)$, the second inequality follows from Proposition~\ref{prop:subsystem:properties}~\ref{item:subsystem:properties:homeo}, and the last inequality also follows from Proposition~\ref{prop:subsystem:properties}~\ref{item:subsystem:properties:homeo}, which shows that for each $\ell$-tile $X^{\ell} \in \Domain{\ell}$ the map from $\set[\big]{ X \in \Domain{k + \ell} \describe F^{k}(X) = X^{\ell} }$ to $\Domain{k}$ induced by Proposition~\ref{prop:cell decomposition: invariant Jordan curve} is injective.
\end{proof}

We record the following well-known lemma and refer the reader to \cite[Lemma~2.11]{mauldin2003graph} for a proof.
\begin{lemma}    \label{lem:subadditive sequence}
	If a sequence $\{a_n\}_{n \in \n}$ of real number is subadditive (i.e., $a_{i + j} \leqslant a_i + a_j$ for all $i, \, j \in\n$), then $\lim_{n \to +\infty} a_n/n$ exists in $\real \cup \{-\infty\}$ and is equal to $\inf\{a_n/n \describe n \in \n\}$.
\end{lemma}

\begin{lemma}    \label{lem:well-defined pressure of subsystems}
	Let $f$, $\mathcal{C}$, $F$ satisfy the Assumptions in Section~\ref{sec:The Assumptions}.
	We assume in addition that $f(\mathcal{C}) \subseteq \mathcal{C}$. 
	Consider $\varphi \in C(S^2)$.
	Then \[
		\pressure[\varphi] = \lim\limits_{n \to +\infty} \frac{1}{n} \log \mathopen{}(Z_{n}(F, \varphi)) \in [-\uniformnorm{\varphi}, +\infty).
	\]
\end{lemma}
\begin{proof}
	By Proposition~\ref{prop:sursubsystem properties}~\ref{item:prop:sursubsystem properties:property of domain and limitset}, $\limitset(F, \mathcal{C}) \ne \emptyset$. 
	Then for each $n \in \n$ we have $\bigcup \Domain{n} \ne \emptyset$ and $Z_{n}(F, \varphi) > 0$. 
	Thus by Lemmas~\ref{lem:partition function is submultiplicative} and \ref{lem:subadditive sequence}, the sequence $\bigl\{ \frac{1}{n} \log \mathopen{}(Z_{n}(F, \varphi)) \bigr\}_{n \in \n}$ is subadditive and \[
		\pressure[\varphi] = \lim_{n \mapping +\infty} \frac{1}{n} \log (Z_{n}(F, \varphi)) = \inf_{n \in \n} \biggl\{ \frac{1}{n} \log (Z_n(F, \varphi)) \biggr\} \in \real \cup \{-\infty\}.
	\]
	It suffices now to prove that $\frac{1}{n} \log (Z_n(F, \varphi))$ is bounded from below by $-\uniformnorm{\varphi}$. Let $x_0 \in \limitset$ be arbitrary. Then for each $n \in \n$, we have $x_0 \in \bigcup \Domain{n}$ and\[
		Z_n(F, \varphi) = \sum_{X^{n} \in \Domain{n}} \myexp[\big]{ \sup\bigl\{S^F_n \varphi(x) \describe x \in X^{n} \bigr\} } 
		\geqslant \myexp[\big]{ S^F_n \varphi(x_0) } 
		\geqslant e^{-n \uniformnorm{\varphi}}.
	\]
	Therefore, we get $\frac{1}{n} \log (Z_n(F, \varphi)) \geqslant - \uniformnorm{\varphi}$ for each $n \in \n$ and $\pressure[\varphi] \geqslant - \uniformnorm{\varphi}$.
\end{proof}

The topological entropy of $F$ can in fact be computed explicitly via tile matrices (defined in Subsection~\ref{sub:Tile matrix}).

\begin{proposition}    \label{prop:topological entropy of subsystem}
	Let $f$, $\mathcal{C}$, $F$ satisfy the Assumptions in Section~\ref{sec:The Assumptions}. 
	We assume in addition that $f(\mathcal{C}) \subseteq \mathcal{C}$. 
	Let $A$ be the tile matrix of $F$ with respect to $\mathcal{C}$.
	Then we have
	\begin{equation}    \label{eq:topological entropy of subsystem}
		 h_{\operatorname{top}}(F) = \log(\rho(A)),
	\end{equation}
	where $\rho(A)$ is the spectral radius of $A$.
\end{proposition}
\begin{rmk}
	The spectral radius $\rho(A)$ can easily be computed from any matrix norm. If for an $(m \times m)$-matrix $B = (b_{ij})$ we set $\sumnorm{B} \define \sum_{i, j = 1}^{m} |b_{ij}|$	for example, then $\rho(A) = \lim_{n \to +\infty} (\sumnorm{A^n})^{1/n}$.
\end{rmk}
\begin{proof}
	By \eqref{eq:pressure of subsystem} and Proposition~\ref{prop:subsystem:properties invariant Jordan curve}~\ref{item:subsystem:properties invariant Jordan curve:relation between color and location of tile}, we have\[
		\begin{split}
			h_{\operatorname{top}}(F) 
			&= \liminf_{n \mapping +\infty} \frac{1}{n} \log(\card{\Domain{n}}) 
			= \liminf_{n \mapping +\infty} \frac{1}{n} \log \sum_{\juxtapose{\colour}{\ccolour} \in \colours} \card{\ccFTile{n}{\colour}{\ccolour}}.
		\end{split}
	\]
	Then it follows from Remark~\ref{rem:tile matrix associated with set of tiles} and Proposition~\ref{prop:power of tile matrix} that \[
		h_{\operatorname{top}}(F) = \liminf_{n \mapping +\infty} \frac{1}{n} \log \sumnorm{A^n}.
	\]
	Since $\rho(A) = \lim_{n \to +\infty} (\sumnorm{A^n})^{1/n}$, we deduce that $h_{\operatorname{top}}(F) = \log(\rho(A))$.
\end{proof}

\subsection{Split Ruelle operators for subsystems}%
\label{sub:Split Ruelle operators}
In this subsection, we define appropriate variants of the Ruelle operator for subsystems (see Definition~\ref{def:split ruelle operator}) on the suitable function spaces in our context. 

\smallskip

Let $f \colon S^2 \mapping S^2$ be an expanding Thurston map with a Jordan curve $\mathcal{C}\subseteq S^2$ satisfying $\post{f} \subseteq \mathcal{C}$. 
Let $\juxtapose{X^0_{\black}}{X^0_{\white}} \in \mathbf{X}^0(f, \mathcal{C})$ be the black $0$-tile and the white $0$-tile, respectively. 

\begin{definition}[Partial split Ruelle operators]    \label{def:partial Ruelle operator}
	Let $f$, $\mathcal{C}$, $F$ satisfy the Assumptions in Section~\ref{sec:The Assumptions}.
	Consider $\varphi \in C(S^2)$.
	We define a map $ \poperator[\varphi]{n} \colon \boufunspace{\ccolour} \mapping \boufunspace{\colour}$, for $\juxtapose{\colour}{\ccolour} \in \colours$, and $n \in \n_{0}$, by
	\begin{equation}    \label{eq:definition of partial Ruelle operators}
		\begin{split}
			\poperator[\varphi]{n}(u)(y) 
			\define & \sum_{ x \in F^{-n}(y) }  \ccndegF{\colour}{\ccolour}{n}{x} u(x) \myexp[\big]{ S_n^F \varphi(x) }   \\
			=& \sum_{ X^n \in \ccFTile{n}{\colour}{\ccolour} }  u\bigl((F^n|_{X^{n}})^{-1}(y)\bigr)  \myexp[\big]{ S_n^F \varphi\bigl((F^n|_{X^{n}})^{-1}(y)\bigr) }
		\end{split}
	\end{equation}
	for each real-valued bounded Borel function $u \in \boufunspace{\ccolour}$ and each point $y \in X^0_{\colour}$. 
\end{definition}
Note that by default, a summation over an empty set is equal to $0$. We will always use this convention in this paper.

If $X^0_{\colour} \subseteq F(\domF)$ for some $\colour \in \colours$, then 
\[
	\mathcal{L}^{(0)}_{F, \varphi, \colour, \ccolour}(u) = \begin{cases} u &\text{if } \ccolour = \colour \\ 0 &\text{if } \ccolour \ne \colour \end{cases}
\]
for each $\ccolour \in \colours$, which means that $\paroperator[\varphi]{0}{\colour}{\colour}$ is the identity map on $\boufunspace{\colour}$ for each $\colour \in \colours$ satisfying $X^0_{\colour} \subseteq F(\domF)$. 

\begin{remark}\label{rem:summation with respect to preimages for interior}
	If we assume in addition that $f(\mathcal{C}) \subseteq \mathcal{C}$, then by Propositions~\ref{prop:cell decomposition: invariant Jordan curve} and \ref{prop:subsystem:properties}~\ref{item:subsystem:properties:homeo}, for all $n \in \n$ and $y \in \inte[\big]{X^0_{\colour}}$, the summation (with respect to tiles) on the right-hand side of \eqref{eq:definition of partial Ruelle operators} becomes a summation with respect to preimages, i.e.,\[
		\mathcal{L}^{(n)}_{F, \varphi, \colour, \ccolour}(u)(y) = \sum_{ x \in F^{-k}(y) \cap \inte{X^{0}_{\ccolour}} } u(x) \myexp[\big]{ S^{F}_{n}\varphi(x) }.
	\]
\end{remark}

\begin{lemma}    \label{lem:well-defined partial Ruelle operator}
	Let $f$, $\mathcal{C}$, $F$ satisfy the Assumptions in Section~\ref{sec:The Assumptions}.
	Consider $\varphi \in C(S^2)$.
	We assume in addition that $f(\mathcal{C}) \subseteq \mathcal{C}$ and $F \in \sursubsystem$.
	Then for all $\juxtapose{n}{k} \in \n_0$, $\juxtapose{\colour}{\ccolour} \in \colours$, and $u \in \confunspace{\ccolour}$, we have
	\begin{align}
		\label{eq:well-defined continuity of partial Ruelle operator}
		\poperator[\varphi]{n} (u) &\in \confunspace{\colour} \qquad\quad \text{ and }  \\
		\label{eq:iteration of partial Ruelle operator}
		\qquad \qquad \poperator[\varphi]{n + k}(u) &= \sum_{\cccolour \in \colours} \paroperator[\varphi]{n}{\colour}{\cccolour} \bigl( \paroperator[\varphi]{k}{\cccolour}{\ccolour}(u) \bigr).  
	\end{align}
\end{lemma}
\begin{proof}
	The case of Lemma~\ref{lem:well-defined partial Ruelle operator} where either $n = 0$ or $k = 0$ follows immediately from Definition~\ref{def:partial Ruelle operator}. 
	Thus, we may assume without loss of generality that $\juxtapose{n}{k} \in \n$.

	The continuity of $\poperator[\varphi]{n}(u)$ follows trivially from \eqref{eq:definition of partial Ruelle operators} and Proposition~\ref{prop:subsystem:properties}~\ref{item:subsystem:properties:homeo}. 
	Note that the number of terms in the summation in \eqref{eq:definition of partial Ruelle operators} is fixed as $y$ in $X^{0}_{\colour}$ varies, and each term is continuous with respect to $y$ when $y \in X^{0}_{\colour}$ by Proposition~\ref{prop:subsystem:properties}~\ref{item:subsystem:properties:homeo}.

	We next establish \eqref{eq:iteration of partial Ruelle operator} by proving that the two sides of \eqref{eq:iteration of partial Ruelle operator} are equal at each point $y \in X^0_{\colour}$. 
	Indeed, by continuity, we only need to prove this for $y \in \inte[\big]{X^0_{\colour}}$. 
	Since $f(\mathcal{C}) \subseteq \mathcal{C}$, each preimage $x \in F^{-k}(y)$ belongs to the set $S^2 \setminus \mathcal{C} = \inte[\big]{X^0_{\black}} \cup \inte[\big]{X^0_{\white}}$ when $y \in \inte[\big]{X^0_{\colour}}$.
	Then by Remark~\ref{rem:summation with respect to preimages for interior}, we get the first two equalities and the second equality of the following:
	\begin{align*}
		\sum_{\cccolour \in \colours} \paroperator[\varphi]{n}{\colour}{\cccolour} \bigl( \paroperator[\varphi]{k}{\cccolour}{\ccolour} (u) \bigr)(y) 
		&= \sum_{\cccolour \in \colours} \sum_{ x \in F^{-n}(y) \cap \inte{X^{0}_{\cccolour}} } e^{S^F_n \varphi(x)}  \sum_{ z \in F^{-k}(x) \cap \inte{X^0_{\ccolour}} } u(z) e^{S^F_{k} \varphi(z)}  \\
		&= \sum_{x \in F^{-n}(y)} \sum_{ z \in F^{-k}(x) \cap \inte{X^0_{\ccolour}} }  u(z) e^{S^F_n\varphi(x) + S^F_{k}\varphi(z)} \\
		&= \sum_{z \in F^{-(n + k)}(y) \cap \inte{X^0_{\ccolour}} } u(z) e^{S^F_{n + k}\varphi(z)} \\
		&= \poperator[\varphi]{n + k}(u)(y),
	\end{align*}
	where the second-to-last equality holds since $x = F^{k}(z)$ and the last equality follows from Remark~\ref{rem:summation with respect to preimages for interior}.
\end{proof}

\begin{definition}[Split Ruelle operators]    \label{def:split ruelle operator}
	Let $f$, $\mathcal{C}$, $F$ satisfy the Assumptions in Section~\ref{sec:The Assumptions}.
	Consider $\varphi \in C(S^2)$.
	The \emph{split Ruelle operator} for the subsystem $F$ and the potential $\varphi$
	\[
		\splopt[\varphi] \colon C\bigl(X^0_{\black}\bigr) \times C\bigl(X^0_{\white}\bigr) \mapping C\bigl(X^0_{\black}\bigr) \times C\bigl(X^0_{\white}\bigr) 
	\]
	on the product space $\splfunspace$ is defined by
	\begin{equation}    \label{eq:def:split ruelle operator}
		\splopt[\varphi]\splfun \define \bigl( 
			\paroperator[\varphi]{1}{\black}{\black}(u_{\black}) + \paroperator[\varphi]{1}{\black}{\white}(u_{\white}), 
			\paroperator[\varphi]{1}{\white}{\black}(u_{\black}) + \paroperator[\varphi]{1}{\white}{\white}(u_{\white})
		\bigr)
	\end{equation}
	for each $u_{\black} \in C\bigl(X^0_{\black}\bigr)$ and each $u_{\white} \in C\bigl(X^0_{\white}\bigr)$.
\end{definition}

Note that by \eqref{eq:well-defined continuity of partial Ruelle operator} in Lemma~\ref{lem:well-defined partial Ruelle operator}, the operator $\splopt[\varphi]$ is well-defined, and it follows immediately from Definition~\ref{def:partial Ruelle operator} that $\splopt[\varphi]^{0}$ is the identity map on $\splfunspace$ if $F(\domF) = S^2$. 
One sees that $\splopt[\varphi] \colon \splfunspace \mapping \splfunspace$ has a natural extension to the space $\splboufunspace$ given by \eqref{eq:def:split ruelle operator} for $u_{\black} \in B\bigl(X^0_{\black}\bigr)$ and $u_{\white} \in B\bigl(X^0_{\white}\bigr)$. 
Moreover, it follows immediately from \eqref{eq:definition of partial Ruelle operators} and \eqref{eq:def:split ruelle operator} that $\splopt[\varphi]$ is a positive, continuous operator on $\splfunspace$ (\resp $\splboufunspace$).

For each color $\colour \in \colours$, we define the projection $\pi_{\colour} \colon \splboufunspace \mapping \boufunspace{\colour}$ by
\begin{equation}    \label{eq:projection on product function spaces}
	\pi_{\colour}\splfun \define u_{\colour}, \qquad \text{for } \splfun \in \splboufunspace.
\end{equation}

We show that the split Ruelle operator $\splopt[\varphi]$ is well-behaved under iteration.

\begin{lemma}    \label{lem:iteration of split-partial ruelle operator}
	Let $f$, $\mathcal{C}$, $F$ satisfy the Assumptions in Section~\ref{sec:The Assumptions}.
	Consider $\varphi \in C(S^2)$.
	We assume in addition that $f(\mathcal{C}) \subseteq \mathcal{C}$ and $F(\domF) = S^2$. 
	Then for all $n \in \n_0$, $u_{\black} \in C\bigl(X^0_{\black}\bigr)$, and $u_{\white} \in C\bigl(X^0_{\white}\bigr)$,
	\begin{equation}    \label{eq:iteration of split-partial ruelle operator}
		\splopt[\varphi]^{n} \splfun = \bigl( \paroperator[\varphi]{n}{\black}{\black}(u_{\black}) + \paroperator[\varphi]{n}{\black}{\white}(u_{\white}), \
		\paroperator[\varphi]{n}{\white}{\black}(u_{\black}) + \paroperator[\varphi]{n}{\white}{\white}(u_{\white}) \bigr).
	\end{equation}
\end{lemma}
\begin{proof}
	We prove \eqref{eq:iteration of split-partial ruelle operator} by induction. The case where $n = 0$ and the case where $n = 1$ both hold by definition. Assume now \eqref{eq:iteration of split-partial ruelle operator} holds for $n = k$ for some $k \in \n$. Then by Definition~\ref{def:split ruelle operator} and \eqref{eq:iteration of partial Ruelle operator} in Lemma~\ref{lem:well-defined partial Ruelle operator}, for each $\colour \in \colours$, we have
	\begin{align*}
		\pi_{\colour}\bigl( \splopt[\varphi]^{k + 1}\splfun \bigr) 
		&= \pi_{\colour} \bigl( 
			\splopt[\varphi]\bigl( 
				\paroperator[\varphi]{k}{\black}{\black}(u_{\black}) + \paroperator[\varphi]{k}{\black}{\white}(u_{\white}), \
				\paroperator[\varphi]{k}{\white}{\black}(u_{\black}) + \paroperator[\varphi]{k}{\white}{\white}(u_{\white}) 
			\bigr) 
		\bigr) \\
		&= \sum_{\ccolour \in \colours} \paroperator[\varphi]{1}{\colour}{\ccolour} \bigl( 
				\paroperator[\varphi]{k}{\ccolour}{\black}(u_{\black}) + \paroperator[\varphi]{k}{\ccolour}{\white}(u_{\white}) 
			\bigr)  \\
		&= \sum_{\cccolour \in \colours} \sum_{\ccolour \in \colours} \paroperator[\varphi]{1}{\colour}{\ccolour} \bigl( \paroperator[\varphi]{k}{\ccolour}{\cccolour}(u_{\cccolour}) \bigr)  \\
		&= \sum_{\cccolour \in \colours} \paroperator[\varphi]{k + 1}{\colour}{\cccolour}(u_{\cccolour}),
	\end{align*}
	for $u_{\black} \in C\bigl(X^0_{\black}\bigr)$ and $u_{\white} \in C\bigl(X^{0}_{\white}\bigr)$. This completes the inductive step, establishing \eqref{eq:iteration of split-partial ruelle operator}.
\end{proof}

\begin{rmk}
	Similarly, one can show that \eqref{eq:iteration of split-partial ruelle operator} in Lemma~\ref{lem:iteration of split-partial ruelle operator} holds for $\splfun \in \splboufunspace$.
\end{rmk}

\subsection{Split spheres}%
\label{sub:Split spheres}

In this subsection, we introduce the notion of the split sphere (see Definition~\ref{def:split sphere}), and set up some identifications and conventions (see Remarks~\ref{rem:disjoint union} and \ref{rem:probability measure in split setting}), which will be used frequently in this paper.

\smallskip

Let $f \colon S^2 \mapping S^2$ be an expanding Thurston map with a Jordan curve $\mathcal{C}\subseteq S^2$ satisfying $\post{f} \subseteq \mathcal{C}$. 
Let $\juxtapose{X^0_{\black}}{X^0_{\white}} \in \mathbf{X}^0(f, \mathcal{C})$ be the black $0$-tile and the white $0$-tile, respectively. 

\begin{definition}    \label{def:split sphere}
	We define the \emph{split sphere} $\splitsphere$ to be the disjoint union of $X^0_{\black}$ and $X^0_{\white}$, i.e., \[
		\splitsphere \define X^0_{\black} \sqcup X^0_{\white} = \bigl\{ (x, \colour) \describe \colour \in \colours, \, x \in X^0_{\colour} \bigr\}.
	\] 
	For each $\colour \in \colours$, let
	\begin{equation}    \label{eq:natural injection into splitsphere}
		i_{\colour} \colon X^0_{\colour} \mapping \splitsphere
	\end{equation}
	be the natural injection (defined by $i_{\colour}(x) \define (x, \colour)$). 
	Recall that the topology on $\splitsphere$ is defined as the finest topology on $\splitsphere$ for which both the natural injections $i_{\black}$ and $i_{\white}$ are continuous. 
	In particular, $\splitsphere$ is compact and metrizable.
	Obviously, a subset $U$ of $\splitsphere$ is open in $\splitsphere$ if and only if its preimage $i_{\colour}^{-1}(U)$ is open in $X^0_{\colour}$ for each $\colour \in \colours$. 
\end{definition}

Let $X$ and $Y$ be normed vector spaces over $\real$. 
Recall that a bounded linear map $T$ from $X$ to $Y$ is said to be an \emph{isomorphism} if $T$ is bijective and $T^{-1}$ is bounded (in other words, $\norm{T(x)} \geqslant C\norm{x}$ for some $C > 0$), and $T$ is called an \emph{isometry} if $\norm{T(x)} = \norm{x}$ for all $x \in X$.

The following is a standard result in functional analysis (see for example, \cite[p.~160]{folland2013real}).

\begin{proposition}[Dual of the product space is isometric to the product of the dual spaces]    \label{prop:Dual of the product is isometric to the product of the dual}
	Let $X$ and $Y$ be normed vector spaces and define $T \colon X^{*} \times Y^{*} \mapping (X \times Y)^{*}$ by $T(u, v)(x, y) = u(x) + v(y)$. Then $T$ is an isomorphism which is an isometry with respect to the norm $\norm{(x, y)} = \max\{\norm{x}, \norm{y}\}$ on $X \times Y$, the corresponding operator norm on $(X \times Y)^{*}$, and the norm $\norm{(u, v)} = \norm{u} + \norm{v}$ on $X^{*} \times Y^{*}$.
\end{proposition}

By Proposition~\ref{prop:Dual of the product is isometric to the product of the dual} and the Riesz representation theorem (see \cite[Theorems~7.17 and 7.8]{folland2013real}), we can identify $\bigl( \splfunspace \bigr)^{*}$ with the product of spaces of finite signed Borel measures $\splmeaspace$, where we use the norm $\norm{\splfun} = \max\{\norm{u_{\black}}, \norm{u_{\white}}\}$ on $\splfunspace$, the corresponding operator norm on $\bigl(\splfunspace\bigl)^{*}$, and the norm $\norm{\splmea} = \norm{\mu_{\black}} + \norm{\mu_{\white}}$ on $\splmeaspace$.

From now on, we write
\begin{align}
	\label{eq:split measure on two separated sets}
	\splmea(A_{\black}, A_{\white}) &\define \mu_{\black}(A_{\black}) + \mu_{\white}(A_{\white}), \\
	\label{eq:Riesz representation}
	\functional{\splmea}{\splfun} 
	&\define \functional{\mu_{\black}}{u_{\black}} + \functional{\mu_{\white}}{u_{\white}}
	= \int_{X^0_{\black}} \! u_{\black} \,\mathrm{d}\mu_{\black} + \int_{X^0_{\white}} \! u_{\white} \,\mathrm{d}\mu_{\white}, 
\end{align}
whenever $\splmea \in \splmeaspace$, $\splfun \in \splboufunspace$, and $A_{\black}$ and $A_{\white}$ are Borel subset of $X^0_{\black}$ and $X^0_{\white}$, respectively.
In particular, for each Borel set $A \subseteq S^2$, we define
\begin{equation}    \label{eq:split measure of set}
	\splmea(A) \define \splmea \bigl(A \cap X^0_{\black}, A \cap X^0_{\white} \bigr)
		= \mu_{\black} \bigl(A \cap X^0_{\black} \bigr) + \mu_{\white} \bigl(A \cap X^0_{\white} \bigr).
\end{equation} 

\begin{remark}    \label{rem:disjoint union}
	In the natural way, the product space $\splfunspace$ (\resp $\splboufunspace$) can be identified with $C(\splitsphere)$ (\resp $B(\splitsphere)$). 
	Similarly, the product space $\splmeaspace$ can be identified with $\mathcal{M}(\splitsphere)$. 
	Under such identifications, we write\[
		\int \! \splfun \,\mathrm{d}\splmea \define \functional{\splmea}{\splfun} \quad \text{and } \quad
		\splfun\splmea \define (u_{\black}\mu_{\black}, u_{\white}\mu_{\white})
	\]
	whenever $\splmea \in \splmeaspace$ and $\splfun \in \splboufunspace$.

	Moreover, we have the following natural identification of $\probmea{\splitsphere}$:\[
		\probmea{\splitsphere} = \bigl\{ \splmea \in \splmeaspace \describe \mu_{\black} \text{ and } \mu_{\white} \text{ are positive measures, } \mu_{\black}\bigl(X^0_{\black}\bigr) + \mu_{\white}\bigl(X^0_{\white}\bigr) = 1 \bigr\}.
	\]
	Here we follow the terminology in \cite[Section~3.1]{folland2013real} that a \emph{positive measure} is a signed measure that takes values in $[0, +\infty]$. 
\end{remark}

\begin{remark}\label{rem:probability measure in split setting}
	It is easy to see that \eqref{eq:split measure of set} defines a finite signed Borel measure $\mu \define \splmea$ on $S^2$. 
	Here we use the notation $\mu$ (\resp $\splmea$) when we view the measure as a measure on $S^2$ (\resp $\splitsphere$), and we will always use these conventions in this paper. 
	In this sense, for $u \in B(S^2)$ we have
	\begin{equation}    \label{eq:split measure on S2}
		\functional{\mu}{u} = \int \! u \,\mathrm{d}\mu = \int \! \splfun \,\mathrm{d}\splmea = \int_{X^0_{\black}} \! u \,\mathrm{d}\mu_{\black} + \int_{X^0_{\white}} \! u \,\mathrm{d}\mu_{\white},
	\end{equation}
	where $u_{\black} \define u|_{X^0_{\black}}$ and $u_{\white} \define u|_{X^0_{\white}}$. Moreover, if both $\mu_{\black}$ and $\mu_{\white}$ are positive measures and $\mu_{\black}\bigl(X^0_{\black}\bigr) + \mu_{\white}\bigl(X^0_{\white}\bigr) = 1$, then $\mu = \splmea$ defined by \eqref{eq:split measure of set} is a Borel probability measure on $S^2$. In view of the identifications in Remark~\ref{rem:disjoint union}, this means that if $\splmea \in \mathcal{P}(\splitsphere)$, then $\mu \in \mathcal{P}(S^2)$.
\end{remark}

For each color $\colour \in \colours$, we define the projection $\pi_{\colour} \colon \splmeaspace \mapping \mathcal{M}\bigl(X^{0}_{\colour}\bigr)$ by
\begin{equation}    \label{eq:projection on product measure spaces}
	\pi_{\colour}\splmea \define \mu_{\colour}, \qquad \text{for } \splmea \in \splmeaspace.
\end{equation}

\subsection{Adjoint operators of split Ruelle operators}%
\label{sub:Adjoint operators of split Ruelle operators}

In this subsection, we investigate the adjoint operators of split Ruelle operators for subsystems.
We collect and prove a few properties of the adjoint operators in Proposition~\ref{prop:dual split operator}, which will be used later.

\smallskip

Let $f$, $\mathcal{C}$, $F$ satisfy the Assumptions in Section~\ref{sec:The Assumptions}. 
We assume in addition that $F(\domF) = S^2$.
Consider $\varphi \in C(S^2)$.
Note that the split Ruelle operator $\splopt[\varphi]$ (see Definition~\ref{def:split ruelle operator}) is a positive, continuous operator on $\splfunspace$. 
Thus, the adjoint operator \[
	\splopt[\varphi]^{*} \colon \bigl( \splfunspace \bigr)^{*} \mapping \bigl( \splfunspace \bigr)^{*}
\]
of $\splopt[\varphi]$ acts on the dual space $\bigl( \splfunspace \bigr)^{*}$ of the Banach space $\splfunspace$. 
As discussed in Subsection~\ref{sub:Split spheres}, we identify $\bigl( \splfunspace \bigr)^{*}$ with the product of spaces of finite signed Borel measures $\splmeaspace$, where we use the norm $\norm{\splfun} = \max\{\norm{u_{\black}}, \norm{u_{\white}}\}$ on $\splfunspace$, the corresponding operator norm on $\bigl( \splfunspace \bigr)^{*}$, and the norm $\norm{\splmea} = \norm{\mu_{\black}} + \norm{\mu_{\white}}$ on $\splmeaspace$.
Then by Remark~\ref{rem:disjoint union}, we can also view $\splopt[\varphi]$ (\resp $\dualsplopt[\varphi]$) as an operator on $C(\splitsphere)$ (\resp $\mathcal{M}(\splitsphere)$).

In the following proposition, we summarize properties of the adjoint operator $\dualsplopt[\varphi]$.
\begin{proposition}    \label{prop:dual split operator}
	Let $f$, $\mathcal{C}$, $F$ satisfy the Assumptions in Section~\ref{sec:The Assumptions}. 
	Consider $\varphi \in C(S^2)$.
	We assume in addition that $f(\mathcal{C}) \subseteq \mathcal{C}$ and $F(\domF) = S^2$. 
	Consider arbitrary $n \in \n$ and $\splmea \in \splmeaspace$. Then the following statements hold:
	\begin{enumerate}[label=\rm{(\roman*)}]
		\smallskip 

		\item     \label{item:prop:dual split operator:act on bound function}
		$\bigl\langle \dualsplopt[\varphi]\splmea, \splfun \bigr\rangle = \functional[\big]{\splmea}{\splopt[\varphi]\splfun}$ for $\splfun \in \splboufunspace$.
		
		\smallskip

		\item     \label{item:prop:dual split operator:sphere version level n}
		For each Borel set $A \subseteq \domain{n}$ on which $F^{n}$ is injective, we have that $F^{n}(A)$ is a Borel set, and
		\begin{equation}  \label{eq:dual split Ruelle operator acts on measure on set A}
			\begin{split}
				\bigl( \dualsplopt[\varphi] \bigr)^{n}\splmea(A)
				&= \sum_{\colour \in \colours} \int_{F^{n}(A) \cap X^0_{\colour} } \! 
					\bigl(\ccndegF{\colour}{}{n}{\cdot \, } \myexp[\big]{ S_{n}^{F}\varphi } \bigr) \circ (F^{n}|_{A})^{-1} 
				\,\mathrm{d} \mu_{\colour}.
			\end{split}
		\end{equation}
		Here $\bigl( \dualsplopt[\varphi] \bigr)^{n}\splmea(A)$ is defined in \eqref{eq:split measure of set}.
		
		\smallskip

		\item     \label{item:prop:dual split operator:split version}
		For each color $\colour \in \colours$ and each Borel set $A_{\colour} \subseteq \domF \cap X^0_{\colour}$ on which $F$ is injective, we have that $F(A_{\colour})$ is a Borel set, and
		\begin{equation}    \label{eq:colour component of dual split Ruelle operator of measure}
			\begin{split}
				\pi_{\colour} \bigl(\dualsplopt[\varphi]\splmea \bigr)(A_{\colour})
				&= \sum_{\ccolour \in \colours} \int_{F(A_{\colour}) \cap X^0_{\ccolour} } \! 
					(\ccdegF{\ccolour}{\colour}{\cdot \, } \myexp{\varphi} ) \circ (F|_{A_{\colour}})^{-1} 
				\,\mathrm{d} \mu_{\ccolour}.
			\end{split}
		\end{equation}

		\smallskip

		\item     \label{item:prop:dual split operator:measure concentrate}
		$\bigl( \dualsplopt[\varphi] \bigr)^n \splmea \bigl( \bigcup \splDomain{n - 1} \bigr) = \bigl( \dualsplopt[\varphi] \bigr)^n \splmea \bigl(\bigcup \splDomain{n} \bigr)$, where 
		\[
			\splDomain{k} \define  \bigcup_{\colour \in \colours} \set[\big]{ i_{\colour}\parentheses[\big]{ X^k } \describe X^k \in \Domain{k}, \, X^k \subseteq X^0_{\colour} }
		\] 
		for each $k \in \n_{0}$. 
		Here $i_{\colour}$ is defined by \eqref{eq:natural injection into splitsphere}.
	\end{enumerate}
\end{proposition}
Recall that a collection $\mathfrak{P}$ of subsets of a set $X$ is a \emph{$\pi$-system} if it is closed under the intersection, i.e., if $\juxtapose{A}{B} \in \mathfrak{P}$ then $A \cap B \in \mathfrak{P}$. 
A collection $\mathfrak{L}$ of subsets of $X$ is a \emph{$\lambda$-system} if the following conditions are satisfied: 
(1) $X \in \mathfrak{L}$. 
(2) If $\juxtapose{B}{C} \in \mathfrak{L}$ and $B \subseteq C$, then $C \setminus B \in \mathfrak{L}$. 
(3) If $A_n \in \mathfrak{L}$, $n \in \n$, with $A_n \subseteq A_{n + 1}$, then $\bigcup_{n \in \n} A_n \in \mathfrak{L}$.
\begin{proof}
	We follow the identifications discussed in Remark~\ref{rem:disjoint union}. 
	Recall from Definition~\ref{def:split sphere} that $\splitsphere = \disjointuniontile$ is the disjoint union of $X^0_{\black}$ and $X^0_{\white}$.
	Then $\splopt[\varphi]$ is a continuous operator on $C(\splitsphere)$.

	\smallskip

	\ref{item:prop:dual split operator:act on bound function} 
	It suffices to show that for each $\mu \in \mathcal{M}(\splitsphere)$ and each Borel set $A \subseteq \splitsphere$,
	\begin{equation}    \label{eq:dual split operator:borel set}
		\functional[\big]{ \dualsplopt[\varphi](\mu) }{ \mathbbm{1}_{A} } 
		= \functional[\big]{\mu}{\splopt[\varphi](\mathbbm{1}_{A})}.
	\end{equation}

	Let $\mathfrak{L}$ be the collection of Borel sets $A \subseteq \splitsphere$ for which \eqref{eq:dual split operator:borel set} holds. Denote the collection of open subsets of $\splitsphere$ by $\mathfrak{G}$. Then $\mathfrak{G}$ is a $\pi$-system.

	We observe from \eqref{eq:def:split ruelle operator} that if $\{u_{n}\}_{n \in \n}$ is a non-decreasing sequence of real-valued continuous functions on $\splitsphere$, then so is $\bigl\{ \splopt[\varphi](u_{n}) \bigr\}_{n \in \n}$.

	By the definition of $\dualsplopt[\varphi]$, we have
	\begin{equation}    \label{eq:dual split operator:continuous function}
		\functional[\big]{\dualsplopt[\varphi](\mu)}{u} 
		= \functional[\big]{\mu}{\splopt[\varphi](u)}
	\end{equation}
	for $u \in C(\splitsphere)$. Fix an open set $U \subseteq \splitsphere$, then there exists a non-decreasing sequence $\{g_{n}\}_{n \in \n}$ of real-valued continuous functions on $\splitsphere$ supported in $U$ such that $g_n$ converges to $\mathbbm{1}_{U}$ pointwise as $n \to +\infty$. 
	Then $\bigl\{\splopt[\varphi](g_n) \bigr\}_{n \in \n}$ is also a non-decreasing sequence of continuous functions, whose pointwise limit is $\splopt[\varphi](\mathbbm{1}_{U})$. By the Lebesgue Monotone Convergence Theorem and \eqref{eq:dual split operator:continuous function}, we can conclude that \eqref{eq:dual split operator:borel set} holds for $A = U$. Thus $\mathfrak{G} \subseteq \mathfrak{L}$.

	We now prove that $\mathfrak{L}$ is a $\lambda$-system. 
	Indeed, since \eqref{eq:dual split operator:continuous function} holds for $u = \mathbbm{1}_{\splitsphere}$, we get $\splitsphere \in \mathfrak{L}$. 
	For each pair $\juxtapose{B}{C} \in \mathfrak{L}$ with $B \subseteq C$, it follows from \eqref{eq:def:split ruelle operator} that $\mathbbm{1}_{C} - \mathbbm{1}_{B} = \mathbbm{1}_{C \setminus B}$ and $\splopt[\varphi](\mathbbm{1}_{C}) - \splopt[\varphi](\mathbbm{1}_{B}) = \splopt[\varphi](\mathbbm{1}_{C} - \mathbbm{1}_{B}) = \splopt[\varphi](\mathbbm{1}_{C \setminus B})$. 
	Thus $C \setminus B \in \mathfrak{L}$. 
	Finally, given $A_n \in \mathfrak{L}$, $n \in \n$, with $A_n \subseteq A_{n + 1}$, and denote $A \define \bigcup_{n \in \n} A_n$. 
	Then $\{\mathbbm{1}_{A_{n}}\}_{n \in \n}$ and $\bigl\{ \splopt[\varphi](\mathbbm{1}_{A_{n}}) \bigr\}_{n \in \n}$ are non-decreasing sequences of real-valued Borel functions on $\splitsphere$ that converge pointwise to $\mathbbm{1}_{A}$ and $\splopt[\varphi](\mathbbm{1}_{A})$, respectively, as $n \to +\infty$. 
	Then by the the Lebesgue Monotone Convergence Theorem, we get $A \in \mathfrak{L}$. Hence $\mathfrak{L}$ is a $\lambda$-system. 

	Recall that Dynkin's $\pi$-$\lambda$ theorem (see for example, \cite[Theorem~3.2]{billingsley2008probability}) states that if $\mathfrak{P}$ is a $\pi$-system and $\mathfrak{L}$ is a $\lambda$-system that contains $\mathfrak{P}$, then the $\sigma$-algebra $\sigma(\mathfrak{P})$ generated by $\mathfrak{P}$ is a subset of $\mathfrak{L}$. Thus by Dynkin's $\pi$-$\lambda$ theorem, the Borel $\sigma$-algebra $\sigma(\mathfrak{G})$ is a subset of $\mathfrak{L}$, i.e., equality~\eqref{eq:dual split operator:borel set} holds for each Borel set $A \subseteq \splitsphere$. This finishes the proof of statement~\ref{item:prop:dual split operator:act on bound function}.

	\smallskip

	\ref{item:prop:dual split operator:sphere version level n}  
	We fix an arbitrary Borel set $A \subseteq \domain{n}$ on which $F^{n}$ is injective. Denote $A_{\black} \define A \cap X^0_{\black}$ and $A_{\white} \define A \cap X^0_{\white}$.

	For each $X^{n} \in \Domain{n}$, it follows from Proposition~\ref{prop:subsystem:properties}~\ref{item:subsystem:properties:homeo} that $F^{n}|_{X^{n}}$ is a homeomorphism from $X^{n}$ onto $F^{n}(X^{n})$, which maps Borel sets to Borel sets. Thus $F^{n}(A)$ is a Borel set since\[
		F^{n}(A) = F^{n} \Bigl( \bigcup_{X^n \in \Domain{n}} A \cap X^{n} \Bigr) = \bigcup_{X^{n} \in \Domain{n}} F^{n}\left(A \cap X^{n}\right) = \bigcup_{X^{n} \in \Domain{n}} F^{n}|_{X^{n}}(A).
	\]

	We now prove \eqref{eq:dual split Ruelle operator acts on measure on set A}. 
	By \eqref{eq:split measure of set}, statement~\ref{item:prop:dual split operator:act on bound function}, and \eqref{eq:Riesz representation}, we get 
	\[
		\bigl( \dualsplopt[\varphi] \bigr)^{n}\splmea(A) 
		= \functional{\splmea}{\splopt[\varphi]^{n}(\mathbbm{1}_{A_{\black}}, \mathbbm{1}_{A_{\white}})} \\
		= \sum_{\colour \in \colours} \functional[\big]{\mu_{\colour}} { \pi_{\colour} \bigl(\splopt[\varphi]^{n}(\mathbbm{1}_{A_{\colour}}, \mathbbm{1}_{A_{\white}}) \bigr) }.
	\]
	Then it suffices to show that for each $\colour \in \colours$ and each $x \in X^0_{\colour}$,
	\[
		\pi_{\colour} \bigl(\splopt[\varphi]^{n}(\mathbbm{1}_{A_{\black}}, \mathbbm{1}_{A_{\white}}) \bigr)(x)
		= \mathbbm{1}_{F^{n}(A)}(x) \cdot \bigl( \ccndegF{\colour}{}{n}{\cdot \, } \myexp[\big]{ S_{n}^{F}\varphi } \bigr) \circ (F^{n}|_{A})^{-1}(x). 
	\]
	Indeed, by \eqref{eq:iteration of split-partial ruelle operator}, \eqref{eq:projection on product function spaces}, and \eqref{eq:definition of partial Ruelle operators},
	\begin{align*}
		\pi_{\colour} \bigl(\splopt[\varphi]^{n} (\mathbbm{1}_{A_{\black}}, \mathbbm{1}_{A_{\white}}) \bigr)(x) 
		&= \paroperator[\varphi]{n}{\colour}{\black}(\mathbbm{1}_{A_{\black}})(x) + \paroperator[\varphi]{n}{\colour}{\white}(\mathbbm{1}_{A_{\white}})(x)     \nonumber\\ 
		&= \sum_{\ccolour \in \colours} \sum_{ X^n \in \ccFTile{n}{\colour}{\ccolour} }   \bigl(\mathbbm{1}_{A_{\ccolour}} \!\cdot \myexp[\big]{ S_{n}^{F}\varphi } \bigr) \circ (F^{n}|_{X^n})^{-1}(x) \nonumber\\
		&= \sum_{\ccolour \in \colours} \sum_{ X^n \in \ccFTile{n}{\colour}{\ccolour} }   \bigl(\mathbbm{1}_{A} \!\cdot \myexp[\big]{ S_{n}^{F}\varphi } \bigr) \circ (F^{n}|_{X^n})^{-1}(x)  \nonumber\\
		&= \sum_{ X^n \in \cFTile{n} }   \bigl(\mathbbm{1}_{A} \!\cdot \myexp[\big]{ S_{n}^{F}\varphi } \bigr) \circ (F^{n}|_{X^n})^{-1}(x) \nonumber\\
		&= \sum_{ y \in F^{-n}(x) } \ccndegF{\colour}{}{n}{y}  \mathbbm{1}_{A}(y) \myexp[\big]{ S_{n}^{F}\varphi(y) } \label{eq:projection split operator split function} \\
		&= \mathbbm{1}_{F^{n}(A)}(x) \bigl(\ccndegF{\colour}{}{n}{\cdot \, } \myexp[\big]{ S_{n}^{F}\varphi }\bigr) \circ (F^{n}|_{A})^{-1}(x), \nonumber
	\end{align*}
	where the third equality follows from the fact that for each $\ccolour \in \colours$ and each $X^n \in \ccFTile{n}{\colour}{\ccolour}$, the point $z = (F^{n}|_{X^{n}})^{-1}(x) \in A$ if and only if $z \in A_{\ccolour}$ since $z \in X^{n} \subseteq X^0_{\ccolour}$, and the last equality holds since $F$ is injective on $A$. 
	Thus, we finish the proof of statement~\ref{item:prop:dual split operator:sphere version level n}.

	\smallskip

	\ref{item:prop:dual split operator:split version}  
	We arbitrarily fix a color $\colour \in \colours$ and a Borel set $A_{\colour} \subseteq \domF \cap X^0_{\colour}$ on which $F$ is injective. Then it follows immediately from statement~\ref{item:prop:dual split operator:sphere version level n} that $F(A_{\colour})$ is a Borel set. 

	We now prove \eqref{eq:colour component of dual split Ruelle operator of measure}. Without loss of generality, we can assume that $\colour = \black$. Then, by \eqref{eq:projection on product measure spaces}, statement~\ref{item:prop:dual split operator:act on bound function}, and \eqref{eq:Riesz representation}, we get
	\begin{align*}
	\pi_{\black} \bigl(\dualsplopt[\varphi]\splmea \bigr)(A_{\black})
		&= \dualsplopt[\varphi]\splmea(A_{\black}, \emptyset) 
		 = \functional{\splmea}{\splopt[\varphi](\mathbbm{1}_{A_{\black}}, 0)} \\
		&= \functional[\big]{\mu_{\black}} { \pi_{\black} \bigl(\splopt[\varphi](\mathbbm{1}_{A_{\black}}, 0) \bigr) } + \functional[\big]{\mu_{\white}}{ \pi_{\white} \bigl(\splopt[\varphi](\mathbbm{1}_{A_{\black}}, 0) \bigr) }.
	\end{align*}
	It suffices to show that for each $\ccolour \in \colours$ and each $x \in X^0_{\ccolour}$,
	\[
		\pi_{\ccolour} \bigl(\splopt[\varphi](\mathbbm{1}_{A_{\black}}, 0) \bigr)(x)
		= \mathbbm{1}_{F(A_{\black})}(x) \cdot ( \ccdegF{\ccolour}{\black}{\cdot \, } \myexp{\varphi} ) \circ (F|_{A_{\black}})^{-1}(x). 
	\]
	Indeed, by \eqref{eq:def:split ruelle operator}, \eqref{eq:projection on product function spaces}, and \eqref{eq:definition of partial Ruelle operators},
	\begin{align*}
		\pi_{\ccolour} \bigl(\splopt[\varphi](\mathbbm{1}_{A_{\black}}, 0) \bigr)(x)
		&= \paroperator[\varphi]{1}{\ccolour}{\black}(\mathbbm{1}_{A_{\black}})(x) \\
		&= \sum_{ X^1 \in \ccFTile{1}{\ccolour}{\black} }  \bigl(\mathbbm{1}_{A_{\black}} \!\cdot \myexp{\varphi} \bigr) \circ (F|_{X^1})^{-1}(x) \\
		&= \sum_{ y \in F^{-1}(x) } \ccdegF{\ccolour}{\black}{y}  \mathbbm{1}_{A_{\black}}(y) \myexp{\varphi(y)} \\
		&= \mathbbm{1}_{F(A_{\black})}(x) (\ccdegF{\ccolour}{\black}{\cdot \, }\myexp{\varphi}) \circ (F|_{A_{\black}})^{-1}(x),
	\end{align*}
	where the last equality holds since $F$ is injective on $A_{\black}$. Thus we finish the proof of statement~\ref{item:prop:dual split operator:split version}.

	\smallskip

	\ref{item:prop:dual split operator:measure concentrate}  
	For convenience we set $\spllimitset^{k} \define \bigcup\splDomain{n} \subseteq \splitsphere$ and $\limitset^{k}_{\colour} \define i_{\colour}^{-1}\bigl( \bigcup \splDomain{k} \bigr) \subseteq X^0_{\colour}$ for each $k \in \n_0$ and each $\colour \in \colours$. 
	Note that $\spllimitset^{k} = \limitset^{k}_{\black} \sqcup \limitset^{k}_{\white}$ and $\limitset^{k}_{\colour} = \bigcup \bigl\{ X^{k} \in \Domain{k} \describe X^{k} \subseteq X^{0}_{\colour} \bigr\}$ for each $k \in \n_{0}$ and each $\colour \in \colours$. 
	In particular, we have $\spllimitset^{0} = \splitsphere$ since $F(\domF) = S^2$.
	By statement~\ref{item:prop:dual split operator:act on bound function} and \eqref{eq:Riesz representation}, we have
	\begin{align*}
		\bigl( \dualsplopt[\varphi] \bigr)^{n} \splmea\bigl( \spllimitset^{n-1} \bigr) 
		&= \bigl( \dualsplopt[\varphi] \bigr)^{n} \splmea\bigl( \limitset^{n-1}_{\black}, \limitset^{n-1}_{\white} \bigr) 
		= \bigl\langle \splmea, \splopt[\varphi]^n \bigl(\mathbbm{1}_{\limitset^{n-1}_{\black}}, \mathbbm{1}_{\limitset^{n-1}_{\white}} \bigr) \bigr\rangle \\
		&=  \bigl\langle \mu_{\black},  \pi_{\black} \bigl(\splopt[\varphi]^n \bigl(\mathbbm{1}_{\limitset^{n-1}_{\black}}, \mathbbm{1}_{\limitset^{n-1}_{\white}} \bigr) \bigr)  \bigr\rangle
		+  \bigl\langle \mu_{\white},  \pi_{\white} \bigl( \splopt[\varphi]^n \bigl(\mathbbm{1}_{\limitset^{n-1}_{\black}}, \mathbbm{1}_{\limitset^{n-1}_{\white}}\bigr) \bigr)  \bigr\rangle, \\ 
		\bigl( \dualsplopt[\varphi] \bigr)^{n} \splmea\bigl( \spllimitset^{n} \bigr) 
		&= \bigl( \dualsplopt[\varphi] \bigr)^{n} \splmea\bigl( \limitset^{n}_{\black}, \limitset^{n}_{\white} \bigr) 
		= \bigl\langle \splmea, \splopt[\varphi]^n \bigl(\mathbbm{1}_{\limitset^{n}_{\black}}, \mathbbm{1}_{\limitset^{n}_{\white}} \bigr) \bigr\rangle \\
		&=  \bigl\langle \mu_{\black},  \pi_{\black} \bigl(\splopt[\varphi]^n \bigl(\mathbbm{1}_{\limitset^{n}_{\black}}, \mathbbm{1}_{\limitset^{n}_{\white}} \bigr) \bigr)  \bigr\rangle
		+  \bigl\langle \mu_{\white},  \pi_{\white} \bigl( \splopt[\varphi]^n \bigl(\mathbbm{1}_{\limitset^{n}_{\black}}, \mathbbm{1}_{\limitset^{n}_{\white}}\bigr) \bigr)  \bigr\rangle.
	\end{align*}
	It suffices to show that for each $\colour \in \colours$,
	\[
		\pi_{\colour} \bigl( \splopt[\varphi]^n \bigl(\mathbbm{1}_{\limitset^{n-1}_{\black}}, \mathbbm{1}_{\limitset^{n-1}_{\white}} \bigr) \bigr) 
		= \pi_{\colour} \bigl( \splopt[\varphi]^n \bigl( \mathbbm{1}_{\limitset^{n}_{\black}}, \mathbbm{1}_{\limitset^{n}_{\white}} \bigr) \bigr).
	\]
	Indeed, by \eqref{eq:projection on product function spaces} in Definition~\ref{def:split ruelle operator}, \eqref{eq:iteration of split-partial ruelle operator} in Lemma~\ref{lem:iteration of split-partial ruelle operator}, and \eqref{eq:definition of partial Ruelle operators} in Definition~\ref{def:partial Ruelle operator},
	\begin{align*}
		\pi_{\colour} \bigl( \splopt[\varphi]^n \bigl(\mathbbm{1}_{\limitset^{n-1}_{\black}}, \mathbbm{1}_{\limitset^{n-1}_{\white}} \bigr) \bigr) 
		&= \paroperator[\varphi]{n}{\colour}{\black} \bigl(\mathbbm{1}_{\limitset^{n-1}_{\black}} \bigr) + \paroperator[\varphi]{n}{\colour}{\white} \bigl(\mathbbm{1}_{\limitset^{n-1}_{\white}}\bigr) \\
		&= \sum_{\ccolour \in \colours} \sum_{ X^n \in \ccFTile{n}{\colour}{\ccolour} }   \bigl( \mathbbm{1}_{ \limitset^{n-1}_{\ccolour} } \!\cdot \myexp[\big]{ S_{n}^{F}\varphi } \bigr) \circ (F|_{X^n})^{-1}   \\
		&= \sum_{\ccolour \in \colours} \sum_{ X^n \in \ccFTile{n}{\colour}{\ccolour} }   \bigl(\mathbbm{1}_{ \limitset^{n}_{\ccolour} } \!\cdot \myexp[\big]{ S_{n}^{F}\varphi } \bigr) \circ (F|_{X^n})^{-1}   \\
		&= \pi_{\colour} \bigl( \splopt[\varphi]^n \bigl( \mathbbm{1}_{\limitset^{n}_{\black}}, \mathbbm{1}_{\limitset^{n}_{\white}} \bigr) \bigr),
	\end{align*}
	where the third equality holds since $X^n \subseteq \limitset^{n}_{\ccolour} \subseteq \limitset^{n-1}_{\ccolour}$ for each $\ccolour \in \colours$ and each $X^n \in \ccFTile{n}{\colour}{\ccolour}$ by Proposition~\ref{prop:subsystem:properties invariant Jordan curve}~\ref{item:subsystem:properties invariant Jordan curve:decreasing relation of domains} and the fact $\spllimitset^{0} = \splitsphere$. 
\end{proof}

\subsection{Eigenmeasures}%
\label{sub:Eigenmeasures}

By applying the Schauder--Tikhonov Fixed Point Theorem, we establish in Theorem~\ref{thm:subsystem:eigenmeasure existence and basic properties} the existence of an eigenmeasure of the adjoint $\dualsplopt$ of the split Ruelle operator $\splopt$.
We also show in Theorem~\ref{thm:subsystem:eigenmeasure existence and basic properties}~\ref{item:thm:subsystem:eigenmeasure existence and basic properties:strongly irreducible vertex zero measure} that if $F$ is strongly irreducible (see Definition~\ref{def:irreducibility of subsystem}), then the set of vertices is of measure zero with respect to such an eigenmeasure.
\smallskip

We follow the conventions discussed in Remarks~\ref{rem:disjoint union} and \ref{rem:probability measure in split setting} in this subsection.
    
\begin{theorem}    \label{thm:subsystem:eigenmeasure existence and basic properties}
	Let $f$, $\mathcal{C}$, $F$, $d$, $\potential$ satisfy the Assumptions in Section~\ref{sec:The Assumptions}. 
	We assume in addition that $f(\mathcal{C}) \subseteq \mathcal{C}$ and $F(\domF) = S^2$. 
	Then there exists a Borel probability measure $\eigmea = (m_{\black}, m_{\white}) \in \mathcal{P}(\splitsphere)$ such that
	\begin{equation}    \label{eq:eigenmeasure for subsystem}
		\dualsplopt\spleigmea = \eigenvalue \spleigmea,
	\end{equation}
	where $\eigenvalue = \big\langle \dualsplopt\spleigmea, \indicator{\splitsphere} \big\rangle$. 
	Moreover, any $\eigmea = \spleigmea \in \mathcal{P}(\splitsphere)$ that satisfies \eqref{eq:eigenmeasure for subsystem} for some $\eigenvalue > 0$ has the following properties:
	\begin{enumerate}[label=\rm{(\roman*)}]
		\smallskip

		\item     \label{item:thm:subsystem:eigenmeasure existence and basic properties:eigenmeasure support on limitset}
		$\eigmea(\limitset(F, \mathcal{C})) = 1$.

		\smallskip

		\item 	  \label{item:thm:subsystem:eigenmeasure existence and basic properties:strongly irreducible vertex zero measure}
		If $F$ is strongly irreducible, then $\eigmea \bigl( \bigcup_{j = 0}^{+\infty}f^{-j}(\post{f}) \bigr) = 0$.
	\end{enumerate}
\end{theorem}
Note that we use the notation $\spleigmea$ (\resp $\eigmea$) to emphasize that we treat the eigenmeasure as a Borel probability measure on $\splitsphere$ (\resp $S^2$).

The proof of Theorem~\ref{thm:subsystem:eigenmeasure existence and basic properties} will be given at the end of this subsection.

Under the same assumption as in Theorem~\ref{thm:subsystem:eigenmeasure existence and basic properties}~\ref{item:thm:subsystem:eigenmeasure existence and basic properties:strongly irreducible vertex zero measure}, we will show $\eigmea\bigl( \bigcup_{j = 0}^{+\infty} f^{-j}(\mathcal{C}) \bigr) = 0$ in Proposition~\ref{prop:subsystem edge Jordan curve has measure zero} by using property~\ref{item:thm:subsystem:eigenmeasure existence and basic properties:strongly irreducible vertex zero measure}.

\smallskip

By some elementary calculations, we have the following results for a $2 \times 2$ matrix.
\begin{lemma}    \label{lem:sumnorm of iteration of matrix}
	Let $A$ be a $2 \times 2$ matrix. 
	We denote by $\sumnorm{A}$ the sum of all the absolute values of entries in $A$. 
	If $A \mathbf{x} = \lambda \mathbf{x}$ for some $\lambda \in \real$ and $0 \ne \mathbf{x} \in \real^{2}$, then $\sumnorm{A^n} \geqslant \abs{\lambda} ^n$ for each $n \in \n$.
\end{lemma}
\begin{proof}
	It suffices to consider the case when $n = 1$.
	Without loss of generality, we can assume that for $\mathbf{x} = (x, y)$ we have $\abs{x} \geqslant \abs{y}$ and $x \ne 0$.
	Then the inequality follows immediately by looking at the first component of $A \mathbf{x}$: $\abs{ \lambda x } = \abs{ a_{11}x + a_{12}y } \leqslant \abs{x} \sumnorm{A}$.
\end{proof}

\begin{proof}[Proof of Theorem~\ref{thm:subsystem:eigenmeasure existence and basic properties}]
	We first show the existence.
	Define $\tau \colon \probmea{\splitsphere} \mapping	\probmea{\splitsphere}$ by \[
		\tau\splmea \define \frac{\dualsplopt\splmea}{ \functional{\dualsplopt\splmea}{\indicator{\splitsphere}} }.
	\]
	Then $\tau$ is a continuous transformation on the non-empty, convex, compact (in the weak$^*$ topology, by Alaoglu's theorem) space $\probmea{\splitsphere}$ of Borel probability measures on $\splitsphere$. 
	By the Schauder--Tikhonov Fixed Point Theorem (see for example, \cite[Theorem~V.10.5]{dunford1988LinearOperators}), there exists a measure $\eigmea = \spleigmea \in \probmea{\splitsphere}$ such that $\tau\spleigmea = \spleigmea$. 
	Thus $\dualsplopt\spleigmea = \eigenvalue \spleigmea$ with $\eigenvalue \define \bigl\langle \dualsplopt\spleigmea, \indicator{\splitsphere} \bigr\rangle$. Note that $\eigenvalue > 0 $ since $F(\domF) = S^2$.

	We now show that any $\eigmea = \spleigmea \in \mathcal{P}(\splitsphere)$ that satisfies \eqref{eq:eigenmeasure for subsystem} for some $\eigenvalue > 0$ has properties~\ref{item:thm:subsystem:eigenmeasure existence and basic properties:eigenmeasure support on limitset} and~\ref{item:thm:subsystem:eigenmeasure existence and basic properties:strongly irreducible vertex zero measure}. 

	We first verify property~\ref{item:thm:subsystem:eigenmeasure existence and basic properties:eigenmeasure support on limitset}.
	For each $n \in \n_0$, we set
	\[
		\splDomain{n} \define  \bigcup_{\colour \in \colours} \bigl\{ i_{\colour}(X^n) \describe X^n \in \Domain{n}, \, X^n \subseteq X^0_{\colour} \bigr\},
	\]
	where $i_{\colour}$ is defined by \eqref{eq:natural injection into splitsphere}. 
	Noting that $F(\domF) = S^2$, we have $\bigcup\splDomain{0} = \splitsphere$.
	Since $\dualsplopt\spleigmea = \eigenvalue \spleigmea$, it follows from Proposition~\ref{prop:dual split operator}~\ref{item:prop:dual split operator:measure concentrate} and induction on $n$ that for each $n \in \n$,
	\[
		\spleigmea\parentheses[\Big]{ \bigcup \splDomain{n} } = \spleigmea \parentheses[\Big]{ \bigcup \splDomain{0} } = \spleigmea \parentheses{ \splitsphere } = 1.
	\]
	Then by Remark~\ref{rem:probability measure in split setting} and \eqref{eq:split measure of set}, for each $n \in \n$, 
	\[
		1 \geqslant \eigmea \parentheses[\Big]{ \domain{n} } = \spleigmea \parentheses[\Big]{ \domain{n} } 
		\geqslant \spleigmea\parentheses[\Big]{ \bigcup \splDomain{n} } = 1.
	\]
	Thus, by Proposition~\ref{prop:subsystem:properties invariant Jordan curve}~\ref{item:subsystem:properties invariant Jordan curve:decreasing relation of domains} and \eqref{eq:def:limitset}, we have 
	\[
		\eigmea(\limitset) = \lim\limits_{n \to +\infty}\eigmea\Bigl(\domain{n}\Bigr) = 1,
	\] 
	which proves property~\ref{item:thm:subsystem:eigenmeasure existence and basic properties:eigenmeasure support on limitset}.

	Next, we verify property~\ref{item:thm:subsystem:eigenmeasure existence and basic properties:strongly irreducible vertex zero measure}. 
	Assume that $F$ is strongly irreducible.

	We will prove that $\eigmea \bigl( \bigcup_{j = 0}^{+\infty}f^{-j}(\post{f}) \bigr) = 0$.
	Since $\bigcup_{j = 0}^{+\infty}f^{-j}(\post{f})$ is a countable set, by property~\ref{item:thm:subsystem:eigenmeasure existence and basic properties:eigenmeasure support on limitset}, the conclusion follows if we can prove that $\eigmea(\{y\}) = 0$ for each $y \in \limitset \cap \bigcup_{j = 0}^{+\infty} f^{-j}(\post{f})$. 

	We claim that it suffices to show that $\eigmea(\{x\}) = 0$ for each periodic $x \in \limitset \cap \post{f}$.
	To see this, let $y \in \limitset \cap \bigcup_{j = 0}^{+\infty} f^{-j}(\post{f})$ be arbitrary.
	We follow the convention that if $p \notin X^0_{\colour}$ for some color $\colour \in \colours$, then $m_{\colour}(\{p\}) = 0$.
	For each $\colour \in \colours$, since $\pi_{\colour} \bigl( \dualsplopt\spleigmea \bigr) = \pi_{\colour}(\eigenvalue \spleigmea) = \eigenvalue m_{\colour}$, by Proposition~\ref{prop:dual split operator}~\ref{item:prop:dual split operator:split version}, we have
	\begin{equation*}    \label{eq:point eigenmeasure relation in split form}
		\begin{split}
			\eigenvalue m_{\colour}(\{y\}) 
			&= \sum_{\ccolour \in \colours} \int_{F(\{y\}) \cap X^0_{\ccolour}} \! ( \ccdegF{\ccolour}{\colour}{\cdot \, } \myexp{\phi} ) \circ \bigl( F|_{\{y\}} \bigr)^{-1}  \,\mathrm{d} m_{\ccolour}  \\
			&= \sum_{\ccolour \in \colours} \ccdegF{\ccolour}{\colour}{y} \myexp{\phi(y)} m_{\ccolour}(\{F(y)\}).
		\end{split}
	\end{equation*}
	By using the notion of local degree matrix (see Definition~\ref{def:subsystem local degree}), we can write the equation above as
	\begin{equation}  \label{eq:point eigenmeasure relation in matrix form}
		\begin{bmatrix}
			m_{\black}(\{y\}) \\ m_{\white}(\{y\}) 
		\end{bmatrix}
		=
		\frac{\myexp{\phi(y)}}{\eigenvalue} \Deg{}{y}
		\begin{bmatrix}
			m_{\black}(\{F(y)\}) \\ m_{\white}(\{F(y)\}) 
		\end{bmatrix}.
	\end{equation}
	Since $F^{n}(y) \in \limitset \subseteq \domF$ for each $n \in \n$ by Proposition~\ref{prop:subsystem:properties invariant Jordan curve}~\ref{item:subsystem:properties invariant Jordan curve:decreasing relation of domains}, we can iterate \eqref{eq:point eigenmeasure relation in matrix form} under $F$. 
	Then it follows from Lemma~\ref{lem:iteration of local degree matrix} and induction that
	\begin{equation}    \label{eq:point eigenmeasure relation in matrix form under iteration}
	 	\begin{bmatrix}
			m_{\black}(\{y\}) \\ m_{\white}(\{y\}) 
		\end{bmatrix}
		=
		\frac{ \myexp{S_n\phi(y)} }{\eigenvalue^n} \Deg{n}{y}
		\begin{bmatrix}
			m_{\black}(\{F^n(y)\}) \\ m_{\white}(\{F^n(y)\}) 
		\end{bmatrix}
	\end{equation} 
	for each $n \in \n$. 
	Hence, since $y \in \bigcup_{j = 0}^{+\infty} f^{-j}(\post{f})$ and $\eigmea(\{y\}) = 0$ if and only if $m_{\black}(\{y\}) = m_{\white}(\{y\}) = 0$, by \eqref{eq:point eigenmeasure relation in matrix form under iteration}, it suffices to show that $\eigmea(\{x\}) = 0$ for each periodic $x \in \limitset \cap \post{f}$.

	It remains to show that $\eigmea(\{x\}) = 0$ for each periodic $x \in \limitset \cap \post{f}$.
	We argue by contradiction and assume that there exists $x \in \limitset \cap \post{f}$ such that $F^{\ell}(x) = x$ for some $\ell \in \n$ and $\eigmea(\{x\}) \ne 0$. 
	Since $\eigmea(\{x\}) = m_{\black}(\{x\}) + m_{\white}(\{x\})$, we may assume without loss of generality that $m_{\black}(\{x\}) > 0$. 

	Since $x \in \limitset \cap \post{f}$ and $F^{\ell}(x) = x$, it follows immediately from \eqref{eq:point eigenmeasure relation in matrix form under iteration} that
	\begin{equation}    \label{eq:point eigenmeasure relation for periodic postcritical point}
		\begin{bmatrix}
			m_{\black}(\{x\}) \\ m_{\white}(\{x\}) 
		\end{bmatrix}
		=
		\frac{ \myexp{S_{\ell}\phi(x)} }{\eigenvalue^{\ell}} \Deg[\big]{\ell}{x}
		\begin{bmatrix}
			m_{\black}(\{x\}) \\ m_{\white}(\{x\}) 
		\end{bmatrix}.
	\end{equation}
	Similarly, by using Proposition~\ref{prop:dual split operator}~\ref{item:prop:dual split operator:split version} repeatedly, for each $k \in \n$ and each $y \in F^{-k\ell}(x)$, we have
	\begin{equation}    \label{eq:point eigenmeasure relation for preimages of periodic postcritical point}
		\begin{bmatrix}
			m_{\black}(\{y\}) \\ m_{\white}(\{y\}) 
		\end{bmatrix}
		=
		\frac{ \myexp{ S_{k \ell}\phi(y) } }{ \eigenvalue^{k \ell} } \Deg[\big]{k\ell}{y}
		\begin{bmatrix}
			m_{\black}(\{x\}) \\ m_{\white}(\{x\}) 
		\end{bmatrix}.
	\end{equation}

	\begin{figure}[H]
		\centering
		\begin{tikzpicture}[x = 20 pt,y = 20 pt]
			\fill[line width = 0.8 pt,color=black!25!white] (-6.672,3.864) -- (-2.552,1.984) -- (-4.372,7.084) -- cycle;
			\fill[line width = 0.8 pt,color=black!25!white] (-2.092,4.184)-- (-2.652,5.524) -- (-0.792,5.524) -- cycle;
			\draw [line width = 0.8 pt] (-2.552,1.984)-- (2.628,7.144);
			\draw [line width = 0.8 pt] (2.628,7.144)-- (-4.372,7.084) node[xshift = 100 pt, yshift = -15 pt] {$X^{k\ell}$};;
			\draw [line width = 0.8 pt] (-4.372,7.084)-- (-2.552,1.984);
			\draw [line width = 0.8 pt] (-2.552,1.984)-- (-6.672,3.864);
			\draw [line width = 0.8 pt] (-6.672,3.864)-- (-4.372,7.084);
			\draw [line width = 0.8 pt] (-6.672,3.864)-- (-6.732,1.124);
			\draw [line width = 0.8 pt] (-6.732,1.124)-- (-2.552,1.984);
			\draw [line width = 0.8 pt] (-2.552,1.984)-- (-3.652,-1.156);
			\draw [line width = 0.8 pt] (-3.652,-1.156)-- (-0.872,-1.056);
			\draw [line width = 0.8 pt] (-0.872,-1.056)-- (-2.552,1.984);
			\draw [line width = 0.8 pt] (-2.092,4.184)-- (-2.652,5.524);
			\draw [line width = 0.8 pt] (-2.652,5.524)-- (-0.792,5.524) node[xshift = -20 pt,yshift = 10 pt] {$X^{k\ell + n}$};
			\draw [line width = 0.8 pt] (-0.792,5.524)-- (-2.092,4.184);
			\draw [fill=black] (-2.552,1.984) circle (2 pt);
			\draw (-2.552,1.984) node[right] {$x$};
			\draw [fill=black] (-2.092,4.184) circle (2 pt);
			\draw (-2.092,4.184) node[below] {$y$};
			\draw (1, 2.1) node {$\bigcup \neighbortile{k\ell}{}{}{x}$};
		\end{tikzpicture}
		\caption{$\bigcup \neighbortile{k\ell}{}{}{x}$, with $\card{\post{f}} = 3$.}
		\label{fig:illustration of T_k}
	\end{figure}

	Recall from Definition~\ref{def:subsystem local degree} and Remark~\ref{rem:number of tile in neighborhood for F} that for each $k \in \n$, $\neighbortile{k \ell}{}{}{x}$ is the collection of $(k\ell)$-tiles of $F$ that intersect $\{x\}$, and contains exactly $\sumnorm[\big]{\Deg[\big]{k \ell}{x}}$ distinct $(k\ell)$-tiles of $F$. 
	Since $F$ is strongly irreducible, it follows from Lemma~\ref{lem:strongly irreducible:tile in interior tile} that for each $k \in \n$ and each $X^{k\ell} \in \neighbortile{k\ell}{}{}{x}$, there exists an integer $n \in \n$ with $n \leqslant n_{F}$ and a black $(k\ell + n)$-tile $X^{k\ell + n}_{\black} \in \bFTile{k\ell + n}$ satisfying $X^{k\ell + n}_{\black} \subseteq \inte[\big]{X^{k\ell}}$. 
	Here $n_{F} \in \n$ is the constant in Definition~\ref{def:irreducibility of subsystem}, which depends only on $F$ and $\mathcal{C}$.
	Then by Proposition~\ref{prop:subsystem:properties}~\ref{item:subsystem:properties:homeo}, there exists a unique $y \in X^{k\ell + n}_{\black} \subseteq \inte[\big]{X^{k\ell}}$ such that \[
		F^{k\ell + n}(y) = x
	\]
	(see Figure~\ref{fig:illustration of T_k}).
	For each $k \in \n$, we denote by $T_k$ the set consisting of one such $y$ from each $X^{k\ell} \in \neighbortile{k\ell}{}{}{x}$, and we have
	\begin{equation}    \label{eq:cardinality of T_k}
		\card{T_{k}} = \sumnorm[\big]{\Deg[\big]{k\ell}{x}}.
	\end{equation}
	Then $\{T_k\}_{k \in \n}$ is a sequence of subsets of $\bigcup_{j = 0}^{+\infty}f^{-j}(\post{f})$. 
	Since $f$ is expanding, we can choose an increasing sequence $\{k_i\}_{i \in \n}$ of integers recursively in such a way that $\bigcup \neighbortile{k_{i + 1}\ell}{}{}{x} \cap \bigcup^{i}_{j = 1} T_{k_{j}} = \emptyset$ for each $i \in \n$. 
	Then $\{T_{k_{i}}\}_{i \in \n}$ is a sequence of mutually disjoint sets. 
	Thus, by \eqref{eq:point eigenmeasure relation for preimages of periodic postcritical point} and Lemma~\ref{lem:distortion_lemma},
	\begin{equation}    \label{eq:temp:thm:subsystem:eigenmeasure existence and basic properties:vertices have measure zero}
		\begin{split}
			&\eigmea \biggl( \bigcup_{j = 0}^{+\infty}f^{-j}(\post{f}) \biggr) \\
			&\qquad \geqslant \sum_{i=1}^{+\infty} \sum_{y \in T_{k_{i}}} \eigmea(\{y\}) \\
			&\qquad= \sum_{i=1}^{+\infty} \sum_{y \in T_{k_{i}}}  
				\begin{bmatrix} 1& 1 \\ \end{bmatrix}  
				\begin{bmatrix} m_{\black}(\{y\}) \\ m_{\white}(\{y\}) \end{bmatrix} \\
			&\qquad= \sum_{i=1}^{+\infty} \sum_{y \in T_{k_{i}}}  
				\frac{ \myexp{ S_{k_i \ell + n} \phi(y) } }{ \eigenvalue^{k_i \ell + n} } 
				\begin{bmatrix}
				 1& 1 \\ 
				\end{bmatrix}
				\Deg[\big]{k_i \ell + n}{y}
				\begin{bmatrix}
					m_{\black}(\{x\}) \\ m_{\white}(\{x\}) 
				\end{bmatrix} \\
			&\qquad= \sum_{i=1}^{+\infty} \sum_{y \in T_{k_{i}}}  
				\frac{ e^{ S_{k_i \ell + n}\phi(y) } }{ e^{S_{k_i \ell}\phi(x)} } 
				\frac{ e^{ S_{k_i \ell }\phi(x) } }{ \eigenvalue^{k_i \ell + n} }
				\begin{bmatrix}
				 1& 1 \\ 
				\end{bmatrix}
				\Deg[\big]{k_i \ell + n}{y}
				\begin{bmatrix}
					m_{\black}(\{x\}) \\ m_{\white}(\{x\}) 
				\end{bmatrix} \\
			&\qquad= \sum_{i=1}^{+\infty} \sum_{y \in T_{k_{i}}}  
				\frac{ e^{ S_{n}\phi ( f^{k_{i}\ell}(y) )} }{ e^{ S_{k_i \ell}\phi(x) - S_{k_i \ell}\phi(y)} }
				\biggl( \frac{ e^{S_{\ell}\phi(x)} }{ \eigenvalue^{\ell} } \biggr)^{k_i} \eigenvalue^{-n}
				\begin{bmatrix}
				 1& 1 \\ 
				\end{bmatrix}
				\Deg[\big]{k_i \ell + n}{y}
				\begin{bmatrix}
					m_{\black}(\{x\}) \\ m_{\white}(\{x\}) 
				\end{bmatrix}\\
			&\qquad\geqslant \sum_{i=1}^{+\infty} \sum_{y \in T_{k_{i}}}  
				\frac{ e^{ - n_F \uniformnorm{\potential}} }{ {\scriptstyle \myexp{\Cdistortion}} }
				\biggl( \frac{ e^{ S_{\ell}\phi(x)} }{ \eigenvalue^{\ell} } \biggr)^{k_i}
				\min\{ 1, \eigenvalue^{-n_{F}} \}
				\begin{bmatrix}
				 1& 1 \\ 
				\end{bmatrix}
				\Deg[\big]{k_i \ell + n}{y}
				\begin{bmatrix}
					m_{\black}(\{x\}) \\ m_{\white}(\{x\}) 
				\end{bmatrix},
		\end{split}
	\end{equation}
	where $C_{1} \geqslant 0$ is the constant defined in \eqref{eq:const:C_1} in Lemma~\ref{lem:distortion_lemma} and depends only on $f$, $\mathcal{C}$, $d$, $\phi$, and $\holderexp$.
	
	To reach a contradiction, it suffices to show that 
	$\eigmea \bigl( \bigcup_{j = 0}^{+\infty}f^{-j}(\post{f}) \bigr) = +\infty$ 
	since $\eigmea$ is a Borel probability measure. 
	For each $i \in \n$ and each $y \in T_{k_{i}}$, we have
	\begin{equation}    \label{eq:temp:thm:subsystem:eigenmeasure existence and basic properties:estimate local degree matrix}
		\begin{split}
			\begin{bmatrix}
				 1& 1 \\ 
				\end{bmatrix}
				\Deg[\big]{k_i \ell + n}{y}
				\begin{bmatrix}
					m_{\black}(\{x\}) \\ m_{\white}(\{x\}) 
				\end{bmatrix}
			&\geqslant \begin{bmatrix}
				 1& 1 \\ 
				\end{bmatrix}
				\Deg[\big]{k_i \ell + n}{y}
				\begin{bmatrix}
					1 \\ 0
				\end{bmatrix} m_{\black}(\{x\})   \\
			&= \ccndegF{\black}{}{k_i \ell + n}{y}  m_{\black}(\{x\}) 
			\geqslant  m_{\black}(\{x\}),
		\end{split}
	\end{equation}
	where the equality follows from \eqref{eq:color degree is the sum of colour-position degree} in Remark~\ref{rem:number of tile in neighborhood for F}, and the last inequality holds since $y \in X^{k_i \ell + n}_{\black}$ for some $X^{k_i \ell + n}_{\black} \in \bFTile{k_i \ell + n}$. 
	Thus, it follows from \eqref{eq:temp:thm:subsystem:eigenmeasure existence and basic properties:vertices have measure zero}, \eqref{eq:temp:thm:subsystem:eigenmeasure existence and basic properties:estimate local degree matrix}, and \eqref{eq:cardinality of T_k} that
	\begin{equation}    \label{eq:temp:inequality for measure of vertices}
		\begin{split}
			\eigmea \biggl( \bigcup_{j = 0}^{+\infty}f^{-j}(\post{f}) \biggr)
			&\geqslant C \sum_{i=1}^{+\infty} \sum_{y \in T_{k_{i}}}  \left( \frac{ \myexp{S_{\ell}\phi(x)} }{ \eigenvalue^{\ell} } \right)^{k_i}  \\
			&= C \sum_{i=1}^{+\infty} \card{T_{k_i}}  \left( \frac{ \myexp{S_{\ell}\phi(x)} }{ \eigenvalue^{\ell} } \right)^{k_i} \\
			&= C \sum_{i=1}^{+\infty} \sumnorm[\big]{\Deg[\big]{k_i \ell}{x}} \left( \frac{ \myexp{S_{\ell}\phi(x)} }{ \eigenvalue^{\ell} } \right)^{k_i},
		\end{split}
	\end{equation}
	where $C \define m_{\black}(\{x\}) \min\{ 1, \eigenvalue^{-n_{F}} \} \myexp[\big]{ - n_F \ell \uniformnorm{\potential} - \Cdistortion } > 0$. 
	It suffices to show that
	\begin{equation}    \label{eq:sumnorm greater than eigenvalue}
		\sumnorm[\big]{\Deg[\big]{k_i\ell}{x}}  \geqslant \bigl( \eigenvalue^{\ell} \big/ \myexp{ S_{\ell}\phi(x) } \bigr)^{k_i}
	\end{equation}
	for each $i \in \n$. 
	We denote by $M$ the matrix $\Deg[\big]{\ell}{x}$ and $\lambda$ the number $\eigenvalue^{\ell} \big/ \myexp{S_{\ell}\phi(x)}$. 
	For each $i \in \n$, by Lemma~\ref{lem:iteration of local degree matrix}, we have $\Deg[\big]{k_i \ell}{x} = M^{k_i}$. 
	Then \eqref{eq:sumnorm greater than eigenvalue} follows from \eqref{eq:point eigenmeasure relation for periodic postcritical point} and Lemma~\ref{lem:sumnorm of iteration of matrix}.
	Combining \eqref{eq:sumnorm greater than eigenvalue} and \eqref{eq:temp:inequality for measure of vertices}, we get\[
		\eigmea \biggl( \bigcup_{j = 0}^{+\infty}f^{-j}(\post{f}) \biggr) \geqslant C \sum_{i=1}^{+\infty} 1 = +\infty.
	\]
	This contradicts the fact that $\eigmea$ is a finite Borel measure. 
	The proof of property~\ref{item:thm:subsystem:eigenmeasure existence and basic properties:strongly irreducible vertex zero measure} is complete.
\end{proof}

\subsection{Eigenfunctions}%
\label{sub:Eigenfunctions}

In this subsection, we establish some useful estimates for split Ruelle operators and construct their eigenfunctions.

\smallskip

We follow the conventions discussed in Remarks~\ref{rem:disjoint union} and \ref{rem:probability measure in split setting} in this subsection.

\begin{proposition}    \label{prop:eigenvalue equals split operator norm constant}
	Let $f$, $\mathcal{C}$, $F$, $d$, $\phi$, $\holderexp$ satisfy the Assumptions in Section~\ref{sec:The Assumptions}. 
	We assume in addition that $f(\mathcal{C}) \subseteq \mathcal{C}$ and $F \in \subsystem$ is irreducible. 
	Let $\spleigmea \in \probmea{\splitsphere}$ be a Borel probability measure defined in Theorem~\ref{thm:subsystem:eigenmeasure existence and basic properties} which satisfies $\dualsplopt\spleigmea = \eigenvalue \spleigmea$ where $\eigenvalue =  \bigl\langle \dualsplopt\spleigmea, \indicator{\splitsphere} \bigr\rangle$. 
	Then for each $\widetilde{x} \in \splitsphere$, we have that $\frac{1}{n} \log \mathopen{}\bigl( \splopt^n\bigl(\indicator{\splitsphere}\bigr)(\widetilde{x}) \bigr)$ converges to $\log \eigenvalue$ as $n$ tends to $+\infty$.
\end{proposition}
\begin{proof}
	Note that by Lemmas~\ref{lem:distortion lemma for subsystem}~\ref{item:lem:distortion lemma for subsystem:uniform bound} and \ref{lem:iteration of split-partial ruelle operator}, for all $n \in \n_0$ and $\widetilde{x}, \, \widetilde{y} \in \splitsphere$, we have
	\begin{equation}    \label{eq:bound on whole split sphere for split operator}
		\Csplratio^{-1} \leqslant \frac{ \splopt^{n}\bigl(\indicator{\splitsphere}\bigr)(\widetilde{x}) }{ \splopt^{n}\bigl(\indicator{\splitsphere}\bigr)(\widetilde{y}) } \leqslant \Csplratio,
	\end{equation}
	where $\Csplratio \geqslant 1$ is the constant depending only on $F$, $\mathcal{C}$, $d$, $\phi$, and $\holderexp$ from Lemma~\ref{lem:distortion lemma for subsystem}~\ref{item:lem:distortion lemma for subsystem:uniform bound}. 
	Since $\bigl\langle \spleigmea, \splopt^{n}\bigl(\indicator{\splitsphere}\bigr) \bigr\rangle =  \bigl\langle \bigl( \dualsplopt \bigr)^{n}\spleigmea,  \indicator{\splitsphere} \bigr\rangle =  \bigl\langle  \eigenvalue^{n}\spleigmea,  \indicator{\splitsphere} \bigr\rangle = \eigenvalue^{n}$, it follows from \eqref{eq:iteration of split-partial ruelle operator} and \eqref{eq:bound on whole split sphere for split operator} that
	\begin{align*}
	\log \eigenvalue &= \lim_{n \to +\infty} \frac{1}{n} \log \int \! \splopt^{n}\bigl(\indicator{\splitsphere}\bigr)(\widetilde{y}) \,\mathrm{d} \spleigmea(\widetilde{y}) \\
		&= \lim_{n \to +\infty} \frac{1}{n} \log \int \! \splopt^{n}\bigl(\indicator{\splitsphere}\bigr)(\widetilde{x}) \,\mathrm{d} \spleigmea(\widetilde{y}) \\
		&= \lim_{n \to +\infty} \frac{1}{n} \log \mathopen{}\bigl( \splopt^{n}\bigl(\indicator{\splitsphere}\bigr)(\widetilde{x}) \bigr)
	\end{align*}
	for each arbitrarily chosen $\widetilde{x} \in \splitsphere$.
\end{proof}

\begin{corollary}    \label{coro:well-define for split operator norm constant}
	Let $f$, $\mathcal{C}$, $F$, $d$, $\phi$, $\holderexp$ satisfy the Assumptions in Section~\ref{sec:The Assumptions}. 
	We assume in addition that $f(\mathcal{C}) \subseteq \mathcal{C}$ and $F \in \subsystem$ is irreducible. 
	Then the following limit exists for each $\widetilde{x} \in \splitsphere$ and is independent of $\widetilde{x}$, which we denote by $\Cnormspl$:
	\[
		\Cnormspl \define \lim_{n \to +\infty} \frac{1}{n} \log \mathopen{}\bigl( \splopt^{n}\bigl(\indicator{\splitsphere}\bigr)(\widetilde{x}) \bigr)
	\]
\end{corollary}
\begin{proof}
	By Theorem~\ref{thm:subsystem:eigenmeasure existence and basic properties}, there exists a measure $\spleigmea \in \probmea{\splitsphere}$ such as the one in Proposition~\ref{prop:eigenvalue equals split operator norm constant}. The limit then clearly depends only on $F$, $\mathcal{C}$, $d$, $\phi$, and $\holderexp$, and in particular, does not depend on the choice of $\spleigmea$.
\end{proof}

We characterize $\Cnormspl$ in terms of iterated preimages.

\begin{proposition}    \label{prop:subsystem preimage pressure}		\def\ncolour{\colour_{n}}
	Let $f$, $\mathcal{C}$, $F$, $d$, $\phi$, $\holderexp$ satisfy the Assumptions in Section~\ref{sec:The Assumptions}. 
	We assume in addition that $f(\mathcal{C}) \subseteq \mathcal{C}$ and $F \in \subsystem$ is irreducible. 
	Then for each sequence $\sequen{x_{n}}$ of points in $S^2$ and each sequence $\sequen{\ncolour}$ of colors in $\colours$ that satisfies $x_{n} \in X^0_{\ncolour}$ for each $n \in \n$, we have
	\begin{equation}    \label{eq:prop:subsystem preimage pressure}
		\Cnormspl = \lim_{n \to +\infty} \frac{1}{n} \log \sum_{ y \in F^{-n}(x_{n}) } \ccndegF{\ncolour}{}{n}{y} \myexp[\big]{S_{n}^{F}\potential(y)}.
	\end{equation}
	Furthermore, if $F$ is strongly irreducible, then for each $x_0 \in \limitset \setminus \mathcal{C}$, we have
	\[
		\Cnormspl = \lim_{n \to +\infty} \frac{1}{n} \log \sum_{y \in (F|_{\limitset})^{-n}(x_{0})} \myexp[\big]{ S_n^{F}\phi(y) }.
	\]
\end{proposition}
Note that if $F \in \subsystem$ is strongly irreducible, then it follows from Propositions~\ref{prop:sursubsystem properties}~\ref{item:prop:sursubsystem properties:property of domain and limitset} and \ref{prop:irreducible subsystem properties}~\ref{item:prop:irreducible subsystem properties:limitset non degenerate to Jordan curve} that $F(\limitset) = \limitset$ and $\limitset \setminus \mathcal{C} \ne \emptyset$.
\begin{proof}	\def\ncolour{\colour_{n}}
	We first prove \eqref{eq:prop:subsystem preimage pressure}.
	Consider arbitrary $\sequen{x_{n}}$ and $\sequen{\ncolour}$ that satisfy the assumptions.
	Denote $\widetilde{x}_{n} \define (x_{n}, \ncolour) \in \splitsphere$ for each $n \in \n$.
	Fix arbitrary point $\widetilde{x} \in S^2$.
	By Lemmas~\ref{lem:distortion lemma for subsystem}~\ref{item:lem:distortion lemma for subsystem:uniform bound} and \ref{lem:iteration of split-partial ruelle operator}, for each $n \in \n$, we have
	\[
		\Csplratio^{-1} \leqslant \frac{ \splopt^{n}\bigl(\indicator{\splitsphere}\bigr)(\widetilde{x}_{n}) }{ \splopt^{n}\bigl(\indicator{\splitsphere}\bigr)(\widetilde{x}) } \leqslant \Csplratio,
	\]
	where $\Csplratio \geqslant 1$ is the constant depending only on $F$, $\mathcal{C}$, $d$, $\phi$, and $\holderexp$ from Lemma~\ref{lem:distortion lemma for subsystem}~\ref{item:lem:distortion lemma for subsystem:uniform bound}.
	Then it follows immediately from Corollary~\ref{coro:well-define for split operator norm constant}, \eqref{eq:definition of partial Ruelle operators} in Definition~\ref{def:partial Ruelle operator}, and \eqref{eq:color degree is the sum of colour-position degree} in Remark~\ref{rem:number of tile in neighborhood for F} that \eqref{eq:prop:subsystem preimage pressure} holds.

	Now we assume that $F$ is strongly irreducible.
	Fix arbitrary $x_{0} \in \limitset \setminus \mathcal{C}$. 
	Without loss of generality we may assume that $x_0 \in \inte[\big]{X^0_{\black}}$. 
	By Corollary~\ref{coro:well-define for split operator norm constant} and \eqref{eq:iteration of split-partial ruelle operator}, it suffices to show that \[
		(F|_{\limitset})^{-n}(x_{0}) = \bigl\{ (F^{n}|_{X^{n}})^{-1}(x_{0}) \describe X^{n} \in \bFTile{n} \bigr\}
	\]
	for each $n \in \n$. Let $n \in \n$ be arbitrary in the following two paragraphs.

	For each $y \in (F|_{\limitset})^{-n}(x_{0})$, we have $y \in \limitset \subseteq \domain{n}$ and $(F|_{\limitset})^{n}(y) = F^{n}(y) = x_{0} \in \inte[\big]{X^0_{\black}}$. 
	Thus by Lemma~\ref{lem:cell mapping properties of Thurston map}~\ref{item:lem:cell mapping properties of Thurston map:ii} and Proposition~\ref{prop:subsystem:properties}~\ref{item:subsystem:properties:homeo}, there exists a unique $n$-tile $X^{n} \in \bFTile{n}$ with $y \in \inte{X^n}$ and $F^{n}(X^n) = X^0_{\black}$. 
	Thus we deduce that $y = (F^{n}|_{X^{n}})^{-1}(x_{0})$ for some $X^n \in \bFTile{n}$.

	To deduce the reverse inclusion, it suffices to show that $(F^{n}|_{X^{n}})^{-1}(x_{0}) \in \limitset$ for each $X^n \in \bFTile{n}$.
	Let $X^{n} \in \bFTile{n}$ be arbitrary and denote $y \define (F^{n}|_{X^{n}})^{-1}(x_{0})$. 
	Then it follows from Proposition~\ref{prop:subsystem:properties invariant Jordan curve}~\ref{item:subsystem:properties invariant Jordan curve:backward invariant limitset outside invariant Jordan curve} and induction that $y \in \limitset$ since $x_{0} \in \limitset \cap \inte[\big]{X^0_{\black}}$ and $y \in F^{-n}(x_0)$.
	
	The proof is complete.
\end{proof}

Let $f$, $\mathcal{C}$, $F$, $d$, $\phi$, $\holderexp$ satisfy the Assumptions in Section~\ref{sec:The Assumptions}. 
We assume in addition that $F \in \subsystem$ is irreducible. 
We define the function
\begin{equation}    \label{eq:def:normed potential}
	\normpotential \define \phi - \Cnormspl \in C^{0,\holderexp}(S^2,d).
\end{equation}
Then 
\begin{equation}    \label{eq:def:normed split operator}
	\normsplopt = e^{-\Cnormspl} \splopt.
\end{equation}
If $\spleigmea$ is an eigenmeasure as in Theorem~\ref{thm:subsystem:eigenmeasure existence and basic properties}, then by Proposition~\ref{prop:eigenvalue equals split operator norm constant} we have
\begin{equation}    \label{eq:eigenmeasure for dual split operator}
	\dualsplopt\spleigmea = e^{\Cnormspl} \spleigmea  \qquad \text{and} \qquad 
	\normdualsplopt\spleigmea = \spleigmea.
\end{equation}
\begin{rmk}
	We will show in Theorem~\ref{thm:existence of f invariant Gibbs measure} that $\Cnormspl = \pressure$.
\end{rmk}

We summarize in the following lemma the properties of $\normsplopt$ that we will need. 
\begin{lemma}    \label{lem:distortion lemma for normed split operator}
	Let $f$, $\mathcal{C}$, $F$, $d$, $\Lambda$, $\phi$, $\holderexp$ satisfy the Assumptions in Section~\ref{sec:The Assumptions}. 
	We assume in addition that $f(\mathcal{C}) \subseteq \mathcal{C}$ and $F \in \subsystem$ is irreducible. 
	Then there exists a constant $\Cnormspllocal \geqslant 0$ depending only on $F$, $\mathcal{C}$, $d$, $\phi$, and $\holderexp$ such that for each $n \in \n$, each $\colour \in \colours$, and each pair of points $\juxtapose{x}{y} \in X^0_{\colour}$, the following inequalities holds:
	\begin{align}    
		\label{eq:first bound for normed split operator}
		\normsplopt^{n}\bigl(\indicator{\splitsphere}\bigr)(\widetilde{x}) \big/ \normsplopt^{n}\bigl(\indicator{\splitsphere}\bigr)(\widetilde{y})
		&\leqslant \myexp[\big]{ C_1 d(x, y)^{\holderexp} }  \leqslant \Csplratio,\\
		\label{eq:second bound for normed split operator}
		\Csplratio^{-1} & \leqslant \normsplopt^{n}\bigl(\indicator{\splitsphere}\bigr)(\widetilde{x}) \leqslant \Csplratio,   \\
		\label{eq:third bound for normed split operator}
		\abs[\big]{\normsplopt^{n}\bigl(\indicator{\splitsphere}\bigr)(\widetilde{x}) - \normsplopt^{n}\bigl(\indicator{\splitsphere}\bigr)(\widetilde{y})}
		&\leqslant \Csplratio \bigl( \myexp[\big]{ C_1 d(x, y)^{\holderexp} } - 1 \bigr) \leqslant \Cnormspllocal d(x, y)^{\holderexp},
	\end{align}
	where $\widetilde{x} \define i_{\colour}(x) = (x, \colour) \in \splitsphere$, $\widetilde{y} \define i_{\colour}(y) = (y, \colour) \in \splitsphere$ (recall Remark~\ref{rem:disjoint union}), $C_1 \geqslant 0$ is the constant in Lemma~\ref{lem:distortion_lemma} depending only on $f$, $\mathcal{C}$, $d$, $\phi$, and $\holderexp$, and $\Csplratio \geqslant 1$ is the constant in Lemma~\ref{lem:distortion lemma for subsystem}~\ref{item:lem:distortion lemma for subsystem:uniform bound} depending only on $F$, $\mathcal{C}$, $d$, $\phi$, and $\holderexp$.
\end{lemma}
\begin{proof}
	Inequality \eqref{eq:first bound for normed split operator} follows immediately from Lemmas~\ref{lem:distortion lemma for subsystem}, \ref{lem:iteration of split-partial ruelle operator}, and \eqref{eq:def:normed split operator}.

	Fix arbitrary $n \in \n$, $\colour \in \colours$, and $\juxtapose{x}{y} \in X^0_{\colour}$.

	By Lemma~\ref{lem:distortion lemma for subsystem}~\ref{item:lem:distortion lemma for subsystem:uniform bound}, \eqref{eq:def:normed split operator}, and \eqref{eq:iteration of split-partial ruelle operator}, we have\[
		0 < \inf_{\widetilde{z} \in \splitsphere} \normsplopt^{n}\bigl(\indicator{\splitsphere}\bigr)(\widetilde{z})  
		\leqslant   \normsplopt^{n}\bigl(\indicator{\splitsphere}\bigr)(\widetilde{x}) 
		\leqslant \sup_{\widetilde{z} \in \splitsphere} \normsplopt^{n}\bigl(\indicator{\splitsphere}\bigr)(\widetilde{z}) < +\infty
	\]
	and $\sup_{\widetilde{z} \in \splitsphere} \normsplopt^{n}\bigl(\indicator{\splitsphere}\bigr)(\widetilde{z}) \leqslant \Csplratio \inf_{\widetilde{z} \in \splitsphere} \normsplopt^{n}\bigl(\indicator{\splitsphere}\bigr)(\widetilde{z})$, where $\Csplratio$ is the constant defined in \eqref{eq:const:Csplratio} in Lemma~\ref{lem:distortion lemma for subsystem}~\ref{item:lem:distortion lemma for subsystem:uniform bound} and depends only on $F$, $\mathcal{C}$, $d$, $\phi$, $\holderexp$. 
	By Theorem~\ref{thm:subsystem:eigenmeasure existence and basic properties}, Proposition~\ref{prop:eigenvalue equals split operator norm constant}, and Corollary~\ref{coro:well-define for split operator norm constant}, we can choose a Borel probability measure $\splmea \in \probmea{\splitsphere}$ such that $\normdualsplopt\splmea = \splmea$. 
	Then we have\[
		\functional[\big]{\splmea}{\normsplopt^{n}\bigl(\indicator{\splitsphere}\bigr)} 
		= \functional[\big]{\bigl(\normdualsplopt\bigr)^{n}\splmea}{\indicator{\splitsphere}} 
		= \functional[\big]{\splmea}{\indicator{\splitsphere}} = 1.
	\]
	Thus $1 \leqslant \sup_{\widetilde{z} \in \splitsphere} \normsplopt^{n}\bigl(\indicator{\splitsphere}\bigr)(\widetilde{z}) \leqslant \Csplratio \inf_{\widetilde{z} \in \splitsphere} \normsplopt^{n}\bigl(\indicator{\splitsphere}\bigr)(\widetilde{z}) \leqslant \Csplratio$ and \eqref{eq:second bound for normed split operator} holds.

	Applying \eqref{eq:first bound for normed split operator} and \eqref{eq:second bound for normed split operator}, we get
	\begin{align*}
		\abs[\big]{\normsplopt^{n}\bigl(\indicator{\splitsphere}\bigr)(\widetilde{x}) - \normsplopt^{n}\bigl(\indicator{\splitsphere}\bigr)(\widetilde{y})}
		&= \normsplopt^{n}\bigl(\indicator{\splitsphere}\bigr)(\widetilde{y}) \abs[\bigg]{ \frac{\normsplopt^{n}\bigl(\indicator{\splitsphere}\bigr)(\widetilde{x})}{\normsplopt^{n}\bigl(\indicator{\splitsphere}\bigr)(\widetilde{y})} - 1 } \\
		&\leqslant \Csplratio \bigl( e^{ C_1 d(x, y)^{\holderexp} } - 1 \bigr) 
		\leqslant \Cnormspllocal d(x, y)^{\holderexp}
	\end{align*}
	for some constant $\Cnormspllocal$ depending only on $C_{1}$, $\Csplratio$, and $\diam{d}{S^{2}}$. 
	Therefore, we establish \eqref{eq:third bound for normed split operator} as the constant $\Cnormspllocal > 0$ depends only on $F$, $\mathcal{C}$, $d$, $\phi$, and $\holderexp$. 
\end{proof}

By Lemma~\ref{lem:distortion lemma for subsystem}, we can construct eigenfunctions of $\normsplopt$.
\begin{proposition}    \label{prop:existence of eigenfunction} 
	Let $f$, $\mathcal{C}$, $F$, $d$, $\Lambda$, $\phi$, $\holderexp$ satisfy the Assumptions in Section~\ref{sec:The Assumptions}. 
	We assume in addition that $f(\mathcal{C}) \subseteq \mathcal{C}$ and $F \in \subsystem$ is irreducible. 
	Then the sequence $\bigl\{ \frac{1}{n} \sum_{j = 0}^{n - 1} \normsplopt^j\bigl(\indicator{\splitsphere}\bigr) \bigr\}_{n \in \n}$ has a subsequential limit (with respect to the uniform norm). Moreover, if $\eigfun \in C(\splitsphere)$ is such a subsequential limit, then
	\begin{align}
		\label{eq:existence:eigenfunction of normed split operator}
		\normsplopt(\eigfun) = \eigfun 	\qquad &\text{and}  \\
		\label{eq:existence:two sides bounds for eigenfunction}
		\Csplratio^{-1} \leqslant \eigfun(\widetilde{x}) \leqslant \Csplratio \qquad &\text{for each } \widetilde{x} \in \splitsphere,  
	\end{align}
	where $\Csplratio \geqslant 1$ is the constant from Lemma~\ref{lem:distortion lemma for subsystem}~\ref{item:lem:distortion lemma for subsystem:uniform bound} depending only on $f$, $\mathcal{C}$, $d$, $\phi$, and $\holderexp$. 
	Furthermore, if we let $\spleigmea \in \probmea{\splitsphere}$ be an eigenmeasure as in Theorem~\ref{thm:subsystem:eigenmeasure existence and basic properties}, then
	\begin{equation}    \label{eq:existence:integration of eigenfunction with respect to eigenmeasure is one}
		\int_{\splitsphere} \! \eigfun \,\mathrm{d}\spleigmea = 1
	\end{equation}
	and $\splmea \define \eigfun \spleigmea \in \probmea{\splitsphere}$ is well-defined as a Borel probability measure on $\splitsphere$.
\end{proposition}
We will show in Theorem~\ref{thm:existence of f invariant Gibbs measure} that a subsequential limit $\eigfun$ as defined above is unique and the sequence $\bigl\{ \frac{1}{n} \sum_{j = 0}^{n - 1} \normsplopt^j\bigl(\indicator{\splitsphere}\bigr) \bigr\}_{n \in \n}$ converges uniformly to $\eigfun \in C(\splitsphere)$.
\begin{proof}
	Define $\widetilde{u}_{n} \define \frac{1}{n} \sum_{j = 0}^{n - 1} \normsplopt^{j}\bigl(\indicator{\splitsphere}\bigr)$ for each $n \in \n$. 
	Then $\{\widetilde{u}_n\}_{n \in \n}$ is a uniformly bounded sequence of equicontinuous functions on $\splitsphere$ by \eqref{eq:second bound for normed split operator} and \eqref{eq:third bound for normed split operator} in Lemma~\ref{lem:distortion lemma for normed split operator}. 
	By the \aalem~Theorem, there exists a continuous function $\eigfun \in C(\splitsphere)$ and an increasing sequence $\{n_i\}_{i\in \n}$ in $\n$ such that $\widetilde{u}_{n_i} \to \eigfun$ uniformly on $\splitsphere$ as $i \to +\infty$.

	To prove \eqref{eq:existence:eigenfunction of normed split operator}, we note that by the definition of $\widetilde{u}_n$ and \eqref{eq:second bound for normed split operator} in Lemma~\ref{lem:distortion lemma for normed split operator}, we have that for each $i \in \n$,\[
		\uniformnorm[\big]{ \normsplopt(\widetilde{u}_{n_{i}}) - \widetilde{u}_{n_{i}} } = \frac{1}{n_{i}} \uniformnorm[\big]{\normsplopt^{n_{i}}\bigl(\indicator{\splitsphere}\bigr) - \indicator{\splitsphere}} 
		\leqslant \frac{1}{n_{i}} (\Csplratio + 1),
	\]
	where $\Csplratio \geqslant 1$ is the constant from Lemma~\ref{lem:distortion lemma for subsystem}~\ref{item:lem:distortion lemma for subsystem:uniform bound} depending only on $f$, $\mathcal{C}$, $d$, $\phi$, and $\holderexp$. By letting $i \to +\infty$, we can conclude that $\uniformnorm[\big]{ \normsplopt(\eigfun) - \eigfun } = 0$. Thus \eqref{eq:existence:eigenfunction of normed split operator} holds.

	By \eqref{eq:second bound for normed split operator}, we have that $\Csplratio^{-1} \leqslant \widetilde{u}_{n}(\widetilde{x}) \leqslant \Csplratio$ for each $n \in \n$ and each $\widetilde{x} \in \splitsphere$. Thus \eqref{eq:existence:two sides bounds for eigenfunction} follows.

	By \eqref{eq:eigenmeasure for dual split operator} and definition of $\widetilde{u}_{n}$, we have $\int \! \widetilde{u}_{n} \,\mathrm{d}\spleigmea = \int \! \indicator{\splitsphere} \,\mathrm{d}\spleigmea = 1$ for each $n \in \n$. Then by the Lebesgue Dominated Theorem, we can conclude that\[
		\int \! \eigfun \,\mathrm{d}\spleigmea = \lim_{i \to +\infty} \int \! \widetilde{u}_{n_{i}} \,\mathrm{d}\spleigmea = 1,
	\]
	establishing \eqref{eq:existence:integration of eigenfunction with respect to eigenmeasure is one}. 
	Therefore, $\eigfun \spleigmea$ is a Borel probability measure on $\splitsphere$.
\end{proof}

\subsection{Invariant Gibbs measures}%
\label{sub:Invariant Gibbs measures}

The goal of this subsection is to establish Theorem~\ref{thm:existence of f invariant Gibbs measure}, namely, the existence of invariant Gibbs measures for subsystems of expanding Thurston maps.
We also investigate the properties of the eigenmeasures (of the adjoint split Ruelle operators) from Theorem~\ref{thm:subsystem:eigenmeasure existence and basic properties}, with the main results being Propositions~\ref{prop:subsystem:f-invariant measure}, \ref{prop:subsystem edge Jordan curve has measure zero}, and \ref{prop:subsystem properties of eigenmeasure}.

\smallskip

In this subsection, we follow the conventions discussed in Remarks~\ref{rem:disjoint union} and \ref{rem:probability measure in split setting}. 
In particular, we use the notation $\eigmea$ (\resp $\spleigmea$) for emphasis when we view the eigenmeasure as a Borel probability measure on $S^2$ (\resp $\splitsphere$), and we follow the same conventions for $\equstate$ (\resp $\splmea$), where $\equstate = \splmea \define \eigfun\spleigmea$ comes from Proposition~\ref{prop:existence of eigenfunction}.

\begin{definition}[Gibbs measures for subsystems]    \label{def:subsystem gibbs measure}
    Let $f$, $\mathcal{C}$, $F$, $d$, $\phi$ satisfy the Assumptions in Section~\ref{sec:The Assumptions}.
    A Borel probability measure $\mu \in \mathcal{P}(S^2)$ is called a \emph{Gibbs measure} with respect to $F$, $\mathcal{C}$, and $\phi$ if there exist constants $P_{\mu} \in \real$ and $C_{\mu} \geqslant 1$ such that for each $n\in \n_0$, each $n$-tile $X^n \in \Domain{n}$, and each $x \in X^n$, we have
    \begin{equation}    \label{eq:def:subsystem gibbs measure}
        \frac{1}{C_{\mu}} \leqslant \frac{\mu(X^n)}{ \myexp{ S_{n}^{F}\phi(x) - nP_{\mu} } } \leqslant C_{\mu}.
    \end{equation}
\end{definition}

One observes that for each Gibbs measure $\mu$ with respect to $F$, $\mathcal{C}$, and $\phi$, the constant $P_{\mu}$ is unique.

\begin{theorem}    \label{thm:existence of f invariant Gibbs measure}
	Let $f$, $\mathcal{C}$, $F$, $d$, $\Lambda$, $\phi$, $\holderexp$ satisfy the Assumptions in Section~\ref{sec:The Assumptions}. 
	We assume in addition that $f(\mathcal{C}) \subseteq \mathcal{C}$ and $F \in \subsystem$ is strongly irreducible.
	Then the sequence $\bigl\{ \frac{1}{n} \sum_{j = 0}^{n - 1} \normsplopt^j\bigl(\indicator{\splitsphere}\bigr) \bigr\}_{n \in \n}$ converges uniformly to a function $\eigfun = \splfun \in \splholderspace$ that satisfies
	\begin{align}
		\label{eq:eigenfunction of normed split operator}
		\normsplopt(\eigfun) &= \eigfun \qquad \text{and} \\
		\label{eq:two sides bounds for eigenfunction}
		\Csplratio^{-1} \leqslant \eigfun(\widetilde{x}) &\leqslant \Csplratio \qquad \text{for } \widetilde{x} \in \splitsphere,
	\end{align}
	where $\Csplratio \geqslant 1$ is the constant from Lemma~\ref{lem:distortion lemma for subsystem}~\ref{item:lem:distortion lemma for subsystem:uniform bound} depending only on $f$, $\mathcal{C}$, $d$, $\phi$, and $\holderexp$. 
	Moreover, if we let $\eigmea = \spleigmea$ be an eigenmeasure as in Theorem~\ref{thm:subsystem:eigenmeasure existence and basic properties}, then
	\begin{equation}    \label{eq:integration of eigenfunction with respect to eigenmeasure is one}
		\int_{\splitsphere} \! \eigfun \,\mathrm{d}\spleigmea = 1,
	\end{equation}
	and $\equstate = \splmea \define \eigfun \spleigmea$ is an $f$-invariant Gibbs measure with respect to $F$, $\mathcal{C}$, and $\phi$, with $\equstate(\limitset(F, \mathcal{C})) = 1$ and
	\begin{equation}    \label{eq:equalities for characterizations of pressure}
		P_{\equstate} = P_{\eigmea} = \pressure = D_{F, \phi} = \lim_{n \to +\infty} \frac{1}{n}\log \bigl( \splopt^{n}\bigl(\indicator{\splitsphere}\bigr)(\widetilde{y}) \bigr),
	\end{equation}
	for each $\widetilde{y} \in \splitsphere$. In particular, $\equstate(U) \ne 0$ for each open set $U \subseteq S^{2}$ with $U \cap \limitset(F, \mathcal{C}) \ne \emptyset$.
\end{theorem}

See \eqref{eq:pressure of subsystem} for the definition of $\pressure$. 
The proof of Theorem~\ref{thm:existence of f invariant Gibbs measure} will be presented at the end of this subsection.

We will show in Theorem~\ref{thm:existence of equilibrium state for subsystem} that a measure $\equstate$ as defined above is, in fact, an equilibrium state for the map $\limitmap = F|_{\limitset}$ and the potential $\phi|_{\limitset}$. 
We will also show in Theorem~\ref{thm:subsystem characterization of pressure} that $\pressure = \fpressure$.

\begin{proposition}    \label{prop:subsystem:f-invariant measure}
	Let $f$, $\mathcal{C}$, $F$, $d$, $\phi$ satisfy the Assumptions in Section~\ref{sec:The Assumptions}. 
	We assume in addition that $f(\mathcal{C}) \subseteq \mathcal{C}$ and $F \in \subsystem$ is irreducible.
	Let $\eigmea = \spleigmea$ be an eigenmeasure as in Theorem~\ref{thm:subsystem:eigenmeasure existence and basic properties} and $\eigfun \in C(\splitsphere)$ be an eigenfunction of $\normsplopt$ from Proposition~\ref{prop:existence of eigenfunction}. Let $\equstate = \splmea \define \eigfun\spleigmea$. Then $\equstate$ is an $f$-invariant Borel probability measure on $S^{2}$.
\end{proposition}
\begin{proof}
	By Proposition~\ref{prop:existence of eigenfunction} and Remark~~\ref{rem:probability measure in split setting}, we get $\equstate \in \probsphere$. 
	It suffices to prove that $\functional{\equstate}{g \circ f} = \functional{\equstate}{g}$ for each $g \in C(S^2)$. 
	Indeed, it follows from \eqref{eq:definition of partial Ruelle operators} and \eqref{eq:def:split ruelle operator} that for each $\widetilde{h} \in C(\splitsphere)$,
	\[
		\normsplopt( \widetilde{h} (\widetilde{g \circ f})) = \widetilde{g} \, \normsplopt(\widetilde{h}),
	\]
	where $\widetilde{g}$ is a continuous function on $\splitsphere$ defined by $\widetilde{g}(\widetilde{x}) \define g(x)$ for each $\widetilde{x} = (x, \colour) \in \splitsphere$, and $\widetilde{g \circ f} \in C(\splitsphere)$ is defined similarly.
	Then by \eqref{eq:eigenmeasure for dual split operator} and \eqref{eq:existence:eigenfunction of normed split operator}, we deduce that for each $g \in C(S^2)$,
	\begin{align*}
		\functional{\equstate}{g \circ f} 
		&= \functional[\big]{\splmea}{\widetilde{g \circ f}} \\
		&= \functional[\big]{\spleigmea}{\eigfun (\widetilde{g \circ f}) }  \\
		&= \functional[\big]{\normdualsplopt\spleigmea}{\eigfun (\widetilde{g \circ f}) } \\
		&= \functional[\big]{\spleigmea}{ \normsplopt(\eigfun (\widetilde{g \circ f})) } \\
		&= \functional[\big]{\spleigmea}{\widetilde{g} \, \normsplopt(\eigfun)}  \\
		&= \functional{\spleigmea}{ \widetilde{g}\, \eigfun } \\
		&= \functional{\eigfun\spleigmea}{ \widetilde{g} } \\
		&= \functional{\equstate}{g}. \qedhere
	\end{align*}
\end{proof}

\begin{proposition}    \label{prop:subsystem edge Jordan curve has measure zero}
	Let $f$, $\mathcal{C}$, $F$, $d$, $\phi$ satisfy the Assumptions in Section~\ref{sec:The Assumptions}. 
	We assume in addition that $f(\mathcal{C}) \subseteq \mathcal{C}$ and $F \in \subsystem$ is strongly irreducible.
	Let $\eigmea = \spleigmea$ be an eigenmeasure as in Theorem~\ref{thm:subsystem:eigenmeasure existence and basic properties} and $\eigfun$ be an eigenfunction of $\normsplopt$ from Proposition~\ref{prop:existence of eigenfunction}. Denote $\equstate = \splmea \define \eigfun\spleigmea$. 
	Then $\eigmea \bigl( \bigcup_{j = 0}^{+\infty} f^{-j}(\mathcal{C}) \bigr) = \equstate \bigl( \bigcup_{j = 0}^{+\infty} f^{-j}(\mathcal{C}) \bigr) = 0$.
\end{proposition}
\begin{proof}
	By Proposition~\ref{prop:subsystem:f-invariant measure}, the measure $\equstate \in \probsphere$ is $f$-invariant. 
	Note that $\mathcal{C} \subseteq f^{-j}(\mathcal{C})$ for each $j \in \n$. 
	Thus we have $\equstate(f^{-j}(\mathcal{C}) \setminus \mathcal{C}) = 0$ for each $j \in \n$. 
	Since $F$ is strongly irreducible, by Definition~\ref{def:irreducibility of subsystem}, for each $\colour \in \colours$, there exists an integer $n_{\colour} \in \n$ and $X_{\colour} \in \cFTile{n_{\colour}}$ such that $X_{\colour} \cap \mathcal{C} = \emptyset$. 
	Then $\partial X_{\colour} \subseteq f^{-n_{\colour}}(\mathcal{C}) \setminus \mathcal{C}$ for each $\colour \in \colours$. 
	So $\equstate(\partial X_{\colour}) = 0$ for each $\colour \in \colours$. 
	Since $\equstate = \splmea = \eigfun\spleigmea$, where $\eigfun \in C(\splitsphere)$ is bounded away from $0$ by \eqref{eq:existence:two sides bounds for eigenfunction} in Proposition~\ref{prop:existence of eigenfunction}, we have $\eigmea(\partial X_{\colour}) = \spleigmea(\partial X_{\colour}) = 0$ for each $\colour \in \colours$.

	We next show that $\eigmea(\mathcal{C}) = \equstate(\mathcal{C}) = 0$. 
	For each $\colour \in \colours$, it follows from Proposition~\ref{prop:subsystem:properties}~\ref{item:subsystem:properties:homeo} that $F^{n_{\colour}}|_{\partial X_{\colour}}$ is a homeomorphism from $\partial X_{\colour}$ to $\mathcal{C}$. Then for each $\colour \in \colours$, by Proposition~\ref{prop:dual split operator}~\ref{item:prop:dual split operator:sphere version level n}, we have
	\begin{align*}
		0 = \eigmea(\partial X_{\colour}) 
		&= \sum_{\ccolour \in \colours} \int_{ \mathcal{C} }   \bigl( \ccndegF{\ccolour}{}{n_{\colour}}{\cdot \, } \myexp[\big]{ S_{n_{\colour}}^{F}\phi } \bigr) \circ (F^{n_{\colour}}|_{\partial X_{\colour}})^{-1} \,\mathrm{d} m_{\ccolour}  \\
		&\geqslant \int_{ \mathcal{C} }   \bigl( \ccndegF{\colour}{}{n_{\colour}}{\cdot \, } \myexp[\big]{ S_{n_{\colour}}^{F}\phi } \bigr) \circ (F^{n_{\colour}}|_{\partial X_{\colour}})^{-1} \,\mathrm{d} m_{\colour}  \\
		&\geqslant \int_{\mathcal{C}} \! \bigl( \myexp[\big]{ S_{n_{\colour}}^{F}\phi } \bigr) \circ (F^{n_{\colour}}|_{\partial X_{\colour}})^{-1} \,\mathrm{d} m_{\colour}   \\
		&\geqslant \myexp{ -n_{\colour}\uniformnorm{\phi} } m_{\colour}(\mathcal{C}) \\
		&\geqslant 0,
	\end{align*}
	where the second inequality holds since $\ccndegF{\colour}{}{n_{\colour}}{y} \geqslant 1$ for each $y \in \partial X_{\colour} \subseteq X_{\colour} \in \cFTile{n_{\colour}}$. 
	Thus $\eigmea(\mathcal{C}) = \spleigmea(\mathcal{C}) = m_{\black}(\mathcal{C}) + m_{\white}(\mathcal{C}) = 0$.
	This implies $\equstate(\mathcal{C}) = 0$.

	Finally, since $\equstate(f^{-j}(\mathcal{C}) \setminus \mathcal{C}) = 0$ for each $j \in \n$,
	we have that $\equstate \bigl( \bigcup_{j = 0}^{+\infty} f^{-j}(\mathcal{C}) \bigr) = 0$.
	This implies $\eigmea\bigl( \bigcup_{j = 0}^{+\infty} f^{-j}(\mathcal{C}) \bigr) = 0$ and finishes the proof.
\end{proof}

\begin{definition}[Jacobian]    \label{def:subsystem Jacobian}
	Let $f$, $\mathcal{C}$, $F$ satisfy the Assumptions in Section~\ref{sec:The Assumptions}. 
	Consider a Borel probability measure $\mu \in \probsphere$ on $S^2$. 
	A Borel function $J \colon \domF \mapping [0, +\infty)$ is a \emph{Jacobian (function)} for $F$ with respect to $\mu$ if for every Borel $A \subseteq \domF$ on which $F$ is injective, the following equation holds:
	\begin{equation}    \label{eq:def:subsystem Jacobian}
		\mu(F(A)) = \int_{A} \! J \,\mathrm{d}\mu.
	\end{equation}
\end{definition}

\begin{proposition}    \label{prop:subsystem properties of eigenmeasure}
	Let $f$, $\mathcal{C}$, $F$, $d$, $\Lambda$, $\phi$, $\holderexp$ satisfy the Assumptions in Section~\ref{sec:The Assumptions}. 
	We assume in addition that $f(\mathcal{C}) \subseteq \mathcal{C}$ and $F \in \subsystem$ is strongly irreducible. 
	Consider an arbitrary $\mu = \splmea \in \probmea{\splitsphere}$.
	If $\dualsplopt\splmea = \eigenvalue \splmea$ for some constant $\eigenvalue > 0$, then the following statements hold:
	\begin{enumerate}
		\smallskip

		\item     \label{item:prop:subsystem properties of eigenmeasure:Jacobian}
		The function $J \colon \domF \mapping [0, +\infty)$ given by $J \define \eigenvalue \myexp{-\phi}$ is a Jacobian for $F$ with respect to $\mu$.

		\smallskip 
		
		\item     \label{item:prop:subsystem properties of eigenmeasure:Gibbs property}
		The measure $\mu$ is a Gibbs measure with respect to $F$, $\mathcal{C}$, and $\phi$ with $P_{\mu} = \pressure = \log \eigenvalue$. 
		Here $\pressure$ is defined in \eqref{eq:pressure of subsystem} from Subsection~\ref{sub:Pressures for subsystems}. 
		In particular, $\sum_{X^n \in \Domain{n}} \mu(X^n) \leqslant 2$.

		\smallskip

		\item     \label{item:prop:subsystem properties of eigenmeasure:forward quasi-invariant and non-singular}
		The map $\limitmap \define F|_{\limitset(F, \mathcal{C})}$ with respect to $\mu$ is forward quasi-invariant (i.e., for each Borel set $A \subseteq \limitset(F, \mathcal{C})$, if $\mu(A) = 0$, then $\mu(\limitmap(A)) = 0$) and non-singular (i.e., for each Borel set $A \subseteq \limitset(F, \mathcal{C})$, if $\mu(A) = 0$, then $\mu(\limitmap^{-1}(A)) = 0$).
	\end{enumerate}
\end{proposition}
\begin{proof}
	Assume that $\dualsplopt\splmea = \eigenvalue \splmea$ for some constant $\eigenvalue > 0$.

	\smallskip

	\ref{item:prop:subsystem properties of eigenmeasure:Jacobian}  
	By Proposition~\ref{prop:dual split operator}~\ref{item:prop:dual split operator:sphere version level n}, for each Borel $A \subseteq \domF$ on which $F$ is injective, we have that $F(A)$ is Borel, and
	\[
		\mu(A) = \eigenvalue^{-1}\dualsplopt\splmea(A)
		= \eigenvalue^{-1} \sum_{\colour \in \colours}  \int_{F(A) \cap X^0_{\colour} } \! 
			(\ccndegF{\colour}{}{}{\cdot \, } \myexp{\phi} ) \circ (F|_{A})^{-1} 
		\,\mathrm{d} \mu_{\colour}.
	\]
	Since $F$ is strongly irreducible, it follows from Proposition~\ref{prop:subsystem edge Jordan curve has measure zero} that $\mu(\mathcal{C}) = 0$. Thus $\mu_{\black}(\mathcal{C}) = \mu_{\white}(\mathcal{C}) = 0$. Note that $\ccndegF{\colour}{}{}{y} = 1$ for each $\colour \in \colours$, each $x \in \inte[\big]{X^0_{\colour}}$, and each $y \in F^{-1}(x)$. Thus by \eqref{eq:split measure on S2} we have
	\begin{equation}    \label{eq:Jacobian:eigmeasure}
		\begin{split}
			\mu(A) 
			&= \eigenvalue^{-1} \sum_{\colour \in \colours}  \int_{F(A) \cap X^0_{\colour} } \! \myexp{\phi} \circ (F|_{A})^{-1} \,\mathrm{d} \mu_{\colour}  \\
			&= \eigenvalue^{-1} \int_{F(A)} \! \myexp{\phi} \circ (F|_{A})^{-1} \,\mathrm{d}\mu 
			= \int_{F(A)} \! \frac{1}{J \circ (F|_{A})^{-1}} \,\mathrm{d}\mu,
		\end{split}
	\end{equation}
	where the function $J$ is given in \ref{item:prop:subsystem properties of eigenmeasure:Jacobian}.

	Since $F$ is injective on each $1$-tile $X^1 \in \Domain{1}$, and both $X^1$ and $F(X^1)$ are closed subsets of $S^2$ by Proposition~\ref{prop:subsystem:properties}~\ref{item:subsystem:properties:homeo}, in order to verify \eqref{eq:def:subsystem Jacobian}, it suffices to assume that $A \subseteq X$ for some $1$-tile $X \in \Domain{1}$. 
	Denote the restriction of $\mu$ on $X$ by $\mu_{X}$, i.e., $\mu_{X}$ assigns $\mu(B)$ to each Borel subset $B$ of $X$.

	Let $\widetilde{\mu}$ be a function defined on the set of Borel subsets of $X$ in such a way that $\widetilde{\mu}(B) = \mu(F(B))$ for each Borel $B \subseteq X$. It is clear that $\widetilde{\mu}$ is a Borel measure on $X$. In this notation, we can write \eqref{eq:Jacobian:eigmeasure} as
	\begin{equation}    \label{eq:restriction measure and forward measure}
		\mu_{X}(A) = \int_{A} \! \frac{1}{J|_{X}} \,\mathrm{d}\widetilde{\mu},
	\end{equation}
	for each Borel $A \subseteq X$.

	By \eqref{eq:restriction measure and forward measure}, we know that $\mu_{X}$ is absolutely continuous with respect to $\widetilde{\mu}$. 
	On the other hand, since $J$ is positive and uniformly bounded away from $+\infty$ on $X$, we can conclude that $\widetilde{\mu}$ is absolutely continuous with respect to $\mu_{X}$. 
	Therefore, by the Radon--Nikodym theorem, for each Borel $A \subseteq X$, we get $\mu(F(A)) = \widetilde{\mu}(A) = \int_{A} \! J|_{X} \,\mathrm{d}\mu_{X} = \int_{A} \! J \,\mathrm{d}\mu$. 
	Thus we finish the proof of statement~\ref{item:prop:subsystem properties of eigenmeasure:Jacobian}.

	\smallskip

	\ref{item:prop:subsystem properties of eigenmeasure:Gibbs property}  
	We observe that
	\begin{equation}    \label{eq:eigenmeasure Jacobian for iteration}
		\mu(F^{m}(B)) = \int_{B} \! \frac{\eigenvalue^{m}}{ \myexp{S_{m}\phi} } \,\mathrm{d}\mu
	\end{equation}
	for all $n \in \n$, $m \in \{0, \, 1, \, \dots, \, n\}$, and Borel set $B \subseteq \domF$ on which $F^n$ is defined and injective.
	Recall from Definition~\ref{def:subsystems} that $F^{n}$ is defined at $x \in \domF$ if and only if $F^{i}(x) \in \domF$ for each $i \in \{0, \, 1, \, \dots, \, n - 1\}$.
	Indeed, by statement~\ref{item:prop:subsystem properties of eigenmeasure:Jacobian}, for a given Borel set $A \subseteq \domF$ on which $F$ is injective, we have\[
		\int_{F(A)} \! g \,\mathrm{d} \mu  
		= \int_{A}\!  \frac{ \eigenvalue (g \circ F) }{ \myexp{\phi} } \,\mathrm{d} \mu
	\]
	for each simple function $g$ on $\domF$, thus also for each integrable function $g$. 

	Fix arbitrary $n\in \n$ and Borel set $B \subseteq \domF$ on which $F^n$ is defined and injective. 
	We establish \eqref{eq:eigenmeasure Jacobian for iteration} for each $m \in \zeroton[n]$ by induction. 
	For $m = 0$, \eqref{eq:eigenmeasure Jacobian for iteration} holds trivially. 
	Assume that (\ref{eq:eigenmeasure Jacobian for iteration}) is established for some $m \in \zeroton[n - 1]$, then since $F^m$ is injective on $F(B)$, we get
	\begin{equation*}
		\mu \parentheses[\big]{ F^{m+1}(B) } = \int_{F(B)} \frac{ \eigenvalue^m } { \myexp{S_m \phi} } \,\mathrm{d} \mu 
		=  \int_{B} \frac{ \eigenvalue^{m+1} } { \myexp{S_{m+1} \phi} } \,\mathrm{d} \mu.
	\end{equation*}
	The induction is now complete. In particular, by Proposition~\ref{prop:subsystem:properties}~\ref{item:subsystem:properties:homeo},
	\begin{equation}    \label{eq:Jacobian property on tiles}
		\mu(F^n(X^n)) = \int_{X^n}\!  \frac{ \eigenvalue^n } { \myexp{S_n \phi} } \,\mathrm{d} \mu
	\end{equation}
	for all $n \in \n$ and $X^n \in \Domain{n}$.

	By \eqref{eq:Jacobian property on tiles} and Lemma~\ref{lem:distortion_lemma}, for all $n \in \n_0$, $X^n \in \Domain{n}$, and $x \in X^n$, we have
	\begin{equation}   \label{eq:gibbs measure:two sides bound for subsystem}
		e^{-C} \mu(F^n(X^n)) \leqslant \eigenvalue^n \mu(X^n) / \myexp{S_n \phi(x)}  \leqslant e^{C} \mu (F^n(X^n)),
	\end{equation}
	where $C \define \Cdistortion \geqslant 0$ and $C_1$ is the constant defined in \eqref{eq:const:C_1} in Lemma~\ref{lem:distortion_lemma} and depends only on $f$, $\mathcal{C}$, $d$, $\phi$, and $\holderexp$.
	Note that $F^n(X^n) \in \Domain{0}$ is either the black $0$-tile $X^0_{\black}$ or the white $0$-tile $X^0_{\white}$. 
	We claim that for each $\colour \in \colours$,
	\begin{equation}    \label{eq:temp:prop:subsystem properties of eigenmeasure:constant for gibbs measure temp form}
		\mu\bigl(X^{0}_{\colour}\bigr) \geqslant \bigl( 1 + e^C \eigenvalue^{n(\ccolour, \colour)} \myexp{ n(\ccolour, \colour) \uniformnorm{\potential} } \bigr)^{-1},
	\end{equation}
	where $\ccolour \in \colours \setminus \{\colour\}$ and $n(\ccolour, \colour) \in \n$ is given by Definition~\ref{def:irreducibility of subsystem}.
	Indeed, since $F$ is irreducible, by Definition~\ref{def:irreducibility of subsystem}, for $\ccolour \in \colours \setminus \{\colour\}$ and $n(\ccolour, \colour) \in \n$, there exists $X^{n(\ccolour, \colour)}_{\ccolour} \in \ccFTile{n(\ccolour, \colour)}{\ccolour}{\colour}$ such that $X^{n(\ccolour, \colour)}_{\ccolour} \subseteq X^0_{\colour}$ and $F^{n(\ccolour, \colour)}\bigl( X^{n(\ccolour, \colour)}_{\ccolour} \bigr) = X^0_{\ccolour}$. 
	Then it follows from \eqref{eq:gibbs measure:two sides bound for subsystem} that 
	\[
		e^{-C} \mu\bigl( X^{0}_{\ccolour} \bigr) 
		\leqslant \eigenvalue^{n(\ccolour, \colour)} \myexp{ n(\ccolour, \colour) \uniformnorm{\potential} } \mu\bigl( X^{n(\ccolour, \colour)}_{\ccolour} \bigr)
		\leqslant \eigenvalue^{n(\ccolour, \colour)} \myexp{ n(\ccolour, \colour) \uniformnorm{\potential} } \mu\bigl( X^{0}_{\colour} \bigr).
	\]
	This implies \eqref{eq:temp:prop:subsystem properties of eigenmeasure:constant for gibbs measure temp form} since $\mu\bigl( X^{0}_{\colour} \bigr) + \mu\bigl( X^{0}_{\ccolour} \bigr) \geqslant \mu(S^2) = 1$. 
	Hence \eqref{eq:def:subsystem gibbs measure} holds, and $\mu$ is a Gibbs measure with respect to $F$, $\mathcal{C}$, and $\phi$ with $P_{\mu} = \log \eigenvalue$.

	Finally, we show that $P_{\mu} = \pressure$. 
	By Theorem~\ref{thm:subsystem:eigenmeasure existence and basic properties}~\ref{item:thm:subsystem:eigenmeasure existence and basic properties:eigenmeasure support on limitset}, we have that for each $n \in \n$,
	\begin{equation}    \label{eq:lower bound for eigenmeasure of tiles of F}
		\sum_{X^n \in \Domain{n}} \mu(X^n) \geqslant \mu(\limitset(F, \mathcal{C})) = 1.
	\end{equation}
	Since $\mu$ is a Gibbs measure with respect to $F$, $\mathcal{C}$, and $\phi$ with $P_{\mu} = \log \eigenvalue$, it follows from \eqref{eq:def:subsystem gibbs measure} that for each $n \in \n_0$,
	\[
		\sum_{X^n \in \Domain{n}} \mu(X^n) \leqslant C_{\mu} e^{-n P_{\mu}} \sum_{X^n \in \Domain{n}} \myexp{  \sup\{ S_n\phi(x) \describe x\in X^n \} }
	\]
	for some constant $C_{\mu} \geqslant 1$. 
	Combining this with \eqref{eq:lower bound for eigenmeasure of tiles of F}, we obtain the inequality $P_{\mu} \leqslant P(F, \phi)$. 
	For the other direction, since $F$ is strongly irreducible, by Theorem~\ref{thm:subsystem:eigenmeasure existence and basic properties}~\ref{item:thm:subsystem:eigenmeasure existence and basic properties:strongly irreducible vertex zero measure}, we get $\mu \bigl( \bigcup_{j = 0}^{+\infty}f^{-j}(\post{f}) \bigr) = 0$. Then it follows from Remark~\ref{rem:intersection of two tiles} and Proposition~\ref{prop:subsystem edge Jordan curve has measure zero} that\[
		\sum_{X^n \in \Domain{n}} \mu(X^n) \leqslant \mu\Bigl( \domain{n} \Bigr) + \mu\biggl( \bigcup_{j = 0}^{+\infty}F^{-j}(\mathcal{C}) \biggr) \leqslant 1.
	\]
	By \eqref{eq:def:subsystem gibbs measure}, we have that for each $n \in \n_{0}$,
	\[
		\sum_{X^n \in \Domain{n}} \mu(X^n) \geqslant C_{\mu}^{-1} e^{-n P_{\mu}} \sum_{X^n \in \Domain{n}} \myexp{  \sup\{ S_n\phi(x) \describe x\in X^n \} }.
	\]	
	This implies that $P_{\mu} \geqslant P(F, \phi)$.
	Thus we get $P_{\mu} = \pressure$ and finish the proof of statement~\ref{item:prop:subsystem properties of eigenmeasure:Gibbs property}.

	\smallskip

	\ref{item:prop:subsystem properties of eigenmeasure:forward quasi-invariant and non-singular}
	Denote $\limitset \define \limitset(F, \mathcal{C})$.
	Fix a Borel set $A \subseteq \limitset$ with $\mu(A) = 0$.
	For each $1$-tile $X^{1} \in \Tile{1}$, the map $\limitmap$ is injective both on $A \cap X^1$ and $\limitmap^{-1}(A) \cap X^1$ by Proposition~\ref{prop:subsystem:properties}~\ref{item:subsystem:properties:homeo}.
	Then it follows from the definition of the Jacobian (recall \eqref{def:subsystem Jacobian}) and statement~\ref{item:prop:subsystem properties of eigenmeasure:Jacobian} that $\mu( \limitmap(A \cap X^1) ) = 0$ and $\mu( \limitmap^{-1}(A) \cap X^1 ) = 0$.
	Thus $\mu(\limitmap(A)) = 0$ and $\mu(\limitmap^{-1}(A)) = 0$.
	This implies that $\limitmap$ is forward quasi-invariant and non-singular with respect to $\mu$.
\end{proof}

Now we can prove the existence of an $f$-invariant Gibbs measure with respect to $F$, $\mathcal{C}$, and $\phi$. 

\begin{proof}[Proof of Theorem~\ref{thm:existence of f invariant Gibbs measure}]
	Let $\eigmea = \spleigmea$ be an eigenmeasure as in Theorem~\ref{thm:subsystem:eigenmeasure existence and basic properties}.

	Define, for each $n \in \n$, $\widetilde{u}_{n} \define \frac{1}{n} \sum_{j = 0}^{n - 1} \normsplopt^{j}\bigl(\indicator{\splitsphere}\bigr)$. 
	Then $\{\widetilde{u}_n\}_{n \in \n}$ is a uniformly bounded sequence of equicontinuous functions on $\splitsphere$ by \eqref{eq:second bound for normed split operator} and \eqref{eq:third bound for normed split operator} in Lemma~\ref{lem:distortion lemma for normed split operator}. 
	By the \aalem~Theorem, every subsequence $\{\widetilde{u}_{n_{k}}\}_{k \in \n}$, which is a uniformly bounded sequence of equicontinuous functions on $\splitsphere$, has a further subsequential limit with respect to the uniform norm. 
	By Proposition~\ref{prop:existence of eigenfunction}, the sequence $\{\widetilde{u}_n\}_{n \in \n}$ has a subsequential limit $\eigfun \in C(\splitsphere)$ satisfying \eqref{eq:eigenfunction of normed split operator}, \eqref{eq:two sides bounds for eigenfunction}, and \eqref{eq:integration of eigenfunction with respect to eigenmeasure is one}.
	In order to prove this theorem, it suffices to show that $\{\widetilde{u}_n\}_{n \in \n}$ has a unique subsequential limit with respect to the uniform norm, and finally justify \eqref{eq:equalities for characterizations of pressure}.

	Suppose that $\eigfunv$ is another subsequential limit of $\{\widetilde{u}_n\}_{n \in \n}$ with respect to the uniform norm. Then $\eigfunv$ is also a continuous function on $\splitsphere$ satisfying \eqref{eq:eigenfunction of normed split operator}, \eqref{eq:two sides bounds for eigenfunction}, and \eqref{eq:integration of eigenfunction with respect to eigenmeasure is one} by Proposition~\ref{prop:existence of eigenfunction}. 
	To prove the uniqueness, it suffices to show that $\eigfun = \eigfunv$ on $\splitsphere = X^0_{\black} \sqcup X^0_{\white}$. 
	The proof proceeds in two steps. 
	We first show that $\eigfun = \eigfunv$ on a subset $\widetilde{\limitset}$ (defined in \eqref{eq:temp:thm:existence of f invariant Gibbs measure:def:split limitset} below) of $\splitsphere$, which satisfies $\spleigmea \bigl(\widetilde{\limitset} \bigr) = 1$. Then we extend this identity $\eigfun = \eigfunv$ to $\splitsphere$. 

	For each $n \in \n_0$, we set
	\begin{equation}    \label{eq:temp:thm:existence of f invariant Gibbs measure:def:split tile of subsystem}
		\splDomain{n} \define  \bigcup_{\colour \in \colours} \bigl\{ i_{\colour}(X^n) \describe X^n \in \Domain{n}, \, X^n \subseteq X^0_{\colour} \bigr\},
	\end{equation}
	where $i_{\colour}$ is defined by \eqref{eq:natural injection into splitsphere}. 
	Let $\spllimitset$ be the subset of $\splitsphere$ defined by 
	\begin{equation}    \label{eq:temp:thm:existence of f invariant Gibbs measure:def:split limitset}
		\spllimitset \define \bigcap\limits_{n \in \n} \bigcup \splDomain{n}.
	\end{equation}

	We claim that $\spleigmea \bigl(\widetilde{\limitset} \bigr) = 1$.
	Indeed, since $\spleigmea$ is an eigenmeasure of $\dualsplopt$, it follows from Proposition~\ref{prop:dual split operator}~\ref{item:prop:dual split operator:measure concentrate} and induction on $n$ that for each $n \in \n$,
	\[
		1 \geqslant \spleigmea\parentheses[\Big]{ \bigcup \splDomain{n} } = \spleigmea \parentheses[\Big]{ \bigcup \splDomain{0} } = \spleigmea \parentheses{ \splitsphere } = 1,
	\]
	where we use the fact that $\bigcup\splDomain{0} = \splitsphere$ since $F$ is irreducible so that $F(\domF) = S^2$.
	Note that by Proposition~\ref{prop:subsystem:properties invariant Jordan curve}~\ref{item:subsystem:properties invariant Jordan curve:decreasing relation of domains} and \eqref{eq:temp:thm:existence of f invariant Gibbs measure:def:split tile of subsystem}, $\bigl\{ \bigcup\splDomain{n} \bigr\}_{n \in \n}$ is a decreasing sequence of sets.
	Thus, by \eqref{eq:temp:thm:existence of f invariant Gibbs measure:def:split limitset}, we get
	\[
		\spleigmea\bigl(\spllimitset\bigr) = \lim_{n \to +\infty} \spleigmea \parentheses[\Big]{ \bigcup\splDomain{n} } = 1.
	\]
	
	Now we show that $\eigfunv(\widetilde{x}) = \eigfun(\widetilde{x})$ for each $\widetilde{x} \in \spllimitset$.
	Let \[
		t \define \sup \bigl\{ s \in \real \describe \eigfun(\widetilde{x}) - s \eigfunv(\widetilde{x}) > 0 \text{ for all } \widetilde{x} \in \splitsphere \bigr\}.
	\]
	It follows from \eqref{eq:two sides bounds for eigenfunction} that $t \in (0, +\infty)$. 
	Then there is a point $\widetilde{y} \in \splitsphere$ such that $\eigfun(\widetilde{y}) - t \eigfunv(\widetilde{y}) = 0$. 
	Without loss of generality, we may assume that $\widetilde{y} = (y, \black)$ for some point $y \in X^0_{\black}$. By \eqref{eq:definition of partial Ruelle operators}, \eqref{eq:def:split ruelle operator}, and the equality \[
		\normsplopt(\eigfun - t \eigfunv) = \eigfun - t \eigfunv,
	\]
	which comes from \eqref{eq:eigenfunction of normed split operator}, we get that $\eigfun(\widetilde{z}) - t \eigfunv(\widetilde{z}) = 0$ for each $\widetilde{z} \in \widetilde{F}^{-1}(\widetilde{y})$, where we define $\widetilde{F}^{-n}(\widetilde{y})$ to be the set
	\begin{equation}    \label{eq:split preimage}
		\bigl\{ (z, \colour) \describe z = \bigl(F^{n}|_{X^{n}_{\black \colour}}\bigr)^{-1}(y), \, X^{n}_{\black \colour} \in \ccFTile{n}{\black}{\colour}, \, \colour \in \colours \bigr\}.
	\end{equation}
	for each $n \in \n$. Inductively, we can conclude that $\eigfun(\widetilde{z}) - t \eigfunv(\widetilde{z}) = 0$ for all $\widetilde{z} \in \bigcup_{i \in \n} \widetilde{F}^{-i}(\widetilde{y})$. 
	Noting that $F$ is irreducible and mimicking the proof of Proposition~\ref{prop:irreducible subsystem properties}~\ref{item:prop:irreducible subsystem properties:preimages is dense in limitset}, we can show that the closure of $\bigcup_{i \in \n} \widetilde{F}^{-i}(\widetilde{y})$ in $\splitsphere$ contains $\widetilde{\limitset}$. 
	Indeed, by \eqref{eq:temp:thm:existence of f invariant Gibbs measure:def:split tile of subsystem}, \eqref{eq:temp:thm:existence of f invariant Gibbs measure:def:split limitset}, and \eqref{eq:definition of expansion}, it suffices to show that for each $n \in \n$ and each $\widetilde{X}^{n} \in \splDomain{n}$, $\widetilde{X}^n \cap \bigcup_{i \in \n} \widetilde{F}^{-i}(\widetilde{y}) \ne \emptyset$. 
	Fix arbitrary $n \in \n$ and $\widetilde{X}^n \in \splDomain{n}$. 
	By \eqref{eq:temp:thm:existence of f invariant Gibbs measure:def:split tile of subsystem}, there exist $\colour \in \colours$ and $X^n \in \Domain{n}$ such that $X^n \subseteq X^0_{\colour}$ and $\widetilde{X}^n  = i_{\colour}(X^n)$. 
	By Proposition~\ref{prop:subsystem:properties}~\ref{item:subsystem:properties:homeo}, $X^n$ is mapped by $F^n$ homeomorphically to a $0$-tile $X^0_{\ccolour}$ for some $\ccolour \in \colours$. 
	Since $F$ is irreducible, by Definition~\ref{def:irreducibility of subsystem}, there exist $k \in \n$ and $Y^{k} \in \Domain{k}$ such that $Y^{k} \subseteq X^0_{\ccolour}$ and $F^{k}\bigl(Y^{k}\bigr) = X^0_{\black}$. 
	Then it follows from Lemma~\ref{lem:cell mapping properties of Thurston map}~\ref{item:lem:cell mapping properties of Thurston map:i} and Proposition~\ref{prop:subsystem:properties}~\ref{item:subsystem:properties:homeo} that $X^{k + n} \define (F^{n}|_{X^n})^{-1}\bigl(Y^{k}\bigr) \in \ccFTile{k + n}{\black}{\colour}$. 
	Denote $z \define (F^{k + n}|_{X^{k + n}})^{-1}(y) \in X^{k + n}$. Then by \eqref{eq:split preimage} we have $(z, \colour) \in \bigcup_{i \in \n} \widetilde{F}^{-i}(\widetilde{y})$.
	Since $z \in X^{k + n} \subseteq X^{n} \subseteq X^0_{\colour}$, we have $( z , \colour ) \in \widetilde{X}^n \cap \bigcup_{i \in \n} \widetilde{F}^{-i}(\widetilde{y}) \ne \emptyset$.
	Thus the closure of $\bigcup_{i \in \n} \widetilde{F}^{-i}(\widetilde{y})$ in $\splitsphere$ contains $\widetilde{\limitset}$. 
	Hence $\eigfun = t \eigfunv$ on $\widetilde{\limitset}$. Since both $\eigfun$ and $\eigfunv$ satisfy \eqref{eq:integration of eigenfunction with respect to eigenmeasure is one} and $\spleigmea(\widetilde{\limitset}) = 1$, we get $t = 1$.
	Thus $\eigfunv(\widetilde{x}) = \eigfun(\widetilde{x})$ for each $\widetilde{x} \in \spllimitset$.

	Following similar arguments to those in Proposition~\ref{prop:irreducible subsystem properties}~\ref{item:prop:irreducible subsystem properties:limitset non degenerate to Jordan curve}, we can show that $\spllimitset \cap \widetilde{X}^{n} \ne \emptyset$ for each $n \in \n$ and each $\widetilde{X}^{n} \in \splDomain{n}$. 
	The key idea is that for any tile $\widetilde{X}^n \in \splDomain{n}$, we can construct a nested sequence of tiles by leveraging the irreducibility of $F$ and the injection $i_{\colour}$ as defined in \eqref{eq:natural injection into splitsphere}.
	This nested sequence has a non-empty intersection, which implies that $\spllimitset \cap \widetilde{X}^{n} \ne \emptyset$. 
	More precisely, given $\widetilde{X}^n = i_{\colour}(X^n)$ for some $\colour \in \colours$, the irreducibility of $F$ allows us to find tiles that map onto $X^0_{\colour}$, enabling us to construct a nested sequence $\{ X^{n + km} \}_{k\in \n_0}$ whose images under $i_{\colour}$ yield a point in $\spllimitset \cap \widetilde{X}^{n}$.

	We next show that $\eigfun = \eigfunv$ on $\splitsphere$. Recall for each $\colour \in \colours$, $\eigfun$ and $\eigfunv$ are continuous on $X^0_{\colour}$ and $X^0_{\colour}$ is compact. 
	Fix an arbitrary number $\varepsilon > 0$, then we can choose $\delta > 0$ such that for each $\colour \in \colours$ and each pair of points $\juxtapose{x}{y} \in X^0_{\colour}$ with $d(x, y) < \delta$, we have $|\eigfun(\widetilde{x}) - \eigfun(\widetilde{y})| < \varepsilon$ and $|\eigfunv(\widetilde{x}) - \eigfunv(\widetilde{y})| < \varepsilon$, where $\widetilde{x} \define (x, \colour)$ and $\widetilde{y} \define (y, \colour)$. 
	By Lemma~\ref{lem:visual_metric}~\ref{item:lem:visual_metric:diameter of cell}, there exists $n \in \n$ such that for each $X^n \in \Domain{n}$, $\diam{d}{X^n} < \delta$. We fix such an integer $n$ in the rest of this paragraph. 
	We also fix an arbitrary point $\widetilde{y}_{\widetilde{X}^{n}} \in \spllimitset \cap \widetilde{X}^{n}$ for each $\widetilde{X}^{n} \in \splDomain{n}$.
	Since there is a natural bijection $X^n \mapsto \widetilde{X}^{n}$ between $\Domain{n}$ and $\splDomain{n}$, we can define $\widetilde{y}_{X^{n}} \define \widetilde{y}_{\widetilde{X}^{n}}$ for each $X^n \in \Domain{n}$. 
	Hence by \eqref{eq:definition of partial Ruelle operators}, \eqref{eq:iteration of split-partial ruelle operator}, Proposition~\ref{prop:subsystem properties of eigenmeasure}~\ref{item:prop:subsystem properties of eigenmeasure:Gibbs property}, and Definition~\ref{def:subsystem gibbs measure}, for each $\colour \in \colours$ and each $\widetilde{x} = (x, \colour) \in i_{\colour}\bigl(X^{0}_{\colour}\bigr)$,
	\begin{align*}
		|\eigfun(\widetilde{x}) - \eigfunv(\widetilde{x})| 
		&= \abs[\big]{ \normsplopt^n(\eigfun - \eigfunv)(\widetilde{x}) }  \\
		&\leqslant \sum_{X^n \in \cFTile{n}} |\eigfun(\widetilde{x}_{X^n}) - \eigfunv(\widetilde{x}_{X^n})| \myexp[\big]{ S_{n}\phi(x_{X^n}) - nP_{\eigmea} }  \\
		&\leqslant C_{\eigmea} \sum_{X^n \in \cFTile{n}} \Bigl( \abs[\big]{ \eigfun(\widetilde{x}_{X^n}) - \eigfun(\widetilde{y}_{X^n}) }  + \abs[\big]{ \eigfun(\widetilde{y}_{X^n}) - \eigfunv(\widetilde{y}_{X^n}) }   \\
		&\qquad \qquad \qquad\qquad \qquad\qquad\qquad \ \ + \abs[\big]{ \eigfunv(\widetilde{y}_{X^n}) - \eigfunv(\widetilde{x}_{X^n}) } \Bigr) \eigmea(X^n)  \\
		&\leqslant C_{\eigmea} \sum_{X^n \in \cFTile{n}} 2 \varepsilon \eigmea(X^n) \\
		&\leqslant 4 \varepsilon C_{\eigmea},
	\end{align*}
	where we write $x_{X^n} \define (F^{n}|_{X^n})^{-1}(x)$ and $\widetilde{x}_{X^n} \define (x_{X^n}, \ccolour)$ with $\ccolour \in \colours$ uniquely determined by the relation $X^n \subseteq X^0_{\ccolour}$. 
	Since $\varepsilon >0$, $\colour \in \colours$, and $\widetilde{x} \in i_{\colour}\bigl(X^{0}_{\colour}\bigr)$ are arbitrary, we conclude that $\eigfun = \eigfunv$ on $\splitsphere$. 
	
	We have proved that $\widetilde{u}_{n}$ converges to $\eigfun$ uniformly as $n \to +\infty$.

	It now follows immediately from \eqref{eq:second bound for normed split operator}, \eqref{eq:third bound for normed split operator}, and the uniform convergence of $\widetilde{u}_{n}$ that $\eigfun \in \splholderspace$.

	By Proposition~\ref{prop:subsystem:f-invariant measure}, $\equstate$ is $f$-invariant. 
	Since $\eigmea$ is a Gibbs measure (with respect to $F$, $\mathcal{C}$ and $\phi$) supported on $\limitset$ with $P_{\eigmea} = \pressure = \Cnormspl$ by Proposition~\ref{prop:subsystem properties of eigenmeasure}~\ref{item:prop:subsystem properties of eigenmeasure:Gibbs property} and \eqref{eq:eigenmeasure for dual split operator}, then by \eqref{eq:two sides bounds for eigenfunction} and Proposition~\ref{thm:subsystem:eigenmeasure existence and basic properties}~\ref{item:thm:subsystem:eigenmeasure existence and basic properties:eigenmeasure support on limitset}, $\equstate$ is also a Gibbs measure (with respect to $F$, $\mathcal{C}$ and $\phi$) supported on $\limitset$ with $P_{\equstate} = P_{\eigmea} = \pressure = D_{F, \phi} = \lim_{n \to +\infty} \frac{1}{n}\log \mathopen{}\bigl( \splopt^{n}\bigl(\indicator{\splitsphere}\bigr)(\widetilde{y}) \bigr)$ for each $\widetilde{y} \in \splitsphere$, establishing \eqref{eq:equalities for characterizations of pressure}.

	Finally, let $U \subseteq S^2$ be an open set with $U \cap \limitset \ne \emptyset$. 
	Then by \eqref{eq:def:limitset} and \eqref{eq:definition of expansion}, there exists $X^k \in \Domain{k}$ with $X^k \subseteq U$ for some $k \in \n$. 
	Since $\equstate$ is a Gibbs measure, we get $\equstate(X^k) > 0$. Thus $\equstate(U) \geqslant \equstate(X^k) > 0$.
\end{proof}
\begin{rmk}
	By an argument similar to that in the proof of the uniqueness of the subsequential limit of $\Bigl\{ \frac{1}{n} \sum_{j = 0}^{n - 1} \normsplopt^j\bigl(\indicator{\splitsphere}\bigr) \Big \}_{n \in \n}$, one can show that $\eigfun$ is the unique eigenfunction, up to a scalar multiple, of $\normsplopt$ corresponding to the eigenvalue $1$.
\end{rmk}

\subsection{Variational Principle and existence of the equilibrium states for subsystems}%
\label{sub:Variational Principle and existence of the equilibrium states for subsystems}

In this subsection, we prove the main results Theorems~\ref{thm:subsystem characterization of pressure} and \ref{thm:existence of equilibrium state for subsystem} of this subsection.  

\smallskip

In the following theorem we establish the Variational Principle for strongly irreducible subsystems with respect to \holder continuous potentials. 
Recall that $F(\limitset) \subseteq \limitset$ by Proposition~\ref{prop:subsystem:properties}~\ref{item:subsystem:properties:limitset forward invariant} and $\limitmap = F|_{\limitset} \colon \limitset \mapping \limitset$.

\begin{theorem}    \label{thm:subsystem characterization of pressure}
	Let $f$, $\mathcal{C}$, $F$, $d$, $\phi$, $\holderexp$ satisfy the Assumptions in Section~\ref{sec:The Assumptions}. 
	We assume in addition that $f(\mathcal{C}) \subseteq \mathcal{C}$ and $F \in \subsystem$ is strongly irreducible. 
	Then we have
	\begin{equation}    \label{eq:subsystem Variational Principle}
		\Cnormspl = \pressure = \fpressure[\potential|_{\limitset}] = \sup \Bigl\{ h_{\mu}(\limitmap) + \int_{\limitset} \! \phi \,\mathrm{d}\mu \describe \mu \in \mathcal{M}(\limitset, \limitmap) \Bigr\},
	\end{equation}
	where $\Cnormspl$ is defined in Corollary~\ref{coro:well-define for split operator norm constant}, $\pressure$ is defined in \eqref{eq:pressure of subsystem}, and $\fpressure[\potential|_{\limitset}]$ is defined in \eqref{eq:def:topological pressure}.
\end{theorem}

The following theorem gives the existence of the equilibrium states for strongly irreducible subsystems and \holder continuous potentials.
\begin{theorem}    \label{thm:existence of equilibrium state for subsystem}
	Let $f$, $\mathcal{C}$, $F$, $d$, $\phi$, $\holderexp$ satisfy the Assumptions in Section~\ref{sec:The Assumptions}. 
	We assume in addition that $f(\mathcal{C}) \subseteq \mathcal{C}$ and $F \in \subsystem$ is strongly irreducible.
	Then there exists an equilibrium state for $\limitmap$ and $\phi|_{\limitset}$. 
	Moreover, any measure $\equstate \in \mathcal{M}(S^{2}, f)$ defined in Theorem~\ref{thm:existence of f invariant Gibbs measure} is an equilibrium state for $\limitmap$ and $\phi|_{\limitset}$, and the map $\limitmap$ with respect to such $\equstate$ is forward quasi-invariant (i.e., for each Borel set $A \subseteq \limitset$, if $\equstate(A) = 0$, then $\equstate(\limitmap(A)) = 0$) and non-singular (i.e., for each Borel set $A \subseteq \limitset$, if $\equstate(A) = 0$, then $\equstate(\limitmap^{-1}(A)) = 0$).
\end{theorem}

The proofs of Theorems~\ref{thm:subsystem characterization of pressure} and~\ref{thm:existence of equilibrium state for subsystem} will be given at the end of this subsection. 
Note that Theorem~\ref{thm:main:thermodynamic formalism for subsystems} follows immediately from Theorems~\ref{thm:subsystem characterization of pressure} and~\ref{thm:existence of equilibrium state for subsystem}.

We use the following convention in this subsection.
\begin{remark}\label{rem:invariant measure equivalent}
	Let $X$ be a non-empty Borel subset of $S^{2}$.
    Given a Borel probability measure $\mu \in \probmea{X}$, by abuse of notation, we can view $\mu$ as a Borel probability measure on $S^{2}$ by setting $\mu(A) \define \mu(A \cap X)$ for all Borel subsets $A \subseteq S^{2}$. 
    Conversely, for each measure $\nu \in \probsphere$ supported on $X$, we can view $\nu$ as a Borel probability measure on $X$.
\end{remark}

\begin{proposition}    \label{prop:subsystem:Gibbs pressure less than topological pressure}
	Let $f$, $\mathcal{C}$, $F$, $d$, $\phi$ satisfy the Assumptions in Section~\ref{sec:The Assumptions}. 
	We assume in addition that $f(\mathcal{C}) \subseteq \mathcal{C}$. 
	Then for each $\limitmap$-invariant Gibbs measure $\mu \in \mathcal{M}(\limitset, \limitmap) \subseteq \probsphere$ with respect to $F$, $\mathcal{C}$, and $\phi$, we have
	\begin{equation}    \label{eq:subsystem:Gibbs pressure less than topological pressure}
		P_{\mu} \leqslant h_{\mu}(\limitmap) + \int_{\limitset} \! \phi \,\mathrm{d}\mu \leqslant \fpressure[\potential|_{\limitset}],
	\end{equation}
	where $\fpressure[\potential|_{\limitset}]$ is the topological pressure of the map $\limitmap = F|_{\limitset} \colon \limitset \mapping \limitset$ with respect to the potential $\phi|_{\limitset}$ (see \eqref{eq:def:topological pressure}). 
\end{proposition}
Recall that $F(\limitset) \subseteq \limitset$ by Proposition~\ref{prop:subsystem:properties}~\ref{item:subsystem:properties:limitset forward invariant}.
\begin{proof}
	Note that the second inequality in \eqref{eq:subsystem:Gibbs pressure less than topological pressure} follows from the Variational Principle~\eqref{eq:Variational Principle for pressure} in Subsection~\ref{sub:thermodynamic formalism}.

	We now use the measurable partitions $O_{n}$, $n \in \n$, of $S^{2}$ that were defined in \eqref{eq:partition induced by cell decomposition}.
	Since $f(\mathcal{C}) \subseteq \mathcal{C}$, it is clear that $O_{k}$ is a refinement of $O_{\ell}$ for $k \geqslant \ell \geqslant 1$.
	Observe that by Proposition~\ref{prop:properties cell decompositions}~\ref{item:prop:properties cell decompositions:cellular} and induction, we can conclude that for each $n \in \n$, 
	\begin{equation}    \label{eq:measurable partition of S2}
		\bigvee_{j = 0}^{n - 1} f^{-j} (O_{1}) = O_{n}.
	\end{equation}
	For $n \in \n$, we define $\widehat{O}_{n} \define \{ U \cap \limitset \describe U \in O_{n} \}$, which is a measurable partition of $\limitset$. Noting that $\limitmap^{-1}(A) = \limitset \cap f^{-1}(A)$ for each subset $A \subseteq \limitset$, it follows from \eqref{eq:measurable partition of S2} that for each $n \in \n$,
	\begin{equation}    \label{eq:measurable partition of limitset subsystem}
		\bigvee_{j = 0}^{n - 1} \limitmap^{-j}(\widehat{O}_{1}) = \widehat{O}_{n}.
	\end{equation}

	Let $\mu \in \mathcal{M}(\limitset, \limitmap) \subseteq \probsphere$ be an $\limitmap$-invariant Gibbs measure with respect to $F$, $\mathcal{C}$, and $\phi$. 
	Then by the Shannon--McMillan--Breiman Theorem (see for example, \cite[Theorem~2.5.4]{przytycki2010conformal}), we have $h_{\mu}(\limitmap, \widehat{O}_{1}) = \int \! F_{\mathcal{I}} \,\mathrm{d}\mu$, where\[
		F_{\mathcal{I}} \define \lim_{n \to +\infty} \frac{1}{n} 
		I_{\mu} \biggl( \bigvee_{j = 0}^{n - 1} \limitmap^{-j}(\widehat{O}_{1}) \biggr) \qquad \mu\text{-a.e.\ and in } L^{1}(\mu),
	\]
	$h_{\mu}\bigl(\limitmap, \widehat{O}_{1}\bigr)$ is defined in \eqref{eq:def:measure-theoretic entropy with respect to partition}, and the information function $I_{\mu}$ is defined in \eqref{eq:def:information function}.

	Note that for all $n \in \n$, $\widehat{U} \in \widehat{O}_{n}$, and $X^n \in \Domain{n}$, either $\widehat{U} \cap X^n = \emptyset$ or $\widehat{U} \subseteq X^n$.

	For $n \in \n_{0}$ and $x \in \limitset$, we denote by $X^{n}(x) \in \Domain{n}$ any one of the $n$-tiles of $F$ containing $x$. Recall that $\widehat{O}_{n}(x)$ denotes the unique set in the measurable partition $\widehat{O}_{n}$ that contains $x$. 
	Since $\widehat{O}_{n}(x) \subseteq X^{n}(x)$, we have that
	\[
		 I_{\mu} \parentheses[\big]{ \widehat{O}_{n} } (x) 
		 = -\log \parentheses[\big]{ \mu \parentheses[\big]{ \widehat{O}_{n}(x) } } 
		 \geqslant -\log \parentheses[\big]{ \mu \parentheses[\big]{ X^{n}(x) } }.
	\]
	Then by \eqref{eq:measurable partition of limitset subsystem} and \eqref{eq:def:subsystem gibbs measure} we deduce that
	\[
		\begin{split}
			\int \! F_{\mathcal{I}} \,\mathrm{d}\mu 
			&= \int \! \lim_{n \to +\infty} \frac{1}{n} I_{\mu} \biggl( \bigvee_{j = 0}^{n - 1} \limitmap^{-j}(\widehat{O}_{1}) \biggr)(x)  \,\mathrm{d}\mu(x) \\
			&= \int \! \lim_{n \to +\infty} \frac{1}{n}  I_{\mu} \parentheses[\big]{ \widehat{O}_{n} }(x)  \,\mathrm{d}\mu(x)	\\
			&\geqslant \limsup_{n \to +\infty} \int \! \frac{1}{n} I_{\mu} \parentheses[\big]{ \widehat{O}_{n} } (x) \,\mathrm{d}\mu(x) \\
			&\geqslant \limsup_{n \to +\infty} \int \! -\frac{1}{n} \log \parentheses[\big]{ \mu \parentheses[\big]{ X^{n}(x) } } \,\mathrm{d}\mu(x)  \\
			&\geqslant \limsup_{n \to +\infty} \int \! \frac{ nP_{\mu} - S_{n}^{F} \phi(x) - \log C_{\mu} }{n}  \,\mathrm{d}\mu(x)  \\
			&= P_{\mu} - \liminf_{n \to +\infty} \frac{1}{n} \int \! S_{n}^{F}\phi(x) \,\mathrm{d}\mu(x)\\
			&= P_{\mu} - \int \! \phi \,\mathrm{d}\mu,
		\end{split}
	\]
	where the last equality follows from \eqref{eq:def:Birkhoff average} and the identity $\int \! \psi \circ \limitmap \,\mathrm{d}\mu = \int \! \psi \,\mathrm{d}\mu$ for each $\psi \in C(\limitset)$, which is equivalent to the fact that $\mu$ is $\limitmap$-invariant. 
	Since $\widehat{O}_{1}$ is a finite measurable partition, the condition than $H_{\mu}(\widehat{O}_{1}) < +\infty$ in \eqref{eq:def:measure-theoretic entropy} is fulfilled. 
	By \eqref{eq:def:measure-theoretic entropy}, we get that\[
		h_{\mu}(\limitmap) \geqslant h_{\mu}(\limitmap, \widehat{O}_{1}) \geqslant P_{\mu} - \int \! \phi \,\mathrm{d}\mu.
	\]
	Therefore, $P_{\mu} \leqslant h_{\mu}(\limitmap) + \int \! \phi \,\mathrm{d}\mu$.
\end{proof}

\begin{proposition}    \label{prop:subsystem invariant measure equivalence}
	Let $f$, $\mathcal{C}$, $F$ satisfy the Assumptions in Section~\ref{sec:The Assumptions}.

    \begin{enumerate}[label=\rm{(\roman*)}]
		\smallskip
		\item    \label{item:prop:subsystem invariant measure equivalence:restriction}  
			If $\mu \in \probsphere$ is $f$-invariant and is supported on $\limitset(F, \mathcal{C})$, then $\mu$ is $\limitmap$-invariant, i.e., $\mu(A) = \mu\bigl(\limitmap^{-1}(A)\bigr)$ for each Borel subset $A$ of $\limitset$.
		
		\smallskip
		
		\item    \label{item:prop:subsystem invariant measure equivalence:extension} 
			If measure $\mu \in \probmea{\limitset}$ is $\limitmap$-invariant, then $\mu \in \probsphere$ is $f$-invariant.
			Moreover, $\mathcal{M}(\limitset, \limitmap) \subseteq \mathcal{M}(S^{2}, f)$.
	\end{enumerate}
\end{proposition}
Recall that $F(\limitset) \subseteq \limitset$ by Proposition~\ref{prop:subsystem:properties}~\ref{item:subsystem:properties:limitset forward invariant}. 
We use the convention in Remark~\ref{rem:invariant measure equivalent}.
\begin{proof}
	\ref{item:prop:subsystem invariant measure equivalence:restriction} 
	Assume that $\mu \in \probsphere$ is $f$-invariant and is supported on $\limitset(F, \mathcal{C})$. Let $A \subseteq \limitset$ be an arbitrary Borel subset. 
	Noting that $(F|_{\limitset})^{-1}(A) = \limitset \cap f^{-1}(A)$, we have
	\[
		\mu\bigl(\limitmap^{-1}(A)\bigr) = \mu \bigl( \limitset \cap f^{-1}(A) \bigr) = \mu \bigl(f^{-1}(A) \bigr) = \mu(A).
	\]
	Thus $\mu$ is $\limitmap$-invariant.

	\smallskip

	\ref{item:prop:subsystem invariant measure equivalence:extension} 
	Assume that $\mu \in \probmea{\limitset}$ is $\limitmap$-invariant. Let $A \subseteq S^{2}$ be an arbitrary Borel subset. 
	Noting that $(F|_{\limitset})^{-1}(A \cap \limitset) = \limitset \cap f^{-1}(A)$, we have
	\[
		\mu \bigl( f^{-1}(A) \bigr) = \mu \bigl( \limitset \cap f^{-1}(A) \bigr) = \mu \bigl(\limitmap^{-1}(A \cap \limitset) \bigr) = \mu(A \cap \limitset) = \mu(A).
	\]
	Thus $\mu$ is $f$-invariant.
\end{proof}

\begin{proposition}    \label{prop:characterization of pressures of subsystems in separated sets}
 	Let $f$, $\mathcal{C}$, $F$, $d$, $\Lambda$ satisfy the Assumptions in Section~\ref{sec:The Assumptions}. 
 	We assume in addition that $f(\mathcal{C}) \subseteq \mathcal{C}$ and $F \in \sursubsystem$. 
 	Consider $\varphi \in C(S^2)$.  
 	Then 
	\begin{equation}  \label{eq:inequality for pressures of subsystems}
		\pressure[\varphi] \geqslant \fpressure[\varphi|_{\limitset}],
	\end{equation}
	where $\fpressure[\varphi|_{\limitset}]$ is the topological pressure of the map $\limitmap \colon \limitset \mapping \limitset$ with respect to the potential $\varphi|_{\limitset}$ (see \eqref{eq:def:topological pressure}).
\end{proposition}
Recall that $F(\limitset) \subseteq \limitset$ by Proposition~\ref{prop:subsystem:properties}~\ref{item:subsystem:properties:limitset forward invariant} and $\limitmap \define F|_{\limitset}$.
\begin{proof}
	By \eqref{eq:def:topological pressure}, it suffices to show that 
	\begin{equation}    \label{eq:temp:characterization of pressures of subsystems in separated sets}
		\pressure[\varphi] \geqslant \lim_{\varepsilon \to 0^{+}} \limsup_{n \to +\infty} \frac{1}{n}\log (N_{d}(\limitmap, \varphi|_{\limitset}, \varepsilon, n)),
	\end{equation}
	where 
	\[
		N_{d}(\limitmap, \varphi|_{\limitset}, \varepsilon, \, m) 
		\define \sup \set[\Big]{\sum_{x \in E} \myexp[\big]{S_m^{F}\varphi(x)} \describe E \subseteq \limitset \text{ is } (m, \varepsilon)\text{-separated with respect to } \limitmap}.
	\]

	For fixed $\varepsilon > 0$ and $n \in \n$, by Lemma~\ref{lem:visual_metric}~\ref{item:lem:visual_metric:diameter of cell}, we have
	\begin{equation}    \label{eq:temp:prop:characterization of pressures of subsystems in separated sets:diam bound}
		\diam{d}{X^{n + i}} \leqslant C \Lambda^{-(n + i)} \qquad \text{for all } i \in \n_0 \text{ and } X^{n + i} \in \Domain{n + i},
	\end{equation}
	where $C \geqslant 1$ is the constant from Lemma~\ref{lem:visual_metric}.
	Let $k = k(\varepsilon)$ be the smallest non-negative integer satisfying $C \Lambda^{-(k + 1)} < \varepsilon$. 
	Let $E \subseteq \limitset$ be an arbitrary $(n, \varepsilon)$-separated set (with respect to $\limitmap$). 
	For each $x \in E$, let $X^{n + k}(x)$ be an element of $\Domain{n + k}$ containing $x$. 
	Next, we prove that the map $x \mapsto X^{n + k}(x)$ is injective by contradiction. 
	Suppose this map is not injective, i.e., there exist two distinct points $\juxtapose{x}{y} \in E$ such that $X^{n + k}(x) = X^{n + k}(y)$. 
	Then by \eqref{eq:temp:prop:characterization of pressures of subsystems in separated sets:diam bound} and Proposition~\ref{prop:subsystem:properties}~\ref{item:subsystem:properties:homeo}, we have\[
		d(F^{i}(x), F^{i}(y)) \leqslant \diam[\big]{d}{F^{i}\bigl(X^{n + k}\bigr)} \leqslant C \Lambda^{-(n + k - i)} \leqslant C \Lambda^{-(k + 1)} < \varepsilon
	\]
	for each $i \in \{0,\, 1\, \dots,\, n-1\}$. 
	This contradicts the fact that $E \subseteq \limitset$ is $(n, \varepsilon)$-separated. 
	Hence the map $x \mapsto X^{n + k}(x)$ is injective. 
	Thus,
	\begin{align*}
	 	\sum_{x \in E} \myexp[\big]{ S_n^{F}\varphi(x) } 
	 	&\leqslant \sum_{X^{n + k} \in \Domain{n + k}} \myexp[\big]{ \sup \bigl\{S^F_n \varphi(x) \describe x \in X^{n + k} \bigr\} }   \\
		&\leqslant \sum_{X^{n + k} \in \Domain{n + k}} \myexp[\big]{ k \uniformnorm{\varphi} + \sup\bigl\{S^F_{n + k} \varphi(x) \describe x \in X^{n + k} \bigr\} }  \\
		&=e^{k \uniformnorm{\varphi}} Z_{n + k}(F, \varphi).
	\end{align*}
	Noting that the above inequality holds for every $(n, \varepsilon)$-separated set $E \subseteq \limitset$ since $k = k(\varepsilon)$ is independent of $E$, we can take the supremum of the left-hand side of the above inequality over all $(n, \varepsilon)$-separated set in $\limitset$. Then it follows from Lemma~\ref{lem:well-defined pressure of subsystems} that \[
		\limsup_{n \to +\infty} \frac{1}{n}\log (N_{d}(\limitmap, \varphi|_{\limitset}, \varepsilon, \, n))
	 	\leqslant \lim_{n \to +\infty}\left( \frac{k\uniformnorm{\varphi}}{n} + \frac{n+k}{n} \frac{ \log (Z_{n+k}(F, \varphi)) }{n+k} \right) 
	 	= \pressure[\varphi].
	\]
	Note that $N_{d}(\limitmap, \varphi|_{\limitset}, \varepsilon, \, n)$ increases as $\varepsilon$ decreases. 
	By letting $\varepsilon$ decrease to $0$, we establish \eqref{eq:temp:characterization of pressures of subsystems in separated sets} and finish the proof.
\end{proof}

\begin{proof}[Proof of Theorem~\ref{thm:subsystem characterization of pressure}]
	Let $\fpressure[\potential|_{\limitset}]$ be the topological pressure of the map $\limitmap \colon \limitset \mapping \limitset$ with respect to the potential $\phi|_{\limitset}$. 
	Then it follows from the Variational Principle (see \eqref{eq:Variational Principle for pressure}) that $\fpressure[\potential|_{\limitset}] = \sup \bigl\{ h_{\mu}(\limitmap) + \int_{\limitset} \! \phi \,\mathrm{d}\mu \describe \mu \in \mathcal{M}(\limitset, \limitmap) \bigr\}$. 
	By Proposition~\ref{prop:characterization of pressures of subsystems in separated sets}, we have $\pressure \geqslant \fpressure[\potential|_{\limitset}]$. 
	Thus, it suffices to show that $\Cnormspl = \pressure \leqslant \fpressure[\potential|_{\limitset}]$.
	
	Let $\equstate \in \probsphere$ be an $f$-invariant Gibbs measure with respect to $F$, $\mathcal{C}$, and $\phi$ with $\equstate(\limitset) = 1$ and $P_{\equstate} = \pressure = \Cnormspl$ from Theorem~\ref{thm:existence of f invariant Gibbs measure}. 
	By Proposition~\ref{prop:subsystem invariant measure equivalence}, $\equstate$ is $\limitmap$-invariant. 
	Thus, by abuse of notation, we have $\equstate \in \mathcal{M}(\limitset, \limitmap)$. 
	Then it follows from Proposition~\ref{prop:subsystem:Gibbs pressure less than topological pressure} that $\pressure = P_{\equstate} \leqslant \fpressure[\potential|_{\limitset}]$.
	The proof is complete.
\end{proof}

\begin{proof}[Proof of Theorem~\ref{thm:existence of equilibrium state for subsystem}]
	Consider an $f$-invariant Gibbs measure $\equstate \in \mathcal{M}(S^{2}, f)$ with respect to $F$, $\mathcal{C}$, and $\phi$ with $\equstate(\limitset) = 1$ and $P_{\equstate} = \pressure$ from Theorem~\ref{thm:existence of f invariant Gibbs measure}. 
	Then by Proposition~\ref{thm:subsystem characterization of pressure} and the Variational Principle (see \eqref{eq:Variational Principle for pressure}), we have $P_{\equstate} = \fpressure[\potential|_{\limitset}]$. 
	By Theorem~\ref{prop:subsystem invariant measure equivalence}, $\equstate$ is $\limitmap$-invariant. 
	Thus by abuse of notation we have $\equstate \in \mathcal{M}(\limitset, \limitmap)$. 
	Then it follows from Proposition~\ref{prop:subsystem:Gibbs pressure less than topological pressure} that $P_{\equstate} = h_{\equstate}(\limitmap) + \int_{\limitset} \! \phi \,\mathrm{d}\equstate = \fpressure[\potential|_{\limitset}]$. 
	Therefore, $\equstate$ is an equilibrium state for $\limitmap$ and $\potential|_{\limitset}$.
	The fact that $\limitmap$ is forward quasi-invariant and non-singular with respect to $\equstate$ follows from \eqref{eq:two sides bounds for eigenfunction} in Theorem~\ref{thm:existence of f invariant Gibbs measure} and Proposition~\ref{prop:subsystem properties of eigenmeasure}~\ref{item:prop:subsystem properties of eigenmeasure:forward quasi-invariant and non-singular}.
\end{proof}

\begin{proof}[Proof of Theorem~\ref{thm:main:thermodynamic formalism for subsystems}]
	The statement follows immediately from Theorems~\ref{thm:subsystem characterization of pressure} and~\ref{thm:existence of equilibrium state for subsystem}. 
\end{proof} 
\section{Large deviation asymptotics for expanding Thurston maps}
\label{sec:Large deviation asymptotics for expanding Thurston maps}

In this section, we establish the large deviation asymptotics (Theorem~\ref{thm:large deviation asymptotics}) for expanding Thurston maps.
In Subsection~\ref{sub:The rate function}, we investigate the rate function, presenting Proposition~\ref{prop:properties of rate function} as our main result, which outlines several properties of this function.
In Subsection~\ref{sub:Pair structures}, we discuss pair structures related to tile structures induced by expanding Thurston maps. 
These are essential for constructing suitable subsystems discussed in Subsection~\ref{sub:Key bounds}. 
Subsection~\ref{sub:Key bounds} focuses on proving key bounds outlined in Proposition~\ref{prop:LDA:key bounds}.
This proof relies on the distortion estimates and thermodynamic formalism for subsystems we established in Sections~\ref{sec:Subsystems} and~\ref{sec:Thermodynamic formalism for subsystems}.
Finally, in Subsection~\ref{sub:Proof of the large deviation asymptotics}, we apply the key bounds from Proposition~\ref{prop:LDA:key bounds} to establish Theorem~\ref{thm:large deviation asymptotics}.

\subsection{The rate function}%
\label{sub:The rate function}

In this subsection, we are going to prove the following results about the rate function $I$ as  defined in \eqref{eq:definition of rate function}.
The proof will be given at the end of this subsection. 

\begin{proposition}    \label{prop:properties of rate function}
    Let $f: S^2 \mapping S^2$ be an expanding Thurston map and $d$ be a visual metric on $S^2$ for $f$. 
    Let $\phi \in C^{0,\holderexp}(S^2,d)$ be a real-valued \holder continuous function with an exponent $\holderexp \in (0,1]$ and not co-homologous to a constant in $C(S^2)$. 
    Let $\mu_{\phi}$ be the unique equilibrium state for the map $f$ and the potential $\phi$.
    Denote $\gamma_{\phi} \define \int \! \phi \,\mathrm{d}\mu_{\phi}$, $\minenergy \define \min\limits_{\mu\in \mathcal{M}(S^2,f)} \int \! \phi \,\mathrm{d}\mu$, and $\maxenergy \define \max\limits_{\mu\in \mathcal{M}(S^2,f)} \int \! \phi \,\mathrm{d}\mu$.
    Then the following statements hold:
    \begin{enumerate}[label=\rm{(\roman*)}]
        \item    \label{item:prop:properties of rate function:non-degenerate} 
            $\gamma_\phi \in (\minenergy,\maxenergy)$.
            In particular, the closed interval $\mathcal{I}_{\phi} \define [\minenergy,\maxenergy]$ cannot degenerate into a singleton set consisting of $\gamma_\phi$.

        \smallskip
        
        \item    \label{item:prop:properties of rate function:expression of rate function} 
            For $\alpha \in (\minenergy,\maxenergy)$,
            \[
                I(\alpha) = P(f, \phi) - P(f, (p')^{-1}(\alpha) \phi ) + ( (p')^{-1}(\alpha) - 1)\alpha,
            \]
            where the function $p \colon \real \mapping \real$ is defined by $p(t) \define P(f, t\phi)$.

        \smallskip

        \item     \label{item:prop:properties of rate function:second order differentiable strictly convex boundary behavior}
            The rate function $I \colon [\minenergy,\maxenergy] \mapping [0, +\infty)$ is twice differentiable and strictly convex on $(\minenergy, \maxenergy)$. Moreover, $I(\alpha) = 0$ if and only if $\alpha = \gamma_\phi$. Furthermore, $\lim\limits_{\alpha \to {\minenergy}^{+}} I'(\alpha) = -\infty$ and $\lim\limits_{\alpha \to {\maxenergy}^{-}} I'(\alpha) = +\infty$.            
    \end{enumerate}
\end{proposition} 

The graph of the rate function $I$ is shown in Figure~\ref{fig:rate function}.
\begin{figure}[H]
    \centering
    \begin{overpic}[width=7cm, tics=20]{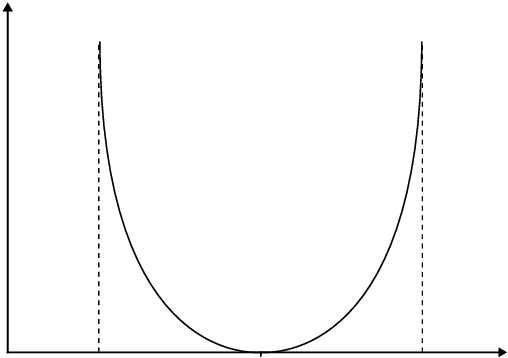}
        \put(16,-3){$\minenergy$}
        \put(49,-3){$\gamma_\phi$}
        \put(79.5,-3){$\maxenergy$}
        \put(103,-1){$\alpha$}
        \put(-12,68){$I(\alpha)$}
    \end{overpic}
    \vspace*{.2cm}
    \caption{The graph of the rate function $I$.}
    \label{fig:rate function}
\end{figure}

Let $f \colon S^2 \mapping S^2$ be an expanding Thurston map and $d$ be a visual metric on $S^2$ for $f$.
If $\phi$ is a \holder continuous function on $S^2$ with respect to the metric $d$, then there is a unique equilibrium state for $f$ and $\phi$ (recall Theorem~\ref{thm:properties of equilibrium state}~\ref{item:thm:properties of equilibrium state:existence and uniqueness}); we denote it by $\mu_{\phi}$. 

Two continuous functions $\varphi$ and $\psi$ are called \emph{co-homologous} in $C(S^2)$ (with respect to $f$) if there exists a continuous function $u \in C(S^2)$ such that $\varphi - \psi = u\circ f - u$.

For a continuous function $\varphi \in C(S^2)$, we define
\[
    \mathcal{I}_{\varphi} \define \set[\bigg]{\int \! \varphi \,\mathrm{d}\mu \describe \mu \in \invmea }.
\]
Note that $\mathcal{I}_{\varphi}$ is a connected closed subset of $\real$.
For each $\alpha \in \mathcal{I}_{\varphi}$, we define
\begin{equation}     \label{eq:H_alpha}
    H_{\varphi}(\alpha) \define \sup \set[\bigg]{ h_{\mu}(f) \describe \mu\in \mathcal{M}(S^2, f) \text{ with } \int \! \phi \,\mathrm{d}\mu = \alpha }.
\end{equation}

The following lemma is an analog of \cite[Lemma~3.3]{sharp2022statistics}. 
The proof is essentially the same, and we include it for the convenience of the reader.

\begin{lemma}    \label{lem:derivative of pressure function related to Birkhoff average}
    Let $f$, $d$, $\phi$, $\gamma_{\phi}$, $\minenergy$, $\maxenergy$ satisfy the Assumptions in Section~\ref{sec:The Assumptions}. 
    We assume in addition that $\phi$ is not co-homologous to a constant in $C(S^2)$.
    Then $\gamma_\phi \in \interior{\mathcal{I}_{\phi}} = (\minenergy,\maxenergy)$ and for each $\alpha \in (\minenergy, \maxenergy)$, there exists a unique number $\xi = \xi(\alpha) \in \real$ such that $H_{\potential}(\alpha) = h_{\mu_{\xi\phi}}(f)$ and 
    \begin{equation}    \label{eq:derivative of pressure function equals to Birkhoff average}
        \frac{\mathrm{d}P(f, t\phi)}{\mathrm{d}t} \bigg|_{t = \xi} = \int \! \phi \,\mathrm{d}\mu_{\xi \phi} = \alpha.   
    \end{equation}
    Here $\mu_{\xi\phi}$ is the equilibrium state for $f$ and $\xi\phi$. 
    Moreover, the supremum in \eqref{eq:H_alpha} is uniquely attained by $\mu_{\xi\phi}$.
\end{lemma}
\begin{proof}
    We write $p(t) \define P(f, t\phi)$ for $t \in \real$ for convenience.
    Then it follows from Theorem~\ref{thm:properties of equilibrium state}~\ref{item:thm:properties of equilibrium state:derivative} that $p'(t) = \int \! \potential \,\mathrm{d} \mu_{t \potential}$ for $t \in \real$.
    It is a standard fact that the function $p \colon \real \mapping \real$ is twice differentiable (the proof is verbatim the same as that of \cite[Theorem~5.7.4]{przytycki2010conformal}). 
    Moreover, since $\phi$ is not co-homologous to a constant, the function $p$ is strictly convex (see \cite[Theorem~2.11.3]{przytycki2010conformal} or \cite[Theorem~1.1]{das2021thermodynamic}).

    Now we consider the set \[
        \mathcal{D} \define \{p'(t)  \describe t \in \real \}  = \biggl\{ \int \! \phi \,\mathrm{d}\mu_{t \phi} \describe t \in \real \biggr\} \subseteq \mathcal{I}_{\phi}.
    \]
    Here $\mathcal{D}$ is an open interval since the function $p$ is twice differentiable and strictly convex. 
    Then we have $\gamma_\phi \in \interior{\mathcal{I}_{\phi}} = (\minenergy,\maxenergy)$ since $\gamma_{\phi} = p'(1) \in \mathcal{D}$.
    
    By the definition of pressure, for all $t \in \real$ and $\mu \in \mathcal{M}(S^2, f)$, $p(t) \geqslant h_{\mu}(f) + t\int \! \phi \,\mathrm{d}\mu$.
    In particular, the graph of the strictly convex function $p$ lies above a line with slope $\int \! \phi \,\mathrm{d}\mu$ (possibly touching it tangentially) so that $\int \! \phi \,\mathrm{d}\mu \in \overline{\mathcal{D}}$. 
    Since $\mu \in \mathcal{M}(S^{2}, f)$ is arbitrary, we have $\mathcal{I}_{\phi} \subseteq \overline{\mathcal{D}}$. 
    Then it follows that $\mathcal{D} = \interior{\mathcal{I}_{\phi}}$. 
    Note that the function $p' \colon \real \mapping \real$ is differentiable and strictly increasing.
    Thus, for each $\alpha\in \interior{\mathcal{I}_{\phi}} = p'(\real)$, there exists a unique number $\xi = \xi(\alpha) \in \real$ with 
    $\alpha = p'(\xi) = \int \! \phi \,\mathrm{d}\mu_{\xi \phi}$.
    Since $\mu_{\xi\phi}$ is the unique equilibrium state for $f$ and $\xi\phi$ (recall Theorem~\ref{thm:properties of equilibrium state}~\ref{item:thm:properties of equilibrium state:existence and uniqueness}), we have, for any $\mu \in \mathcal{M}(S^2, f)$ with $\mu \ne \mu_{\xi\phi}$, $h_{\mu_{\xi\phi}}(f) + \xi \int \! \phi \,\mathrm{d}\mu_{\xi\phi} > h_{\mu}(f) + \xi\int \! \phi \,\mathrm{d}\mu$.
    In particular, if $\int \! \phi \,\mathrm{d}\mu = \alpha$ then $h_{\mu_{\xi\phi}}(f) > h_{\mu}(f)$. This means that the supremum in \eqref{eq:H_alpha} is uniquely attained by $h_{\mu_{\xi\phi}}$. 
    Therefore, $\mu_{\xi\phi}$ is the unique measure with the desired properties.
\end{proof}

\begin{corollary}    \label{coro:properties of xi}
    Let $f$, $d$, $\phi$, $\gamma_{\phi}$, $\minenergy$, $\maxenergy$ satisfy the Assumptions in Section~\ref{sec:The Assumptions}. 
    We assume in addition that $\phi$ is not co-homologous to a constant in $C(S^2)$.
    Let $\xi \colon (\minenergy, \maxenergy) \mapping \real$ be the function defined via Lemma~\ref{lem:derivative of pressure function related to Birkhoff average}. 
    Then this function $\xi \colon (\minenergy, \maxenergy) \mapping \real$ is differentiable and strictly increasing. 
    Moreover, we have $\xi(\gamma_\phi) = 1$ and $\lim\limits_{\alpha \to {\minenergy}^{+}} \xi(\alpha) = -\infty$, $\lim\limits_{\alpha \to {\maxenergy}^{-}} \xi(\alpha) = +\infty$.
\end{corollary}
\begin{proof}
    We write $p(t) \define P(f, t \phi)$ for $t \in \real$. 
    Then the function $p \colon \real \mapping \real$ is twice differentiable and strictly convex.
    Note that it follows from Lemma~\ref{lem:derivative of pressure function related to Birkhoff average} that $\xi(\alpha) = (p')^{-1}(\alpha)$ for each $\alpha \in (\minenergy, \maxenergy)$.
    Therefore, the function $\xi \colon (\minenergy, \maxenergy) \mapping \real$ is differentiable and strictly increasing, and $\lim\limits_{\alpha \to {\minenergy}^{+}} \xi(\alpha) = -\infty$, $\lim\limits_{\alpha \to {\maxenergy}^{-}} \xi(\alpha) = +\infty$.

    For $\alpha = \gamma_\phi$, by \eqref{eq:H_alpha}, we have 
    \begin{align*}
        H_{\potential}(\gamma_\phi) 
        &= \sup \biggl\{ h_{\mu}(f) + \int \! \phi \,\mathrm{d}\mu \describe \mu\in \mathcal{M}(S^2,f), \, \int \! \phi \,\mathrm{d}\mu = \gamma_\phi \biggr\} - \gamma_\phi   \\
        &\leqslant P(f, \phi) - \gamma_\phi = h_{\mu_{\phi}}(f) \leqslant H_{\potential}(\gamma_\phi) = h_{\mu_{\xi(\gamma_\phi)\phi}}(f).
    \end{align*}
    Thus $h_{\mu_{\phi}}(f) = h_{\mu_{\xi(\gamma_\phi)\phi}}(f) = H_{\potential}(\gamma_\phi)$. 
    Therefore, it follows from Lemma~\ref{lem:derivative of pressure function related to Birkhoff average} that $\xi(\gamma_\phi) = 1$.
\end{proof}

Now we can prove Proposition~\ref{prop:properties of rate function}.

\begin{proof}[Proof of Proposition~\ref{prop:properties of rate function}]
    \ref{item:prop:properties of rate function:non-degenerate} 
    Statement~\ref{item:prop:properties of rate function:non-degenerate} follows immediately from Lemma~\ref{lem:derivative of pressure function related to Birkhoff average}.  

    \smallskip

    \ref{item:prop:properties of rate function:expression of rate function}
    Fix arbitrary $\alpha \in (\minenergy, \maxenergy)$. 
    By Lemma~\ref{lem:derivative of pressure function related to Birkhoff average} and \eqref{eq:H_alpha}, we have 
    \[
        H_{\potential}(\alpha) = h_{\mu_{\xi(\alpha)\phi}}(f) = P(f, \xi(\alpha)\phi) - \xi(\alpha)\alpha.
    \]
    Thus we get\[
        I(\alpha) = P(f, \phi) - \alpha -H_{\potential}(\alpha) = P(f, \phi) - P(f, \xi(\alpha)\phi) + (\xi(\alpha) - 1)\alpha
    \]
    by the definition of the rate function $I$ (see \eqref{eq:definition of rate function}).
    By Lemma~\ref{lem:derivative of pressure function related to Birkhoff average}, we have $\xi(\alpha) = (p')^{-1}(\alpha)$.
    Hence, statement~\ref{item:prop:properties of rate function:expression of rate function} holds.

    \smallskip

    \ref{item:prop:properties of rate function:second order differentiable strictly convex boundary behavior}  
    Let $\alpha \in (\minenergy, \maxenergy)$. 
    We write $p(t) \define P(f, t\phi)$ for $t \in \real$. 
    Then the function $p \colon \real \mapping \real$ is twice differentiable and strictly convex.
    By statement~\ref{item:prop:properties of rate function:expression of rate function}, we can write $I(\alpha)$ as
    \begin{equation}    \label{eq:temp:proof:prop:properties of rate function:rate function}
        I(\alpha) = p(1) - \alpha + \xi(\alpha)\alpha - p(\xi(\alpha)).
    \end{equation}
    Then it follows from Corollary~\ref{coro:properties of xi} that the rate function $I$ is differentiable on $(\minenergy, \maxenergy)$.

    Differentiating the rate function $I(\alpha)$ with respect to $\alpha$ and noting that $p'(\xi(\alpha)) = \alpha$ (Lemma~\ref{lem:derivative of pressure function related to Birkhoff average}), we obtain $I'(\alpha) = -1 + \xi'(\alpha)\alpha + \xi(\alpha) - p'(\xi(\alpha)) \xi'(\alpha) = \xi(\alpha) - 1$.
    Thus by Corollary~\ref{coro:properties of xi}, $I$ is twice differentiable and strictly convex since $I''(\alpha) = \xi'(\alpha) > 0$, and it follows from Corollary~\ref{coro:properties of xi} that $\lim\limits_{\alpha \to {\minenergy}^{+}} I'(\alpha) = -\infty$ and $\lim\limits_{\alpha \to {\maxenergy}^{-}} I'(\alpha) = +\infty$. Moreover, since $\xi(\gamma_{\phi}) = 1$ and $I'(\gamma_{\phi}) = 0$, by the strict convexity of $I$ and \eqref{eq:temp:proof:prop:properties of rate function:rate function}, $I(\alpha) = 0$ if and only if $\alpha = \gamma_\phi$.
\end{proof}

\subsection{Pair structures}%
\label{sub:Pair structures}
In this subsection, we discuss pair structures associated with the tile structures induced by an expanding Thurston map, which will be used to build appropriate subsystems in Subsection~\ref{sub:Key bounds}.

\begin{definition}[Pair structures]    \label{def:pair structures}
    Let $f$, $\mathcal{C}$, $e^0$ satisfy the Assumptions in Section~\ref{sec:The Assumptions}.
    For each $n\in \n$, we can pair a white $n$-tile $X^n_{\white} \in \mathbf{X}^n_{\white}$ and a black $n$-tile $X^n_{\black} \in \mathbf{X}^n_{\black}$ whose intersection $X^n_{\white} \cap X^n_{\black}$ contains an $n$-edge contained in $f^{-n}(e^0)$. 
    We call $X^n_{\white} \cup X^n_{\black}$ an \emph{$n$-pair} (with respect to $f$, $\mathcal{C}$, and $e^{0}$), and define the \emph{set of $n$-pairs} (with respect to $f$, $\mathcal{C}$, and $e^{0}$), denoted by $\mathbf{P}^n(f,\mathcal{C},e^0)$, to be
    \begin{equation}    \label{eq:definition of n-pairs}
        \mathbf{P}^n(f,\mathcal{C}, e^0) \define \bigl\{ X^n_{\white} \cup X^n_{\black} \describe X^n_{\white} \in \mathbf{X}^n_{\white}, \, X^n_{\black} \in \mathbf{X}^n_{\black}, \, X^n_{\white} \cap X^n_{\black} \cap f^{-n}\bigl(e^0\bigr) \in \mathbf{E}^n(f,\mathcal{C}) \bigr\}.
    \end{equation}
\end{definition}
Figure~\ref{fig:pair example} illustrates the structure of $n$-pairs. 
There are a total of $(\deg f)^n$ such pairs, and each $n$-tile is in precisely one such pair (see Lemma~\ref{lem:pairs are disjoint}).

\begin{figure}[H]
    \centering
    \vspace*{.5cm}
    \begin{overpic}[width=12cm, tics=5]{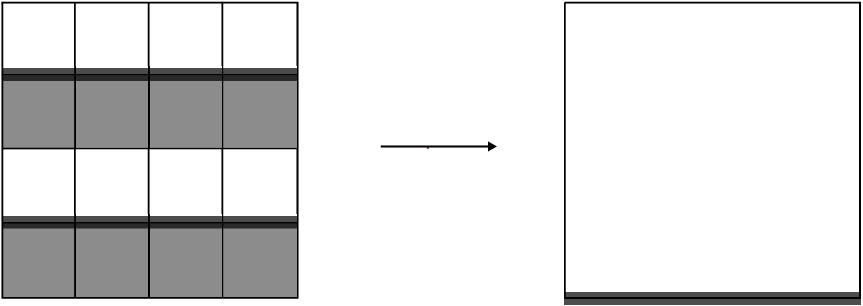}
        \put(13,38){$\mathbf{D}^n(f,\mathcal{C})$}
        \put(78,38){$\mathbf{D}^0(f,\mathcal{C})$}
        \put(49,22){$f^n$}
        \put(82,-3.5){$e^0$}
    \end{overpic}
    \vspace{.1cm}
    \caption{The graph of $n$-pairs.}
    \label{fig:pair example}
\end{figure} 

\begin{remark}\label{rem:basic property for n-pair}
    For each integer $n \in \n$, one sees that $f^n(P^n) = S^2$ for each $n$-pair $P^n \in \mathbf{P}^n(f,\mathcal{C}, e^0)$. This basic property is crucial for our constructions and proofs in this section (for example, see Proposition~\ref{prop:LDA:key bounds}). Indeed, this is the main reason why we define and use the pair structures.
\end{remark}

From now on, if the map $f$, the Jordan curve $\mathcal{C}$ and the $0$-edge $e^0$ are clear from the context, we will sometimes omit $(f,\mathcal{C}, e^0)$ in the notation above.

\begin{lemma}    \label{lem:pairs are disjoint}
    Let $f$, $\mathcal{C}$, $e^0$ satisfy the Assumptions in Section~\ref{sec:The Assumptions}. Then for each $n\in\n$ and any two distinct $n$-pairs $\juxtapose{P^n}{\widetilde{P}^n} \in \mathbf{P}^n$, their interiors are disjoint.
\end{lemma}
\begin{proof}
    We argue by contradiction and assume that there exist two distinct $n$-pairs $\juxtapose{P^n}{\widetilde{P}^n} \in \mathbf{P}^n$ for some $n\in\n$ such that their interiors intersect.
    Then it follows from Lemma~\ref{lem:intersection of cells}~\ref{item:lem:intersection of cells:intersection with union of tiles} that their intersection contains an $n$-tile $X^{n} \in \tile{n}$.
    Since $P^{n} \ne \widetilde{P}^n$, Remark~\ref{rem:intersection of two tiles} implies that $X^{n}$ contains two distinct $n$-edges $e^n, \widetilde{e}^n \in \Edge{n}$ satisfying $f^{n}(e^{n}) = f^{n}(\widetilde{e}^n) = e^{0}$, which contradicts Proposition~\ref{prop:properties cell decompositions}~\ref{item:prop:properties cell decompositions:cellular}.
\end{proof}

\begin{corollary}    \label{coro:union of pairs are the whole sphere}
    Let $f$, $\mathcal{C}$, $e^0$ satisfy the Assumptions in Section~\ref{sec:The Assumptions}. Then for each $n\in \n$, we have $\bigcup \mathbf{P}^n = S^2$ and $\card{\mathbf{P}^n} = (\deg{f})^n$.
\end{corollary}
\begin{proof}
    For each $n$-tile $X^n \in \mathbf{X}^n$, by Proposition~\ref{prop:properties cell decompositions}~\ref{item:prop:properties cell decompositions:cellular} and  the definition of $\mathbf{P}^n$, there exists an $n$-pair $P^n \in \mathbf{P}^n$ such that $X^n \subseteq P^n$. Thus $S^2 = \bigcup \mathbf{X}^n \subseteq \bigcup \mathbf{P}^n$ and we get $\bigcup\mathbf{P}^n = S^2$. By Lemma~\ref{lem:pairs are disjoint} and Proposition~\ref{prop:properties cell decompositions}~\ref{item:prop:properties cell decompositions:cardinality}, we have $\card{\mathbf{P}^n} = \card{\mathbf{X}^n}/2 = (\deg{f})^n$.
\end{proof}

\begin{lemma}    \label{lem:pair for iterate of f}
    Let $f$, $\mathcal{C}$, $e^0$ satisfy the Assumptions in Section~\ref{sec:The Assumptions}. 
    Then $\mathbf{P}^k(f^{n}, \mathcal{C}, e^0) = \mathbf{P}^{kn}(f, \mathcal{C}, e^0)$ for each $\juxtapose{n}{k} \in \n$.
\end{lemma}
\begin{proof}
    It follows immediately from Proposition~\ref{prop:properties cell decompositions}~\ref{item:prop:properties cell decompositions:iterate of cell decomposition} that $\mathbf{P}^k(f^{n}, \mathcal{C}, e^0) = \mathbf{P}^{kn}(f, \mathcal{C}, e^0)$ for $\juxtapose{n}{k} \in \n$.
\end{proof}

We formulate the next lemma to prove Lemma~\ref{lem:subsystem is strongly primitive for high level pair}. Recall $U^{n}(x)$ is the \emph{$n$-bouquet of $x$} (see \eqref{eq:Un bouquet of point}).
\begin{lemma}    \label{lem:pair neighbor disjoint with Jordan curve}
    Let $f$, $\mathcal{C}$, $d$, $e^0$ satisfy the Assumptions in Section~\ref{sec:The Assumptions}. We assume in addition that $f(\mathcal{C}) \subseteq \mathcal{C}$. 
    Then there exists an integer $M \in \n$ depending only on $f$, $\mathcal{C}$, $d$, and $e^0$ such that for each color $\colour \in \colours$, there exists an $M$-pair $P^{M}_{\colour} \in \mathbf{P}^{M}$ such that for each integer $n \geqslant M$ and each $x \in P^{M}_{\colour}$, we have $U^{n}(x) \subseteq \inte[\big]{X^0_{\colour}}$.
\end{lemma}
\begin{proof}
    We first show that there exists an integer $m\in \n$ satisfying the requirements, then let the integer $M$ be the smallest one among all such integers so that $M$ depends only on $f$, $\mathcal{C}$, $d$, and $e^0$. 

    For two fixed point $x_{\black} \in \inte[\big]{X^{0}_{\black}}$ and $x_{\white} \in \inte[\big]{X^{0}_{\white}}$, we can find a sufficiently small number $r > 0$ such that $B_{d}(x_{\colour}, r) \subseteq \inte[\big]{X^0_{\colour}}$ for each $\colour \in \colours$. 
    By Lemma~\ref{lem:visual_metric}~\ref{item:lem:visual_metric:bouquet bounded by ball} and~\ref{item:lem:visual_metric:ball bounded by bouquet}, there exists an integer $m \in \n$ such that $2K \Lambda^{-m} \leqslant r$ and $U^m(x_{\colour}) \subseteq B_{d}(x_{\colour}, r)$ for each $\colour \in \colours$, where $K$ is the constant from Lemma~\ref{lem:visual_metric}.
    By Corollary~\ref{coro:union of pairs are the whole sphere}, for each color $\colour \in \colours$ there exists an $m$-pair $P^m_{\colour} \in \mathbf{P}^m$ containing $x_{\colour}$. 
    Then for each $y \in P^m_{\colour} \subseteq U^m(x_{\colour})$ and each $z \in U^m(y)$, by Lemma~\ref{lem:visual_metric}~\ref{item:lem:visual_metric:bouquet bounded by ball}, we have\[
        d(z,x_{\colour}) \leqslant d(z, y) + d(y, x_{\colour}) < K\Lambda^{-m} + K\Lambda^{-m} = 2K\Lambda^{-m} \leqslant r.
    \]
    Since $f(\mathcal{C}) \subseteq \mathcal{C}$, this implies that for each integer $n \geqslant m$ and each $y \in P^{m}_{\colour}$, we have $U^n(y) \subseteq U^m(y) \subseteq B_{d}(x_{\colour}, r) \subseteq \inte[\big]{X^{0}_{\colour}}$.
\end{proof}

\subsection{Key bounds}
\label{sub:Key bounds}
The main goal in this subsection is to establish Proposition~\ref{prop:LDA:key bounds}, namely, an exponential upper bound for the measure of the set $P^n(\alpha)$ with respect to the equilibrium state $\mu_{\phi}$, where $P^n(\alpha)$ is defined in Definition~\ref{def:pair with respect to alpha}. Roughly speaking, we use pairs defined in Subsection~\ref{sub:Pair structures} to cover the set in \eqref{eq:large deviation asymptotics} and establish an upper bound for the measure of those tiles by the Gibbs property of $\mu_{\phi}$.

\begin{definition}    \label{def:pair with respect to alpha}
    Let $f$, $\mathcal{C}$, $d$, $\phi$, $e^0$, $\gamma_{\phi}$, $\minenergy$, $\maxenergy$ satisfy the Assumptions in Section~\ref{sec:The Assumptions}. 
    We assume in addition that $\phi$ is not co-homologous to a constant in $C(S^2)$.
    For each $\alpha \in (\minenergy, \maxenergy) \setminus \{\gamma_\phi\}$ and each $n\in \n$, we define
    \begin{equation}    \label{eq:definition P^n(alpha)}
        \mathbf{P}^n(\alpha) \define 
        \begin{cases}
            \set[\big]{P^n \in \mathbf{P}^n(f,\mathcal{C}, e^0) \describe \max_{x \in P^n} n^{-1} S_n \phi(x) \geqslant \alpha}   &\text{ if } \gamma_\phi < \alpha < \maxenergy;\\
            \set[\big]{P^n \in \mathbf{P}^n(f,\mathcal{C}, e^0) \describe \min_{x \in P^n} n^{-1} S_n \phi(x) \leqslant \alpha}   &\text{ if } \minenergy <  \alpha < \gamma_\phi,
        \end{cases}
    \end{equation}
    and $P^n(\alpha) \define \bigcup\mathbf{P}^n(\alpha)$.
\end{definition}

We remark that the sets $\mathbf{P}^n(\alpha)$ and $P^{n}(\alpha)$ defined above depend on the choice of $e^{0}$.

\begin{remark}\label{rem:pair with respect to alpha}
    For each $\alpha \in (\minenergy, \maxenergy) \setminus \{\gamma_\phi\}$ and each integer $n\in \n$, it is easy to see that $P^{n}(\alpha)$ is non-empty and is a union of some $n$-tiles in $\Tile{n}$. Thus by Definition~\ref{def:subsystems} and Proposition~\ref{prop:properties cell decompositions}~\ref{item:prop:properties cell decompositions:iterate of cell decomposition}, the map $f^n|_{P^n(\alpha)} \colon P^{n}(\alpha) \mapping S^{2}$ is a subsystem of $f^{n}$ with respect to $\mathcal{C}$.
    Recall that $\limitset(f^n|_{P^n(\alpha)}, \mathcal{C})$ is the tile maximal invariant set associated with $f^n|_{P^n(\alpha)}$ with respect to $\mathcal{C}$ (see Definition~\ref{def:subsystems}). 
    Denote $\limitset \define \limitset\bigl(f^n|_{P^n(\alpha)}, \mathcal{C}\bigr)$.
    Then by Proposition~\ref{prop:subsystem:properties}~\ref{item:subsystem:properties:limitset forward invariant}, we have $f^n(\limitset) \subseteq \limitset$.
\end{remark}

The next lemma shows that the subsystem $f^{n}|_{P^{n}(\alpha)}$ (with respect to $f^{n}$ and $\mathcal{C}$) is strongly primitive (see Definition~\ref{def:primitivity of subsystem}) for each $\alpha \in (\minenergy, \maxenergy) \setminus \{ \gamma_{\phi} \}$ and each sufficiently large integer $n \in \n$.

\begin{lemma}    \label{lem:subsystem is strongly primitive for high level pair}
    Let $f$, $\mathcal{C}$, $d$, $\phi$, $e^0$, $\gamma_{\phi}$, $\minenergy$, $\maxenergy$ satisfy the Assumptions in Section~\ref{sec:The Assumptions}. 
    We assume in addition that $f(\mathcal{C}) \subseteq \mathcal{C}$ and $\phi$ is not co-homologous to a constant in $C(S^2)$.
    Then for each $\alpha \in (\minenergy, \maxenergy) \setminus \{\gamma_\phi\}$, there exists an integer $N \in \n$ depending only on $f$, $\mathcal{C}$, $d$, $\phi$, $e^{0}$, and $\alpha$ such that for each integer $n \geqslant N$ and each color $\colour \in \colours$, there exists an $n$-pair $P^{n}_{\colour} \in \mathbf{P}^n(\alpha)$ such that $P^{n}_{\colour} \subseteq \inte[\big]{X^0_{\colour}}$. 
    In particular, the subsystem $f^{n}|_{P^{n}(\alpha)}$ (with respect to $f^{n}$ and $\mathcal{C}$) is strongly primitive.
\end{lemma}
\begin{proof}
    Let $\alpha \in (\minenergy, \maxenergy) \setminus \{\gamma_\phi\}$ be arbitrary. 
    Without loss of generality, we may assume that $\gamma_\phi < \alpha < \maxenergy$. 

    We first prove the following claim, which follows from the definition of $\maxenergy$ and the compactness of $\mathcal{M}(S^{2}, f)$ (equipped with the weak$^*$ topology).
    
    \smallskip
    
    \emph{Claim~1.} For each integer $m \in \n$, there exists $y \in S^2$ such that $\frac{1}{m}S_m \phi(y) \geqslant \maxenergy$.
    
    \smallskip

    To establish Claim~1, we argue by contradiction and assume that there exists an integer $m\in \n$ such that $\frac{1}{m}S_m\phi(y) < \maxenergy$ for all $y \in S^2$. Since $\phi$ is continuous and $S^2$ is compact, there exists a number $\delta > 0$ such that $\frac{1}{m} S_m\phi(y) < \maxenergy - \delta$ for all $y \in S^2$. Then for each $\mu \in \mathcal{M}(S^2, f)$ we have \[
        \int \! \frac{1}{m}S_m\phi \,\mathrm{d}\mu \leqslant \int \! (\maxenergy - \delta) \,\mathrm{d}\mu = \maxenergy - \delta.
    \]
    
    Note that $\mathcal{M}(S^2,f)$ is compact in the weak* topology and the continuous map $\mu \mapsto \int \! \phi \,\mathrm{d}\mu$ maps $\mathcal{M}(S^2,f)$ onto the closed interval $[\minenergy, \maxenergy]$. Hence there exists an $f$-invariant measure $\nu\in \mathcal{M}(S^2, f)$ such that $\int \! \phi \,\mathrm{d}\nu = \maxenergy$. Since
    \begin{align*}
        \int \! \frac{1}{m}S_m\phi \,\mathrm{d}\nu 
        &= \frac{1}{m}\int \! \sum_{j=0}^{m-1} \phi\circ f^j \,\mathrm{d}\nu 
        = \frac{1}{m} \sum_{j=0}^{m-1} \int \! \phi \,\mathrm{d} f_{*}^{j} \nu 
        = \frac{1}{m}\cdot m \int \! \phi \,\mathrm{d}\nu = \int \! \phi \,\mathrm{d}\nu,
    \end{align*}
    we get $\maxenergy = \int \! \phi \,\mathrm{d}\nu = \int \! \frac{1}{m}S_m\phi \,\mathrm{d}\nu \leqslant \maxenergy - \delta$, which is a contradiction, and so Claim~1 follows.

    \smallskip

    By Lemma~\ref{lem:pair neighbor disjoint with Jordan curve}, there exists an integer $M \in \n$ depending only on $f$, $\mathcal{C}$, $d$, and $e^0$ such that for each color $\colour \in \colours$, there exists an $M$-pair $P^{M}_{\colour} \in \mathbf{P}^{M}$ such that for each integer $n \geqslant M$ and each $x \in P^{M}_{\colour}$, we have $U^{n}(x) \subseteq \inte[\big]{X^0_{\colour}}$. 

    We fix such an integer $M$ and the corresponding $M$-pairs $P^{M}_{\black}$ and $P^{M}_{\white}$ in the following.

    Set \begin{equation}    \label{eq:temp definition of N}
        N \define \left\lfloor M \cdot \frac{\maxenergy - \inf_{x \in S^2} \phi(x)}{\maxenergy - \alpha} \right\rfloor + 1.
    \end{equation}
    Note that the integer $N$ depends only on $f$, $\mathcal{C}$, $d$, $\phi$, $e^0$, and $\alpha$. Moreover, we have $N \geqslant M + 1$ since $\inf_{x \in S^2}\phi(x) \leqslant \minenergy < \alpha$.

    \smallskip
    
    \emph{Claim~2.} For each integer $n \geqslant N$ and each color $\colour \in \colours$, there exists an $n$-pair $P^n_{\colour} \in \mathbf{P}^n(\alpha)$ such that $P^n_{\colour} \subseteq \inte[\big]{X^0_{\colour}}$.
    
    \smallskip

    To establish Claim~2, let integer $n \geqslant N$ and color $\colour \in \colours$ be arbitrary. Since $N \geqslant M + 1$, we have $n - M \in \n$. By Claim~1 there exists $y \in S^2$ such that $\frac{1}{n - M}S_{n - M} \phi(y) \geqslant \maxenergy$.
    Then there exists $x_{\colour} \in P^M_{\colour}$ such that $f^M(x_{\colour}) = y$ since $f^{M}(P^M_{\colour}) = S^2$. 
    Thus we have
    \begin{align*}
        \frac{1}{n}S_n\phi(x_{\colour})
        &= \frac{1}{n} S_M\phi(x_{\colour}) + \frac{1}{n} S_{n-M} \phi(y) \\
        &\geqslant \frac{M}{n} \inf_{z \in S^2}\phi(z) + \frac{n-M}{n} \maxenergy \\
        &= \maxenergy - \frac{M}{n} \bigl(\maxenergy - \inf_{z \in S^2}\phi(z) \bigr)   \\
        &\geqslant  \maxenergy - \frac{M}{N} \bigl(\maxenergy - \inf_{z \in S^2}\phi(z) \bigr) \\
        &\geqslant  \maxenergy - (\maxenergy - \alpha) \\
        &= \alpha,
    \end{align*}
    where the last inequality follows from the definition of $N$ (see \eqref{eq:temp definition of N}) and the fact that $\lfloor t\rfloor + 1 \geqslant t$ for all $t \in \real$. By Corollary~\ref{coro:union of pairs are the whole sphere}, there exists an $n$-pair $P^n_{\colour} \in \mathbf{P}^n$ containing $x_{\colour}$. Thus we have $x_{\colour} \in P^n_{\colour} \in \mathbf{P}^n(\alpha)$ since $\frac{1}{n}S_n\phi(x_{\colour}) \geqslant \alpha$. Noting that $x_{\colour} \in P^M_{\colour}$ and $n \geqslant M$, by Lemma~\ref{lem:pair neighbor disjoint with Jordan curve}, we get $U^n(x_{\colour}) \subseteq \inte[\big]{X^0_{\colour}}$. 
    Then it follows from the definition of $U^{n}(x_{\colour})$ and $P^{n}_{\colour}$ that $x_{\colour} \in P^n_{\colour} \subseteq U^{n}(x_{\colour}) \subseteq \inte[\big]{X^0_{\colour}}$, and so Claim~2 follows.

    \smallskip
    
    By Claim~2, it follows immediately from Definition~\ref{def:primitivity of subsystem} and Definition~\ref{def:pair structures} that the subsystem $f^{n}|_{P^{n}(\alpha)}$ (with respect to $f^{n}$ and $\mathcal{C}$) is strongly primitive for each integer $n \geqslant N$.
\end{proof}

We record this result below for the convenience of the reader.

\begin{proposition}[Z.~Li \cite{li2018equilibrium}]    \label{prop:equilibrium state is gibbs measure}
    Let $f$, $\mathcal{C}$, $d$, $\phi$, $\mu_{\potential}$ satisfy the Assumptions in Section~\ref{sec:The Assumptions}.
    Then $\mu_{\phi}$ is a Gibbs measure with respect to $f$, $\mathcal{C}$, and $\phi$, with the constant $P_{\mu_{\phi}} = P(f, \phi)$, i.e., there exists a constant $C_{\mu_{\phi}} \geqslant 1$ such that for each $n\in \n_0$, each $n$-tile $X^n \in \mathbf{X}^n(f,\mathcal{C})$, and each $x \in X^n$, we have
    \begin{equation}    \label{eq:equilibrium state is gibbs measure}
        \frac{1}{C_{\mu_{\phi}}} \leqslant \frac{\mu_{\phi}(X^n)}{ \myexp{ S_{n}\phi(x) - n P(f,\phi) } } \leqslant C_{\mu_{\phi}}.
    \end{equation}
\end{proposition}

The following lemma is an immediate consequence of Lemma~\ref{lem:distortion_lemma} and Proposition~\ref{prop:equilibrium state is gibbs measure}.

\begin{lemma}    \label{lem:equilibrium measure of pair and inf of Sn_phi over P^n(alpha)}
    \def\tempconst{D(\potential)}
    Let $f$, $\mathcal{C}$, $d$, $\phi$, $\holderexp$, $\mu_{\phi}$, $e^0$ satisfy the Assumptions in Section~\ref{sec:The Assumptions}.
    Denote $\tempconst \define \Cdistortion$ and $C \define 2C_{\mu_{\phi}} e^{ \tempconst }$, where $C_{\mu_{\phi}} \geqslant 1$ is from Proposition~\ref{prop:equilibrium state is gibbs measure} and $C_1 \geqslant 0$ is given by \eqref{eq:const:C_1} in Lemma~\ref{lem:distortion_lemma}.
    Then the following statements hold:
    \begin{enumerate}[label=\rm{(\roman*)}]
    \smallskip 

    \item     \label{item:lem:equilibrium measure of pair and inf of Sn_phi over P^n(alpha):upper bound for measure of pair}
        For each integer $n \in \n$ and each $n$-pair $P^n \in \mathbf{P}^n$, we have
        \[
            \mu_{\phi}(P^n) \leqslant C e^{-P(f, \phi)n} \inf_{x\in P^n} e^{S_n \phi(x)}.
        \]

    \smallskip
    
    \item     \label{item:lem:equilibrium measure of pair and inf of Sn_phi over P^n(alpha):bound for potential}
        We assume in addition that $\phi$ is not co-homologous to a constant in $C(S^2)$.
        Then for each $n \in \n$ and each $\alpha \in (\gamma_\phi, \maxenergy)$, we have\[
            \inf_{x \in P^n(\alpha)} S_n \phi(x) \geqslant  n\alpha - 2 \tempconst.
        \]
        Similarly, for each $n \in \n$ and each $\alpha \in (\minenergy, \gamma_\phi)$, we have $\sup_{x \in P^n(\alpha)} S_n \phi(x) \leqslant n\alpha + 2 \tempconst$.
    \end{enumerate}
\end{lemma}
\begin{proof}
    \def\tempconst{D(\potential)}
    For each $n\in \n$ and each $n$-pair $P^n = X^n_{\black} \cup X^n_{\white} \in \mathbf{P}^n$, denote $e^n \define X^n_{\black} \cap X^n_{\white}$ and fix an arbitrary point $x_{e} \in e^n$.
    
    \ref{item:lem:equilibrium measure of pair and inf of Sn_phi over P^n(alpha):upper bound for measure of pair}
    It follows from Lemma~\ref{lem:distortion_lemma} that
    \begin{equation}    \label{eq:temp:inf Sn_phi over Pn}
        \inf_{x \in P^n} S_n\phi(x) = \min \Bigl\{ \inf_{x \in X^n_{\black}} S_n\phi(x), \inf_{x \in X^n_{\white}} S_n\phi(x) \Bigr\} \geqslant S_n\phi(x_e) - \tempconst.
    \end{equation}
    Then by Proposition~\ref{prop:equilibrium state is gibbs measure}, we deduce that
    \[
        \mu_{\phi}(P^n) \leqslant \mu_{\phi}(X^n_{\black}) + \mu_{\phi}(X^n_{\white})  \leqslant  2 C_{\mu_{\phi}} e^{-P(f,\phi)n} e^{S_n\phi(x_e)} 
        \leqslant C e^{-P(f, \phi)n} \inf_{x\in P^n} e^{S_n \phi(x)}.
    \]

    \smallskip

    \ref{item:lem:equilibrium measure of pair and inf of Sn_phi over P^n(alpha):bound for potential} 
    Consider $\alpha \in (\gamma_\phi, \maxenergy)$.
    For each $P^n = X^n_{\black} \cup X^n_{\white} \in \mathbf{P}^n(\alpha)$, there exists $x_{0} \in P^n$ such that $S_n\phi(x_{0}) \geqslant n \alpha$ (recall \eqref{eq:definition P^n(alpha)}).
    Since $x_e \in X^n_{\black} \cap X^n_{\white}$, by Lemma~\ref{lem:distortion_lemma}, $S_n\phi(x_e) \geqslant S_n \phi(x_{0}) - \tempconst$.
    Thus by~\eqref{eq:temp:inf Sn_phi over Pn}, we obtain
    \begin{equation*}
        \inf_{x \in P^n(\alpha)} S_n\phi(x) = \min_{P^n \in \mathbf{P}^n(\alpha)} \bigl\{ \inf_{x \in P^n} S_n\phi(x) \bigr\} 
        \geqslant n\alpha -2 \tempconst.
    \end{equation*}
    For $\alpha \in (\minenergy, \gamma_\phi)$, the proof is similar.
\end{proof}

Now we are ready to prove the main results of this section.
\begin{proposition}[Key bounds]    \label{prop:LDA:key bounds}
    \def\tempconst{D(\potential)}
    Let $f \colon S^2 \mapping S^2$ be an expanding Thurston map and $\mathcal{C} \subseteq S^2$ be a Jordan curve containing $\post{f}$ with the property that $f(\mathcal{C}) \subseteq \mathcal{C}$. 
    Let $d$ be a visual metric on $S^2$ for $f$.
    Let $\phi \in C^{0,\holderexp}(S^2,d)$ be a real-valued \holder continuous function with an exponent $\holderexp \in (0,1]$ and not co-homologous to a constant in $C(S^2)$. 
    Let $\mu_{\phi}$ be the unique equilibrium state for the map $f$ and the potential $\phi$. 
    Denote $\tempconst \define \Cdistortion$ and $C \define 2C_{\mu_{\phi}} e^{ \tempconst }$, where $C_{\mu_{\phi}} \geqslant 1$ and $C_1 \geqslant 0$ are the constants from Proposition~\ref{prop:equilibrium state is gibbs measure} and \eqref{eq:const:C_1} in Lemma~\ref{lem:distortion_lemma}, respectively.
    Fix a $0$-edge $e^0 \in \mathbf{E}^0(f, \mathcal{C})$ and denote $\gamma_{\phi} \define \int \! \phi \,\mathrm{d}\mu_{\phi}$, $\minenergy \define \min\limits_{\mu\in \mathcal{M}(S^2,f)} \int \! \phi \,\mathrm{d}\mu$, and $\maxenergy \define \max\limits_{\mu\in \mathcal{M}(S^2,f)} \int \! \phi \,\mathrm{d}\mu$.
    Then the following statements hold:
    \begin{enumerate}[label=\emph{(\roman*)}]
        \smallskip

        \item     \label{item:prop:LDA:key bounds:max}
        For each $\alpha\in (\gamma_\phi,\maxenergy)$, there exists an integer $N \in \n$ depending only on $f$, $\mathcal{C}$, $d$, $\phi$, $e^0$, and $\alpha$ such that for each integer $n \geqslant N$, there exists a measure $\mu \in \mathcal{M}(S^2, f)$ such that 
        \[
            \mu_{\phi} (P^n(\alpha)) \leqslant C e^{ n(P_{\mu}(f, \phi) - P(f,\phi)) } \quad \text{ and }\quad 
            \int \! \phi \,\mathrm{d}\mu \in \biggl[ \alpha - \frac{2 \tempconst }{n}, \maxenergy \biggr].
        \] 

        \item     \label{item:prop:LDA:key bounds:min}
        For each $\alpha\in (\minenergy,\gamma_\phi)$, there exists an integer $N \in \n$ depending only on $f$, $\mathcal{C}$, $d$, $\phi$, $e^0$, and $\alpha$ such that for each integer $n \geqslant N$, there exists a measure $\mu \in \mathcal{M}(S^2, f)$ such that 
        \[
        \mu_{\phi} (P^n(\alpha)) \leqslant C e^{ n(P_{\mu}(f, \phi) - P(f,\phi)) } \quad \text{ and }\quad \int \! \phi \,\mathrm{d}\mu \in \left[\minenergy, \alpha + \frac{2 \tempconst }{n} \right].
        \] 
    \end{enumerate}
\end{proposition}
\begin{proof}
    \def\tempconst{D(\potential)}
    We first consider the case where $\alpha \in (\gamma_{\phi}, \maxenergy)$. Let $\alpha \in (\gamma_{\phi}, \maxenergy)$ be arbitrary.
    
    By Lemma~\ref{lem:subsystem is strongly primitive for high level pair}, there exists an integer $N \in \n$ depending only on $f$, $\mathcal{C}$, $d$, $\phi$, $e^0$, and $\alpha$ such that for each integer $n \geqslant N$, the subsystem $f^{n}|_{P^{n}(\alpha)}$ (with respect to $f^{n}$ and $\mathcal{C}$) is strongly primitive.
    Let integer $n \geqslant N$ be arbitrary. Denote $\limitset \define \limitset(f^n|_{P^n(\alpha)}, \mathcal{C})$. 
    Then it follows from Propositions~\ref{prop:sursubsystem properties}~\ref{item:prop:sursubsystem properties:property of domain and limitset} and~\ref{prop:irreducible subsystem properties}~\ref{item:prop:irreducible subsystem properties:limitset non degenerate to Jordan curve} that $f^{n}(\limitset) = \limitset$ and $\limitset \setminus \mathcal{C} \ne \emptyset$.

    Let $y_{0} \in \limitset \setminus \mathcal{C}$ be arbitrary. By Proposition~\ref{prop:subsystem preimage pressure} and Theorem~\ref{thm:subsystem characterization of pressure}, we have
    \begin{equation}     \label{eq:temp:prop:LDA:key bounds:Variational Principle}
    \begin{split}
        \sup_{\nu \in \mathcal{M}(\limitset, f^n|_{\limitset})}  \Bigl\{ h_{\nu}(f^n|_{\limitset}) + \int \! S_n\phi \,\mathrm{d}\nu \Bigr\} 
        = \lim_{m \to +\infty} \frac{1}{m} \log \!\!\! \sum_{x\in (f^n|_{\limitset})^{-m}(y_0)} \!\!\! e^{\sum_{k=0}^{m-1}S_{n}\phi(f^{nk}(x))}.
    \end{split}
    \end{equation}
    Here $\mathcal{M}(\limitset, f^n|_{\limitset})$ denotes the set of $f^n|_{\limitset}$-invariant Borel probability measures on $\limitset$ endowed with the weak$^*$ topology, and $h_{\nu}(f^n|_{\limitset})$ denotes the measure-theoretic entropy of $f^n|_{\limitset}$ for $\nu$. For the summand inside the logarithm in \eqref{eq:temp:prop:LDA:key bounds:Variational Principle}, we have
    \begin{equation}     \label{eq:temp:prop:LDA:key bounds:summation of preimages}
        \sum_{x\in (f^n|_{\limitset})^{-m}(y_0)} \!\!\! e^{\sum_{k = 0}^{m - 1}S_{n}\phi( f^{nk}(x) )}
                =\prod_{i=0}^{m-1} \sum_{y_{i+1}\in (f^n|_{\limitset})^{-1}(y_i)} \!\!\!\!\! e^{S_n\phi(y_{i+1})} .
    \end{equation}

    \smallskip
    
    \emph{Claim.} For each point $y \in \limitset \setminus \mathcal{C}$, we have $\card{(f^n|_{\limitset})^{-1}(y)} = \card{\mathbf{P}^n(\alpha)}$, and each $n$-pair $P^n \in \mathbf{P}^n(\alpha)$ contains exactly one preimage $x \in (f^n|_{\limitset})^{-1}(y)$, which satisfies $x \in \limitset \setminus \mathcal{C}$.
    
    \smallskip

    To establish this Claim, we consider an arbitrary point $y\in \limitset \setminus \mathcal{C}$. Without loss of generality we may assume that $y \in \inte[\big]{X^0_{\black}}$. Then by Proposition~\ref{prop:properties cell decompositions}, we have $\card{f^{-n}(y)} = (\deg f)^n = \card{\mathbf{X}^n_{\black}}$, and each black $n$-tile $X^n_{\black} \in \mathbf{X}^n_{\black}$ contains exactly one preimage $x \in f^{-n}(y)$, which satisfies $x \in \inte{X^n_{\black}}$. Thus each $n$-pair $P^n \in \mathbf{P}^n(\alpha)$ contains exactly one preimage $x \in f^{-n}(y) \cap P^n(\alpha)$, which satisfies $x \in \inte{P^n}$, and we have\[
        \card{f^{-n}(y) \cap P^n(\alpha)} = \card{\mathbf{P}^n(\alpha)}.
    \]
    Let preimage $x \in f^{-n}(y) \cap P^n(\alpha)$ be arbitrary. Noting that $f^{-n}(y) \cap P^n(\alpha) = (f^{n}|_{P^{n}(\alpha)})^{-1}(y)$ and $y \in \limitset \setminus \mathcal{C}$, by Proposition~\ref{prop:subsystem:properties invariant Jordan curve}~\ref{item:subsystem:properties invariant Jordan curve:backward invariant limitset outside invariant Jordan curve}, we have $x \in \limitset \setminus \mathcal{C}$. 
    Since $(f^n|_{\limitset})^{-1}(y) = f^{-n}(y) \cap \limitset = f^{-n}(y) \cap P^n(\alpha)$, the claim follows.

    \smallskip

    By the claim, we know that all the preimages $y_i$ in the summation in \eqref{eq:temp:prop:LDA:key bounds:summation of preimages} belong to $\limitset \setminus \mathcal{C}$. 
    Moreover, for each point $y \in \limitset \setminus \mathcal{C}$, every $n$-pair $P^n \in \mathbf{P}^n(\alpha)$ contains exactly one preimage $x \in (f^n|_{\limitset})^{-1}(y)$, and every preimage $x \in (f^n|_{\limitset})^{-1}(y)$ is contained in a unique $n$-pair $P^n \in \mathbf{P}^n(\alpha)$. 
    Thus, we get the first two inequalities of the following:
    \begin{align*}
        \sum_{x\in (f^n|_{\limitset})^{-m}(y_0)} \! e^{\sum_{k = 0}^{m - 1}S_{n}\phi( f^{nk}(x) )}
        &=\prod_{i=0}^{m-1} \sum_{y_{i+1}\in (f^n|_{\limitset})^{-1}(y_i)} \!\!\!\!\! e^{S_n\phi(y_{i+1})} \\
        &\geqslant \Bigl( \inf_{y \in \limitset \setminus \mathcal{C}} \sum_{x\in (f^n|_{\limitset})^{-1}(y)} e^{S_n\phi(x)} \Bigr)^{m} \\
        &\geqslant \Bigl( \sum_{P^n \in \mathbf{P}^{n}(\alpha)} \inf_{x \in P^n} e^{S_n\phi(x)} \Bigr)^m \\
        &\geqslant \parentheses[\Big]{  C^{-1} e^{nP(f,\phi)} \sum_{P^n \in \mathbf{P}^n(\alpha)} \mu_{\phi}(P^n) }^{m}  \\
        &= \parentheses[\big]{ C^{-1} e^{nP(f,\phi)} \mu_{\phi} (P^n(\alpha)) }^{m}.
    \end{align*}
    The last inequality follows from Lemma~\ref{lem:equilibrium measure of pair and inf of Sn_phi over P^n(alpha)}~\ref{item:lem:equilibrium measure of pair and inf of Sn_phi over P^n(alpha):upper bound for measure of pair} and the last equality follows from Lemma~\ref{lem:pairs are disjoint} and Theorem~\ref{thm:properties of equilibrium state}~\ref{item:thm:properties of equilibrium state:edge measure zero}.
    Taking logarithms of both sides, dividing by $m$, and plugging the result into the previous inequality, we get
    \begin{equation*}
        \lim_{m \to +\infty} \frac{1}{m} \log \parentheses[\bigg]{ \sum_{x\in (f^n|_{\limitset})^{-m}(y_0)}  \myexp[\Big]{ \sum_{k=0}^{m-1}S_{n}\phi\bigl( f^{nk}(x) \bigr) } }  
        \geqslant \log \parentheses[\big]{ \mu_{\phi}(P^n(\alpha)) }  + nP(f, \phi) - \log C.
    \end{equation*}
    Plugging this inequality into \eqref{eq:temp:prop:LDA:key bounds:Variational Principle} yields
    \begin{equation}    \label{eq:temp:prop:LDA:key bounds:supremum measure-theoretic pressure inequality}
        \sup_{\nu \in \mathcal{M}(\limitset, f^n|_{\limitset})} \Bigl\{ h_{\nu}(f^n|_{\limitset}) + \int \! S_n\phi \,\mathrm{d}\nu \Bigr\}  
        \geqslant \log \parentheses[\big]{ \mu_{\phi}(P^n(\alpha)) } + nP(f, \phi) - \log C.
    \end{equation}
    
    By Theorem~\ref{thm:existence of equilibrium state for subsystem} and Proposition~\ref{prop:subsystem invariant measure equivalence}~\ref{item:prop:subsystem invariant measure equivalence:extension}, there exists an equilibrium state $\widehat{\mu} \in \mathcal{M}(\limitset, f^n|_{\limitset}) \subseteq \mathcal{M}(S^{2}, f^{n})$ that attains the supremum in \eqref{eq:temp:prop:LDA:key bounds:supremum measure-theoretic pressure inequality}. 
    Denote $\mu \define \frac{1}{n} \sum_{i = 0}^{n - 1} f_{*}^{i}  \widehat{\mu}$. 
    Then $\mu \in \mathcal{M}(S^2, f)$ and we have\[
        \int \! \phi \,\mathrm{d}\mu 
        = \frac{1}{n} \int \sum_{i = 0}^{n - 1} \phi \,\mathrm{d} f_{*}^{i} \widehat{\mu}
        = \frac{1}{n} \int \sum_{i = 0}^{n - 1} \phi \circ f^{i} \,\mathrm{d}\widehat{\mu} 
        = \frac{1}{n} \int \! S_n\phi \,\mathrm{d}\widehat{\mu}.
    \]
    By Lemma~\ref{lem:equilibrium measure of pair and inf of Sn_phi over P^n(alpha)}~\ref{item:lem:equilibrium measure of pair and inf of Sn_phi over P^n(alpha):bound for potential}, we have $\inf_{x \in P^n(\alpha)} S_n \phi(x) \geqslant  n\alpha - 2 \tempconst$.
    Noting that $\supp{\widehat{\mu}} \subseteq \limitset \subseteq P^n(\alpha)$, we have\[
        \int \! \phi \,\mathrm{d}\mu = \frac{1}{n}\int \! S_n\phi \,\mathrm{d}\widehat{\mu} \geqslant \alpha - 2 \tempconst n^{-1}.
    \]
    Thus the measure $\mu$ satisfies $\int \! \phi \,\mathrm{d}\mu \in \left[\alpha - 2 \tempconst n^{-1}, \maxenergy \right]$.

    By \eqref{eq:measure-theoretic entropy well-behaved under iteration} and \eqref{eq:measure-theoretic entropy is affine}, we have
    \begin{equation}    \label{eq:temp:entropy equality}
        n h_{\mu}(f) = h_{\mu}(f^n) = \frac{1}{n} \sum_{i = 0}^{n - 1} h_{f_{*}^{i} \widehat{\mu}}(f^n).
    \end{equation}
    We now show that $h_{f_{*}^{i} \widehat{\mu}}(f^n) = h_{\widehat{\mu}}(f^{n})$ for each $i \in \{0, \, 1 \, \dots, \, n-1\}$. 
    Indeed, the measure $f_{*} \widehat{\mu}$ is $f^{n}$-invariant and the triple $(S^{2}, f^{n}, f_{*}\widehat{\mu})$ is a factor of $(S^{2}, f^{n}, \widehat{\mu})$ by the map $f$. 
    It follows that $h_{f_{*}\widehat{\mu}}(f^{n}) \leqslant h_{\widehat{\mu}}(f^{n})$ (see for example, \cite[Proposition~4.3.16]{katok1995introduction}). 
    Iterating this and noting that $f_{*}^{n} \widehat{\mu} = (f^{n})_{*}\widehat{\mu} = \widehat{\mu}$ by $f^{n}$-invariance of $\widehat{\mu}$, we obtain\[
        h_{\widehat{\mu}}(f^{n}) = h_{f_{*}^{n} \widehat{\mu}}(f^{n}) 
        \leqslant h_{f_{*}^{n - 1} \widehat{\mu}}(f^{n}) \leqslant \cdots 
        \leqslant h_{f_{*} \widehat{\mu}}(f^{n}) \leqslant h_{\widehat{\mu}}(f^{n}).
    \]
    Hence $h_{f_{*}^{i} \widehat{\mu}}(f^n) = h_{\widehat{\mu}}(f^{n})$ for each $i \in \{0, \, 1 \, \dots, \, n-1\}$. Combining this with \eqref{eq:temp:entropy equality}, we obtain \[
        n h_{\mu}(f) = h_{\widehat{\mu}}(f^{n}) =  h_{\widehat{\mu}}(f^n|_{\limitset}).
    \]
    Thus
    \begin{equation*}
        n \parentheses[\bigg]{ h_{\mu}(f) + \int \! \phi \,\mathrm{d}\mu }
        = h_{\widehat{\mu}}(f^n|_{\limitset}) + \int \! S_n\phi \,\mathrm{d}\widehat{\mu}
        \geqslant \log \mathopen{}\bigl(\mu_{\phi}(P^n(\alpha))\bigr) + nP(f, \phi) - \log C,
    \end{equation*}
    i.e., $\log (\mu_{\phi}(P^n(\alpha))) \leqslant n (P_{\mu}(f, \phi) - P(f, \phi) ) + \log C$.
    This completes the proof of Proposition~\ref{prop:LDA:key bounds}~\ref{item:prop:LDA:key bounds:max}.

    For $\alpha\in (\minenergy,\gamma_\phi)$, the proof is similar.
\end{proof}

\subsection{Proof of the large deviation asymptotics}
\label{sub:Proof of the large deviation asymptotics}
In this subsection, we establish large deviation asymptotics for expanding Thurston maps.
More precisely, we first prove the results under the assumption that there exists an $f$-invariant Jordan curve $\mathcal{C}$ with $\post{f} \subseteq \mathcal{C}$.
Then by Lemma~\ref{lem:invariant_Jordan_curve}, we remove this assumption and prove Theorem~\ref{thm:large deviation asymptotics}.

\begin{proposition}    \label{prop:large deviation asymptotics with invariant Jordan curve}
    Let $f \colon S^2 \mapping S^2$ be an expanding Thurston map and $\mathcal{C} \subseteq S^2$ be a Jordan curve containing $\post{f}$ with the property that $f(\mathcal{C}) \subseteq \mathcal{C}$. 
    Let $d$ be a visual metric on $S^2$ for $f$. 
    Let $\phi \in C^{0,\holderexp}(S^2,d)$ be a real-valued \holder continuous function with an exponent $\holderexp \in (0,1]$ and not co-homologous to a constant in $C(S^2)$. 
    Let $\mu_{\phi}$ be the unique equilibrium state for the map $f$ and the potential $\phi$. 
    Denote $\gamma_{\phi} \define \int \! \phi \,\mathrm{d}\mu_{\phi}$, $\minenergy \define \min\limits_{\mu\in \mathcal{M}(S^2,f)} \int \! \phi \,\mathrm{d}\mu$, and $\maxenergy \define \max\limits_{\mu\in \mathcal{M}(S^2,f)} \int \! \phi \,\mathrm{d}\mu$.
    Then for each $\alpha \in (\minenergy, \maxenergy)$, there exists an integer $N \in \n$ depending only on $f$, $\mathcal{C}$, $d$, $\phi$, and $\alpha$ such that for each integer $n \geqslant N$, \[
        \mu_{\phi}\Bigl(\Bigl\{ x \in S^2 \describe \sgn{\alpha - \gamma_{\phi}} \frac{1}{n}S_n\phi(x) \geqslant \sgn{\alpha - \gamma_{\phi}} \alpha \Bigr\} \Bigr)
        \leqslant C_{\alpha} e^{-I(\alpha)n},
    \]
    where $C_{\alpha} > 0$ is a constant depending only on $f$, $\mathcal{C}$, $d$, $\phi$, $\holderexp$, and $\alpha$. 
    Quantitatively, we choose\[
            C_{\alpha} \define 2C_{\mu_{\phi}} \myexp[\big]{ ( 1 + 2\abs{I'(\alpha)} ) \Cdistortion },
    \]
    where $C_{\mu_{\phi}} \geqslant 1$ is the constant from Proposition~\ref{prop:equilibrium state is gibbs measure} and $C_1 \geqslant 0$ is the constant defined in \eqref{eq:const:C_1} in Lemma~\ref{lem:distortion_lemma}.
\end{proposition}
\begin{proof}
    \def\tempconst{D(\potential)}
    We denote $\tempconst \define \Cdistortion$ and $C = 2C_{\mu_{\phi}}e^{\tempconst }$ in this proof.

    We first consider the case where $\alpha \in (\gamma_\phi, \maxenergy)$. Let $\alpha \in (\gamma_\phi, \maxenergy)$ be arbitrary.

    For each $0$-edge $e^0 \in \mathbf{E}^0(f,\mathcal{C})$, by Proposition~\ref{prop:LDA:key bounds}, there exists an integer $N_{e^0} \in \n$ depending only on $f$, $\mathcal{C}$, $d$, $\phi$, $\alpha$, and $e^0$ such that for each integer $n \geqslant N_{e^0}$, there exists a measure $\mu \in \mathcal{M}(S^2, f)$ satisfying \[
        \mu_{\phi} (P^n(\alpha)) \leqslant C e^{(P_{\mu}(f, \phi) - P(f,\phi))n} \quad \text{ and }\quad
        \int \! \phi \,\mathrm{d}\mu \in \bigl[ \alpha - 2 \tempconst n^{-1}, \maxenergy \bigr].
    \] 
    Since $\mathbf{E}^0(f,\mathcal{C})$ is a finite set, there exists an $0$-edge $\widetilde{e}^0 \in \mathbf{E}^0(f,\mathcal{C})$ such that \[
        N_{\widetilde{e}^0} = \min \bigl\{ N_{e^0} : e^0 \in \mathbf{E}^0(f,\mathcal{C}) \bigr\}.
    \]
    We fix such $0$-edge $\widetilde{e}^0$ and let \[
        N \define \max \{ N_{\widetilde{e}^0}, \lceil 2 \tempconst / (\alpha - \gamma_{\phi}) \rceil \},
    \]
    where $\lceil x \rceil$ denotes the smallest integer no less than $x$. Note that this integer $N$ satisfies $\gamma_{\phi} + \frac{2\tempconst }{N} \leqslant \alpha$ and depends only on $f$, $\mathcal{C}$, $d$, $\phi$, and $\alpha$.

    For each integer $n \geqslant N$, by the definition and properties of the rate function $I$ (see \eqref{eq:definition of rate function} and Proposition~\ref{prop:properties of rate function}), we have that
    \begin{equation*}
        \begin{split}
            \mu_{\phi} \parentheses[\big]{ \{x \in S^2 \describe n^{-1} S_n\phi(x) \geqslant \alpha \} } 
            &\leqslant \mu_{\phi}(P^n(\alpha))  \\
            &\leqslant C e^{(P_{\mu}(f, \phi) - P(f,\phi))n}  \\
            &\leqslant C \myexp[\big]{ \sup \set{ P_{\nu}(f,\phi) - P(f,\phi) \describe \nu \in \mathcal{A}_{n} } n }    \\
            &= C \myexp[\big]{  -\inf \set{ I(\eta) \describe \eta \in [ \alpha - 2 \tempconst n^{-1}, \maxenergy ] } n }  \\
            &= C \myexp[\big]{-I \parentheses[\big]{ \alpha - 2 \tempconst n^{-1} } n} \\
            &\leqslant C e^{2 \tempconst I'(\alpha)} e^{-I(\alpha)n},
        \end{split}
    \end{equation*}
    where $\mathcal{A}_{n} \define \set{\nu \in \invmea \describe \int \! \phi \,\mathrm{d}\nu \in [ \alpha - 2 \tempconst n^{-1}, \maxenergy ] }$.
    The last inequality is shown as follows. 
    By Taylor's formula, there exists $\theta \in \big[- 2\tempconst n^{-1}, 0\big]$ such that 
    \[
        I \bigl( \alpha - 2\tempconst n^{-1} \bigr) 
        = I(\alpha) - 2\tempconst n^{-1} I'(\alpha + \theta).
    \]
    Since $I$ is convex and $C^1$, $I'$ is increasing on $(\gamma_\phi, \maxenergy)$ and $I'(\alpha + \theta) \leqslant I'(\alpha)$. 
    Hence,
    \[
        I \bigl(\alpha - 2\tempconst n^{-1} \bigr) 
        \geqslant I(\alpha) - 2\tempconst n^{-1} I'(\alpha).
    \]
    By choosing $C_{\alpha} \define 2C_{\mu_{\phi}} \myexp[\big]{ ( 1 + 2I'(\alpha) ) \tempconst }$, we deduce that
    \begin{equation*}
        \mu_{\phi} \parentheses[\big]{ \set{ x \in S^2 \describe n^{-1} S_n\phi(x) \geqslant \alpha } }
        \leqslant C_{\alpha} e^{-I(\alpha)n}.
    \end{equation*}

    \smallskip

    For $\alpha\in (\minenergy,\gamma_\phi)$, the proof is similar. In this case, we choose $C_{\alpha} = 2C_{\mu_{\phi}} \myexp[\big]{ ( 1 - 2I'(\alpha) ) \tempconst }$.  
    
    \smallskip

    For $\alpha = \gamma_{\phi}$, the conclusion holds trivially since $C_{\gamma_{\phi}} \geqslant 1$ and $I(\gamma_{\phi}) = 0$. 
    This completes the proof.
\end{proof}

Now we can prove the Theorem~\ref{thm:large deviation asymptotics}. The key point of the proof is to iterate the map and apply Proposition~\ref{prop:large deviation asymptotics with invariant Jordan curve}.

\begin{proof}[Proof of Theorem~\ref{thm:large deviation asymptotics}]
    We first consider the case where $\alpha \in (\gamma_\phi,\maxenergy)$. 

    By Lemma~\ref{lem:invariant_Jordan_curve} we can find a sufficiently high iterate $F \define f^{K}$ of $f$ such that there exists an $F$-invariant Jordan curve $\mathcal{C} \subseteq S^{2}$ with $\post{F} = \post{f} \subseteq \mathcal{C}$. 
    Then $F$ is also an expanding Thurston map (recall Remark~\ref{rem:Expansion_is_independent}). 
    We fix such an integer $K$, a map $F$, and a Jordan curve $\mathcal{C}$.

    Let $\phi \in \holderspacesphere$ for some $\holderexp \in (0,1]$.
    Denote $\Phi \define S_{K}^{f} \phi$.
    Then $\Phi \in \holderspacesphere$ since $f$ is Lipschitz continuous with respect to $d$ (see \cite[Lemma~3.12]{li2018equilibrium}).
    Let $\mu_{F, \Phi}$ and $I_{F, \Phi}$ denote the unique equilibrium state and the rate function, respectively, for the map $F$ and the potential $\Phi$. 
    Recall from Subsection~\ref{sub:thermodynamic formalism} that $P(F, \Phi) = K P(f, \phi)$. 
    Then it follows from $P_{\mu_{\phi}}(F, \Phi) = KP_{\mu_{\phi}}(f, \phi)$ and the uniqueness of the equilibrium state that $\mu_{F, \phi} = \mu_{\phi}$. 
    Since $\potential$ is not co-homologous to a constant in $C(S^2)$ (with respect to $f$), Theorem~\ref{thm:properties of equilibrium state}~\ref{item:thm:properties of equilibrium state:co-homologous} implies that $\mu_{\potential} \ne \mu_{0}$, where $\mu_{0}$ denotes the measure of maximal entropy of $f$.
    Note that $\mu_{0}$ is also the measure of maximal entropy of $F$.
    Hence, it follows from Theorem~\ref{thm:properties of equilibrium state}~\ref{item:thm:properties of equilibrium state:co-homologous} that $\Phi$ is not co-homologous to a constant in $C(S^2)$ (with respect to $F$).

    We next show that $I_{F, \Phi}(K \widetilde{\alpha}) = KI(\widetilde{\alpha})$ for each $\widetilde{\alpha} \in (\minenergy, \maxenergy)$. 
    By Proposition~\ref{prop:properties of rate function}~\ref{item:prop:properties of rate function:expression of rate function} and Lemma~\ref{lem:derivative of pressure function related to Birkhoff average}, we have\[
        I(\widetilde{\alpha}) = P(f, \phi) - P(f, \xi(\widetilde{\alpha})\phi) + (\xi(\widetilde{\alpha}) - 1)\widetilde{\alpha}
    \]
    and
    \[
        I_{F, \Phi}(K \widetilde{\alpha}) = P(F, \Phi) - P(F, \xi_{F}(K \widetilde{\alpha})\Phi) + (\xi_{F}(K \widetilde{\alpha}) - 1)K \widetilde{\alpha},
    \]
    where $\xi(\widetilde{\alpha})$ and $\xi_{F}(K \widetilde{\alpha})$ are defined by\[
        \widetilde{\alpha} = \frac{\mathrm{d}P(f, t\phi)}{\mathrm{d}t} \bigg|_{t = \xi(\widetilde{\alpha})} \quad \text{ and } \quad 
        K\widetilde{\alpha} = \frac{\mathrm{d}P(F, t\Phi)}{\mathrm{d}t} \bigg|_{t = \xi_{F}(K\widetilde{\alpha})},
    \]
    respectively. Since $P(F, t\Phi) = K P(f, t\phi)$ for each $t \in \real$, we get\[
        \widetilde{\alpha} = \frac{1}{K}\frac{\mathrm{d}P(F, t\Phi)}{\mathrm{d}t} \bigg|_{t = \xi_{F}(K\widetilde{\alpha})}
        = \frac{\mathrm{d}P(f, t\phi)}{\mathrm{d}t} \bigg|_{t = \xi_{F}(K\widetilde{\alpha})}.
    \]
    Thus $\xi_{F}(K \widetilde{\alpha}) = \xi(\widetilde{\alpha})$ by the uniqueness of $\xi(\widetilde{\alpha})$ (see Lemma~\ref{lem:derivative of pressure function related to Birkhoff average}). Then it follows immediately from the expressions of $I(\widetilde{\alpha})$ and $I_{F, \Phi}(K \widetilde{\alpha})$ that $I_{F, \Phi}(K \widetilde{\alpha}) = KI(\widetilde{\alpha})$ and $I_{F, \Phi}'(K \widetilde{\alpha}) = I'(\widetilde{\alpha})$.

    Applying Proposition~\ref{prop:large deviation asymptotics with invariant Jordan curve}, we obtain the large deviation asymptotics for the map $F$ and the potential $\Phi$. 
    Thus for each $\widetilde{\alpha} \in [\gamma_\phi,\maxenergy)$ there exists an integer $M \in \n$ such that for each integer $m \geqslant M$, 
    \begin{equation}    \label{eq:large deviation asymptotics for iterate map}
        \mu_{\phi} \bigl( \bigl\{ x \in S^2 \describe m^{-1} S_m^{F}\Phi(x) \geqslant K \widetilde{\alpha} \bigr\} \bigr)
        \leqslant C_{K\widetilde{\alpha}} e^{-I_{F, \Phi}(K \widetilde{\alpha})m}
        = C_{K\widetilde{\alpha}} e^{- K I(\widetilde{\alpha})m},
    \end{equation}
    where $C_{K\widetilde{\alpha}} = 2C_{\mu_{\potential}} \myexp[\big]{ D(\Phi)( 1 + 2I_{F, \Phi}'(K \widetilde{\alpha}) ) } = 2C_{\mu_{\phi}} \myexp[\big]{ D(\Phi) ( 1 + 2I'(\widetilde{\alpha}) )}$.
    Here $C_{\mu_{\potential}}$ is the constant from Proposition~\ref{prop:equilibrium state is gibbs measure} and the constant $D(\Phi)$ dependents only on $F$, $\mathcal{C}$, $d$, $\Phi$, and $\holderexp$.
    We will derive the large deviation asymptotics for $f$ and $\phi$ from \eqref{eq:large deviation asymptotics for iterate map}. 
    Indeed, for each integer $m \geqslant M$ and each $k \in \{0,\, 1\, \dots, \, K- 1\}$, we have
    \begin{equation}    \label{eq:set inequality from iterate map to original map}
        \begin{split}
            &\bigl\{ x \in S^2 \describe S_{mK + k}^{f}\phi(x) \geqslant (mK + k) \alpha \bigr\} \\
            &\qquad \subseteq \bigl\{ x \in S^2 \describe S_{mK}^{f}\phi(x) \geqslant (mK + k) \alpha - k \uniformnorm{\phi} \bigr\} \\
            &\qquad \subseteq \bigl\{ x \in S^2 \describe S_{mK}^{f}\phi(x) \geqslant mK \alpha - 2K\uniformnorm{\phi} \bigr\} \\
            &\qquad = \bigl\{ x \in S^2 \describe m^{-1} S_m^{F}\Phi(x) \geqslant K \bigl( \alpha - 2 m^{-1} \uniformnorm{\phi} \bigr) \bigr\}.
        \end{split}
    \end{equation}
    Put
    \begin{equation}    \label{eq:sufficiently large constant N}
        N \define K \max \{ M, \left\lceil 2 \uniformnorm{\phi} / ( \alpha - \gamma_{\phi} ) \right\rceil \}.
    \end{equation}
    Note that the integer $N$ satisfies $\alpha - \frac{2\uniformnorm{\phi}}{N/K} \geqslant \gamma_{\phi}$. 

    Recall that $\alpha \in (\gamma_\phi,\maxenergy)$.
    For each integer $n \geqslant N$, we can write $n = mK + k$ for some integers $k \in \zeroton[K - 1]$ and $m \geqslant M$ satisfying $\alpha - \frac{2\uniformnorm{\phi}}{m} \geqslant \gamma_{\phi}$. Then by \eqref{eq:set inequality from iterate map to original map} and \eqref{eq:large deviation asymptotics for iterate map}, we get the first two inequalities of the following:
    \begin{equation}     \label{eq:estimate of deviation set from iterate}
        \begin{split}
            &\mu_{\phi}\bigl( \bigl\{ x \in S^2 \describe n^{-1} S_n^{f}\phi(x) \geqslant \alpha  \bigr\} \bigr) \\
            &\qquad = \mu_{\phi} \bigl( \bigl\{ x \in S^2 \describe S_{mK + k}^{f}\phi(x) \geqslant (mK + k) \alpha \bigr\}  \bigr)  \\
            &\qquad \leqslant \mu_{\phi} \bigl( \bigl\{ x \in S^2 \describe m^{-1} S_m^{F}\Phi(x) \geqslant K \bigl( \alpha - 2 m^{-1} \uniformnorm{\phi} \bigr) \bigr\} \bigr)  \\
            &\qquad \leqslant 2C_{\mu_{\phi}} \myexp[\big]{ D(\Phi) ( 1 + 2 I'( \alpha - 2 m^{-1} \uniformnorm{\phi} ) )} \,  e^{- m K I \parentheses[\big]{ \alpha - 2 m^{-1} \uniformnorm{\phi} } }  \\
            &\qquad \leqslant C_{\alpha} e^{-nI(\alpha)},
        \end{split}
    \end{equation}
    where $C_{\alpha} \define 2C_{\mu_{\phi}} \myexp[\big]{ D(\Phi) ( 1 + 2I'(\alpha) ) + 2K\uniformnorm{\phi}I'(\alpha) + KI(\alpha)}$.
    The last inequality in \eqref{eq:estimate of deviation set from iterate} is shown as follows. By Taylor's formula, there exists $\theta \in \bigl[ -\frac{2\uniformnorm{\phi}}{m}, 0 \bigr]$ such that\[
        I \bigl( \alpha - 2 m^{-1} \uniformnorm{\phi} \bigr) = I(\alpha) - 2 m^{-1} \uniformnorm{\phi} I'(\alpha + \theta).
    \]
    Since $I$ is convex and $C^1$, $I'$ is increasing on $(\gamma_\phi, \maxenergy)$ and $I'(\alpha + \theta) \leqslant I'(\alpha)$. 
    Hence,
    \[
        I \bigl( \alpha - 2 m^{-1} \uniformnorm{\phi} \bigr) 
        \geqslant I(\alpha) - 2 m^{-1} \uniformnorm{\phi} I'(\alpha).
    \]
    Then \[
        e^{- mK I \bigl( \alpha - 2 m^{-1} \uniformnorm{\phi} \bigr)}
        \leqslant e^{-mKI(\alpha) + 2K\uniformnorm{\phi}I'(\alpha) } 
        = e^{2K\uniformnorm{\phi}I'(\alpha) + kI(\alpha)} e^{-nI(\alpha)}  
        \leqslant e^{2K\uniformnorm{\phi}I'(\alpha) + KI(\alpha)} e^{-nI(\alpha)}.
    \]
    Since $I' \bigl(\alpha - 2 m^{-1} \uniformnorm{\phi} \bigr) \leqslant I'(\alpha)$, we obtain the last inequality in \eqref{eq:estimate of deviation set from iterate}. 
    Therefore, we conclude that for each $\alpha \in (\gamma_{\phi}, \maxenergy)$, there exist an integer $N \in \n$ (given by \eqref{eq:sufficiently large constant N}) and a constant $C_{\alpha} > 0$ such that for each integer $n \geqslant N$, \[
        \mu_{\phi}\bigl( \bigl\{ x \in S^2 \describe n^{-1} S_n^{f}\phi(x) \geqslant \alpha  \bigr\} \bigr) \leqslant C_{\alpha} e^{-I(\alpha)n}.
    \]

    \smallskip

    For $\alpha\in (\minenergy,\gamma_\phi)$, the proof is similar.

    \smallskip

    For $\alpha = \gamma_{\phi}$, the conclusion holds trivially since $I(\gamma_{\phi}) = 0$. 
    This completes the proof.
\end{proof} 


\section*{Acknowledgements}%
\label{sec:Acknowledgements}

The authors are grateful to the anonymous referees for their meticulous reading of the manuscript and the valuable comments.
Z.~Li and X.~Shi were partially supported by NSFC Nos.~12101017, 12090010, 12090015, 12471083, and BJNSF No.~1214021. \
Y.~Zhang was partially supported by NSFC Nos.~12161141002 and~12271432. 

\printbibliography

\end{document}